\documentclass[11pt,reqno]{amsart}
\usepackage{geometry} 

\usepackage{amsmath,amssymb,amsthm, mathrsfs,appendix}
\usepackage{xcolor}
\usepackage{graphicx} 
\usepackage{tikz}
\usetikzlibrary{positioning} 
\usepackage{pgfplots}
\allowdisplaybreaks

\usepackage{marvosym}
\usepackage{cancel}

\usepackage{latexsym}
\usepackage[nobysame]{amsrefs}[11pts]
\usepackage{bm}
\usepackage{enumerate, enumitem}
\usepackage{indentfirst}
\usepackage[colorlinks,
           linkcolor=black,
          anchorcolor=black,
         citecolor=black
        ]{hyperref}

\usepackage[left, pagewise]{lineno}

\newcommand{\dy}{\mathrm{d}y}

\newcommand{\lss}{c}
\newcommand{\Clss}{Q}
\newcommand{\subclss}{\zeta}
\newcommand{\dive}{\mathrm{div}}
\newcommand{\smallc}{\nu}
\newcommand{\md}{\mathbb{D}}

\newtheorem{Theorem}{Theorem}[section]
\newtheorem{Lemma}{Lemma}[section]
\newtheorem{Proposition}{Proposition}
\newtheorem{Remark}{Remark}

\numberwithin{Remark}{section}
\numberwithin{Proposition}{section}
\numberwithin{Definition}{section}
\numberwithin{Lemma}{section}
\numberwithin{equation}{section}
\numberwithin{Theorem}{section}

\begin{document}
\title[Degenetate compressible Navier-Stokes equations]{Development of Implosions of Solutions \\
to the Three-Dimensional  Degenerate Compressible Navier-Stokes Equations}
\date{\today} 

\author{Gui-Qiang G. Chen}
\address[Gui-Qiang G. Chen]{Mathematical Institute, University of Oxford, Oxford, OX2 6GG, UK.} \email{\tt gui-qiang.chen@maths.ox.ac.uk}

\author{Lihui Liu}
\address[Lihui Liu]{School of Mathematical Sciences, Shanghai Jiao Tong University, Shanghai, 200240, China; Mathematical Institute, University of Oxford, Oxford, OX2 6GG, UK.}
\email{\tt lzy1604@sjtu.edu.cn}

\author{Shengguo Zhu }
\address[Shengguo Zhu]{School of Mathematical Sciences, CMA-Shanghai  and MOE-LSC,  Shanghai Jiao Tong University, Shanghai 200240, P. R. China}
 \email{\tt  zhushengguo@sjtu.edu.cn}

\begin{abstract}
A fundamental open problem in the theory of the multidimensional compressible Navier–Stokes equations 
is whether smooth solutions can develop singularities---such as cavitation or implosion---in finite time. 
In the case of constant viscosity coefficients, recent remarkable results show that 
there exist some smooth initial data for which the corresponding smooth solutions 
of the barotropic flow undergo finite-time implosion at the origin in three spatial dimensions, 
with the density $\rho$ blowing up to infinity  (see \cites{MPI, buckmaster, shao, shijia}).
On the other hand, when the viscosity coefficients depend linearly 
on $\rho$ (as in the shallow water equations), it has also been established recently that, 
for general large spherically symmetric initial data (even allowing vacuum states),
the solutions of the barotropic flow remain globally regular in two and three spatial 
dimensions (see \cites{CZZ, CZZ1}).
These observations indicate that the behavior of the  multi-dimensional solutions depends sensitively on the structures of the viscosity coefficients, thereby complicating the analysis of singularity formation.
In this paper, we analyze the case of
nonlinear density-dependent viscosity coefficients (\textit{i.e.}, $\rho^\delta$ with an exponent $\delta>0$). 
We identify a value $\delta^*(\gamma)<\frac12$, depending on the adiabatic exponent $\gamma$, such that, 
for every $0<\delta< \delta^*(\gamma)$, 
there exists a class of smooth initial data for which the corresponding  smooth solutions 
of the barotropic flow implode (with infinite density) in finite time at the origin in three spatial dimensions. 
Moreover, we assume that the initial density is strictly positive, 
in order to rule out 
the possibility that the corresponding implosion is an artifact of the vacuum.
The key point is to prove that, in this regime, the degenerate viscous terms are not strong enough 
to suppress the convective mechanism driving the implosion. 
However, verifying this issue is challenging due to fundamental difficulties 
arising from the degenerate structure. 
Specifically, for flows with constant viscosity coefficients (as in \cites{MPI, buckmaster, shao, shijia}), the dissipative effect in the momentum equation remains spatially uniform. In contrast, in our setting, the dissipative effect increases significantly in regions of large density. 
Therefore, one would expect the viscous terms to suppress any tendency toward implosion (just as in \cites{CZZ,CZZ1}).
To overcome this obstacle,  we first provide a pointwise estimate of $\rho$ by a region segmentation method, and then establish the spatial decay estimates for the gradient of the velocity via carefully designed weighted estimates for high-order derivatives  and interpolation inequalities. 
The corresponding decay rate is fast enough to absorb the singularity of $\rho$  and then leads to a uniform-in-time control of the viscous terms.
\end{abstract}

\date{March 23, 2026}
\subjclass[2020]{Primary: 35A01, 35B40; Secondary: 35B65, 35A09} 
\keywords{Compressible Navier-Stokes equations,   Degenerate viscosity, Three-Dimensions, Smooth solutions, Singularity formation, Implosion, Cauchy problem}.

\maketitle

\tableofcontents

\section{Introduction}\label{Section1}
We are concerned with the formation of singularities in compressible viscous flows 
and prove that there exists a class of smooth initial data for which the corresponding 
smooth solutions of the Cauchy problem for the three-dimensional (3-D) compressible 
Navier–Stokes equations (\textbf{CNS}) undergo finite-time implosion 
within a physical parameter regime.
More specifically, we consider the motion of a barotropic, compressible, viscous Newtonian polytropic 
fluid in  $\mathbb{R}^3$, governed by the following \textbf{CNS}{\rm:}
\begin{equation}\label{eq:1.1}
\begin{cases}
\partial_t\rho+\dive_x(\rho u)=0,\\[6pt]
\partial_t(\rho u)+\dive_x(\rho u\otimes u)
+\nabla_x P =\dive_x \mathbb{S},
\end{cases}
\end{equation}
and  the initial and far-field conditions:
\begin{equation}\label{eqs:CauchyInit}
\begin{split}
&(\rho,u)(0, x)=(\rho_0,u_0)(x)
\qquad\,\, \text{for $x \in \mathbb{R}^3$},\\[3pt]
\displaystyle
&(\rho,u)(t, x)
\to (\bar\rho,0)
\qquad \ \ \ \ \quad\, \text{as $|x|\to \infty$ for $t\ge 0$}.
\end{split}
\end{equation}
Here, $t\geq 0$ denotes the time, 
$x=(x_1, x_2, x_3)^{\top}\in \mathbb{R}^3$ is the Eulerian spatial coordinates, 
$\rho\geq 0$ is the fluid density, $\bar\rho>0$ is a constant,   $u=(u_1, u_2, u_3)^\top$ $\in \mathbb{R}^3$ is  
the fluid velocity, and $P$ denotes the fluid pressure. For polytropic gases, the constitutive relation is given by
\begin{equation}
\label{eq:1.2}
P=A\rho^{\gamma}, 
\end{equation}
where $A>0$ is  the entropy  constant and  $\gamma>1$ is the adiabatic exponent. 
Without loss of generality, by scaling, we may take $A=\frac{1}{\gamma}$.
The viscous stress tensor $\mathbb{S}$ takes the form{\rm:}
\begin{equation}\label{eq:1.1t}
\mathbb{S}=2\mu(\rho)D_x(u)+\lambda(\rho)\,\dive_x u\,\mathbb{I}_3,
\end{equation} 
where $D_x(u)=\frac{1}{2}\big(\nabla_x u+(\nabla_x u)^\top\big)$ is the deformation tensor, $\mathbb{I}_3$ is the $3 \times 3$ identity matrix,
\begin{equation}\label{fandan}
\mu(\rho)=a_1  \rho^\delta,\quad \lambda(\rho)=a_2 \rho^\delta\qquad\,\,\mbox{for some  constant $\delta\geq 0$},
\end{equation}
$\mu(\rho)$ is the shear viscosity coefficient, 
$\lambda(\rho)+\frac{2}{3}\mu(\rho)$ is the bulk viscosity coefficient,  
and $a_1$ and $a_2$ are both constants satisfying
\begin{equation}\label{viscoeff}
a_1 >0, \qquad 2a_1+3a_2 \geq 0. 
\end{equation}

For rarefied gases, the full \textbf{CNS}  can be formally derived from the Boltzmann equation 
via the Chapman--Enskog expansion (see Chapman--Cowling \cite{chap} and Li--Qin \cite{tlt}). 
Under proper physical assumptions,  the viscosity coefficients $(\mu,\lambda)$ and the thermal conductivity coefficient $\kappa$ are all functions 
of the absolute temperature $\vartheta$. In fact,  
for the cutoff inverse power force models, if the intermolecular potential varies as $\ell^{-\varkappa}$,
where $\ell$ is the intermolecular distance and $\varkappa$ is a positive constant, then   
\begin{equation}\label{eq:1.6g}
\mu(\vartheta)=b_1 \vartheta^{\frac{1}{2}+ b},\quad \lambda(\vartheta)=b_2 \vartheta^{\frac{1}{2}+ b}, 
\quad  \kappa(\vartheta)=b_3 \vartheta^{\frac{1}{2}+ b}\qquad\,\, \text{with}\,\,\,  b=\frac{2}{\varkappa} \in [0,\infty),
\end{equation}
for some constants $b_i$, $i=1,2,3$.
In particular, for the ionized gas,
$\varkappa=1$;
for Maxwellian molecules,
$\varkappa=4$;
while $\varkappa=\infty$ for rigid elastic spherical molecules (see \S 10 of \cite{chap}).
For barotropic and polytropic fluids, such a dependence is inherited through the laws of Boyle and Gay-Lussac{\rm:}
\begin{equation*}
P=\hat{A}\rho \vartheta=A\rho^\gamma \qquad \text{for a constant } \hat{A}>0,
\end{equation*}
\textit{i.e.}, $\vartheta=A\hat{A}^{-1}\rho^{\gamma-1}$, 
and  the viscosity coefficients become functions of $\rho$ taking form $\eqref{fandan}$; see Liu--Xin--Yang \cite{taiping}.  
Note that there exist other physical models satisfying the density-dependent viscosity assumption \eqref{fandan}, such as the Korteweg system, the shallow water equations, the lake equations, the quantum Navier-Stokes system, among others; see \cites{Dunn,gpm,jun,Kort, Mar,Gent}.

The local well-posedness of strong solutions of the Cauchy problem for \textbf{CNS} has been thoroughly studied.
When  $\inf_x {\rho_0(x)}>0$, the local well-posedness of smooth solutions in the inhomogeneous Hilbert spaces $H^s$ ($s>\frac{n}{2}+1$)  follows from the standard symmetric hyperbolic-parabolic structure that 
satisfies the well-known Kawashima's condition (\textit{cf.} Kawashima \cite{KA} and Nash \cite{nash}).  
However, such approaches fail in the presence of vacuum due to the possible degeneracies 
from the time evolution and spatial dissipation as follows{\rm:}
\begin{equation}\label{dege}
\displaystyle
 \underbrace{\rho(\partial_t u+u\cdot \nabla_x u)}_{Degenerate   \ time \ evolution}+\nabla_x P= \underbrace{\text{div}(2\mu(\rho)D_x(u)+\lambda(\rho)\,\dive_x u \,\mathbb{I}_n)}_{Degenerate\ spatial  \ dissipation}.
\end{equation}
Generally, a vacuum appears in the far-field under some physical requirements such as a finite total mass in $ \mathbb{R}^n$. 
For the multidimensional   (M-D) \textbf{CNS} with constant viscosity coefficients ($\delta=0$ in (\ref{fandan})),  by introducing some compatibility conditions, the local well-posedness of smooth solutions 
with vacuum in the homogeneous Sobolev spaces has been established successfully; 
see Salvi--Stra\v skraba \cite{Salvi} and Cho--Kim  \cite{guahu}, among others. 
On the other hand,  for M-D \textbf{CNS} with degenerate viscosity coefficients ($\delta>0$ in (\ref{fandan})),
by introducing some enlarged reformulations to deal with the degenerate spatial dissipation,   
the local well-posedness of smooth solutions with vacuum in the inhomogeneous Sobolev spaces 
has been established by Zhu \cite{zhuthesis}, Li--Pan--Zhu \cites{sz3,sz333}, Xin--Zhu \cite{zz2}, and the references cited therein. 

Another fundamental problem in the theory of fluid dynamics is the question of global regularity in $\mathbb{R}^n$ ($n\geq 2$). 
Under proper smallness assumptions on the initial data, 
some important progress has been made on the global well-posedness of smooth solutions 
to the M-D \textbf{CNS};
see Matsumura--Nishida \cite{MN}, Danchin \cite{danchin1}, and Huang--Li--Xin \cite{HX1}  for flows with constant viscosity coefficients, and Sundbye \cite{Sundbye2}, Wang--Xu \cite{weike}, and Xin--Zhu \cite{zz} for flows with degenerate viscosity coefficients. 
On the other hand, for the general smooth initial data, very few results are available concerning
the global well-posedness of smooth solutions to the M-D \textbf{CNS}. 
Recently, when the viscosity coefficients depend linearly on the mass 
density $\rho$ (as in the shallow water equations, \textit{i.e.}, the viscous Saint-Venant system 
with $\delta=1$ and $a_2=0$ in \eqref{fandan}), 
Chen--Zhang--Zhu \cite{CZZ} have established global well-posedness of spherically symmetric smooth solutions of the Cauchy problem in $\mathbb{R}^n$ ($n=2$ or $3$) for the following two classes of initial density profiles{\rm:}
 \begin{enumerate}[leftmargin=2.5em,
  align=left,
  labelsep=0.2em,
  labelwidth=1.1em]
        \item[\rm{(i)}] Strictly positive density in the whole space $\mathbb{R}^n$:
        $$
        \inf_{x\in \mathbb{R}^n} \rho_0(x)>0;
        $$
        \item[\rm{(ii)}] Positive density decaying at infinity:
        $$   
        \rho_0(x) >0 \quad \text{for $x\in \mathbb{R}^n$},  \qquad\quad
        \rho_0(x) \rightarrow 0 
        \quad \text{as $|x|\rightarrow \infty$}.      
        $$
    \end{enumerate}
Furthermore, for the degenerate \textbf{CNS} \eqref{eq:1.1} with $\delta=1$ in \eqref{fandan},
even when the initial density vanishes in some open set, under the setting of the vacuum free boundary problem,  Chen--Zhang--Zhu have proved in the recent work \cite{CZZ1} that  
the solutions with large data of spherical symmetry still remain globally regular and unique 
in two and three spatial dimensions. 
It is worth emphasizing that, for the global-in-time well-posedness established 
in \cite{CZZ1}, the physical vacuum boundary condition is allowed for the class of initial
data considered, and the corresponding solutions remain smooth all the way up 
to and including the moving vacuum boundary.
For the general M-D problem with large initial data, 
although many important results on the global existence of weak solutions 
of the Cauchy problem have been obtained (see \cites{bvy,fu1, fu3,JZ,lz, lions, vayu}), 
the issue of uniqueness remains widely open due to the limited regularity of these weak solutions.

We remark
that there are several results concerning the singularity formation for the M-D \textbf{CNS}.
For the flow with constant viscosity coefficients,  Xin \cite{zx} showed that any smooth solutions in the inhomogeneous Sobolev spaces to the non-isentropic \textbf{CNS} with zero thermal conductivity cannot exist globally in time, since it may blow up in finite time, provided that
the density is compactly supported, which holds regardless of the size of the initial data. 
Subsequently, Rozanova \cite{olga} showed that the smooth solutions of the Cauchy problem for 
\textbf{CNS}  in three or higher spatial dimensions, with conserved total mass, finite total energy, and finite moment of inertia, lose their initial smoothness in finite time, 
even when the initial density is not compactly supported. 
On the other hand, it remains unclear whether the class of solutions considered 
in \cites{olga, zx} admits the local-in-time existence in $\mathbb{R}^n$ ($n\geq 2$).
For the barotropic \textbf{CNS} with degenerate viscosity coefficients ($\delta>1$  in \eqref{fandan}), it was shown in \cites{sz333,zz} that,  
for certain classes of 3-D initial data with vacuum in some open set, one can construct corresponding local smooth solutions in the inhomogeneous Sobolev spaces that 
break down in finite time, regardless of the size of the initial data. 

These observations indicate that, in the presence of vacuum, 
under the setting of the Cauchy problem, 
viscous effects are insufficient to prevent the singularity formation of 
classical solutions to \textbf{CNS} for the cases   $\delta=0$ and $\delta>1$ in \eqref{fandan}. However, the results obtained in \cites{olga,zx,sz333,zz} 
show that the lifespan of the corresponding classical solutions is finite, 
typically via contradiction arguments, 
and thus do not provide a detailed description of the underlying blow-up mechanism.
If the viscous effects are formally neglected in compressible flow, 
the motion of the corresponding barotropic inviscid flow is governed by 
the compressible Euler equations (system \eqref{eq:1.1} with $a_1=a_2=0$).
For this inviscid system, the mechanism of singularity formation has been studied extensively and deeply.
In particular,  significant progress has been made over the past decades
on the shock formation;
see, for example, the works of Lax \cite{lax}, Lebaud \cite{lebaud},  
Sideris \cite{sideris}, 
Chen--Dong \cite{chen-dong}, Yin \cite{yin},
Christodoulou \cite{Christodoulou}, Luk--Speck \cite{Luk}, Christodoulou--Miao \cite{ChristodoulouMiao},
Chen--Chen--Zhu \cite{gengzhu1}, Buckmaster--Shkoller--Vicol \cites{bsv1, bsv2}, 
Shkoller--Vicol \cite{sv}, and the references cited therein. 

In addition to the shock formation, another mechanism of singularity formation in 
compressible flows, known as implosion, has been observed. 
In this scenario, either the density or the velocity becomes unbounded at the blow-up time.
Recently, Merle--Rapha\"el--Rodnianski--Szeftel \cite{merle1} constructed $C^\infty$ smooth, 
spherically symmetric imploding solutions to the barotropic compressible Euler equations 
for $\gamma>1$ in $\mathbb{R}^3$, 
except for at most countably many values of $\gamma$, 
based on a self-similar reformulation to the equations. 
Specifically, denoting $\alpha=\frac{\gamma-1}{2}$ and introducing the variable
$$
\lss=\alpha^{-1}\rho^{\alpha},
$$
the Euler equations (system \eqref{eq:1.1} with $a_1=a_2=0$) can be rewritten  as
\begin{equation}\label{eq:1.21}
\begin{cases}
\partial_t \lss=-u\cdot \nabla_x \lss -\alpha \lss \,\dive_x u,\\[6pt]
 \partial_t u = -u\cdot \nabla_x u -\alpha \lss \nabla_x \lss.
\end{cases}
\end{equation}
Then, for fixed constants  $T>0$ and $\Lambda>1$, one can  introduce  the vector function $(\Clss, U)$
by
\begin{equation}\label{scaling1}
\begin{aligned}
    \lss(t,x)& = \frac{(T-t)^{\frac{1}{\Lambda}-1}}{\Lambda} \Clss (-\frac{\log (T-t)}{\Lambda}, \frac{x}{(T-t)^{\frac{1}{\Lambda}}} ),\\
    u(t,x)&= \frac{(T-t)^{\frac{1}{\Lambda}-1}}{\Lambda}U(-\frac{\log (T-t)}{\Lambda}, \frac{x}{(T-t)^{\frac{1}{\Lambda}}} ),
\end{aligned}    
\end{equation}
 and  the new self-similar time and space coordinates $(\tau, y)${\rm:}
\begin{equation}\label{scaling-coordinat1}
    \tau=-\frac{\log (T-t)}{\Lambda},\quad y= \frac{x}{(T-t)^{\frac{1}{\Lambda}}}=e^{\tau}x,
\end{equation}
with the initial time $\tau_0=-\frac{\log T}{\Lambda}>0$,
where $y=(y_1, y_2, y_3)^{\top}\in \mathbb{R}^3$ and $\tau\in [\tau_0,\infty)$.
In the  self-similar coordinates $(\tau, y)$, the compressible Euler equations become
\begin{equation}\label{selfsimilar eq1}
    \begin{cases}
    \displaystyle
        \partial_\tau \Clss =-(\Lambda-1)\Clss -(y+U)\cdot\nabla_y \Clss - \alpha \Clss \,\dive_y U,\\[6pt]
        \displaystyle
        \partial_\tau U = -(\Lambda-1)U-(y+U)\cdot\nabla_y U- \alpha \Clss\, \nabla_y \Clss.
    \end{cases}
\end{equation}
In \cite{merle1}, for   $\gamma \in (1,\, \infty )$ except for at most countably many points, there exists a discrete sequence   $\{\Lambda_n\}_{n=1}^\infty$ satisfying 
\begin{equation}\label{range of Lambda}
    1<\Lambda_n <\Lambda^*(\gamma), \qquad\,\, \lim_{n\rightarrow \infty} \Lambda_n=\Lambda^*(\gamma),
    \end{equation}
with 
\begin{equation}\label{def:Lambda*}
    \Lambda^*(\gamma)=\begin{cases}
        1+\frac{2}{(1+\sqrt{\frac{2}{\gamma-1}})^2} & \,\,\text{for} \,\, 1<\gamma <\frac{5}{3},\\[6pt]
        \frac{3\gamma-1}{2+\sqrt{3}(\gamma-1)} & \,\,\text{for} \,\, \gamma\geq \frac{5}{3},
    \end{cases}
\end{equation}
such that, for every $\Lambda_n$,  nontrivial $C^\infty$ smooth spherically symmetric profiles $(\overline \Clss_n, \overline U_n)$ were constructed 
by solving 
\begin{equation}\label{Profilen}
    \begin{cases}
        (\Lambda_n-1)\overline \Clss_n +(y+\overline U_n)\cdot\nabla_y \overline \Clss_n + \alpha \overline \Clss_n \, \dive_y \overline U_n =0,\\[6pt]
        (\Lambda_n-1)\overline U_n+(y+\overline U_n)\cdot\nabla_y \overline U_n+ \alpha \overline \Clss_n\, \nabla_y \overline \Clss_n =0.
    \end{cases}
\end{equation}
These profiles are regular at the origin and decay at spatial infinity. In
particular,
\[
\overline{\Clss}_n(0)=\Clss_*>0,\qquad  \overline U_n(0)=0,\
\qquad (\overline{\Clss}_n,\overline U_n)(y)\to (0,0)\quad \text{as $|y|\to\infty$}.
\]
The existence of such self-similar profiles implies that the corresponding $C^\infty$ solutions to the compressible Euler equations \eqref{eq:1.21} 
undergo finite-time implosion. For the 3-D barotropic \textbf{CNS} with constant viscosity coefficients 
and spherical symmetry, based on the existence of the $C^\infty$ smooth self-similar solutions 
to the compressible Euler equations  in \cite{merle1}, 
 Merle--Rapha\"el--Rodnianski--Szeftel \cite{MPI} proved  that, for   $\gamma \in \big(1,\, 1+\frac{2}{\sqrt{3}}\big)$ except for at most countably many points,  there exists a class of smooth finite-energy initial data with far-field vacuum such that
 the corresponding solutions undergo finite-time implosion, with the density blowing up to infinity.

It is worth noting that the case $\gamma = \frac{5}{3}$, corresponding to a monatomic gas,  
 was excluded in \cites{merle1,MPI} 
due to a triple-point degeneracy in the underlying analysis. 
This gap was later filled by Shao--Wang--Wei--Zhang \cite{shao} by introducing a novel renormalization
for the autonomous system \eqref{atuosys} in Appendix~\ref{appendix B}. In particular, it was proved in \cite{shao} that, for $\gamma=\frac{5}{3}$,
the sequence $\{\Lambda_n\}_{n=1}^{\infty}$ described in \eqref{range of Lambda}--\eqref{Profilen}  
also exists, which can be used to show the finite-time blow-up  of solutions of the 3-D  barotropic  \textbf{CNS}.
Moreover, for 3-D barotropic flow, Buckmaster--Cao-Labora--Gómez-Serrano in \cite{buckmaster}
constructed smooth self-similar imploding solutions to the compressible Euler equations for all $\gamma>1$,
and subsequently proved the finite-time blow-up of solutions of \textbf{CNS} with strictly positive initial density 
and spherical symmetry  when $\gamma=\frac{7}{5}$ (corresponding to a diatomic gas).
On the other hand,  for the 3-D barotropic \textbf{CNS} without any symmetry assumption, 
 Cao-Labora--Gómez-Serrano--Shi--Staffilani \cite{shijia}  constructed some smooth solutions 
 that remain strictly away from vacuum and nevertheless develop an imploding finite-time 
 singularity in $\mathbb{T}^3$ or $\mathbb{R}^3$. 
 These developments indicate that, although viscosity generally has a regularizing effect, 
 the self-similar blow-up mechanisms identified in the inviscid setting 
 may still persist or influence the dynamics in weakly viscous regimes.

These results naturally raise the question of whether similar finite-time singularities persist
in the physically important case that the viscosity coefficients depend on the density 
in a power law  (\textit{i.e.}, the form: $\rho^\delta$ with an exponent $\delta>0$). 
Indeed, when vacuum appears in some open set, finite-time blow-up phenomena have been observed 
for the case $\delta>1$ in \eqref{fandan} (see \cites{sz333,zz}).
However, from a physical standpoint, the concept of fluid velocity itself loses meaning in the region where the density vanishes. 
In the derivation of hydrodynamic equations from first principles, 
a fundamental underlying assumption is that 
 the fluid is non-dilute and can be described as a continuum. This continuum hypothesis fails in vacuum regions, where the hydrodynamic 
description no longer applies to the evolution of thermodynamic states.
Consequently, the dynamics within vacuum regions cannot be meaningfully analyzed using classical hydrodynamic models. 
Based on the above considerations, in the setting of the Cauchy problem,  it is natural to expect the
viscous system may fail to be globally well-posed when vacuum is present in an open set at the initial time. 
From this viewpoint, the finite-time blow-up results obtained for the Cauchy problem
in \cites{sz333, zz} for the case $\delta>1$ in \eqref{fandan} 
are perhaps not unexpected.
What is of genuine interest, however, is the regime in which the initial data remain strictly away from the vacuum in the domain under consideration.
In contrast to the case $\delta>1$, 
for the degenerate \textbf{CNS} \eqref{eq:1.1} with $\delta=1$ in \eqref{fandan} 
in two and three spatial dimensions, as mentioned above, no matter for the Cauchy problems that 
are strictly away from the vacuum or  allow the far-field vacuum in \cite{CZZ}, 
or the vacuum free boundary problem that contains the physical vacuum in \cite{CZZ1}, 
the corresponding smooth solutions with large initial data of spherical symmetry remain 
globally regular and unique. 
These results demonstrate that the qualitative behavior of solutions 
to the M-D degenerate \textbf{CNS} depends sensitively on the viscosity exponent $\delta$, 
thereby substantially complicating the analysis of singularity formation 
in the density-dependent viscosity setting.

In this paper, we study the Cauchy problem for the M-D \textbf{CNS} with the viscosity coefficients determined by \eqref{fandan} involving $\rho^\delta$ with $\delta>0$. 
A fundamental difficulty in this problem arises from the fact that the dissipative effects become 
increasingly strong in the regions of large density. From a formal perspective, 
one would therefore expect the dissipative terms to dominate the dynamics 
and suppress any tendency toward implosion.

Beyond this physical intuition, the rigorous analysis presents substantial mathematical challenges. 
After performing an Euler-type self-similar transformation to \textbf{CNS}, 
the structure of the equations changes fundamentally.
In particular, the original viscous terms are transformed into a degenerate nonlinear elliptic operator, together with additional coupling terms between the velocity field and the density.
As a result, the standard energy methods and uniform elliptic estimates are no longer directly 
applicable.
To overcome these difficulties, the analysis requires a delicate interplay of 
carefully designed weighted estimates across different spatial regimes.

The main contribution of this paper is to show that the above intuition 
fails in a certain physical parameter regime.
More precisely, when the viscosity exponent $\delta$ is smaller than a threshold value 
$\delta^*(\gamma)$, 
depending on the adiabatic exponent $\gamma$, the nonlinear viscous terms are not strong enough 
to counterbalance the convective effects, leading to finite-time implosion. 
Our main result on  the   Cauchy problem  \eqref{eq:1.1}{\rm--}\eqref{viscoeff}
can be stated as follows:

\begin{Theorem}\label{Thm1.1}  Let  $(\overline{\Clss}, \overline{U})$ be  the $C^\infty$  self-similar profile 
solving 
 \begin{equation}\label{Profile}
 \begin{cases}
        (\Lambda-1)\overline \Clss +(y+\overline U)\cdot\nabla_y \overline \Clss + \alpha \overline \Clss \, \dive_y \overline U =0,\\[6pt]
        (\Lambda-1)\overline U+(y+\overline U)\cdot\nabla_y \overline U+ \alpha \overline \Clss\, \nabla_y \overline \Clss =0,
    \end{cases}
\end{equation}
which is obtained in {\rm Lemma \ref{lem:existofselfsimilar}} {\rm(}see {\rm Appendix \ref{appendix B}}{\rm)}, and let 
  $\delta^*(\gamma)$ be   a constant given by
\begin{equation}\label{def:delta*}
\begin{aligned}
\delta^*(\gamma)=\begin{cases}
     \frac{\gamma+1}{4}-\frac{\sqrt{2(\gamma-1)}}{2} &\text{for}\,\,\,\, 1 < \gamma <\frac{5}{3},\\[6pt]
     \frac{1-(2\sqrt{3}-3)\gamma}{2(3-\sqrt{3})} &\text{for}\,\,\,\, \frac{5}{3} \leq \gamma <1+\frac{2}{\sqrt{3}}.
\end{cases}
\end{aligned}
\end{equation}
If  parameters  $(\gamma,\delta)$ satisfy
any one of the following  conditions $(P_1)$--$(P_2)${\rm :}
\begin{itemize}
\item[$(\rm P_1)$] 
\begin{equation}\label{canshu2}
1<\gamma<1+\frac{2}{\sqrt{3}},\quad\, 0<\delta<\frac{1}{2},\quad\, 
\delta_{\rm dis}=\frac{(1-\delta)(\Lambda-1)}{\alpha}+\Lambda-2>0,\end{equation}
   where $\Lambda\in (1,\Lambda^*(\gamma))$ is the  scaling parameter introduced in {\rm Lemma \ref{thm:existence_profiles};}
\smallskip
\smallskip
\item[$(\rm P_2)$] 
\begin{equation}\label{canshu1}
\gamma\in (1,1+\frac{2}{\sqrt{3}})\setminus \mathcal{J},\quad
0< \delta < \delta^*(\gamma),
\end{equation}
where the set $\mathcal{J}$ {\rm (}possibly empty{\rm )}
contains at most countably many points excluding $\{\frac{5}{3}, \frac{7}{5}\}$ and is introduced in {\rm Lemma \ref{thm:existence_profiles}}, \end{itemize}
then, for two constants $T, \sigma>0$ sufficiently small,    
there are $C^{\infty}$ initial data $(\rho_0, u_0)$ with $\rho_0 >\sigma$ such  that the corresponding smooth  solution $(\rho,u)(t,x)$ of the   Cauchy problem  \eqref{eq:1.1}{\rm--}\eqref{viscoeff} blows up at time $T${\rm:} for any $\varepsilon>0$,
\begin{equation}\label{densityimplosion}
\quad \ \ \ \ \lim_{t\to T}\rho(t,0)=\infty,
\qquad  \lim_{t\to T} \sup_{|x|\leq \varepsilon}|u(t,x)|=\infty.
\end{equation}  
  More precisely, for any fixed $y \in \mathbb{R}^3$, 
  \begin{equation}\label{limitrelation}
        \begin{aligned}
            \lim_{t\to T}\,\big(\alpha^{-1} \Lambda(T-t)^{1-\frac{1}{\Lambda}}\big)^{\frac{1}{\alpha}}\rho(t,(T-t)^{\frac{1}{\Lambda}}y)&= \overline{\Clss}(|y|)^{\frac{1}{\alpha}},\\
            \lim_{t\to T}\, \Lambda(T-t)^{1-\frac{1}{\Lambda}}u(t,(T-t)^{\frac{1}{\Lambda}}y)&= \overline{U}(|y|).
        \end{aligned}
    \end{equation}
 Moreover, there exists a finite-codimension set of initial data satisfying the above conclusions {\rm(}see {\rm{Appendix \ref{appendix D} }}for more details{\rm)}.
\end{Theorem}

Furthermore, we consider the period problem in a torus of period $10$:
$$\mathbb{T}^3_{10}:=(\mathbb{R}/10 \mathbb{Z})^3$$
with the fundamental domain $[-5, 5)^{3}$, and take $0$ as the center of the torus (without loss of generality, by scaling). 
The same conclusion also holds for the periodic problem in $\mathbb{T}^3_{10}$:
\begin{equation}\label{periodicproblem}
\begin{cases}
\displaystyle
\partial_t\rho+\dive_x(\rho u)=0,\\[6pt]
\displaystyle
\partial_t(\rho u)+\dive_x(\rho u\otimes u)
+\nabla_x P =\dive_x \mathbb{S},\\[6pt]
(\rho,u)(0, x)=(\rho_0,u_0)(x)
\qquad\,\, \text{for $x \in \mathbb{T}_{10}^3$},
\end{cases}
\end{equation}
which can be stated as follows:

\begin{Theorem}\label{Thm1.2peordic}
Let  $(\overline{\Clss}, \overline{U})$ be  the $C^\infty$  self-similar profile 
solving \eqref{Profile},
which is obtained in {\rm Lemma \ref{lem:existofselfsimilar}} {\rm(}see {\rm Appendix \ref{appendix B}}{\rm)}{\rm ;} 
and let $\delta^*(\gamma)$ be a constant given by \eqref{def:delta*}.
If parameters  $(\gamma,\delta)$ satisfy condition 
 $(P_1)$ or condition $(P_2)$ shown in \eqref{canshu2}--\eqref{canshu1},
then  for two constants $T, \sigma>0$ sufficiently small, 
there are $C^{\infty}$ initial data $(\rho_0, u_0)$ with $\rho_0 >\sigma$ such  that the corresponding smooth  solution $(\rho,u)(t,x)$ of  problem   \eqref{periodicproblem}  blows up at time $T${\rm:}
   for any $\varepsilon>0$, \eqref{densityimplosion} holds.
  More precisely, for any fixed $y \in \mathbb{T}_{10}^3$, \eqref{limitrelation} holds.
 Moreover, there exists a finite-codimension set of initial data satisfying the above conclusions {\rm(}see {\rm{Appendix \ref{appendix D} }}for more details{\rm)}.
\end{Theorem}
In what follows, we provide a detailed proof for Theorem \ref{Thm1.1}. 
The proof of Theorem \ref{Thm1.2peordic} follows similar arguments, and we therefore only include 
some remarks on the proof of this situation at the end of this paper.

We now make some remarks on the results in Theorems \ref{Thm1.1}--\ref{Thm1.2peordic}.
\begin{Remark}
We make some remarks on the set $\mathcal{J}$ appearing in condition $(P_2)$ of {\rm Theorem~\ref{Thm1.1}}.
In \cite{merle1}, it is shown that, for $\gamma \in (1, \infty)\setminus \mathcal{J}_0$, 
there exists a discrete sequence $\{\Lambda_n\}_{n=1}^{\infty}$ satisfying \eqref{range of Lambda} such that system \eqref{Profilen} admits a smooth, spherically symmetric profile $(\overline{\Clss}_n, \overline{U}_n)$ for each $\Lambda_n$, where $\mathcal{J}_0$ is a {\rm (}possibly empty{\rm )} exceptional set containing at most countably many points{\rm ;} 
see {\rm Theorem~1.1 in \cite{merle1}}. 
Moreover, it was shown in \cite{shao} and \cite{buckmaster} that such sequences $\{\Lambda_n\}_{n=1}^{\infty}$ 
also exist for $\gamma=\frac{5}{3}$ and $\gamma =\frac{7}{5}$, respectively.
Based on the arguments presented in {\rm \S 8}, we can show that $\frac{5}{3}, \frac{7}{5} \notin \mathcal{J}$.
\end{Remark}

\begin{Remark}
The constant, $\delta^*(\gamma)$, in {\rm{Theorems~\ref{Thm1.1}--\ref{Thm1.2peordic}}} 
has the following asymptotic behavior{\,\rm:}
\begin{equation*}
    \delta^*(\gamma)\to \frac{1}{2} \quad \text{as $\gamma \to 1$}, \qquad \quad \delta^*(\gamma)\to 0 \quad 
    \text{as $\gamma \to 1+\frac{2}{\sqrt{3}}$}.
\end{equation*}
\end{Remark}
 
\begin{Remark}
    It follows from the properties of profile $(\overline{\Clss}, \overline{U})$ shown in {\rm Lemma \ref{lem:existofselfsimilar}} that $\overline{\Clss}(0) >0$, and there exists some finite $y^*$ with $|\overline{U}(y
    ^*)|>0$. Then \eqref{limitrelation} implies that, for any $\varepsilon > 0$, 
    $$
    \lim_{t\to T}\rho(t,0)=\infty,
\qquad    
    \lim_{t\to T}\sup_{|x|\leq \varepsilon}|u(t,x)| \geq \lim_{t\to T}\frac{1}{\Lambda}(T-t)^{\frac{1}{\Lambda}-1}|\overline{U}(y^*)|= \infty.
    $$  
\end{Remark}

\begin{Remark}
{\rm{Theorems~\ref{Thm1.1}--\ref{Thm1.2peordic}}} can be applied to several physical models. For example, we consider a class of barotropic cut-off inverse power force models for which $\delta = (\frac{1}{2} + b)(\gamma-1)$ with $b\in [0, \infty)$. 
We list below three representative choices of parameter $b${\rm:}
    \begin{enumerate}[leftmargin=2.5em,
  align=left,
  labelsep=0.2em,
  labelwidth=1.1em]
        \item[\rm(i)] for rigid elastic spherical molecules: $b=0$ and $\delta =\frac{1}{2}(\gamma-1)$, 
        it is required that $1< \gamma < 7-4\sqrt{2}${\rm ;}
        \item[\rm(ii)] for Maxwellian molecules: $b=\frac{1}{2}$ and $\delta =\gamma-1$, it is required that $1< \gamma < 1+\frac{2}{9}${\rm ;}
        \item[\rm(iii)] for the ionized gas:  $b=2$ and $\delta =\frac{5}{2}(\gamma-1)$, 
        it is required that $1< \gamma < 1+\frac{22-4\sqrt{10}}{81}$.
    \end{enumerate}
\end{Remark}

The rest of this paper is organized as follows{\rm:} In \S \ref{Section2}, we introduce the notations, present a new reformulation of \eqref{eq:1.1}, and outline the proof of the main results. 
In \S \ref{Section3}, we establish the desired local-in-time well-posedness for the 3-D reformulated problem. 
In \S \ref{Section4}--\S\ref{Section7}, we establish the global-in-time well-posedness 
for the reformulated 
system with smooth initial data chosen from a finite-codimensional manifold by a bootstrap argument.
In \S \ref{Section8}, we analyze the singularity formation for \textbf{CNS} in $\mathbb{R}^3$, 
as stated in Theorem \ref{Thm1.1}. 
In \S \ref{remarkperiodic}, we provide some remarks on the proof of Theorem \ref{Thm1.2peordic},
concerning the development of implosions of solutions of the  periodic problems. 
Finally, in Appendices A-D,  we present several auxiliary lemmas used throughout the analysis, 
as well as some properties of the self-similar profile $(\overline{\Clss}, \overline{U})$ of the steady Euler equations, as verified in \cites{buckmaster, shijia, merle1,shao}.

\section{Notations, Reformulations, and Outline of the Proof}\label{Section2}

In this section, we first introduce the notations in \S\ref{section-notaions}, 
which will be frequently used throughout this paper. 
Next, in \S\ref{section-reformulation}, we present a suitable reformulation of the Cauchy problem \eqref{eq:1.1}--\eqref{viscoeff}, which renders the problem amenable to analysis.
Finally, in \S\ref{section-outline}, we outline the main strategies and key new ideas 
of our analysis, based on this reformulation. 

\subsection{Notations}\label{section-notaions}
The following notations will be frequently used in this paper.
\subsubsection{Notations on coordinates and operators}\label{operator}
\begin{itemize}\smallskip
\item 
Let $f$ be any function defined on a measurable subset of $\mathbb{R}^3$ with independent variables $z=(z_1,z_2,z_3)^\top$. For any integer $k\geq 1$ and an ordered multi-index   
$\beta=(\beta_1,\beta_2,\cdots,\beta_k)$ with 
$\beta_i\in\{1,2,3\}$ for $i=1,\cdots,k$, we  define the length of $\beta$ by $|\beta|=k$, and denote
\begin{equation*}
\begin{split}
 & \partial^{z}_{\beta}f =\partial_{\beta_1, \beta_2, \cdots, \beta_k} f=\frac{\partial^k}{\partial z_{\beta_1}\partial{z_{\beta_2}}\cdots \partial{z_{\beta_k}}} f,\\
 & \nabla_{z} f=(\partial_{z_1} f,\partial_{z_2} f,\partial_{z_3} f)^\top,\quad \Delta_{z} f=\sum_{i=1}^3 \frac{\partial^2}{\partial^2 _{z_i}}f,\\
&   \nabla^k_{z} f \ \text{denotes one generic} \ \partial^{z}_{\beta} f \ \text{with} \      |\beta|=k  \text{ for integer }k\geq 2,\\
&   |\nabla^k_{z} f|=\Big(\sum_{|\beta|=k}|\partial^{z}_{\beta}f|^2\Big)^\frac{1}{2} \quad\text{ for integer } k\geq 1.
  \end{split}
\end{equation*}
Moreover, we define the  $\beta$--related, ordered multi-index $\beta^{(j)}$ ($j=1,\cdots, k$):
\begin{equation}\label{indexbetai}
\beta^{(j)}=(\beta_1, \beta_2, \cdots,\beta_{j-1}, \beta_{j+1}, \cdots, \beta_k),
\end{equation}
and then 
\begin{equation}\label{indexbetajdaoshu}
\begin{split}
\partial^{z}_{\beta^{(j)}}f
=&\frac{\partial^{k-1}}{\partial z_{\beta_1}\partial z_{\beta_2} \cdots\partial z_{\beta_{j-1}} \partial z_{\beta_{j+1}} \cdots \partial z_{\beta_k}} f\quad \text{for}~ j=1,\cdots,k.
\end{split}
\end{equation}

\smallskip
\item Let $f=(f_1,f_2,f_3)^\top: \Omega\subset \mathbb{R}^3 \to \mathbb{R}^3$ 
($\Omega$ is a measurable set) be a vector function with independent 
variables $z=(z_1,z_2,z_3)^\top$ and $\mathcal X \in \{\partial^{z}_{\beta}, \partial^{z}_{\beta^{(j)}},\Delta_{z},\nabla_{z}^k\}$.
Then, for any integer $k\geq 1$ and an ordered multi-index   
$\beta=(\beta_1, \beta_2, \cdots,\beta_k)$ with 
$\beta_i\in\{1,2,3\}$ for $i=1,\cdots,k$,
\begin{align*}
&\mathcal Xf=\big(\mathcal Xf_1,\mathcal Xf_2,\mathcal Xf_3\big)^\top, \ \quad \nabla_{z}f=\begin{pmatrix} 
\partial_{z_1} f_1 & \partial_{z_2} f_1 & \partial_{z_3} f_1\\[2mm]
\partial_{z_1} f_2 & \partial_{z_2} f_2 &  \partial_{z_3} f_2 \\[2mm]
\partial_{z_1} f_3 & \partial_{z_2} f_3 & \partial_{z_3} f_3
\end{pmatrix},\\
&|\nabla^k_{z} f|=\Big(\sum_{i=1}^3\sum_{|\beta|=k}|\partial_{\beta}^{z}f_i|^2\Big)^\frac{1}{2}.
\end{align*}
Moreover, it follows that $\dive_{z}\,f= \partial_{z_1}f_1+\partial_{z_2}f_2+\partial_{z_3}f_3$.

In particular, for the derivatives with respect to the self-similar spatial 
coordinates  $y=(y_1,y_2,y_3)^\top$ introduced in \eqref{scaling-coordinat1}, 
we use the notation: 
$$
(\partial_{\beta}, \partial_{\beta^{(j)}},\nabla,\Delta,\nabla^k,\dive)
=(\partial^{y}_{\beta}, \partial^{y}_{\beta^{(j)}},\nabla_{y},\Delta_y,\nabla_y^k, \dive_{y}).
$$

\smallskip
\item
Let $\beta$ be a multi-index, and let $\mathcal A$ be a linear operator.
For any sufficiently smooth function $f$ defined on a measurable subset of $\mathbb{R}^3$
with independent variables $z=(z_1,z_2,z_3)^\top$, we define the commutator between
$\partial^{z}_\beta$ and $\mathcal A$ by
\[
[\partial^{z}_\beta,\,\mathcal A]f
:= \partial^{z}_\beta(\mathcal A f) - \mathcal A(\partial^{z}_\beta f).
\]
Additionally, if $f=(f_1,f_2,f_3)^\top: \Omega\subset \mathbb{R}^3 \to \mathbb{R}^3$ ($\Omega$ is a measurable set) is a vector function with the independent variables $z=(z_1,z_2,z_3)^\top$, then
$$
[\partial_\beta^{z},\,\mathcal A]f=([\partial_\beta^{z}, \,\mathcal A]f_1,[\partial_\beta^{z}, \,\mathcal A]f_2,[\partial_\beta^{z}, \,\mathcal A]f_3)^\top.
$$
\end{itemize}

\subsubsection{Notations on function spaces}\label{function}

\begin{itemize}\smallskip
\item For $R>0$, we denote by
$B(0,R):=\{z\in\mathbb{R}^3 : \, |z|<R\}$
the open ball in $\mathbb{R}^3$ centered at the origin with radius $R$,  $B^c(0,R):=\{z\in\mathbb{R}^3 : \, |z|\geq R\}$  the corresponding exterior region of $B(0,R)$, and $\partial B(0,R):=\{z\in\mathbb{R}^3 : \, |z|=R\}$ the boundary of $B(0,R)$.

\smallskip
\item 
For any (scalar or vector valued) functions $f$ and $g$ defined in $\mathbb{R}^3$ and for some function spaces $X_1$ and $X_2$, if the independent variables of $(f,g)$ are $z=(z_1,z_2,z_3)^\top$, we adopt the following simplified notation for the standard  Sobolev spaces:
\begin{equation*}\begin{split}
 & \quad \ \ \|f\|_{H^s}=\|f\|_{H^s(\mathbb{R}^3)},\ \  \|f\|_{L^p}=\|f\|_{L^p(\mathbb{R}^3)},\ \ 
 \|f\|_{X_1 \cap X_2}=\|f\|_{X_1}+\|f\|_{X_2},\ \  
 \\[6pt]
 &\quad  \ \ \ D^{k,r}=\big\{f\in L^1_{\rm{loc}}(\mathbb{R}^3): \,\|f\|_{D^{k,r}(\mathbb{R}^3)}=\|\nabla^k_{z}f\|_{L^r(\mathbb{R}^3)}<\infty\big\},\ \ D^k=D^{k,2},   \\[6pt] 
 &\quad  \ \ \ \|f\|_{D^{k,r}}=\|f\|_{D^{k,r}(\mathbb{R}^3)},\quad \|(f,g)\|_{X_1}=\|f\|_{X_1}+
 \|g\|_{X_1}.
 \end{split}
\end{equation*}

\smallskip
\item For an open set $\Omega\subset\mathbb{R}^3$ and any integer $m\ge 0$, we define 
\begin{align*}
  \quad  \ \  H_0^m(\Omega; \mathbb{R}) :=H_0^m(\Omega), \quad
H_0^m(\Omega; \mathbb{R}^3) :=\{\, f=(f_1,f_2,f_3)^\top:\, f_i\in H_0^m(\Omega),\ i=1,2,3 \}, 
\end{align*} 
where $H_0^m(\Omega)$ denotes the closure of $C_c^\infty(\Omega)$ in $H^m(\Omega)$.

\smallskip

\item We define the following function space{\rm:}
\begin{align}
    X&:=H_0^m(B(0, 3C_0); \mathbb{R}) \times H_0^m(B(0, 3C_0); \mathbb{R}^3)\label{X space} \\
    &~ =\{(Q(z), U(z)): \, \Clss(z)\in H_0^m(B(0, 3C_0);\mathbb{R}), U(z) \in H_0^m(B(0, 3C_0);\mathbb{R}^3)\},\nonumber
\end{align}
where $m$ is a sufficiently large positive integer that will be  determined  later (see Lemma~\ref{prop:maxdissmooth} in Appendix~\ref{appendix C}),  and $C_0$ is a sufficiently large constant that will be determined later (see Lemma~\ref{C0choice}).
For $(\Clss_1,U_1)$, $(\Clss_2,U_2)\in X$, the corresponding  inner product  is given by 
\begin{equation*}
\begin{split}
&\langle (\Clss_1, U_1), (\Clss_2, U_2) \rangle_X \\\
&=\int_{B(0, 3C_0)} \big( \nabla^m_{z} \Clss_1 \cdot \nabla^m_{z} \Clss_2 +  \nabla^m_{z} U_1 \cdot \nabla^m_{z} U_2 + \Clss_1 \Clss_2+ U_1 \cdot U_2  \big)\,\text{d}z\\
&=\int_{B(0, 3C_0)} \big(\sum_{|\beta|=m}(\partial_\beta^{z} \Clss_1  \partial_\beta^{z} \Clss_2 +  \partial_\beta^{z} U_1 \cdot \partial_\beta^{z} U_2) + \Clss_1 \Clss_2+ U_1 \cdot U_2  \big)\,\text{d}z.
\end{split}
\end{equation*}
We obtain the inherited norm{\rm:}
\begin{equation}\label{spacexnorm}
\begin{split}
 \| (\Clss, U) \|_X^2 &= \int_{B(0, 3C_0)} \big(|\nabla^m_{z}\Clss |^2 + |\nabla^m_{z}U|^2 + \Clss^2  +  |U|^2\big)\,\text{d}z\\
&=\int_{B(0, 3C_0)} \big(\sum_{|\beta|=m}(|\partial_\beta^{z} \Clss|^2 +  |\partial_\beta^{z} U|^2) + \Clss^2+ |U|^2  \big)\,\text{d}z.
\end{split}
\end{equation}

\item We recall that the Hilbert space $X$ in \eqref{X space} admits the invariant decomposition
\begin{equation}\label{Xdecomposition}
X = V_{\mathrm{sta}} \oplus V_{\mathrm{uns}},
\end{equation}
where $V_{\mathrm{uns}}$ is finite-dimensional and consists of smooth functions.
This decomposition is established in Lemma~\ref{lemma:abstract_result} in Appendix~\ref{appendix C}.
We denote by
\[
P_{\mathrm{sta}}: X \to V_{\mathrm{sta}},
\qquad
P_{\mathrm{uns}}: X \to V_{\mathrm{uns}},
\]
the corresponding projections onto the stable subspace $V_{\mathrm{sta}}$ and the unstable subspace  $V_{\mathrm{uns}}$, respectively.
\end{itemize}

\subsubsection{Other notations}\label{othernotation}

\begin{itemize}
\item The \emph{Japanese bracket} $\langle \bullet\rangle$ is defined by $\langle z\rangle := \sqrt{1+|z|^2}\text{ for } z\in\mathbb{R}^3.$
\smallskip
\item  For any  integrable function $f(z)$ in $\mathbb{R}^3$,
$$ \int_{\mathbb{R}^3}  f \,{\rm d}z   =\int f \,{\rm d}z.$$

\smallskip
\item We write $\mathsf{a} \ll \mathsf{b}$ if $\mathsf{a}$ is far smaller than $\mathsf{b}$,
and write
$\mathsf{a} \gg \mathsf{b}$ if $\mathsf{a}$ is far larger than $\mathsf{b}$.
\smallskip

    \item 
The function, $\mathscr{X}(z)$, defined in $\mathbb{R}^3$ denotes a smooth spherically symmetric cutoff function satisfying
\begin{equation}\label{cutoff function}
\operatorname{supp}\mathscr{X} \subset B(0,1), \qquad\,
\mathscr{X}(z) = 1 \,\,\text{ on $B(0,\tfrac{1}{2})$}, \qquad\,
|\nabla_{z} \mathscr{X}| \le 10.
\end{equation}
For convenience, we define  $\widehat{X}(z,\tau) := \mathscr{X}(e^{-\tau}z)$.\

\smallskip
\item We introduce two smooth, spherically symmetric cutoff functions. Let $\chi_1(\varsigma): \mathbb R \rightarrow [0, 1]$ be supported on $[ -\frac32 C_0, \frac32 C_0 ]$ and satisfy $\chi_1 (\varsigma) = 1$ for $|\varsigma| \leq C_0$, where $C_0$ is a sufficiently large constant that will be determined later (see Lemma~\ref{C0choice}). Let $\chi_2(\varsigma) : \mathbb R \rightarrow [0, 1]$ be  supported on $[ -\frac52 C_0, \frac52 C_0 ]$ and satisfy $\chi_2 (\varsigma) = 1$ for $|\varsigma| \leq 2 C_0$, with $\chi_2(\varsigma)>0 $ in $(-\frac52 C_0, \frac52 C_0 )$.

\begin{figure}[h]
\centering
\includegraphics[width=0.4\textwidth]{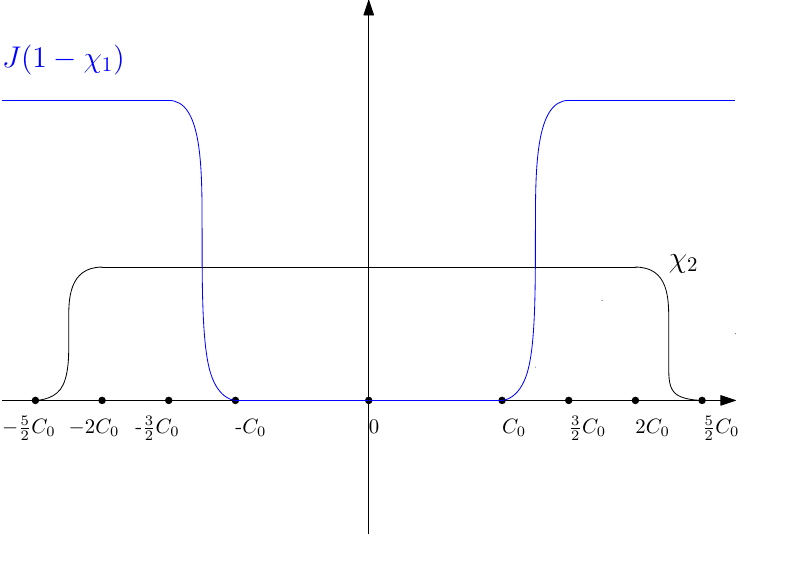}
\caption{The graph of $J(1-\chi_1)$ and $\chi_2$, where $J$ is a sufficient large constant determined in Lemma~\ref{prop:maxdissmooth}.}
\label{fig:cutoff}
\end{figure}

\item We also employ a spherically symmetric weight function $\phi \in C^1(\mathbb{R}^3)$ defined as
\begin{align}\label{weightdefi}
\phi(z)=\begin{cases}
1 & \text{ for } |z|\leq R_0,\\[6pt]
\frac{|z|^{2(1-\eta)}}{2R_0^{2(1-\eta)}} &\text{ for } |z|\geq 4R_0.
\end{cases}
\end{align}
Here $0<\eta \ll 1$ is a small constant to be fixed later in the bootstrap argument (see \S \ref{Section6}), 
and constant $R_0>0$ will be chosen later (see  \eqref{Rochoice estimate}). 
\end{itemize}

\subsection{Reformulation}\label{section-reformulation}
Throughout \S \ref{Section2}--\S \ref{remarkperiodic},   $C\geq 1$ denotes a generic constant depending only on $(\Lambda,\gamma, \delta, a_1, a_2)$, and $C(\nu_1,\cdots\!,\nu_k) \geq 1$ denotes a generic  constant depending on $C$
and parameters $\nu_1,\cdots\!,\nu_k$, which may be different at each occurrence.

\subsubsection{Reformulation based on the self-similar scaling}
First, denote  $\alpha =\frac{\gamma-1}{2}$ and  $\lss = \frac{1}{\alpha}\rho^{\alpha}$.
Then the Cauchy problem \eqref{eq:1.1}{\rm--}\eqref{viscoeff} can be rewritten as  
\begin{equation}\label{eq:1.2a}
\begin{cases}
\partial_t\lss=-u\cdot \nabla_x \lss -\alpha \lss \, \dive_x u,\\[4pt]
 \partial_tu = -u\cdot \nabla_x u -\alpha \lss \nabla_x \lss+ \alpha^{\frac{\delta-1}{\alpha}}\lss^{\frac{\delta-1}{\alpha}}L_xu+\frac{\delta}{\alpha}\alpha^{\frac{\delta-1}{\alpha}}\lss^{\frac{\delta-1-\alpha}{\alpha}}\nabla_x \lss \cdot \md_x(u),\\[4pt]
(c,u)(0,x)=(c_0,u_0)(x):=(\alpha^{-1}\rho^\alpha_0,u_0)(x)
\qquad \ \   \text{for $x \in \mathbb{R}^3$},\\[4pt]
\displaystyle
(c,u)(t,x)
\to (\bar{c},0):=(\alpha^{-1}\bar\rho^\alpha,0)
\qquad \ \    \qquad\qquad  \text{as $|x|\to \infty$\, for $t\ge 0$},
\end{cases}
\end{equation}
where
\begin{equation}\label{Lame operator}
    L_x u=a_1 \Delta_x u + (a_1+a_2)\nabla_x \dive_x u,\quad 
    \md_x(u)=a_1\big(\nabla_x u+(\nabla_x u)^\top\big)+a_2\,\dive_x u\,\mathbb{I}_3.
\end{equation}

Second, for fixed constants  $T>0$ and $\Lambda>1$, we  introduce the self-similar scaling{\rm:}
\begin{equation}\label{scaling-coordinat}
    \tau=-\frac{\log (T-t)}{\Lambda},\quad y= \frac{x}{(T-t)^{\frac{1}{\Lambda}}}=e^{\tau}x,\quad\tau_0=-\frac{\log T}{\Lambda},
\end{equation}
and 
\begin{equation}\label{scaling}
\begin{aligned}
    \lss(x,t)& = \frac{(T-t)^{\frac{1}{\Lambda}-1}}{\Lambda}\Clss (-\frac{\log (T-t)}{\Lambda} , \frac{x}{(T-t)^{\frac{1}{\Lambda}}}),\\
    u(x,t)&= \frac{(T-t)^{\frac{1}{\Lambda}-1}}{\Lambda}U( -\frac{\log (T-t)}{\Lambda} , \frac{x}{(T-t)^{\frac{1}{\Lambda}}}).
\end{aligned}    
\end{equation}
Then the reformulated system for $(Q,U)$ is given by
\begin{equation}\label{selfsimilar eq}
    \begin{cases}
    \displaystyle
        \partial_\tau \Clss =- (\Lambda-1)\Clss-(y+U)\cdot\nabla \Clss- \alpha \Clss \,\dive U,\\[4pt]
       \displaystyle
        \partial_\tau U =- (\Lambda-1)U-(y+U)\cdot\nabla U- \alpha \Clss \nabla \Clss\\[4pt]
        \qquad\quad\,+\, C_{\rm{dis}}e^{-\delta_{\rm{dis}}\tau}\big(\Clss^{\frac{\delta-1}{\alpha}}L(U)+\frac{\delta}{\alpha}\Clss^{\frac{\delta-1-\alpha}{\alpha}}\nabla \Clss \cdot \md(U)
        \big),
    \end{cases}
\end{equation}
where 
\begin{equation}\label{Lame operator-y}
    L(U)=a_1 \Delta U + (a_1+a_2)\nabla \dive U,\quad\,\,\,
    \md(U)=a_1\big(\nabla U+(\nabla U)^\top\big)+a_2\,\dive U\,\mathbb{I}_3,
\end{equation}
and constants $C_{\rm{dis}}$ and $\delta_{\rm{dis}}$ are defined by
\begin{equation}\label{S-para}
    C_{\rm{dis}}= \alpha^{\frac{\delta-1}{\alpha}}\Lambda^{\frac{1-\delta+\alpha}{\alpha}}>0, \quad \quad \delta_{\rm{dis}}= \frac{(1-\delta)(\Lambda-1)}{\alpha}+(\Lambda-2)>0.
\end{equation}
Next, we consider the Cauchy problem for \eqref{selfsimilar eq}--\eqref{S-para} with the  initial and far-field conditions{\rm:}
\begin{align}
(\Clss, U)(\tau_0,y)
&= (\Clss_0, U_0)(y)
:= \big(\Lambda e^{-(\Lambda-1)\tau_0}\lss_0,\,
       \Lambda e^{-(\Lambda-1)\tau_0}u_0\big)(e^{-\tau_0}y)
\quad &\text{for} \,\,\, y\in \mathbb{R}^3,
\label{SSinitial}\\[3pt]
(\Clss, U)(\tau,y)
&\longrightarrow \big(\Clss^*,0\big)
:= \big(\Lambda e^{-(\Lambda-1)\tau}\bar{\lss},\,0\big)
\quad \ \ \ \ \ \ \ \ \ \ \ \  \text{as } |y|\to\infty \ &\text{for} \,\,\, \tau \geq \tau_0.
\label{SSfar}
\end{align}

\begin{Remark}
    The explicit form of the threshold exponent \eqref{def:delta*}, derived from the condition that $\delta_{\rm{dis}}>0$, is imposed to guarantee that, after the reformulation, the dissipative terms exhibit the required temporal decay, thereby allowing {\rm \textbf{CNS}} to develop an implosion analogous to that of inviscid flows. 
\end{Remark}

\subsubsection{Reformulation  based on $C^\infty$ smooth spherically symmetric profiles of the steady Euler equations}\label{reformulationsteady}

Let $(\overline{\Clss}, \overline{U})$ be the $C^\infty$  self-similar profile  solving \eqref{Profile} obtained 
in {\rm Lemma \ref{lem:existofselfsimilar}} (see {\rm Appendix \ref{appendix B}}) 
with the scaling parameter $\Lambda \in (1, \Lambda^*(\gamma))$ introduced in {\rm Lemma~\ref{thm:existence_profiles}}.
Since $(\overline{\Clss}, \overline{U})$ 
is spherically symmetric,  then there exists a scalar function $\overline{\mathcal{U}}$ such that 
\[
\overline{\Clss}(y)=\overline\Clss(r),\quad 
\overline{U}(y)=\overline{\mathcal U}(r)\frac{y}{r} \qquad\,\, \mbox{for $r=|y|$}.
\]
In fact, in the proof for Theorem \ref{Thm1.1},  
\begin{itemize}
\item
the study on  the development of implosions for the original Cauchy problem of \textbf{CNS} can be reduced to a stability analysis around the cutoff self-similar profile $(\widehat{X}\overline \Clss, \widehat{X}\overline U)$,
\end{itemize}
where $\widehat{X}$ denotes the cutoff function introduced in \eqref{cutoff function}.
Motivated by this viewpoint, the key issue is to analyze the mathematical structure satisfied by the corresponding perturbation:
\begin{equation}\label{truncated profile}
     \widetilde{\Clss}= \Clss- \widehat{X}\overline \Clss, \quad \quad  \widetilde{U}= U- \widehat{X}\overline U.
\end{equation}
According to  \eqref{Profile} and \eqref{selfsimilar eq}, 
 $(\widetilde{\Clss}, \widetilde{U})$ satisfies the following
equations{\rm:}
\begin{align}
    \partial_\tau \widetilde{\Clss} =&\underbrace{-(\Lambda-1)\widetilde{\Clss}- (y + \widehat{X}\overline{U})\cdot \nabla \widetilde{\Clss}- \alpha(\widehat{X}\overline{\Clss})\,\dive \,\widetilde{U}- \widetilde{U}\cdot \nabla (\widehat{X}\overline{\Clss}) - \alpha\widetilde{\Clss}\,\dive (\widehat{X}\overline{U})}_{\mathcal{L}_\subclss^e(\widetilde{\Clss}, \widetilde{U})}\nonumber\\
    & \underbrace{- \widetilde{U}\cdot \nabla \widetilde{\Clss} - \alpha \widetilde{\Clss}\,\dive \,\widetilde{U}}_{\mathcal{N}_\subclss(\widetilde{\Clss}, \widetilde{U})}\label{eq: PerS}\\
    &\underbrace{-(\widehat{X}^2- \widehat{X})\overline{U}\cdot \nabla \overline{\Clss}- (1+\alpha)\widehat{X}\overline{U}\cdot \nabla\widehat{X}\,\overline{\Clss}-(\widehat{X}^2- \widehat{X})\overline{\Clss}\,\dive \,\overline{U}}_{\mathcal{E}_\subclss},  \nonumber\\
    \partial_\tau \widetilde{U} =&\underbrace{-(\Lambda-1)\widetilde{U}- (y + \widehat{X}\overline{U})\cdot \nabla \widetilde{U}- \alpha(\widehat{X}\overline{\Clss})\nabla \widetilde{\Clss}- \widetilde{U}\cdot \nabla (\widehat{X}\overline{U}) - \alpha\widetilde{\Clss}\nabla (\widehat{X}\overline{\Clss})}_{\mathcal{L}_u^e(\widetilde{\Clss}, \widetilde{U})}\nonumber\\
    &\underbrace{-\widetilde{U}\cdot \nabla \widetilde{U} - \alpha \widetilde{\Clss}\nabla \widetilde{\Clss}}_{\mathcal{N}_u(\widetilde{\Clss}, \widetilde{U})}\label{eq: PerU}\\
    &\underbrace{-(\widehat{X}^2- \widehat{X})\overline{U}\cdot \nabla \overline{U}-\widehat{X}\overline{U}\cdot \nabla\widehat{X}\,\overline{U}-\alpha(\widehat{X}^2- \widehat{X})\, \overline{\Clss} \nabla \overline{\Clss} -\alpha \widehat{X}\overline{\Clss}\, \nabla\widehat{X}\overline{\Clss} }_{\mathcal{E}_u} \nonumber \\
    &\underbrace{+C_{\rm{dis}}e^{-\delta_{\rm{dis}}\tau}\big(\Clss^{\frac{\delta-1}{\alpha}}L(U)+\frac{\delta}{\alpha}\Clss^{\frac{\delta-1-\alpha}{\alpha}}\nabla \Clss \cdot \md(U)
        \big)}_{\mathcal{F}_{\rm{dis}}}, \nonumber
\end{align}
where we have used the fact that $\partial\tau\widehat{X}= -y\cdot \nabla \widehat{X}$. Here 
\begin{itemize}
\item
$\mathcal{L}^{e}=(\mathcal{L}_{\subclss}^{e}, \mathcal{L}_{u}^{e})$ denotes the linear part in 
the time evolution equations of $(\widetilde \Clss, \widetilde U)$; 
\smallskip
\item $\mathcal{N}=(\mathcal{N}_{\subclss}, \mathcal{N}_{u})$ collects the nonlinear terms (at least quadratic in 
$(\widetilde \Clss, \widetilde U)$); 
\smallskip
\item $\mathcal{E}=(\mathcal{E}_{\subclss}, \mathcal{E}_{u})$ contains all the terms 
that depend only on the  known 
profile and \\
time-independent coefficients (\textit{e.g.} $\widehat X$, $\overline \Clss $, $\overline U$);

\item $\mathcal{F}_{\mathrm{dis}}$ denotes the exponentially decaying 
dissipative term produced by the regularization.
\end{itemize}

The properties of the linear operator $\mathcal{L}^{e}$ play an important role in our analysis. 
In Appendix  \ref{appendix C}, we will see that $\mathcal{L}^{e}$ exhibits some dissipative properties in bounded spatial regions.   
To capture this feature and facilitate precise estimates, we fix a constant
$C_0>0$, depending only on the self-similar profile
$(\overline{\Clss}, \overline{U})$ and  $\Lambda$, and then decompose $\mathbb{R}^3$ into  $B(0,C_0)$ and $B^c(0,C_0)$.

However, such dissipative properties fail in the far-fields. To further localize the analysis, we define the truncated linear operator $\mathcal L:=(\mathcal L_\subclss, \mathcal L_u)$ by
\begin{equation} \label{eq:cutoffL}
\begin{cases}
\displaystyle
    \mathcal L_\subclss(\widetilde{\Clss}, \widetilde{U}) = \chi_2 \mathcal{L}_\subclss^e(\widetilde{\Clss}, \widetilde{U}) - J(1-\chi_1)\widetilde{\Clss}, \\[6pt]
    \displaystyle
    \mathcal L_u(\widetilde{\Clss}, \widetilde{U}) = \chi_2 \mathcal{L}_u^e(\widetilde{\Clss}, \widetilde{U}) - J(1-\chi_1)\widetilde{U},
\end{cases}
\end{equation}
where $\chi_1$ and $\chi_2$ are two smooth cutoff functions defined in \S \ref{othernotation} and $J$ is a sufficiently large constant defined in Lemma~\ref{prop:maxdissmooth} so as to enforce the  uniform dissipation of $\mathcal L$ in the exterior region $B^c(0,C_0)$. According to the definition of $\mathcal L$ in \eqref{eq:cutoffL}, $\mathcal L=\mathcal{L}^{e}$ in  $B(0,C_0)$. 
 
As shown in \cite{shijia}, for $m$  determined in Lemma \ref{prop:maxdissmooth}, in the Hilbert
space $H_0^m(B(0, 3C_0); \mathbb{R})\times H_0^m(B(0, 3C_0); \mathbb{R}^3)$ 
defined in \eqref{X space} based on the choice of $(m,C_0)$,  operator $\mathcal{L}$
 is maximally dissipative, modulo a finite-rank perturbation corresponding to the unstable modes.  
 
Now we are ready to state the desired global stability theory of problem \eqref{selfsimilar eq}--\eqref{SSfar}.

\begin{Theorem}\label{globalexist}
We assume that $1< \gamma <1+\frac{2}{\sqrt{3}}$, $\Lambda \in (1, \Lambda^*(\gamma))$ is the scaling parameter introduced in {\rm Lemma~\ref{thm:existence_profiles}} such that there exists a smooth, spherically symmetric self-similar profile $(\overline{\Clss}, \overline{U})$ solving \eqref{Profile}, 
$\delta\in (0,\frac{1}{2})$ satisfies $\delta_{\rm dis}>0$, 
$(m, J)$ are  two sufficiently large constants determined in {\rm Lemma~\ref{prop:maxdissmooth}}, 
and  $C_0>1$ is  a sufficiently large constant depending on profile $(\overline{\Clss}, \overline{U})$ and parameter $\Lambda$. 
Then there exist positive constants $\smallc_1, E$, and $K$, which depend only on $(m, J)$ 
and satisfy 
  $$
  K\geq 4, \qquad \smallc_1 \ll \frac{1}{E} \ll \frac{1}{K}\ll 1,
  $$ 
  such that, if the initial data function $(\Clss_0,U_0)$ satisfies
    $$
  \inf_{y\in \mathbb{R}^3}\Clss_0>0,\quad   \Clss_0-\Clss^* \in   H^K(\mathbb{R}^3), \quad U_0\in  H^K(\mathbb{R}^3),     $$
    and 
    \begin{equation}\label{initialdata1}
        \max\big\{\|\Clss_0-\widehat{X}\overline{\Clss}\|_{L^{\infty}}, \ \|U_0-\widehat{X}\overline{U}\|_{L^{\infty}}\big\} \leq \smallc_1, \quad   \Clss_0 \geq \frac{\smallc_1}{2}, \quad  \langle y \rangle^{\Lambda}(|\nabla \Clss_0|+|\nabla U_0|)< \infty,
    \end{equation}
\begin{equation}\label{initialdata2}
       \max\big\{\|P_{\rm sta}(\chi_2(\Clss_0-\widehat{X}\overline{\Clss}))\|_{X}, \  \|P_{\rm sta}(\chi_2(U_0-\widehat{X}\overline{U}))\|_{X}\} \leq \smallc_1,
    \end{equation}
    \begin{equation}\label{initialdata3}
        E_K(\tau_0)=\int\big(|\nabla^K \Clss(\tau_0,y)|^2+|\nabla^K U(\tau_0,y)|^2\big)
        \phi^K(y) \,\dy \leq \frac{E}{2},
    \end{equation}
    where the function space $X$ is defined in \eqref{X space} based on the choice of $(m,C_0)$, 
    and the  weight function $\phi \in C^1(\mathbb{R}^3)$ is defined in \eqref{weightdefi},
    then  the Cauchy problem \eqref{selfsimilar eq}{\rm--}\eqref{SSfar} admits a unique global-in-time smooth solution $(\Clss, U)$ satisfying
    \begin{equation}
        \begin{split}
            &\inf_{ y\in \mathbb{R}^3} \Clss(\tau,y)>0, \quad 
\Clss-\Clss^*\in C([\tau_0,\tau];H^K(\mathbb{R}^3)),\\
&\partial_\tau \Clss\in C([\tau_0,\tau];H^{K-1}(\mathbb{R}^3)),\ \ U\in C([\tau_0,\tau];H^K(\mathbb{R}^3))\cap L^2([\tau_0,\tau];H^{K+1}(\mathbb{R}^3)),\\[2pt]
& \partial_\tau U\in C([\tau_0,\tau];H^{K-2}(\mathbb{R}^3))\cap L^2([\tau_0,\tau];H^{K-1}(\mathbb{R}^3)),
\end{split}
    \end{equation}
and
    \begin{equation}
  \max\big\{\|\Clss-\widehat{X}\overline{\Clss}\|_{L^\infty}, \  \|U-\widehat{X}\overline{U}\|_{L^\infty}\big\}\leq \frac{\smallc_0}{100}e^{-\varepsilon(\tau-\tau_0)}, 
  \qquad   E_K(\tau) \leq E,
    \end{equation}
for all $\tau>\tau_0$, where $\smallc_0>0$ is a sufficiently small positive constant depending only on $E$, and $\varepsilon>0$ is a small positive constant depending only on $\delta_{\rm dis}$, so that they satisfy $$\smallc_0^{3/2}\ll \smallc_1 \ll \smallc_0 \ll \frac{1}{E}\ll \frac{1}{K}\ll \frac{1}{m}\ll \varepsilon \ll 1.$$
    Moreover, there exists a finite-codimension set of initial data satisfying the above conditions {\rm(}see {\rm{Appendix \ref{appendix D} }}for more details{\rm)}.
\end{Theorem}
\begin{Remark}[Choice and Hierarchy of Parameters]\label{choiceofparameters}
Throughout the proof of {\rm{Theorem~\ref{globalexist}}}, the parameters are chosen in the following order
and satisfy the relations below{\rm :}

\begin{enumerate}[leftmargin=2.5em,
  align=left,
  labelsep=0.2em,
  labelwidth=1.1em]
\item[\rm{(i)}] The constant $\varepsilon>0$ depends only on $\delta_{\rm dis}$ and is chosen such that
\[
\varepsilon < \frac{1}{4}\,\delta_{\rm dis}=O(1).
\]

\item[\rm{(ii)}] The integer $K$ is chosen sufficiently large so that
\[
K > \max\{3^m\, C(J,m), \,500\},
\]
where $J$ and $m$ are fixed as in {\rm Lemma~\ref{prop:maxdissmooth}}.

\item[\rm{(iii)}] The energy bound $E$ is taken large enough to satisfy
\[
E > C3^{2K^2}.
\]

\item[\rm{(iv)}] The constant $\smallc_0>0$ is chosen sufficiently small, depending on $E$, such that
\[
\smallc_0 < E^{-1}.
\]

\item[\rm{(v)}] The constant $\smallc_1>0$ is chosen to satisfy 
\[
\smallc_0^{\frac{3}{2}} \ll \smallc_1 \ll \smallc_0;
\]
for instance, one may take $\smallc_1 = \smallc_0^{\frac{4}{3}}$.
\end{enumerate}

It is worth noting that all parameters are chosen independently of time $\tau$.
\end{Remark}

\subsection{Outline of the Proof}\label{section-outline}
As mentioned in \S \ref{reformulationsteady}, 
\begin{itemize}
\item the study on the finite-time implosion of the smooth solution $(\rho,u)(t,x)$ of the Cauchy problem \eqref{eq:1.1}--\eqref{viscoeff} for the original  \textbf{CNS}  can be reduced to a global-in-time stability analysis around the cutoff self-similar profile $(\widehat{X}\overline \Clss, \widehat{X}\overline U)(y)$ for solution $(Q, U)(\tau,y)$ of the reformulated problem \eqref{selfsimilar eq}{\rm--}\eqref{SSfar}.
\end{itemize}
 Thus, in order to prove Theorem \ref{Thm1.1},  
\begin{itemize}
\item [\rm{(i)}]
First, based on the  self-similar solutions to the steady Euler equations, we apply the same self-similar scaling to \textbf{CNS}, 
thus obtaining the reformulated system~\eqref{selfsimilar eq}. 

\item [\rm{(ii)}] Second,  we establish a global-in-time stability analysis around profile $(\widehat{X}\overline \Clss, \widehat{X}\overline U)$ 
for solution $(Q,U)$ of the reformulated problem \eqref{selfsimilar eq}{\rm--}\eqref{SSfar} (see Theorem \ref{globalexist}). Moreover,  we  can  recover a solution $(\rho,u)$ to  \textbf{CNS} \eqref{eq:1.1} on $[0,T)$ (see \S \ref{Section3}--\S \ref{Section7}). 

\item [\rm{(iii)}] Finally, based on the global stability result obtained above,  
in the original $(t, x)$--variables, solution $(\rho,u)$ to \eqref{eq:1.1} satisfies the perturbative bounds:
\begin{equation*}
    \begin{aligned}
       \Big|\Lambda(T-t)^{1-\frac{1}{\Lambda}}\lss(t, x)- \mathscr{X}(x)\overline{\Clss}(\frac{x}{(T-t)^{\frac{1}{\Lambda}}})\Big|\leq \frac{\smallc_0}{100}\Big(\frac{T-t}{T}\Big)^{\frac{\varepsilon}{\Lambda}},\\
        \Big|\Lambda(T-t)^{1-\frac{1}{\Lambda}} u(t, x)- \mathscr{X}(x)\overline{U}(\frac{x}{(T-t)^{\frac{1}{\Lambda}}})\Big|\leq \frac{\smallc_0}{100}\Big(\frac{T-t}{T}\Big)^{\frac{\varepsilon}{\Lambda}},
    \end{aligned}
\end{equation*}
for all $(t,x) \in (0,T)\times \mathbb{R}^3$, where 
\((\lss,u)\)=$(\alpha^{-1}\rho^\alpha,u)$.
These estimates show that the  solution to \textbf{CNS} remains a small perturbation of the self-similar Euler profile, thereby showing the implosions of the smooths solutions to the original \textbf{CNS} \eqref{eq:1.1} in finite time and then  completing the proof of Theorem~\ref{Thm1.1} (see \S \ref{Section8}). 
\end{itemize}

According to the above discussion, the main part of the proof for Theorem \ref{Thm1.1} 
is to prove Theorem \ref{globalexist}.
In fact, for establishing such kind of global-in-time stability results \textbf{CNS}, 
motivated by the classical work in  Matsumura--Nishida \cite{MN}, Wang--Xu \cite{weike}, among others, 
the following two points are crucial: 
Within the life span of the solutions of problem \eqref{selfsimilar eq}{\rm--}\eqref{SSfar},
 \begin{enumerate}
\item [\rm{(i)}]    $Q$ is uniformly lower and upper bounded;

\item  [\rm{(ii)}]  $(Q,U)$ can maintain the smallness condition as that of the initial data.
 \end{enumerate}
Such information can usually be obtained by using the so-called bootstrap argument; see \cites{MN,weike}. 

Then we show how the above two points can be obtained. 
First, we notice the following:
\begin{enumerate}
\item [\rm{(i)}]the truncated linear operator $\mathcal{L}$  generates a continuous semigroup with uniform decay estimates on the stable subspace $V_{\rm{sta}}$ defined in \eqref{Xdecomposition};

\item  [\rm{(ii)}] by the definition of cutoff functions $\chi_1$ and $\chi_2$, the linear operator $\mathcal{L}^e$ in \eqref{eq: PerS}--\eqref{eq: PerU} agrees exactly with $\mathcal{L}$ in $B(0,C_0)$, 
 \end{enumerate}
 and  henceforth carry out the analysis at the level of 
 the perturbation system \eqref{eq: PerS}--\eqref{eq: PerU}. 
To overcome the difficulties caused by the degeneracy  and strong coupling  between the velocity and the density 
in the spatial dissipation, the main ingredients of our analysis consist of the following:
\begin{enumerate}
\item [\S\ref{subsub:2.3.1}:]
establishing spatial decay estimates for $(\Clss, U)$ by the region segmentation method and 
some carefully designed  weighted energy estimates (see \S \ref{Section5});

\item[\S\ref{subsub:2.3.2}:] establishing the  lower-order temporal decay estimates of perturbation $(\widetilde{\Clss}, \widetilde U)$ and establishing uniform-in-time bounds for the associated lower-order norms by  the region segmentation method and a bootstrap argument (see \S \ref{subsection6.1});

\item[\S\ref{subsub:2.3.3}:] establishing higher-order weighted energy estimates 
through some carefully designed weighted energy estimates  and the repulsive structure of the self-similar profile (see \S \ref{subsection6.2}); 

\item[\S\ref{subsub:2.3.4}:] controlling the unstable modes by selecting suitable initial data through a fixed-point argument (see \S \ref{subsection6.3}).
\end{enumerate}

\subsubsection{Spatial decay estimates} \label{subsub:2.3.1}
First, based on  the properties of $(\overline{\Clss}, \overline{U})$, 
the weighted energy estimate for perturbation $(\widetilde \Clss, \widetilde U)$ is established:
$$
\widetilde{E}_K(\tau):=\int \big(|\nabla^{K}\widetilde{\Clss}(\tau, y)|^2+|\nabla^{K}\widetilde{U}(\tau,y)|^2\big)\phi^{K}(y)\,\text{d}y\leq 2E.
$$

Second, the spatial decay estimates for $\Clss$ and $(\nabla \Clss,\nabla \widetilde \Clss)$ are derived by a region segmentation method.  
We choose $R_0>0$, depending only on $\smallc_0$ and profile
$(\overline{\Clss},\overline U)$, such that $|\overline{\Clss}|+|\overline U|\leq C\smallc_0$ in the exterior region $B^c(0,R_0)$. 
For any $|y_0|\ge R_0$, define
\[
\omega_{\Clss,y_0}(\tau)
:=e^{(\Lambda-1)(\tau-\bar\tau)}
\Clss(\tau,\,y_0e^{(\tau-\bar\tau)}),
\]
where $y_0$ and $\bar\tau$ are chosen such that either $|y_0|=R_0$ or $\bar\tau=\tau_0$.
Then, by \eqref{selfsimilar eq}$_1$, $\omega_{\Clss, y_0}(\tau)$ satisfies
\begin{align}\label{jieduanQ}
    \partial_{\tau}\omega_{\Clss,y_0}(\tau)
    = -e^{(\Lambda-1)(\tau-\bar\tau)}(U \cdot \nabla \Clss+\alpha \Clss \, \dive \, U)(\tau,\,y_0 e^{(\tau-\bar\tau)}).
\end{align}
Note that $(\Clss, U)=(\widehat{X}\overline{\Clss}+\widetilde{\Clss},\, \widehat{X}\overline{U}+\widetilde{U})$.
By applying the interpolation inequalities, the right-hand side of \eqref{jieduanQ} 
can be controlled by the upper bounds of the cutoff profile $(\widehat{X}\overline \Clss,\widehat{X}\overline U)$ and perturbation $(\widetilde \Clss, \widetilde U)$ in $B^c(0,R_0)$ together with $K$-th order weighted energy estimates for $(\widetilde{\Clss},\widetilde{U})$.
Integrating \eqref{jieduanQ} with respect to $\tau $ yields the pointwise upper and lower bounds for $\Clss$ along the trajectories $y=y_0e^{\tau-\bar\tau}$ so that,
in the region $B^c(0, R_0)$,
$$ 
\frac{\smallc_1}{C}\big\langle \frac{y}{R_0} \big\rangle^{-(\Lambda-1)} 
\leq \Clss \leq C\smallc_0 \max \Big\{ \big\langle \frac{y}{R_0} \big\rangle^{-(\Lambda-1)}, \,e^{-(\tau-\tau_0)(\Lambda-1)} \Big\}. 
$$
Applying $\nabla $ to \eqref{selfsimilar eq}$_1$ and repeating the above argument, we further obtain
$$
|\nabla \Clss|+|\nabla \widetilde{\Clss}|
\le C\big\langle \frac{y}{R_0}\big\rangle^{-\Lambda},
\qquad |y|\ge R_0.
$$
In particular, this decay rate is faster than the one directly  obtained by
interpolation between $\|\Clss\|_{L^\infty}$ and the $K$-th order weighted energy estimates for $(\Clss, U)$.
In the interior region $B(0,R_0)$, the bounds follow from interpolation between the upper bounds and the $K$-th order weighted energy estimates for $(\widetilde{\Clss},\widetilde{U})$, together with the properties of $(\overline{\Clss},\overline{U})$ and $\widehat{X}$.

Finally, we turn to the spatial decay of $\nabla U$. In the interior region $B(0, R_0)$, the derivation of the desired estimates for $\nabla U$ is identical to that for $\nabla \Clss$. 
In the exterior region $B^c(0,R_0)$, additional difficulties arise. Applying $\partial_{y_i}$ to \eqref{selfsimilar eq}$_2$, 
the viscous term takes the form:
\begin{equation*}
\begin{aligned}
     \partial_{y_i} \mathcal{F}_{\rm{dis}}=C_{\rm{dis}}e^{-\delta_{\rm{dis}}\tau}\Big(&\Clss^{\frac{\delta-1}{\alpha}}L(\partial_{y_i} U)+ \frac{\delta-1}{\alpha} \Clss^{\frac{\delta-1}{\alpha}-1}\partial_{y_i} \Clss L(U)+ \frac{\delta}{\alpha}\Clss^{\frac{\delta-1}{\alpha}-1}\nabla \Clss \cdot \md(\partial_{y_i} U)\Big.\\
     & + \Big.\frac{\delta}{\alpha} \Clss^{\frac{\delta-1}{\alpha}-1}\nabla (\partial_{y_i} \Clss) \cdot\md(U) +\frac{\delta}{\alpha}\big(\frac{\delta-1}{\alpha}-1\big) \Clss^{\frac{\delta-1}{\alpha}-2}\,\partial_{y_i} \Clss\,\nabla \Clss \cdot \md(U) \Big).
\end{aligned}
\end{equation*}
Introducing $\omega_{U,y_0,i}(\tau) := e^{\Lambda(\tau-\bar{\tau})} \partial_{y_i} U(\tau,  y_0 e^{\tau-\bar\tau})$, we derive 
\begin{equation}\label{eq:omega-U}
\partial_{\tau} \omega_{U,y_0,i} = -e^{\Lambda(\tau-\bar{\tau})} \big( (\partial_{y_i} U \cdot \nabla) U + (U\cdot \nabla )(\partial_{y_i} U) + \alpha \partial_{y_i}  \Clss \nabla \Clss + \alpha \Clss \nabla (\partial_{y_i} \Clss) \big)+ \partial_{y_i} \mathcal{F}_{\rm{dis}}.
\end{equation}
For estimating the right-hand side of \eqref{eq:omega-U}, the main difficulties
arise from the viscous terms containing factors of form
$\Clss^{-1}\nabla^k\Clss$, with $k=1,2$.
A direct interpolation between $\Clss$ itself and its $K$-th order weighted energy does not provide a sufficiently fast decay rate. However, we have already shown that $|\nabla\Clss|$ exhibits faster spatial decay than $\Clss$ itself.
Then, according to  the Gagliardo-Nirenberg interpolation  inequality \eqref{eq:GNresultinfty}, we obtain the pointwise bound
$$|\Clss^{-1}\nabla^k\Clss|\leq C(\smallc_0)\big\langle \frac{y}{R_0}\big\rangle^{-1+\eta-\Lambda}\qquad \text{for}\ \ k=1,2.$$
All remaining terms in \eqref{eq:omega-U} can be handled
by the lower bound of $\Clss$, as well as the upper bounds and the $K$-th order weighted energy for $(\Clss,U)$.
Combining all the above bounds, we obtain
$$ |\nabla U|+ |\nabla \widetilde U| \leq C\big\langle \frac{y}{R_0} \big\rangle^{-\Lambda} 
\qquad \text{for all $y \in B^c (0, R_0)$.}
$$

\subsubsection{Lower-order temporal decay estimates}\label{subsub:2.3.2}
First, the spatial decay estimates for $(\Clss,U)$ established above, together with the properties of the cutoff function $\widehat X$,
yield the temporal decay bounds for terms
$\mathcal E$, $\mathcal N$, and
$\mathcal F_{\mathrm{dis}}$.
These bounds provide the basic preparatory estimates for the temporal decay
analysis of the perturbation $(\widetilde{\Clss},\widetilde U)$ itself and  its fourth-order spatial derivatives.

Next, we treat the interior and exterior regions separately.  Recall that 
$V_{\mathrm{sta}}$ and $V_{\mathrm{uns}}$ are invariant subspaces of  
$\mathcal L$.  
On the stable space $V_{\mathrm{sta}}$,  $\mathcal L$ is dissipative.
On the unstable space $V_{\mathrm{uns}}$, which is finite-dimensional with
$\dim V_{\mathrm{uns}}=N$, we assume that coefficients $\{\hat{k}_i\}$ are chosen so that the evolution of the unstable initial component 
$\sum_{i=1}^{N}\hat{k}_i(\varphi_{i,\subclss},\varphi_{i,u})$ remains controlled:
$$
\Big\|\sum_{i=1}^{N}k_i(\tau)(\varphi_{i,\subclss}, \varphi_{i,u})\Big\|_{X}
   \le \smallc_1 e^{-\varepsilon(\tau-\tau_0)} \qquad \text{for all}\,\,\, \tau\ge \tau_0.
$$
It is observed that the linear operator $\mathcal L^e$ appearing in
\eqref{eq: PerS}--\eqref{eq: PerU} agrees with $\mathcal L$ in  $B(0,C_0)$. As a consequence, the above properties of dissipativity and boundedness
remain valid for $\mathcal L^e$ when restricted to $B(0,C_0)$. This allows us to identify the interior region as $B(0, C_0)$.
Combining these properties with the decay estimates  obtained for 
$\mathcal E$, $\mathcal N$, and $\mathcal F_{\mathrm{dis}}$, 
it follows that, for $0\leq j\leq 4$, 
\begin{equation*}
\|\nabla^{j}\widetilde{\Clss}\|_{L^{\infty}(B(0,C_0))}+\|\nabla^{j}\widetilde{U}\|_{L^{\infty}(B(0,C_0))}\leq C\frac{\smallc_1}{\smallc_g}e^{-\varepsilon(\tau-\tau_0)}.
\end{equation*}

On the other hand, in the exterior region $B^c(0, C_0)$, we rewrite the perturbed  system \eqref{eq: PerS}--\eqref{eq: PerU} 
along the trajectories: $y = y_0 e^{\tau-\bar{\tau}}$. Following the same strategy as the one used in the derivation of the spatial decay estimates for  $\Clss$  in \S \ref{subsub:2.3.1}, and combining the  Gagliardo-Nirenberg interpolation inequality with  the  bounds for $\mathcal E$, $\mathcal N$, and $\mathcal F_{\mathrm{dis}}$ obtained previously, we can obtain 
the corresponding temporal decay bounds in the exterior region $B^c(0, C_0)$: 
\begin{equation*}
\max\{|\widetilde{\Clss}|, \ |\widetilde{U}|\} \leq \frac{{\smallc_0}}{200}e^{-\varepsilon(\tau-\tau_0)}\qquad \text{for all}\,\,\, \tau\ge \tau_0.
\end{equation*} 
We now turn to the temporal decay estimates for the fourth-order spatial derivatives of 
$(\widetilde{\Clss},\widetilde{U})$. The main difficulty arises from the 
derivatives of the nonlinear terms $\mathcal N$ and the dissipative term 
$\mathcal F_{\mathrm{dis}}$. For $\mathcal N$, a Gagliardo-Nirenberg interpolation inequality is employed to obtain the required bounds.  
For $\mathcal F_{\mathrm{dis}}$, the spatial decay estimates established earlier ensure the 
integrability with respect to $y$, while the factor
$e^{-\delta_{\mathrm{dis}}\tau}$ provides the necessary temporal decay.  As a consequence, we conclude
\begin{equation*}
\max\big\{\|\nabla^4\widetilde{\Clss}\|_{L^2(B^c(0, C_0))},\, \|\nabla^4\widetilde{U}\|_{L^2(B^c(0, C_0))}\big\}
\leq  \frac{\smallc_0}{2}e^{-\varepsilon(\tau-\tau_0)}\qquad \text{for all}\,\,\, \tau\ge \tau_0.
\end{equation*}

\subsubsection{$K$-th order weighted  estimates}\label{subsub:2.3.3}
Now we give the $K$-th order weighted energy estimate:
$$
E_K(\tau)= \int \big(|\nabla^K \Clss(\tau, y)|^2+|\nabla^K U(\tau, y)|^2\big)
\phi^K(y)\,\dy\leq E,
$$
where the  weight function $\phi \in C^1(\mathbb{R}^3)$ is  defined in \eqref{weightdefi}.
Applying $\partial_\beta$ ($\beta=(\beta_1, \beta_2, \cdots,\beta_K)$ with $|\beta|=K$) to
\eqref{selfsimilar eq}
and denoting  $W_\beta := (\partial_\beta\Clss, \partial_\beta U)^\top$ yield
\begin{equation}\label{eq:sym-K}
(\partial_\tau+\Lambda-1+K)W_\beta+\sum_{j=1}^3 A_j(\Clss,U;y)\,\partial_{y_j}W_\beta
=\mathcal{C}_\beta+\binom{0}{\partial_\beta\mathcal F_{\rm dis}},
\end{equation}
where $A_j(\Clss, U; y), j=1,2,3$, are symmetric matrices, 
and $\mathcal{C}_\beta$ collects the lower-order terms:
\[
\mathcal{C}_\beta:=
-\binom{[\partial_\beta,\,U\cdot\nabla]\Clss}{[\partial_\beta,\,U\cdot\nabla]U}
-\alpha\binom{[\partial_\beta,\,\Clss]\,\dive U}{[\partial_\beta,\,\Clss]\nabla\Clss}.
\]
By  the symmetric structure of  \eqref{eq:sym-K}, the construction  of the weight function $\phi$, 
and the repulsivity of the radial and angular
components of profile $(\overline\Clss,\overline U)$, the convective terms can be controlled by  the weighted energy $E_K$.
More precisely, the required repulsivity is guaranteed by the conditions:
\begin{align*}
  -\partial_r (\widehat{X}\overline{\mathcal{U}})
  + \alpha |\partial_r (\widehat{X}\overline \Clss) |
  + \frac{\nabla \phi \cdot (y+\widehat{X}\overline U)
  + | \nabla \phi | \widehat{X}\overline \Clss}{2\phi}
  &\leq 1 - \frac{\min \{ \tilde \eta, \eta \}}{2}, \\
  -\frac{\widehat{X}\overline{\mathcal{U}}}{r}
  + \alpha |\partial_r (\widehat{X}\overline \Clss) |
  + \frac{\nabla \phi \cdot (y+\widehat{X}\overline U)
  + | \nabla \phi | \widehat{X}\overline \Clss}{2\phi}
  &\leq 1 - \frac{\min \{ \tilde \eta , \eta \}}{2},
\end{align*}
where $r=|y|$, and $0<\tilde{\eta}, \eta<1$ are two small constants arising from the repulsivity properties \eqref{radial repulsivity}--\eqref{angular repulsivity} and the definition of the weight function in \eqref{weightdefi}, respectively.
Moreover, the terms $\mathcal{C}_\beta$ and $\partial_\beta \mathcal{F}_{\rm{dis}}$ can be handled by direct calculations  or  integration by parts.

As a result, the weighted energy $E_K(\tau)$ satisfies the differential inequality:
\begin{equation*}
\Big(\frac12\frac{\mathrm d}{\mathrm d\tau} + \Lambda-1+K\Big)E_K
\;\le\;
 K\Big(1-\frac{\min\{\tilde\eta,\eta\}}{2}\Big)E_K+ CE
+ e^{-\delta_{\rm dis}(\tau-\tau_0)},
\end{equation*}
which, along with  the Gr\"onwall inequality,  yields the desired bound for $E_K(\tau)$.

\subsubsection{Control of the unstable modes of $(\overset{\approx}{\Clss},\overset{\approx}{U})$ by the selection of  initial data.}\label{subsub:2.3.4}  
Since the linear operator $\mathcal L$ generates a continuous semigroup on space $X$ consisting of functions supported in region $B(0, 3C_0)$,
we consider the truncated equations \eqref{eq:truncated} with the initial data $( \overset{\approx}{\Clss}_{0},\overset{\approx}{U}_{0})$. Under the decomposition $X = V_{\mathrm{sta}} \oplus V_{\mathrm{uns}}$, the initial data for  \eqref{eq:truncated} can be decomposed  
into a stable component and an unstable one:
\begin{equation}
\label{trun-inidata}
\overset{\approx}
{\Clss}_{0}=\chi_2\widetilde{\Clss}_0^{*}+\sum_{i=1}^{N}\hat{k}_i
\varphi_{i,\subclss},\qquad 
\overset{\approx}
{U}_{0}=\chi_2\widetilde{U}_0^{*}+\sum_{i=1}^{N}\hat{k}_i\varphi_{i,u},
\end{equation}
where $\{(\varphi_{i,\subclss}, \varphi_{i,u})\}_{i=1}^{N}$ is a normalized basis of the unstable  subspace $V_{\mathrm{uns}}$ constructed in Lemma~\ref{prop:maxdissmooth} in Appendix~\ref{appendix C} and coefficients $\{\hat k_i\}_{i=1}^N$ are to be determined.
Since the stable component of $(\overset{\approx}{\Clss}, \overset{\approx}{U})$ has already been controlled as above, it remains to analyze the unstable modes.
We therefore focus on the unstable modes associated with operator $\mathcal L$.
Denote
$$
k(\tau)=P_{\rm{uns}}(\overset{\approx}{\Clss}(\cdot,\tau), \overset{\approx}{U}(\cdot,\tau))=\sum_{i=1}^{N}k_i(\tau)(\varphi_{i,\subclss},\varphi_{i,u}).
$$
Since $V_{\mathrm{uns}}$ is invariant with respect to the linear operator $\mathcal L$, the unstable coefficients satisfy
\begin{equation}\label{ODEofki}
\displaystyle
\partial_\tau k(\tau) = \mathcal L k(\tau) + P_{\rm{uns}}(\chi_2 \mathcal F), \qquad 
\displaystyle
k(\tau_0) = \big(\sum_{i=1}^{N}\hat{k}_i
\varphi_{i,\subclss},\,\sum_{i=1}^{N}\hat{k}_i
\varphi_{i,u}\big),
\end{equation}
where $\mathcal{F}= \mathcal{N}+\mathcal{E}+\mathcal{F}_{\mathrm{dis}}$.
This differential equation is generally unstable, and solutions with arbitrary initial data may grow exponentially.
For the metric, $\Upsilon$, determined in Lemma~\ref{prop:maxdissmooth}, let 
$$
\widetilde{\mathcal R}(\tau):=\big\{k(\tau)\in V_{\mathrm{uns}} : \, \|k(\tau)\|_{\Upsilon}\leq \smallc_1^{\frac{11}{10}} e^{-\frac{4}{3}\varepsilon(\tau-\tau_0)}\big\}.
$$
By the spectral properties of $\mathcal L$ restricted to $V_{\mathrm{uns}}$,
Lemma~\ref{outgoingpro} shows that any trajectory reaching the boundary of $\widetilde{\mathcal R}(\tau)$ immediately exits the region.
This outgoing property provides  a topological selection argument based on Brouwer’s fixed point theorem, which yields at least one choice of initial unstable coefficients $\{\hat k_i\}_{i=1}^N$ for which the corresponding solution remains in $\widetilde{\mathcal R}(\tau)$ for all $\tau\ge\tau_0$.

\section{Local-In-Time Well-Posedness of the Three-Dimensional  Smooth Solutions}\label{Section3}

This section is devoted to establishing the desired local-in-time well-posedness of smooth solutions of the  Cauchy problem \eqref{selfsimilar eq}--\eqref{SSfar}. In particular, we provide a higher-order weighted regularity estimate for the corresponding solutions, which plays a crucial role in initiating the bootstrap argument and in controlling the degenerate dissipative structure introduced by the reformulation. Moreover, throughout \S \ref{Section3}--\S \ref{remarkperiodic}, we always assume that   $C\geq 1$ denotes  a generic constant depending only on $(\Lambda,\gamma, \delta, a_1, a_2)$, and $C(\nu_1,\cdots\!,\nu_k) \geq 1$ denotes  a generic  constant depending on $C$
and parameters $\nu_1,\cdots\!,\nu_k$, which may be different at each occurrence.

Our main result in this section can be stated as follows:

\begin{Theorem}\label{Local(U,S)}
Assume that parameters $(\Lambda,\alpha,\delta, a_1,a_2)$ satisfy 
\begin{equation}\label{physical}
\Lambda>1,\quad \alpha>0, \quad 0<\delta <1, \quad a_1>0, \quad 2a_1+3a_2\geq 0,
\end{equation}
and  $K\geq 4$ is an integer.   
Let the initial data $(\Clss_0,U_0)$ satisfy 
\begin{equation}\label{initial'}
\inf_{y\in \mathbb{R}^3} \Clss_0(y)>0, \quad  \Clss_0-\Clss^*\in H^s(\mathbb{R}^3),\quad  U_0\in H^s(\mathbb{R}^3),
\end{equation}
for some integer $s\geq K$.
Then there exists a time $\tau_*>\tau_0$ such that the Cauchy 
problem \eqref{selfsimilar eq}{\rm--}\eqref{SSfar} admits 
a unique smooth  solution $(\Clss, U)(\tau, y)$ in $[\tau_0,\tau_*]\times\mathbb{R}^3$  
satisfying 
\begin{equation}\label{regularity'}
\begin{split}
&\inf_{(\tau, y)\in [\tau_0,\tau_*]\times\mathbb{R}^3} \Clss(\tau,y)>0, \ \ 
\Clss-\Clss^*\in C([\tau_0,\tau_*];H^s(\mathbb{R}^3)),\\
&\partial_\tau \Clss \in C([\tau_0,\tau_*];H^{s-1}(\mathbb{R}^3)),\ \ U\in C([\tau_0,\tau_*];H^s(\mathbb{R}^3))\cap L^2([\tau_0,\tau_*];H^{s+1}(\mathbb{R}^3)),\\[2pt]
& \partial_\tau U\in C([\tau_0,\tau_*];H^{s-2}(\mathbb{R}^3))\cap L^2([\tau_0,\tau_*];H^{s-1}(\mathbb{R}^3)). \end{split}
\end{equation}
Moreover, if $E_{K}(\tau_0)< \infty$, 
\begin{equation}\label{weightenergy}
   \text{\rm ess}\!\!\!\sup_{\tau_0\leq \tau \leq \tau_*} E_K(\tau) < \infty.
\end{equation}
\end{Theorem}

\subsection{The Local Well-Posedness of Smooth Solutions of  the Cauchy Problem  \eqref{eq:1.2a}}
In order to study the local-in-time well-posedness of smooth solutions of  the Cauchy problem  \eqref{selfsimilar eq}--\eqref{SSfar}, we first consider  the Cauchy problem  \eqref{eq:1.2a}.

\begin{Theorem}\label{zth1-po}
Assume that  the physical parameters $(\gamma,\delta, a_1,a_2)$ satisfy 
\begin{equation}\label{physical'}
\gamma>1, \quad 0<\delta<1, \quad a_1>0, \quad 2a_1+3a_2\geq 0,
\end{equation}
and $K\geq 4$ is an integer.
If the initial data $(\lss_0,u_0)$ satisfy 
\begin{equation}\label{initial}
\inf_{x\in \mathbb{R}^3} \lss_0(x)>0, \quad  \lss_0-\bar{\lss}\in H^s(\mathbb{R}^3),\quad  u_0\in H^s(\mathbb{R}^3),
\end{equation}
for some integer $s\geq K$ and constant $\bar{\lss}>0$, then there exists 
$T_*>0$ such that  the Cauchy problem \eqref{eq:1.2a} admits a unique smooth solution $(\lss,u)(t,x)$ in $[0,T_*]\times\mathbb{R}^3$  satisfying 
\begin{equation}\label{regularity}
\begin{split}
&\inf_{(t,x)\in [0,T_*]\times\mathbb{R}^3} \lss(t,x)>0, \ \ 
\lss-\bar{\lss}\in C([0,T_*];H^s(\mathbb{R}^3)),\\
&\partial_t\lss\in C([0,T_*];H^{s-1}(\mathbb{R}^3)),\ \ u\in C([0,T_*];H^s(\mathbb{R}^3))\cap L^2([0,T_*];H^{s+1}(\mathbb{R}^3)),\\[2pt]
& \partial_t u\in C([0,T_*];H^{s-2}(\mathbb{R}^3))\cap L^2([0,T_*];H^{s-1}(\mathbb{R}^3)). 
\end{split}
\end{equation}
\end{Theorem}
Since one can obtain the positive lower bound of $\lss$ via its transport equation $\eqref{eq:1.2a}_1$ in some finite time, $\eqref{eq:1.2a}_2$ becomes a parabolic system for $u$ with the source term $\lss^{\frac{\delta-1-\alpha}{\alpha}}\nabla_x \lss \cdot \md_x(u)$.  
Then, by adapting the classical method in Matsumura-Nishida \cite{MN} and exploiting the equations in \eqref{eq:1.2a}, 
one can establish the desired local well-posedness of smooth solutions when $\bar{\lss}>0$. For brevity, we omit the details of proof.

\subsection{The Local-In-Time Well-Posedness of Smooth Solutions of  the Cauchy  
Problem  \eqref{selfsimilar eq}--\eqref{SSfar}}
Now we are ready to give the proof for Theorem \ref{Local(U,S)}.

\begin{proof} We divide the proof into three steps. 
In what follows, we focus on the case $s=K$, since the other cases 
follow by analogous arguments.

\smallskip
\textbf{1.} By the assumption of Theorem~3.1, the initial data $(\Clss_0,U_0)$ satisfy
\begin{equation}\label{eq:SS-initial-assumption}
\inf_{y\in \mathbb{R}^3} \Clss_0(y)>0, \qquad  \Clss_0-\Clss^*\in H^K(\mathbb{R}^3),\qquad  U_0\in H^K(\mathbb{R}^3),
\end{equation}
where $\Clss^*=\Clss^*(\tau_0)$ is given by \eqref{SSfar}.

It follows from \eqref{SSinitial} that
\begin{equation}\label{eq:initial-transform-inverse}
c_0(x)=\Lambda^{-1}e^{(\Lambda-1)\tau_0}Q_0(e^{\tau_0}x),
\qquad
u_0(x)=\Lambda^{-1}e^{(\Lambda-1)\tau_0}U_0(e^{\tau_0}x).
\end{equation}
Based on the relation: $x=e^{-\tau_0}y$, we obtain
\begin{equation*}
    \inf_{x\in \mathbb{R}^3}c_0(x)=\inf_{y \in \mathbb{R}^3}c_0(e^{-\tau_0}y) = \Lambda^{-1}e^{(\Lambda-1)\tau_0}\inf_{y\in \mathbb{R}^3}Q_0(y)>0.
\end{equation*}
 Let $\beta=(\beta_1,\beta_2,\cdots,\beta_k)$ be a multi-index with 
$\beta_i\in\{1,2,3\}$ for $i=1,\cdots,k$, where $k\in [0,K]$ is an integer. By the chain rule and the fact that 
$\partial_{x_i}y_j = e^{\tau_0}\delta_{ij}$, one has
$$
\partial^x_\beta u_0(x)
= \Lambda^{-1} e^{(\Lambda-1)\tau_0} e^{|\beta|\tau_0}\,
  \,\partial_\beta^y U_0(y).
$$
Since $\text{d}x=e^{-3\tau_0}\,\text{d}y$, it follows that
\begin{equation*}
\begin{aligned}
\|\partial^x_\beta u_0\|_{L^2_x}^2
&= \Lambda^{-2} e^{2(\Lambda-1)\tau_0} e^{2|\beta|\tau_0}
\int |\partial^y_\beta U_0(y)|^2 \, e^{-3\tau_0} \text{d}y \\
&= \Lambda^{-2} e^{2(\Lambda-1)\tau_0} e^{(2|\beta|-3)\tau_0}\,
\|\partial^y_\beta U_0\|_{L^2_y}^2.
\end{aligned}
\end{equation*}
Summing over $|\beta|\leq K$ yields
\begin{equation*}
\|u_0\|_{H^K_x}
\leq C(K, \tau_0)\, \|U_0\|_{H^K_y},
\end{equation*}
where
\begin{equation*}
C(K,\tau_0)=\Lambda^{-1} e^{(\Lambda-1)\tau_0}
\Big(\sum_{|\beta|\leq K} e^{\frac{2|\beta|-3}{2}\tau_0}\Big).
\end{equation*}
Thus, $u_0 \in H^K(\mathbb{R}^3)$. Similarly, we can also  show  that $\lss_0-\bar{\lss}\in H^K(\mathbb{R}^3)$.

\smallskip
\textbf{2.}
According to Step 1 and Theorem \ref{zth1-po}, there exist both $T_*>0$ and 
a unique smooth solution $(\lss,u)$ in $[0,T_*]\times \mathbb{R}^3$ of  \eqref{eq:1.2a}. 
Now we introduce the self-similar variables for some $T>0$:
\[
\tau=-\frac{\log(T-t)}{\Lambda},
\qquad
y=\frac{x}{(T-t)^{\frac1\Lambda}}=e^\tau x.
\]
Denote   $\tau_*=-\frac{\log (T-T_*)}{\Lambda}$. 
According to \eqref{scaling}, the vector function $(\Clss,U)(\tau, y)$ is well-defined 
in $[\tau_0,\tau_*]\times \mathbb{R}^3$.
It is direct to check that $(\Clss,U)$
satisfies the Cauchy problem \eqref{selfsimilar eq}--\eqref{SSfar} pointwise. 
Next, we need to show that \eqref{regularity'}--\eqref{weightenergy} hold.

Let $\beta=(\beta_1,\beta_2,\cdots,\beta_k)$ be a multi-index with 
$\beta_i\in\{1,2,3\}$ for $i=1,\cdots,k$, where $k\in [0,K]$ is an integer.
By the chain rule and the identity $\partial_{y_j}x_i=e^{-\tau}\delta_{ij}$,
one has 
\[
\partial^y_\beta U(\tau, y)
=\Lambda e^{-(\Lambda-1)\tau}e^{-|\beta|\tau}
(\partial^x_\beta u)(t, x),
\qquad x=e^{-\tau}y.
\]
Since $\text{d}y=e^{3\tau}\,\text{d}x$, it follows that
\[
\begin{aligned}
\|\partial^y_\beta U(\tau, \cdot)\|_{L^2_y}^2
&=\Lambda^2 e^{-2(\Lambda-1)\tau}e^{-2|\beta|\tau}
\int|\partial^x_\beta u(t, x)|^2 e^{3\tau}\,\text{d}x \\
&=\Lambda^2 e^{-2(\Lambda-1)\tau}e^{(3-2|\beta|)\tau}
\|\partial^x_\beta u(t, \cdot)\|_{L^2_x}^2 .
\end{aligned}
\]
Summing over all $|\beta|\le K$, we obtain
\[
\|U(\tau, \cdot)\|_{H^K_y}
\le C(K, \tau)\,\|u(t, \cdot)\|_{H^K_x},
\qquad t=T-e^{-\Lambda\tau},
\]
where
\[
C(K, \tau)
=\Lambda e^{-(\Lambda-1)\tau}
\Big(\sum_{|\beta|\le K}e^{\frac{3-2|\beta|}{2}\tau}\Big)
\]
is bounded on the finite interval $[\tau_0,\tau_*]$. 

Define the scaling operator $\mathcal{S}_{\tau}f(y):=f(e^{-\tau}y)$. Then
$$
\|\mathcal S_\tau f\|_{H^K} \leq \Big(\sum_{|\beta|\le K}e^{\frac{3-2|\beta|}{2}\tau}\Big)\|f\|_{H^K}, \qquad\,\,\text{ for all $f\in H^K(\mathbb R^3)$ and $\tau \in [\tau_0, \tau_*]$}. 
$$
Moreover, the family, $\{\mathcal S_\tau\}_{\tau\in\mathbb R}$, satisfies the group property:
$$
\mathcal S_{\tau_2}f(y)=f(e^{-\tau_2}y)=f(e^{-\tau_1-(\tau_2-\tau_1)}y)=\mathcal S_{\tau_2-\tau_1}(\mathcal S_{\tau_1}f(y)),\qquad \mathcal S_0f(y)=f(y),
$$
for all $\tau_1, \tau_2 \in \mathbb{R}$.

Indeed, for $f\in C_c^\infty(\mathbb R^3)$, the map: $\tau\mapsto \mathcal S_\tau f$
is continuous in $H^K(\mathbb R^3)$ by dominated convergence.
Since $C_c^\infty(\mathbb R^3)$ is dense in $H^K(\mathbb R^3)$ and $\mathcal S_\tau$ is locally bounded,
the strong continuity extends to all $f\in H^K(\mathbb R^3)$.
Consequently, $\{\mathcal S_\tau\}_{\tau\in\mathbb R}$ forms a strongly continuous one-parameter group on $H^{K}(\mathbb{R}^3).$
Therefore, for $\tau \in [\tau_0, \tau_*]$
\[
U(\tau,\cdot)=\Lambda e^{-(\Lambda-1)\tau}\mathcal S_{\tau}(u(t(\tau),\cdot))
\qquad \mbox{with $t(\tau):=T-e^{-\Lambda\tau}$}.
\]

Since the map: $t\mapsto u(t,\cdot)$ is continuous in $H^K(\mathbb R^3)$,
the change of variables in time $\tau\mapsto t(\tau)$ is continuous,
and $\{\mathcal S_\tau\}_{\tau\in\mathbb R}$ is a strongly continuous group on $H^K(\mathbb R^3)$,
we conclude 
\[
U\in C([\tau_0,\tau_*];H^K(\mathbb R^3)).
\]
Moreover, it follows from  the relation: $\text{d}t=\Lambda(T-t)\,\text{d}\tau$,
\eqref{scaling}, and  
$u\in L^2([0,T_*];H^{K+1}(\mathbb R^3))$ that 
\[
U\in L^2([\tau_0,\tau_*];H^{K+1}(\mathbb R^3)).
\]

The analogous arguments can be applied to $\Clss$ and the corresponding  time derivatives to 
further obtain
\begin{equation}\label{eq:reg-SS}
\begin{aligned}
&\inf_{(\tau,y)\in[\tau_0,\tau_*]\times\mathbb R^3}\Clss(\tau,y)>0, 
\quad  \Clss-\Clss^*\in C([\tau_0,\tau_*];H^K(\mathbb R^3)),\\
&\partial_\tau \Clss\in C([\tau_0,\tau_*];H^{K-1}(\mathbb R^3)),
 \quad 
 \partial_\tau U\in
C([\tau_0,\tau_*];H^{K-2}(\mathbb R^3))
\cap L^2([\tau_0,\tau_*];H^{K-1}(\mathbb R^3)).
\end{aligned}
\end{equation}

\smallskip
\textbf{3.}
Let $f: \mathbb{R}^{+} \to \mathbb{R}$ be a non-increasing $C^2$ function satisfying
\begin{align*}
f(s)=\begin{cases}
1 & s \in [0,\frac{1}{2}],\\[2pt]
\text{non-negative polynomial} & s \in [\frac{1}{2}, 1],\\[2pt]
e^{-s} & s \ge 1.
\end{cases}
\end{align*}
For any $R>1$, we define $f_R: \mathbb{R}^3 \to \mathbb{R}$ by $f_R(y)=f (\frac{|y|}{R})$. 
Clearly, there exists a constant $C>0$ such that 
\begin{align*}
|f'(s)| \le C f(s).
\end{align*}
Then, for any $p\ge 0$, 
\begin{align}\label{601}
|y|^p f_R(y) \leq C,  \qquad
\lim_{|y|\to \infty} |y|^p f_R(y) =0.
\end{align}

Let $\beta=(\beta_1, \beta_2, \cdots,\beta_K)$ be a multi-index with 
$\beta_i\in\{1,2,3\}$ for $i=1,\cdots,K$. Applying $\partial_\beta$ to \eqref{selfsimilar eq}  
and using the identity: $\partial_\beta ((y \cdot \nabla) f) = K \partial_\beta f + (y \cdot \nabla) \partial_\beta f$, 
we have 
\begin{equation}\label{eq:K-thderivative'}
\begin{aligned}
    &(\partial_\tau +\Lambda-1+K)\partial_{\beta}\Clss +(y\cdot \nabla) \partial_{\beta}\Clss + \partial_{\beta}(U\cdot \nabla \Clss)+\alpha\partial_{\beta}(\Clss\,\dive U)=0,\\[0.3em]
     &(\partial_\tau +\Lambda-1+K)\partial_{\beta}U +(y\cdot \nabla) \partial_{\beta}U + \partial_{\beta}(U\cdot \nabla U)+\alpha\partial_{\beta}(\Clss \nabla \Clss)= \partial_{\beta}\mathcal{F}_{\rm dis},
\end{aligned}
\end{equation}
where 
$$\mathcal{F}_{\rm dis}= C_{\rm{dis}}e^{-\delta_{\rm{dis}}\tau}\big(\Clss^{\frac{\delta-1}{\alpha}}L(U)+\frac{\delta}{\alpha}\Clss^{\frac{\delta-1-\alpha}{\alpha}}\nabla \Clss \cdot \md(U)
        \big).$$
 Now, we define a new energy
\begin{equation}
    E_{K,R}(\tau)= \int 
    \big(| \nabla^K \Clss |^2+ | \nabla^K U |^2  \big) \phi^K f_R  \,\dy= \sum_{|\beta|=K}  \int 
    \big( | \partial_\beta \Clss |^2+| \partial_\beta U |^2 \big) \phi^K f_R \,\dy,
\end{equation}
where $\partial_\beta U=(\partial_\beta U_1, \partial_\beta U_2, \partial_\beta U_3)^\top$, and  the  weight function $\phi \in C^1(\mathbb{R}^3)$ is  defined in \eqref{weightdefi}.

Multiplying each equation of \eqref{eq:K-thderivative'} by $\phi^K f_R \,\partial_\beta \Clss $ and $\phi^K f_R \,\partial_\beta U$ respectively, summing over all $|\beta|=K$, and integrating over $\mathbb{R}^3$, we obtain
\begin{align}\label{eq:EKR}
        &\big(\frac{1}{2}\frac{\mathrm d}{\mathrm d\tau} + \Lambda-1+K\big)E_{K,R} \nonumber \\
       &= -\sum_{|\beta| =K}\int  \big(\partial_{\beta} \Clss\, (y\cdot\nabla) \partial_{\beta} \Clss + \sum_{i=1}^3 \partial_{\beta} U_i\, (y\cdot\nabla) \partial_{\beta}U_i \big)\phi^Kf_R \,\dy
 \nonumber \\
        &\quad -\sum_{|\beta| =K} \int  \big(\partial_{\beta} \Clss \, \partial_{\beta}(U\cdot \nabla \Clss)+ \sum_{i=1}^3 \partial_{\beta} U_i\, \partial_{\beta}(U \cdot \nabla U_i) \big) \phi^Kf_R \,\dy  \\
        &\quad - \sum_{|\beta| =K} \int \alpha \big(\partial_{\beta} \Clss \,\partial_{\beta}(\Clss \dive\, U) + \sum_{i=1}^3 \partial_{\beta} U_i\, \partial_{\beta}(\Clss \partial_{y_i} \Clss) \big)\phi^K f_R \,\dy \nonumber \\
       & \quad +\sum_{|\beta| =K}C_{\rm{dis}}e^{-\delta_{\rm{dis}}\tau}\int \partial_{\beta}U\cdot \partial_{\beta}\Big( \Clss^{\frac{\delta-1}{\alpha}}L(U) +\frac{\delta}{\alpha}\Clss^{\frac{\delta-1}{\alpha}-1}\nabla \Clss \cdot \md(U) \Big)\phi^K f_R\,\dy \nonumber \\
      & := -J_1-J_2-J_3+C_{\rm{dis}}e^{-\delta_{\rm{dis}}\tau}J_4. \nonumber
\end{align}
Integrating by parts,  we obtain
\begin{align}\label{eq:EstJ1}
        -J_1  &=-\frac{1}{2}\sum_{|\beta| =K}\int  y 
        \cdot \big( \nabla |\partial_\beta \Clss|^2 + \sum_{i=1}^{3} \nabla|\partial_\beta U_i|^2\big)\phi^K f_R\,\dy \nonumber \\
        & =\frac{1}{2}\sum_{|\beta| =K}\int \frac{\dive(y \phi^K f_R  )}{\phi^K f_R}\big( |\partial_\beta \Clss|^2 + |\partial_\beta U|^2\big)\phi^K f_R\,\dy \nonumber \\
        & =  \frac{1}{2} \sum_{|\beta| =K} \int \Big(3+\frac{Ky \cdot \nabla\phi }{\phi}+ \frac{|y|}{Rf_R}f'(\frac{|y|}{R})\Big)\big(  |\partial_\beta \Clss|^2 +|\partial_\beta U|^2\big)\phi^K f_R\,\dy  \\
        & \leq C E_{K,R}+ C(K)\int \big(|\nabla^K \Clss|^2 + |\nabla^K U|^2\big)|y|\phi^K f_R \,\dy \nonumber \\
        & \leq CE_{K,R}+ C(K)\|(\Clss, U)\|_{D^K}^2.\nonumber
\end{align}

 We now turn to the estimates for  $J_2$ and $J_3$. 
 First, combining the results obtained in {Steps 1--2}, 
  we conclude that there exists a constant $C(K, \Clss_0, U_0, \tau_*)$ such that $(\Clss, U)$ satisfies
 \begin{align}
     \inf_{(\tau,y)\in[\tau_0,\tau_*]\times\mathbb R^3}\Clss(\tau,y) \ge C^{-1}(K, \Clss_0, U_0, \tau_*), \label{loc:lowerboundQ}\\ \sup_{\tau\in[\tau_0, \tau_*]}\|(\Clss-\Clss^*, U)\|_{H^K} \leq C(K, \Clss_0, U_0, \tau_*)\label{loc:HKbound}.
 \end{align}
 It follows from \eqref{loc:HKbound}, $K\ge 4$, and the Sobolev embedding that
 \begin{align}\label{loc:Linftybound}
    \sum_{j=0}^2\|(\nabla^j \Clss, \nabla^j U)\|_{L^\infty} \leq C(K, \Clss_0, U_0, \tau_*) \qquad \text{for all $\tau \in [\tau_0, \tau_*]$}.
 \end{align}
 Second, in the  Gagliardo-Nirenberg inequality \eqref{eq:GNresult}, denoting
 $$ 
 \varphi = f_R^{\frac{1}{4K}}\phi^{\frac{1}{2}}, \  \ \psi = 1, \ \ p=\infty, \ \ q=2,\ \  l=K, \ \ \bar{r} = \frac{2(K+1)}{\ell}, \ \ \theta = \frac{\ell - 3/\bar{r}}{K - 3/2}, 
 $$ 
 we obtain 
 \begin{equation}\label{interpol1}
    \begin{aligned}
        \|\langle y \rangle^{-\varepsilon}f_R^{\frac{\theta}{4}}\phi^{\frac{K\theta}{2}}\nabla^{\ell} U\|_{L^{\bar{r}}} 
        \leq C(\varepsilon, K)\big(\|U\|_{L^{\infty}}^{1-\theta} \|f_R^{\frac 14}\phi^{\frac K2}\nabla^K U\|_{L^{2}}^{\theta} + \|U\|_{L^{\infty}}\big).
    \end{aligned}
\end{equation}
Then, for $\ell +\ell^* =K+1$ with $\ell,\ \ell^*\ge 1$, and $r^*=\frac{2(K+1)}{\ell^*}$, according to \eqref{601}, \eqref{interpol1},  
and the H\"older inequality, we have
    \begin{align}\label{interpol2}
      &\|f_R^{\frac 1 2}\phi^{\frac K2}|\nabla^{\ell}U||\nabla^{\ell^*}U|\|_{L^2} \nonumber \\
      &=\big\||\nabla^{\ell}U||\nabla^{\ell^*}U|(f_R^{\frac{1}{4}}\phi^{\frac{K}{2}})^ \frac{\ell-3/{\bar{r}}}{K-3/2}(f_R^{\frac{1}{4}}\phi^{\frac{K}{2}})^\frac{\ell^*-3/{r^*}}{K-3/2}\phi^{-\frac{K}{2(K-3/2)}}f_R^{\frac{K-5/2}{4(K-3/2)}}\big\|_{L^2}\nonumber \\
       &  \leq C \big\|\langle y \rangle^{-\frac{1}{10}}(f_R^{\frac{1}{4}}\phi^{\frac{K}{2}})^\frac{\ell-3/{\bar{r}}}{K-3/2}\nabla^{\ell}U\big\|_{L^{\frac{2(K+1)}{\ell}}}  \big\|\langle y \rangle^{-\frac{1}{10}}(f_R^{\frac{1}{4}}\phi^{\frac{K}{2}})^\frac{\ell^*-3/{r^*}}{K-3/2}\nabla^{\ell^*}U\big\|_{L^{\frac{2(K+1)}{\ell^*}}}  \\
        & \leq  C(K,\Clss_0, U_0,\tau_*)+C(K,\Clss_0, U_0,\tau_*)\Big(\int \big(|\nabla^K \Clss|^2 + |\nabla^K U|^2\big)\phi^K f_R^{\frac12} \,\dy\Big)^\frac{K-1/2}{2K-3} \nonumber \\
         & \leq  C(K,\Clss_0, U_0,\tau_*)\big(1+ \|(\Clss, U)\|_{D^K}^\frac{K-1/2}{K-3/2}\big).\nonumber
\end{align}
Therefore, it follows from the H\"older inequality, \eqref{loc:Linftybound}, and \eqref{interpol2} that  
    \begin{align}\label{Est:J21}
        & -\sum_{|\beta| =K} \int f_R\phi^K \sum_{i=1}^3 \partial_\beta 
       (U \cdot \nabla U_i) \partial_\beta U_i \,\dy  \nonumber \\
       &\leq C(K,\Clss_0, U_0,\tau_*) \big\|f_R^{\frac12}\phi^{\frac K2}\nabla^K U\big\|_{L^2}\big(\|U\|_{L^\infty}\big\|f_R^{\frac 12}\phi^{\frac K2}\nabla^{K+1} U\big\|_{L^2} +\|\nabla U\|_{L^\infty}\big\|f_R^{\frac 12}\phi^{\frac K2}\nabla^K U\big\|_{L^2} \big)\nonumber \\
       &\quad  + C(K,\Clss_0, U_0,\tau_*)\big\|f_R^{\frac 12}\phi^{\frac K2}\nabla^K U\big\|_{L^2} \Big(\sum_{\ell,\ell^* \geq 2, \, \ell + \ell^* = K+1}   \big\|f_R^{\frac 12}\phi^{\frac K2}|\nabla^{\ell}U||\nabla^{\ell^*}U|\big\|_{L^2} \Big)  \\
       & \leq C(K,\Clss_0, U_0,\tau_*)\Big( 1+E_{K,R}+ E_{K,R}^{\frac 12} \big( \|f_R^{\frac 12}\phi^{\frac K2}\nabla^{K+1} U\|_{L^2} +\|(\Clss, U)\|_{D^K}^\frac{K-1/2}{K-3/2}\big)\Big).\nonumber
    \end{align}
By the same argument as in \eqref{Est:J21}, we derive 
\begin{equation}
\begin{aligned}
        &- \sum_{|\beta| =K} \int \alpha f_R\phi^K \partial_{\beta} \Clss \,\partial_{\beta}(\Clss\,\dive\, U)  \,\dy  \\
        & \leq C(K,\Clss_0, U_0,\tau_*) \Big(1+E_{K,R}+ E_{K,R}^{\frac 12} \big( \|f_R^{1/2}\phi^{K/2}\nabla^{K+1}U \|_{L^2} +\|(\Clss, U)\|_{D^K}^\frac{K-1/2}{K-3/2}\big)\Big).\label{Est:J31}
\end{aligned}
\end{equation}
Moreover, we obtain from integration by parts, the H\"older inequality, \eqref{loc:Linftybound}, and \eqref{interpol2} that 
    \begin{align}\label{Est:J22}
        & -\sum_{|\beta| =K} \int f_R \phi^K  \partial_\beta 
       (U \cdot \nabla \Clss) \partial_\beta \Clss \,\dy \nonumber \\
        &\leq - \frac{1}{2}\sum_{|\beta| =K}\int f_R \phi^K 
       U \cdot \nabla |\partial_\beta \Clss|^2 \,\dy+ C(K,\Clss_0, U_0,\tau_*)\|\nabla U\|_{L^\infty}\big\|f_R^{\frac 12}\phi^{\frac K2}\nabla^K \Clss\big\|_{L^2}^2 \nonumber \\
       & \quad +C(K,\Clss_0, U_0,\tau_*)\|\nabla \Clss\|_{L^\infty}\big\|f_R^{\frac 12}\phi^{\frac K2}\nabla^K U\big\|_{L^2} \big\|f_R^{\frac 12}\phi^{\frac K2}\nabla^K \Clss\big\|_{L^2}  \nonumber \\
       & \quad + C(K,\Clss_0, U_0,\tau_*)\big\|f_R^{\frac 12}\phi^{\frac K2}\nabla^K \Clss\big\|_{L^2} \Big(\sum_{\ell, \ell^* \geq 2, \ \ell + \ell^* = K+1}   \big\|f_R^{\frac 12}\phi^{\frac K2}|\nabla^{\ell}\Clss||\nabla^{\ell^*} U|\big\|_{L^2} \Big) \\
       & \leq \frac{1}{2} \sum_{|\beta| =K} \int \Big( \dive\,U + K\frac{\nabla \phi \cdot U}{\phi}+\frac{U\cdot y}{R|y|f_R}f'(\frac{|y|}{R})\Big)  |\partial_\beta \Clss|^2 \phi^K f_R\,\text{d}y \nonumber \\
       & \quad\, +C(K,\Clss_0, U_0,\tau_*)\Big(E_{K,R}+E_{K,R}^{\frac{1}{2}}\|(\Clss, U)\|_{D^K}^\frac{K-1/2}{K-3/2}\Big) \nonumber \\
       & \leq C(K,\Clss_0, U_0,\tau_*)\Big( E_{K,R}+E_{K,R}^{\frac{1}{2}}\| (\Clss, U) \|_{D^K}^\frac{K-1/2}{K-3/2}\Big),\nonumber
    \end{align}
and
    \begin{align}\label{Est:J32}
        & -\sum_{|\beta| =K} \int f_R\phi^K  \partial_{\beta} U\cdot \partial_{\beta}(\Clss \nabla \Clss) \,\dy \nonumber \\
       & \leq C(K)\sum_{|\beta| =K} \int \big|\nabla(f_R\phi^K  \partial_{\beta} U)\big|\, \big| \nabla^{K-1}(\Clss \nabla \Clss)\big| \,\dy  \\
       & \leq C(K,\Clss_0, U_0,\tau_*)\big( E_{K,R} + (1+\|(\Clss, U)\|_{D^K}) \big\|f_R^{\frac 12}\phi^{\frac K2}\nabla^{K+1} U\big\|_{L^2} \big).\nonumber
    \end{align}
Consequently, combining \eqref{Est:J21}--\eqref{Est:J32} yields
\begin{equation}\label{Est:J2,3}
\begin{aligned}
    -J_2-J_3 &\leq  C(K,\Clss_0, U_0,\tau_*)\Big(1+E_{K,R}+ E_{K,R}^{\frac 12} \big( \big\|f_R^{\frac 12}\phi^{\frac K2}\nabla^{K+1} U\big\|_{L^2} +\|(\Clss, U)\|_{D^K}^\frac{K-1/2}{K-3/2}\big)\Big) \\
    &\quad +\, C(K,\Clss_0, U_0,\tau_*)\big(1+\|(\Clss, U)\|_{D^K}\big) \|f_R^{\frac 12}\phi^{\frac K2}\nabla^{K+1} U\|_{L^2}.
\end{aligned}
\end{equation}
To estimate $J_4$, we decompose it into five terms:
    \begin{align}\label{EqJ4}
        J_4 =& \sum_{|\beta| =K}\int \partial_{\beta}U \cdot  \partial_{\beta}\Big( \Clss^{\frac{\delta-1}{\alpha}}L(U) +\frac{\delta}{\alpha}\Clss^{\frac{\delta-1}{\alpha}-1}\nabla \Clss \cdot \md(U) \Big)\phi^Kf_R \,\dy \nonumber \\
        =& \sum_{|\beta| =K} \int \Clss^{\frac{\delta-1}{\alpha}}\partial_{\beta}U\cdot  \partial_\beta L(U)\phi^K f_R \,\dy\,+ \sum_{|\beta| =K} \int \partial_{\beta}U\cdot  \big([\partial_\beta,\Clss^{\frac{\delta-1}{\alpha}}] L(U)\big)\phi^K f_R \,\dy \nonumber\\
         &+\frac{\delta}{\alpha}\sum_{|\beta| =K}\int  \partial_\beta U \cdot  \Big([\partial_\beta, \Clss^{\frac{\delta-1}{\alpha}-1}] \big(\nabla \Clss \cdot \md(U)\big)\Big) \phi^K f_R \,\dy \nonumber\\
        &+\frac{\delta}{\alpha}\sum_{|\beta| =K}\int  \partial_\beta U \cdot \big( \Clss^{\frac{\delta-1}{\alpha}-1}(\partial_{\beta}\nabla \Clss) \cdot \md(U)\big) \phi^K f_R \,\dy \\
        &+\frac{\delta}{\alpha}\sum_{|\beta| =K}\int  \partial_\beta U \cdot \Clss^{\frac{\delta-1}{\alpha}-1}\, \big( \partial_\beta(\nabla \Clss\cdot \md(U))-(\partial_{\beta}\nabla \Clss) \cdot \md(U)\big) \phi^K f_R \,\dy \nonumber\\
        :=& \, J_{41}+J_{42}+J_{43}+J_{44}+J_{45}.\nonumber
    \end{align}
For $J_{41}$, it follows from integration by parts, \eqref{601}, \eqref{loc:lowerboundQ}, \eqref{loc:Linftybound}, and the H\"older inequality that  
\begin{align}\label{Est:J41}
    J_{41}&=  \sum_{|\beta| =K} \int \Clss^{\frac{\delta-1}{\alpha}}\partial_{\beta}U\cdot  \partial_\beta L(U)\phi^K f_R \,\dy \nonumber\\
    & =-\sum_{|\beta| =K} \int   \Clss^{\frac{\delta-1}{\alpha}}\big(a_1|\nabla \partial_\beta U|^2+(a_1+a_2)|\dive \partial_\beta U|^2\big) \phi^K f_R\,\dy \nonumber\\
    & \quad -\sum_{|\beta| =K} \int \big(\nabla(\Clss^{\frac{\delta-1}{\alpha}}\phi^K\,f_R) \cdot \partial_\beta U\big)\cdot \partial_{\beta}\md(U) \,\dy \nonumber\\
    & \leq -\sum_{|\beta| =K} \int   \Clss^{\frac{\delta-1}{\alpha}}\big(a_1|\nabla \partial_\beta U|^2+(a_1+a_2)|\dive \partial_\beta U|^2\big) \phi^K f_R \,\dy \\
    & \quad +C\sum_{|\beta| =K} \int \Clss^{\frac{\delta-1}{\alpha}}\big((\frac{\nabla \Clss}{\Clss}+\frac{K\nabla \phi}{\phi}+ \frac{y}{R|y|f_R}f'(\frac{|y|}{R}))\cdot \partial_\beta U\big)\cdot \partial_{\beta}\md(U)  \phi^K f_R \,\dy \nonumber\\
    & \leq -\int   \Clss^{\frac{\delta-1}{\alpha}}\big(a_1|\nabla^{K+1} U|^2+(a_1+a_2)|\dive \nabla^K U|^2\big) \phi^K f_R \,\dy \nonumber\\
    &\quad +C(K, \Clss_0, U_0, \tau_*)E_{K,R}^{\frac{1}{2}}\|f_R^{\frac 12}\phi^{\frac K2}\nabla^{K+1}U\big\|_{L^2}.\nonumber
\end{align}

To facilitate the estimates of the lower-order derivative terms, we introduce 
a family of multi-indices $\boldsymbol{\bar\beta}^i$, $i=0,1,\cdots\!,\ell$ for some   integer  $\ell\in [1,K]$, which satisfy 
\begin{equation}\label{range:barbetai}
\begin{split}
&|\boldsymbol{\bar\beta}^0| + | \boldsymbol{\bar\beta}^1 |  + \cdots + | \boldsymbol{\bar\beta}^\ell |= K+2,
\qquad 1\le |\boldsymbol{\bar\beta}^i|\le K-1 \,\,\,\,\, \text{for all $0\leq i \leq \ell$}.
\end{split}
\end{equation}
Here $|\boldsymbol{\bar\beta}^i|$ denotes the length of multi-indices $\boldsymbol{\bar\beta}^i$ as defined in \S \ref{operator} for each $i=0,1,\cdots\!,\ell$. 

It follows from the H\"older inequality, \eqref{loc:lowerboundQ}, and  \eqref{interpol1}  that 
\begin{equation}
\begin{aligned}\label{Est:lowerorder}
        &\big\|\Clss^{\frac{\delta-1}{\alpha}}\,\partial_{\boldsymbol{\bar\beta}^0} U \,\frac{\partial_{\boldsymbol{\bar\beta}^1} \Clss}{\Clss} \frac{\partial_{\boldsymbol{\bar\beta}^2} \Clss}{\Clss} \cdots \frac{\partial_{\boldsymbol{\bar\beta}^\ell} \Clss}{\Clss}  \phi^{\frac K2}f_R^{\frac 12} \big\|_{L^2} \\
        & \leq  C(K,\Clss_0, U_0,\tau_*)\Big\|\langle y \rangle^{-\frac{|\boldsymbol{\bar\beta}^0|}{2(K+2)}}\big(f_R^{\frac{1}{4}}\phi^{\frac{K}{2}}\big)^\frac{|\boldsymbol{\bar\beta}^0|-3/r_{0}}{K-3/2}\nabla^{|\boldsymbol{\bar\beta}^0|}U\Big\|_{L^{\frac{2(K+2)}{|\boldsymbol{\bar\beta}^0|}}} \\
        &\quad \times \prod_{1\leq j\leq \ell}\Big\|\langle y \rangle^{-\frac{|\boldsymbol{\bar\beta}^j|}{2(K+2)}}\big(f_R^{\frac{1}{4}}\phi^{\frac{K}{2}}\big)^\frac{|\boldsymbol{\bar\beta}^j|-3/r_{j}}{K-3/2}\nabla^{|\boldsymbol{\bar\beta}^j|}\Clss\Big\|_{L^{\frac{2(K+2)}{|\boldsymbol{\bar\beta}^j|}}}\\
        & \leq C(K,\Clss_0, U_0,\tau_*)+C(K,\Clss_0, U_0,\tau_*)\|(\Clss, U)\|_{D^K}^\frac{K+1/2}{K-3/2},
        \end{aligned}
    \end{equation}
where $r_j = \frac{2(K+2
)}{|\boldsymbol{\bar\beta}^j|}, j=0,1, \cdots,\ell$.
Then, based on the H\"older inequality, \eqref{loc:lowerboundQ}, \eqref{loc:Linftybound}, and \eqref{range:barbetai}--\eqref{Est:lowerorder}, we have
\begin{align}\label{Est:J42}
    J_{42} &=  \sum_{|\beta| =K} \int \partial_{\beta}U\cdot  \big([\partial_\beta,\Clss^{\frac{\delta-1}{\alpha}}] L(U)\big)\phi^K f_R \,\dy \nonumber  \\
    &\leq C(K,\Clss_0, U_0,\tau_*)E_{K,R}
       +C(K,\Clss_0, U_0,\tau_*)
       E_{K,R}^{\frac 12}\big\|f_R^{\frac 12}\phi^{\frac K2}\nabla^{K+1}U\big\|_{L^2}\nonumber\\[1mm]
    & \quad +C(K,\Clss_0, U_0,\tau_*)E_{K,R}^{\frac 12}\,\,\max_{\substack{\sum_{j=0}^{\ell}|\boldsymbol{\bar\beta}^j|
    = K+2 \\  | \boldsymbol{\bar\beta}^0 |\ge 3 }} \Big\|\Clss^{\frac{\delta-1}{\alpha}} \partial_{\boldsymbol{\bar\beta}^0} U \,\frac{\partial_{\boldsymbol{\bar\beta}^1} \Clss}{\Clss} \frac{\partial_{\boldsymbol{\bar\beta}^2} \Clss}{\Clss} \cdots \frac{\partial_{\boldsymbol{\bar\beta}^\ell} \Clss}{\Clss}  \phi^{\frac K2}f_R^{\frac 12} \Big\|_{L^2}\Big) \nonumber\\
    &\leq C(K,\Clss_0, U_0,\tau_*)\Big(1+E_{K,R} +E_{K,R}^{\frac 12}\big(\big\|f_R^{\frac 12}\phi^{\frac K2}\nabla^{K+1}U\big\|_{L^2}+\|(\Clss, U)\|_{D^K}^\frac{K+1/2}{K-3/2}\big)\Big).
\end{align} 
Following the argument used in \eqref{Est:J42}, we obtain from the H\"older inequality, \eqref{loc:lowerboundQ}, \eqref{loc:Linftybound}, and \eqref{range:barbetai}--\eqref{Est:lowerorder} that
\begin{align}\label{Est:J43}
    J_{43} &=   \frac{\delta}{\alpha}\sum_{|\beta| =K}\int  \partial_\beta U \cdot  \big([\partial_\beta, \Clss^{\frac{\delta-1}{\alpha}-1}] (\nabla \Clss \cdot \md(U))\big) \phi^K f_R \,\dy \nonumber\\
    &\leq C(K,\Clss_0, U_0,\tau_*)\Big(E_{K,R} 
    +E_{K,R}^{\frac 12}\,\,\max_{ \sum_{j=0}^\ell|\boldsymbol{\bar\beta}^j|
   = K+2 } \Big\|\Clss^{\frac{\delta-1}{\alpha}} \partial_{\boldsymbol{\bar\beta}^0} U \frac{\partial_{\boldsymbol{\bar\beta}^1} \Clss}{\Clss} \frac{\partial_{\boldsymbol{\bar\beta}^2} \Clss}{\Clss} \cdots \frac{\partial_{\boldsymbol{\bar\beta}^\ell} \Clss}{\Clss}  \phi^{\frac K2} f_R^{\frac 12}\Big\|_{L^2}\Big) \nonumber \\
    &\leq C(K,\Clss_0, U_0,\tau_*)\Big(1+ E_{K,R}+E_{K,R}^{\frac 12}\big(\big\|f_R^{\frac 12}\phi^{\frac K2}\nabla^{K+1}U\big\|_{L^2}+\|(\Clss, U)\|_{D^K}^\frac{K+1/2}{K-3/2}\big)\Big). 
    \end{align} 
For $J_{44}$, it follows from integration  by parts,  \eqref{loc:lowerboundQ}, \eqref{loc:Linftybound}, and the H\"older inequality that 
\begin{align}\label{Est:J44}
J_{44}&= \frac{\delta}{\alpha}\sum_{|\beta| =K}\int  \partial_\beta U \cdot \big( \Clss^{\frac{\delta-1}{\alpha}-1}(\partial_{\beta}\nabla \Clss) \cdot \md(U)\big) \phi^K f_R \,\dy \nonumber\\
&=-\frac{\delta}{\alpha}\sum_{|\beta| =K}\int \partial_{\beta} \Clss \,\dive \big( \partial_\beta U \cdot \md(U) \Clss^{\frac{\delta-1}{\alpha}-1}\phi^K f_R\big)  \,\dy \nonumber\\
& \leq C(K,\Clss_0, U_0,\tau_*)\Big( E_{K,R}+E_{K,R}^{\frac 12}\big\|f_R^{\frac 12}\phi^{\frac K2}\nabla^{K+1}U\big\|_{L^2}\Big).
\end{align}
For $J_{45}$, let $\ell=2$ in \eqref{range:barbetai}. Then it follows from \eqref{loc:lowerboundQ}, \eqref{loc:Linftybound}, 
and \eqref{Est:lowerorder} that
\begin{align}\label{Est:J45}
    J_{45}&=\frac{\delta}{\alpha}\sum_{|\beta| =K}\int  \partial_\beta U \cdot \Clss^{\frac{\delta-1}{\alpha}-1}\, \big( \partial_\beta(\nabla \Clss\cdot \md(U))-(\partial_{\beta}\nabla \Clss) \cdot \md(U)\big) \phi^K f_R \,\dy \nonumber\\
    &\leq C(K, \Clss_0, U_0, \tau_*)\Big( E_{K,R}+E_{K,R}^{\frac 12}\big\|f_R^{\frac 12}\phi^{\frac K2}\nabla^{K+1}U\big\|_{L^2}\Big)\\
    & \quad +C(K, \Clss_0, U_0, \tau_*)E_{K,R}^{\frac 12}\max_{\substack{ |\boldsymbol{\bar\beta}^0| + | \boldsymbol{\bar\beta}^1 | = K+2 \\ 3\le | \boldsymbol{\bar\beta}^0 |\leq K-1}} \Big\|\Clss^{\frac{\delta-1}{\alpha}} \partial_{\boldsymbol{\bar\beta}^0} U \frac{\partial_{\boldsymbol{\bar\beta}^1} \Clss}{\Clss} \phi^{\frac K2} f_R^{\frac 12}\Big\|_{L^2}\nonumber\\
    & \leq  C(K,\Clss_0, U_0,\tau_*)\Big(1+ E_{K,R}+E_{K,R}^{\frac 12}\big(\big\|f_R^{\frac 12}\phi^{\frac K2}\nabla^{K+1}U\big\|_{L^2}+\|(\Clss, U)\|_{D^K}^\frac{K+1/2}{K-3/2}\big)\Big).\nonumber
\end{align}
Consequently, substituting \eqref{eq:EstJ1}, \eqref{Est:J2,3}, and \eqref{Est:J41}--\eqref{Est:J45} into \eqref{eq:EKR}, together with \eqref{loc:HKbound} and the Young inequality, yields
\begin{equation}
\begin{aligned}
    \frac{\mathrm d}{\mathrm d\tau} E_{K,R}+ \|f_R^{\frac 12}\phi^{\frac K2}\nabla^{K+1}U\|_{L^2}^2 &\leq C(K,\Clss_0, U_0,\tau_*)\Big(1+E_{K,R}+ \|(\Clss, U)\|_{D^K}^\frac{2K+1}{K-3/2}\Big)\\
    & \leq C(K,\Clss_0, U_0,\tau_*)\big(1+E_{K,R}\big).
\end{aligned}
\end{equation}
Applying the Gr\"onwall inequality, we deduce 
$$
E_{K,R}(\tau)\leq e^{C(K,\Clss_0, U_0,\tau_*)(\tau-\tau_0)}E_{K,R}(\tau_0)
+e^{C(K,\Clss_0, U_0,\tau_*)(\tau-\tau_0)}-1.
$$
By the definition of $E_{K, R}$, we see that $E_{K, R}(\tau_0)\leq E_{K}(\tau_0)$. Therefore, for any $R>1$,
\begin{equation}\label{EKRbound}
E_{K,R}(\tau)\leq e^{C(K,\Clss_0, U_0,\tau_*)(\tau-\tau_0)}E_{K}(\tau_0)
+e^{C(K,\Clss_0, U_0,\tau_*)(\tau-\tau_0)}-1.
\end{equation}
Moreover, since 
$$|\nabla^K U|^2\phi^K f_R(y)  \to |\nabla^K U|^2\phi^K \qquad \mbox{as $R\to\infty$}$$
for all $y\in \mathbb{R}^3$,  Fatou's lemma yields
\begin{align}\label{EKbound}
\text{ess}\!\!\!\sup_{\tau_0\leq \tau \leq \tau_*} E_K(\tau) 
\le \text{ess}\!\!\!\sup_{\tau_0\leq t \leq \tau_*}  \liminf_{R \to \infty}E_{K,R}(\tau) \le C(K,\Clss_0, U_0,\tau_*).
\end{align}

This completes the proof of Theorem \ref{Local(U,S)}.
\end{proof}

\section{Choices of the Related Parameters and 
Assumptions in Bootstrap Argument} \label{Section4}
In \S \ref{Section4}--\S \ref{Section6},
we 
establish 
the global-in-time energy estimates for the smooth 
solutions of the Cauchy problem \eqref{selfsimilar eq}--\eqref{SSfar},
restricted to a finite-codimensional manifold of smooth initial data, via a bootstrap argument. 
Let $(\Clss,U)$ be the unique smooth solution 
of the Cauchy problem \eqref{selfsimilar eq}--\eqref{SSfar} on $[\tau_0,\tau_*]\times\mathbb{R}^3$,
whose 
local well-posedness has been established in Theorem \ref{Local(U,S)}. 
In this section, we first fix several parameters determined by 
profile $(\overline{\Clss}, \overline{U})$, 
which is the smooth spherically symmetric solution to the steady Euler equations  
constructed in {\rm Lemma \ref{lem:existofselfsimilar}} (see {\rm Appendix \ref{appendix B}}).
We then introduce the \textit{a priori}  assumptions in the bootstrap argument.

\subsection{Choices of the Related Parameters}

First, in the rest of this paper, we always assume that 
$1< \gamma <1+\frac{2}{\sqrt{3}}$, $\Lambda \in (1, \Lambda^*(\gamma))$ is introduced in {\rm Lemma~\ref{thm:existence_profiles}}, and $\delta\in (0,\frac{1}{2})$ satisfies $\delta_{\rm dis}>0$.
We emphasize that the local well-posedness result established earlier
only imposes a requirement on parameter $K$, \textit{i.e.} $K\ge 4$.

Second, we choose some related  parameters as follows{\rm:}
\begin{equation}\label{Para-ineq}
    \frac{1}{\tau_0} \ll \smallc_0^{\frac{3}{2}} \ll \smallc_1 \ll \smallc_0 \ll\frac{1}{E} \ll \frac{1}{K} \ll \frac{1}{m}\ll \eta \ll \smallc_g = \frac{25}{12}\varepsilon \ll \delta_{\rm{dis}} = O(1),
\end{equation}
and parameter $K$ is also  chosen sufficiently large so that
\begin{equation}\label{Para-ineq1}
    \quad   \frac{1}{K} \ll \tilde{\eta},  
\end{equation}
where  $\varepsilon>0$ is the small parameter appearing in Theorem~\ref{globalexist}, $\widetilde{\eta}$ is the constant appearing in Lemma~\ref{prop:maxdissmooth}, and $\eta$ arises from the weight function
defined in \eqref{weightdefi}. These parameters $\varepsilon, K, E, \smallc_0$, and $\smallc_1$ are chosen with the hierarchy stated in Remark~\ref{choiceofparameters}. In particular, we take 
$$
K > \max\{3^m\, C(J,m), \,500\},
$$
which ensures $K\ge 4$. Therefore, the above assumptions automatically satisfy 
the local existence conditions,
and the local solution constructed previously remains valid.

Next, we choose parameter $R_0>0$ appearing in 
\eqref{weightdefi} sufficiently large so that 
\begin{equation}\label{Rochoice estimate}
\begin{split}
\overline{\Clss}(y)\geq 2\smallc_0  &\qquad\text{ for all $y \in B(0, R_0)$,}\\
|\nabla \overline{\Clss}(y)|+ |\nabla \overline{U}(y)|\leq \frac{\smallc_1}{2}, \,\,\,
\max\{\overline{\Clss}(y), |\overline{U}|(y)\}\leq C\smallc_0 &\qquad\text{ for all $y \in B^c(0, R_0).$}
\end{split}
\end{equation} 
The existence of $R_0$ follows from $\smallc_0^{\frac{3}{2}}\ll \smallc_1\ll \smallc _0 \ll 1 $ 
and the decay of profile \eqref{profile decay}. Therefore, $R_0$
may be chosen so that
\begin{align}\label{choiceR0}
C^{-1} R_0^{1-\Lambda}\leq \smallc_0 \leq C R_0^{1-\Lambda}.
\end{align}
Moreover,  the initial time $\tau_0$ is chosen sufficiently
large so as to ensure
$$
C(\smallc_0)\leq e^{\delta_{\rm dis}\tau_0} .
$$

Finally, parameter $C_0$ is determined in the following lemma, which plays a key role in the definition of the Hilbert space $X$  in \S \ref{X space}.

\begin{Lemma}\label{C0choice}
Let parameter  $C_0$  be sufficiently large, which  depends only on $\Lambda, \alpha$,
the self-similar profile $(\overline{\Clss}, \overline{U})$, and the smooth cutoff function $\widehat{X}$.
Then  profile $(\overline{\Clss},\overline{U})$ satisfies the following conditions for all $\tau\geq \tau_0 ${\rm:}  
\begin{align}\label{profilecondition1}
|\alpha\widehat{X}\overline{\Clss}|+ |\widehat{X}\overline{U}|+ |\alpha\nabla(\widehat{X}\overline{\Clss})|+ |\nabla(\widehat{X}\overline{U})|&\leq\frac{1}{100 C_1} \quad\, \text{ for $y \in B^c(0, C_0)$},\\
\label{profilecondition4}
\|\nabla(\widehat{X}\overline{\Clss})\|_{L^{\infty}(B^c(0, C_0))}+\|\nabla(\widehat{X}\overline{U})\|_{L^{\infty}(B^c(0, C_0))} &\leq \frac{1}{ 20000}, \\
\label{profilecondition2}
\|\nabla^{2}(\widehat{X}\overline{\Clss})\|_{L^{8}(B^c(0, C_0))}+\|\nabla^2(\widehat{X}\overline{U})\|_{L^{8}(B^c(0, C_0))} &\leq \frac{1}{ 20000},\\
\label{profilecondition3}
\sum_{j=3}^{5}(\|\nabla^j(\widehat{X}\overline{\Clss})\|_{L^{2}(B^c(0, C_0))}+\|\nabla^j(\widehat{X}\overline{U})\|_{L^{2}(B^c(0, C_0))})&\leq \frac{1}{ 20000},
\end{align}
where $C_1$ is determined by \begin{equation}\label{parameterchosen}
\frac{1}{100}=\frac{32}{\Lambda-1}\frac{1}{C_1}\Big(\frac{1}{100}\Big)^{\frac{1}{20}}.
\end{equation}
\end{Lemma}
\begin{proof}
    Since $|\widehat{X}|\leq 1$ and $|\nabla \widehat{X}|\leq 2e^{-\tau}$, \eqref{profilecondition1} and \eqref{profilecondition4} follow directly from the decay properties of $(\overline{\Clss},\overline{U})$ stated in \eqref{profile decay}.

    We next verify \eqref{profilecondition2} and \eqref{profilecondition3}. For those terms in which no derivative falls on the cutoff function $\widehat{X}$, the desired decay again follows immediately from \eqref{profile decay}.
     Moreover, from the definition of $\widehat{X}$, we have
    \begin{equation*}
      \text{supp} (\nabla\widehat{X}) \subset B ( 0, e^\tau ) \cap B^c \big( 0,\frac{e^\tau}{2} \big), \qquad   |\nabla^{j}\widehat{X}|\leq C(j) e^{-j\tau}.
    \end{equation*}
    Then, for $|\nabla^j\widehat{X}||\nabla^k \overline{\Clss}|$, 
    \begin{align*}
        \||\nabla^j\widehat{X}||\nabla^k \overline{\Clss}|\|^q_{L^q(\mathbb{R}^3)} 
        & \leq C(j,k,q)e^{-jq\tau }\int_{\frac{1}{2}e^\tau}^{e^\tau} r^{-q(\Lambda-1)-qk+2}\,\mathrm{d}r \\
        & \leq C(j,k,q)e^{-jq\tau }e^{(-q(\Lambda-1)-qk+3)\tau}\\
        & \leq C(j,k,q)e^{-(2q-3)\tau} \leq C(j,k,q)e^{-\tau}.
    \end{align*}
    Since $j\ge 1$, $k\ge 0$, $2\leq j+k \leq 5$,  $q\in \{2,8\}$, 
    and $\tau \ge \tau_0 \gg 1$, 
    those terms with at least one derivative on $\widehat{X}$ are sufficiently small.
\end{proof}
 Here the $L^{8}$--estimate \eqref{profilecondition2} will be used in the energy estimate of 
 $$\max\big\{ \|\nabla^4\widetilde{\Clss}\|_{L^{2}(B^c(0, C_0))}, \   \|\nabla^4\widetilde{U}\|_{L^{2}(B^c(0, C_0))}\big\}$$
through the  H\"older inequality.

\subsection{Assumptions in the Bootstrap Argument}
\label{subsection4.2}

We first  introduce  the initial perturbation
$$
(\widetilde{\Clss}_0, \widetilde{U}_0):=(\Clss_0-\widehat{X}\overline{\Clss}, \, U_0-\widehat{X}\overline{U}),
$$
which will be used throughout the subsequent analysis, 
where $\widehat{X}$ denotes the cutoff function introduced in \eqref{cutoff function}.

According to Lemma~\ref{prop:maxdissmooth}, the Hilbert space $X$ admits the direct sum decomposition
\begin{equation*}
 X = V_{\rm sta}\oplus V_{\rm uns},   
\end{equation*}
where $V_{\rm sta}$ and $V_{\rm uns}$ denote the stable and unstable subspaces associated with the truncated linear operator $\mathcal L$, respectively.

Since space $X$ consists of functions supported in ball $B(0,3C_0)\subset\mathbb R^3$, whereas perturbation
$(\widetilde{\Clss}_0,\widetilde U_0)$ is defined on the whole space
$\mathbb R^3$, we first localize the perturbation by using the cutoff
function $\chi_2$ that is defined in \S \ref{othernotation}. We then write the initial perturbation
$(\widetilde{\Clss}_0,\widetilde U_0)$ as the sum of a stable part $( \widetilde{\Clss}_0^*,\widetilde U_0^*)$
and an unstable component, where the stable part is chosen such that
\begin{equation}\label{eq:tilde_is_stable}
(\chi_2 \widetilde{\Clss}_0^*,\,\chi_2 \widetilde U_0^*) \in V_{\rm sta}.
\end{equation}

On the other hand, the unstable subspace $V_{\rm uns}$ is finite-dimensional.
Let $\{(\varphi_{i,\subclss},\varphi_{i,u})\}_{i=1}^N$ be an orthonormal basis of $V_{\rm uns}$ constructed in Lemma~\ref{prop:maxdissmooth} in Appendix~\ref{appendix C}. The initial perturbation can be expressed as
\begin{align}\label{eq:tildeplusunstable}
\begin{cases}
\displaystyle
\widetilde{\Clss}_0=\widetilde{\Clss}_0^{*}+\sum_{i=1}^{N}\hat{k}_i
\varphi_{i,\subclss},\\
\displaystyle
\widetilde{U}_0=\widetilde{U}_0^{*}+\sum_{i=1}^{N}\hat{k}_i
\varphi_{i,u},
\end{cases}
\end{align}
where coefficients $\{\hat{k}_i\}$ parameterize the unstable component of the initial data and will be chosen later so as to control the growth of the unstable modes.  

Moreover, the stable component
$(\widetilde{\Clss}_0^*,\widetilde U_0^*)$
is chosen such that the full initial perturbation
$(\widetilde{\Clss}_0,\widetilde U_0)$
satisfies the following conditions{\rm:}
\begin{equation} \label{initialdatacondition1}
\begin{split}
\max\big\{\|\widetilde{\Clss}_0\|_{L^{\infty}},\,\|\widetilde{U}_0\|_{L^{\infty}}, \,\|\chi_2\widetilde{\Clss}^*_0\|_{X},\,\|\chi_2\widetilde{U}^*_0\|_{X}\big\}\leq  \smallc_1, \\
E_{K}(\tau_0)\leq\frac{E}{2},  \quad\, \widetilde{\Clss}_0+\widehat{X}\overline{\Clss}\geq  \frac{\smallc_1}{2} \big\langle\frac{y}{R_0} \big\rangle^{1-\Lambda}.
\end{split}
\end{equation}
In addition, the initial data are assumed to satisfy the spatial decay condition:
\begin{equation}\label{initialdatacondition2}
\langle y\rangle^{\Lambda}\big(|\nabla(\widetilde{\Clss}_0+\widehat{X} \overline{\Clss})|+|\nabla(\widetilde{U}_0+\widehat{X} \overline{U})|\big) < \infty.
\end{equation}

In order to use the analysis for the truncated linear operator 
$$
\mathcal{L}=(\mathcal{L}_\subclss, \mathcal{L}_u) = \chi_2 (\mathcal{L}_\subclss^{e}, \mathcal{ L}_u^{e}) - J(1-\chi_1),
$$ 
we consider an additional system, which we call the \textit{truncated equations}, \textit{i.e.},  $\eqref{eq:truncated}_1$--$\eqref{eq:truncated}_2$. Now we consider the corresponding  Cauchy problem:
\begin{equation}\label{eq:truncated}
\begin{cases}
\displaystyle
\partial_\tau  \overset{\approx}{\Clss}= \mathcal L_\subclss(\overset{\approx}{\Clss},\overset{\approx}{U}) + \chi_2 \mathcal F_\subclss(\widetilde{\Clss},\widetilde{U}),  \\
\displaystyle
\partial_\tau \overset{\approx}{U} = \mathcal L_u(\overset{\approx}{\Clss},\overset{\approx}{U}) + \chi_2 \mathcal F_u(\widetilde{\Clss},\widetilde{U}),\\
\displaystyle
(\overset{\approx}{\Clss},\overset{\approx}{U})(\tau_0, y)= (\overset{\approx}{\Clss}_{0},\overset{\approx}{U}_{0})(y)\qquad \text{for $y \in \mathbb{R}^3$},\\
(\overset{\approx}{\Clss},\overset{\approx}{U})(\tau, y)\to (0, 0)\qquad \ \ \ \   \ \ \ \ \  
\text{as $|y|\to \infty$ \,  for $\tau \ge \tau_0$},
\end{cases}
\end{equation} 
with $\mathcal F_\subclss = \mathcal{N}_\subclss + \mathcal{E}_\subclss$ and $\mathcal F_u = \mathcal{N}_u + \mathcal{E}_u + \mathcal{F}_{\rm dis}$, 
where $(    
\overset{\approx}{\Clss}_{0},
\overset{\approx}{U}_{0})$ are given  as follows{\rm:}
\begin{align}\label{initialdatatruncated}
\begin{cases}
\displaystyle
\overset{\approx}{\Clss}_{0}=\chi_2\widetilde{\Clss}_0^{*}+\sum_{i=1}^{N}\hat{k}_i
\varphi_{i,\subclss},\\
\displaystyle
\overset{\approx}
{U}_{0}=\chi_2\widetilde{U}_0^{*}+\sum_{i=1}^{N}\hat{k}_i\varphi_{i,u}.
\end{cases}
\end{align}

\begin{Remark}\label{local-truncation}
    Equations \eqref{eq:truncated}$_1${\rm--}\eqref{eq:truncated}$_2$ constitute a first-order symmetric hyperbolic system. Since {\rm Theorem~\ref{Local(U,S)}} ensures that $(\widetilde{\Clss},\widetilde{U})=(\Clss-\widehat{X}\overline{\Clss}, \, U-\widehat{X}\overline{U})$ is a smooth vector function on 
    interval $[\tau_0, \tau_*]$, the standard well-posedness theory for symmetric hyperbolic systems {\rm(}see {\rm \S 7.3} in \cite{evans}{\rm)} implies that  the Cauchy problem  \eqref{eq:truncated}{\rm--}\eqref{initialdatatruncated} admits a unique smooth solution $(\overset{\approx}{\Clss},\overset{\approx}{U})$ in $[\tau_0, \tau_*]\times \mathbb{R}^3$.
\end{Remark}

An important relation between $(\overset{\approx}{\Clss}, \overset{\approx}{U})$ and $(\widetilde{\Clss},\widetilde{U})$ is observed  in the following lemma:

\begin{Lemma}\label{truncatedestimate} 
Suppose that $(\overset{\approx}{\Clss}, \overset{\approx}{U})$ is a smooth solution to  \eqref{eq:truncated} 
in $[\tau_0, \tau_*]\times \mathbb{R}^3$ obtained in {\rm Remark~\ref{local-truncation}}. 
If, at $\tau = \tau_0$, 
$$
(\overset{\approx}{\Clss}_0, \overset{\approx}{U}_0) = (\widetilde{\Clss}_0,\widetilde{U}_0) \qquad \text{for all $y \in B(0, C_0)$},
$$ 
then 
$$
(\overset{\approx}{\Clss}, \overset{\approx}{U}) = (\widetilde{\Clss},\widetilde{U})\qquad 
\text{  for all $(\tau, y)\in [\tau_0, \tau_*]\times B(0, C_0)$}.
$$
\end{Lemma}

\begin{proof}
     Denote 
     $$(\delta_\Clss, \delta_U)=(\widetilde{\Clss}-\overset{\approx}{\Clss}, \,\widetilde{U}-\overset{\approx}{U}).$$
    Then we have
    \begin{equation}\label{eq:difference}
    \begin{aligned}
       \partial_\tau \delta_\Clss+(\Lambda-1)\delta_\Clss &= -(y+\widehat{X}\overline{U})\cdot \nabla\delta_\Clss-\alpha(\widehat{X}\overline{S})\,\dive \, \delta_U-\delta_U\cdot\nabla(\widehat{X}\overline{\Clss}) -\alpha \delta_\Clss\,\dive(\widehat{X}\overline{U}),\\
       \partial_\tau\delta_U +(\Lambda-1)\delta_U &= -(y+\widehat{X}\overline{U})\cdot \nabla\delta_U-\alpha(\widehat{X}\overline{\Clss})\nabla \delta_\Clss-\delta_U\cdot\nabla(\widehat{X}\overline{U}) -\alpha \delta_\Clss\nabla(\widehat{X}\overline{\Clss})
       \end{aligned}
    \end{equation}
    in a ball $B(0, C_0)$ and $(\delta_\Clss, \delta_U)|_{\tau=\tau_0}=(0, 0)$. 

    Notice that profile $(\overline \Clss(y), \overline U(y))$ is spherically symmetric.
    Then there is a scalar function $\overline{\mathcal{U}}(|y|)$ such that 
    \begin{equation}\label{radialrep}
(\overline \Clss, \overline U)(y)=(\overline{\Clss}(|y|),\,\overline{\mathcal{U}}(|y|)\frac{y}{|y|}).
\end{equation}
 Multiplying \eqref{eq:difference}$_1$ and \eqref{eq:difference}$_2$ by $\delta_\Clss$ and $\delta_{U}$ respectively, summing them up and integrating over $B(0, C_0)$, 
it follows from \eqref{radialrep}
and the main properties of $(\overline \Clss(y), \,\overline U(y), \,\overline{\mathcal{U}}(|y|))$ 
shown in Lemma \ref{prop:profiles-R}  that 
\begin{align*}
        &\frac{1}{2}\frac{\text{d}}{\text{d}\tau}\int_{B(0, C_0)} \big(\delta_\Clss^2 +|\delta_U|^2 \big)\,\text{d}y\\
        &= \int_{B(0, C_0)}\Big(-\frac{1}{2}(y+\overline{U})\cdot\nabla(\delta_\Clss^2)-\alpha\overline{\Clss}\delta_\Clss\,\dive\,\delta_U -(\Lambda-1)\delta_\Clss^2-\delta_\Clss\delta_U\cdot \nabla\overline{\Clss} 
        -\alpha\delta_\Clss^2\,\dive\,\overline{U}\Big)\,\dy\\
        &\quad+\int _{B(0, C_0)}\Big(-\frac{1}{2}(y+\overline{U})\cdot \nabla(|\delta_U|^2)-\alpha\overline{\Clss}\nabla \delta_\Clss\cdot \delta_U -(\Lambda-1)|\delta_U|^2\\
        &\qquad\qquad\qquad\quad   -(\delta_U\cdot \nabla)\overline{U}\cdot \delta_U 
           -\alpha\delta_\Clss\nabla\overline{\Clss}\cdot \delta_U \Big)\, \dy\\
        &=\int_{B(0, C_0)}\Big(\frac{3+\dive\,\overline{U}}{2}\big(\delta_\Clss^2+ |\delta_U|^2\big)+\alpha \nabla \overline{\Clss}\cdot(\delta_\Clss\delta_U)-(\Lambda-1)\big(\delta_\Clss^2+|\delta_U|^2\big)\\
        &\qquad\qquad\quad\quad - \delta_U\cdot \nabla\overline{U}\cdot \delta_U-(\alpha+1)\delta_\Clss\delta_U\cdot \nabla\overline{\Clss}- \alpha\delta_\Clss^2\,\dive\,\overline{U}\Big)\,\dy \\
        &\quad\,\, -\int_{\partial B(0, C_0)}\Big(\frac{|y|+\overline{\mathcal U}}{2}\big(\delta_\Clss^2+|\delta_U|^2\big) +\alpha\overline{\Clss}\delta_\Clss(\delta_U \cdot \bm{n}) \Big)\,\text{d}S\\
        &\leq  C \int_{B(0, C_0)}\big(\delta_\Clss^2+|\delta_U|^2\big)\,\dy
        - \int_{\partial B(0, C_0)}\frac{|y|+\overline{\mathcal U}-\alpha\overline{\Clss}}{2}\big(\delta_\Clss^2+|\delta_U|^2\big)\, \text{d}S\\
        &\leq  C \int_{B(0, C_0)}\big(\delta_\Clss^2+|\delta_U|^2\big)\,\dy,
\end{align*}
where $\bm{n}$ is  the outward unit  normal vector to $\partial B(0,C_0)$.
 It follows from the Gr\"onwall inequality and $(\delta_\Clss,\,\delta_U)|_{\tau=\tau_0}=(0, 0)$ in $B(0, C_0)$ that 
$$
(\delta_\Clss, \,\delta_U)(\tau, y)=(0, 0)\qquad  
\text{  for all $(\tau, y)\in [\tau_0, \tau_*]\times B(0, C_0)$}.
$$
\end{proof}

To establish all the required estimates, it suffices to prove Proposition~\ref{prop:bootstrap}.

\begin{Proposition} \label{prop:bootstrap}
For the initial data satisfying \eqref{eq:tilde_is_stable}{\rm--}\eqref{initialdatacondition2},
assume that \eqref{Para-ineq}{\rm--}\eqref{choiceR0} hold and that, for all $\tau\in [\tau_0, \tau_*]$, the following estimates hold{\rm :}
\begin{align}\label{unsestimate1}
\|P_{\rm{uns}}( \overset{\approx}{\Clss}, \overset{\approx}{U})\|_{X} &\leq \smallc_1 e^{-\varepsilon(\tau-\tau_0)},\\
\label{lowerestimate1}
\max\big\{|\widetilde{\Clss}|, \  |\widetilde{U}|\big\} &\leq \frac{\smallc_0 }{100}e^{-\varepsilon(\tau-\tau_0)}, \\ 
\label{lowerestimate2}
\max\big\{\|\nabla^4\widetilde{\Clss}\|_{L^2(B^c(0, C_0))}, \ \|\nabla^4\widetilde{U}\|_{L^2(B^c(0, C_0))}\big\} &\leq \smallc_0 e^{-\varepsilon(\tau-\tau_0)}, \\
\label{higherestimate3}
E_{K}=\int \big(|\nabla^{K}\Clss|^2+|\nabla^{K}U|^2\big)\phi^{K}\,{\rm d}y &\leq E.
\end{align}
Then,  for $\tau\in [\tau_0,\tau_*]$,
\begin{align}\label{unsestimate1'}
\|P_{\rm{uns}}( \overset{\approx}{\Clss}, \overset{\approx}{U})\|_{X} &\leq \smallc_1^{\frac{21}{20}} e^{-\frac{4\varepsilon}{3}(\tau-\tau_0)},\\
\label{lowerestimate1'}
\max\big\{|\widetilde{\Clss}|,\  |\widetilde{U}|\big\} &\leq  \frac{\smallc_0}{200}e^{-\varepsilon(\tau-\tau_0)}, \\ 
\label{lowerestimate2'}
\max\big\{\|\nabla^4\widetilde{\Clss}\|_{L^2(B^c(0, C_0))}, \  \|\nabla^4\widetilde{U}\|_{L^2(B^c(0, C_0))}\big\} &\leq \frac{\smallc_0}{2} e^{-\varepsilon(\tau-\tau_0)}, \\
\label{higherestimate3'}
E_{K}=\int\big(|\nabla^{K}\Clss|^2+|\nabla^{K}U|^2\big)\phi^{K}\,{\rm d}y &\leq \frac{E}{2}.
\end{align}
Here $P_{\rm{uns}}$ is a projection onto $V_{\rm{uns}}$ described in \rm{\S~\ref{function}}.
\end{Proposition}

    The proof of Proposition~\ref{prop:bootstrap} is carried out via a bootstrap argument, whose details are presented in \S\ref{Section6}. Specifically, estimate \eqref{lowerestimate1'} follows from Lemmas~\ref{linfitnityinside}--\ref{boosstrapestimate1}, while \eqref{lowerestimate2'} is a consequence of Lemmas~\ref{linfitnityinside} and \ref{boosstrapestimate2}. The higher-order estimate \eqref{higherestimate3'} is established in Lemma~\ref{boosstrapestimate3}. Finally, the control of the unstable modes \eqref{unsestimate1'} is proved in \S\ref{subsection6.3}.

\section{Spatial Decay Estimates}\label{Section5}
In this section, we derive the spatial decay estimates under the bootstrap assumptions introduced above. 
The derived decay rates allow us to complete the bootstrap argument in the next section.  
In this section, the initial data function $(\widetilde{\Clss}_0, \widetilde{U}_0)$ under consideration 
is constructed in  \eqref{eq:tilde_is_stable}--\eqref{initialdatacondition2}, 
$1< \gamma <1+\frac{2}{\sqrt{3}}$, $\Lambda \in (1, \Lambda^*(\gamma))$ is introduced in {\rm Lemma~\ref{thm:existence_profiles}}, $\delta\in (0,\frac{1}{2})$ satisfies $\delta_{\rm dis}>0$,
and assumptions \eqref{Para-ineq}--\eqref{choiceR0} and \eqref{unsestimate1}--\eqref{higherestimate3} 
always hold.

We first obtain an energy estimate of perturbation $(\widetilde{\Clss}, \widetilde{U})$.

\begin{Lemma}\label{per-U,S}  
For all $\tau\in[\tau_0, \tau_*]$,
\begin{equation}\label{perturbationenergy}
\widetilde{E}_{K}(\tau):=\int\big(|\nabla^{K}\widetilde{\Clss}|^2+|\nabla^{K}\widetilde{U}|^2\big)\phi^{K}{\rm d}y\leq 2E.
\end{equation}
\end{Lemma}

\begin{proof}
Since $(\widetilde{\Clss},\widetilde{U})=(\Clss,U)-(\widehat{X}\overline{\Clss},\widehat{X}\overline{U})$ 
and $K\ll E$, from \eqref{higherestimate3}, we only need to show
\begin{equation}\label{profilekestimate}
\int\big(|\nabla^{K}\big(\widehat{X}\overline{\Clss})|^2+|\nabla^{K}(\widehat{X}\overline{U})|^2\big)\phi^{K}\text{d}y\leq C(K).
\end{equation}
The decay of profile \eqref{profile decay} and the definition of $\phi$ \eqref{weightdefi} lead to
$$
\big(|\nabla^{K} \overline{\Clss} |^2+|\nabla^{K} \overline{U} |^2 \big) \phi^{K}\leq C(K)\langle y\rangle^{-2(\Lambda-1)-2K\eta}.
$$
Since $\frac{1}{K}\ll \eta$, 
we claim 
\begin{equation}\label{profilekestimate1}
\int\big( |\nabla^{K}\overline{\Clss}|^2+|\nabla^{K}\overline{U}|^2 \big) \phi^{K}\text{d}y\leq  C(K).
\end{equation}
The definition of $\widehat{X}$ gives
\begin{equation}\label{estimatehatX}
\text{supp}\, \widehat{X} \subset B ( 0, e^\tau ), \,\,\,\,  
\text{supp}\, (1-\widehat{X}) \subset B^c \big( 0,\frac{e^\tau}{2} \big), \,\,\,\, 
\text{supp}\, \nabla\widehat{X} \subset B ( 0, e^\tau ) \cap B^c \big( 0,\frac{e^\tau}{2} \big), 
\end{equation}
and
\begin{equation}\label{estimatehatX2}
\|\nabla^{j}\widehat{X}\|_{L^{\infty}}\leq C(j) e^{-j\tau}.
\end{equation}
Applying \eqref{weightdefi}, \eqref{estimatehatX}--\eqref{estimatehatX2}, and the decay properties of $(\overline{\Clss},\overline{U})$ stated in \eqref{profile decay} yields 
\begin{align}\label{profilekestimate2}
&\sum_{0\leq j\leq K-1}\int |\nabla^{K-j} \widehat{X}|^2 \big( |\nabla^{j} \overline{\Clss} |^2+|\nabla^{j} \overline{U} |^2 \big)\phi^{K}\text{d}y \nonumber \\
&\leq C(K) \sum_{0\leq j\leq K-1}e^{-2(K-j)\tau}\int_{\frac{1}{2}e^{\tau}}^{e^\tau}r^{-2(\Lambda-1)-2j}r^{2(1-\eta)K}r^{2}\text{d}r\\
&\leq C(K) \sum_{0\leq j\leq K-1}e^{-2(K-j)\tau+(3+2(1-\eta)K-2(\Lambda-1)-2j)\tau}\nonumber\\
&\leq C(K) e^{-(2K\eta-3+2(\Lambda-1))\tau}\leq C(K).\nonumber
\end{align}
Thus, \eqref{profilekestimate} follows from \eqref{profilekestimate1} and \eqref{profilekestimate2}.
\end{proof}

We now establish pointwise lower and upper bounds for $\Clss$.

\begin{Lemma} \label{lemma:Sbounds} When $y \in \mathbb{R}^3$,
\begin{equation*}
    \Clss \geq \frac{\smallc_1}{C}\big\langle \frac{y}{R_0} \big\rangle^{-(\Lambda-1)}\qquad \text{for } \tau\in[\tau_0, \tau_*].
\end{equation*}
When $|y| \ge R_0$, 
\begin{equation*}
\Clss \leq C\smallc_0 \max \big\{ \big\langle \frac{y}{R_0} \big\rangle^{-(\Lambda-1)}, 
\,e^{-(\tau-\tau_0)(\Lambda-1)} \big\}\qquad \text{for } \tau\in[\tau_0, \tau_*].
\end{equation*}
\end{Lemma}
\begin{proof}
We divide the proof into three steps.

\smallskip
\textbf{1.} For $|y|< R_0$, it follows from \eqref{Rochoice estimate} and \eqref{lowerestimate1} that $\Clss\geq 2\smallc_0-\|\widetilde{\Clss}\|_{L^\infty} \geq \smallc_0$.

\smallskip
\textbf{2.} Define $\omega_{\Clss,y_0}(\tau):=e^{(\Lambda-1)(\tau-\bar\tau)}\Clss(\tau, y_0e^{(\tau-\bar\tau)})$. It follows from \eqref{selfsimilar eq} that
\begin{equation}\label{eq: wy_01}
\begin{aligned}
    \partial_{\tau}\omega_{\Clss,y_0}(\tau)&= (\Lambda-1)e^{(\Lambda-1)(\tau-\bar\tau)}\Clss(\tau, y_0e^{(\tau-\bar\tau)})+e^{(\Lambda-1)(\tau-\bar\tau)}\partial_{\tau}\Clss(\tau, y_0e^{(\tau-\bar\tau)})\\
    &\quad + e^{(\Lambda-1)(\tau-\bar\tau)}y_0e^{(\tau-\bar\tau)} \cdot \nabla \Clss(\tau, y_0e^{(\tau-\bar\tau)})\\
    &= -e^{(\Lambda-1)(\tau-\bar\tau)}(U \cdot \nabla \Clss)(\tau, y_0e^{(\tau-\bar\tau)})-\alpha e^{(\Lambda-1)(\tau-\bar\tau)}(\Clss \,\dive \, U)(\tau, y_0e^{(\tau-\bar\tau)}).
\end{aligned}
\end{equation}
It follows from \eqref{Para-ineq}, \eqref{lowerestimate1}, \eqref{perturbationenergy}, and the Gagliardo-Nirenberg inequality \eqref{eq:GNresultinfty_simplified} that
\begin{equation*}
 |\nabla \widetilde{\Clss}| \leq C(K)\big( \|\widetilde{\Clss}\|_{L^\infty}^{1-\frac{1}{K-3/2}}\|\phi^{\frac{K}{2}}\nabla^K \widetilde{\Clss}\|_{L^2}^{\frac{1}{K-3/2}}\phi^{\frac{-K}{2K-3}} + \|\widetilde{\Clss}\|_{L^\infty}\, \langle y \rangle^{-1}\big),
\end{equation*}
which leads to
\begin{equation}\label{ES:weighted 1derPerS}
   \|\phi^{\frac{1}{2}}\nabla \widetilde{\Clss}\|_{L^{\infty}} 
   \leq C(K)\big(\smallc_0^{1-\frac{1}{K-3/2}}E^{\frac{1}{2K-3}}+ \smallc_0\big) 
   \leq \smallc_0^{\frac{9}{10}}.
\end{equation}
Moreover, for $|y|\ge R_0$, we obtain from \eqref{profile decay}, \eqref{choiceR0}, and \eqref{estimatehatX}--\eqref{estimatehatX2} that
\begin{equation}\label{ES:weighted 1derS}
\begin{aligned}
\phi^{\frac{1}{2}} | \nabla (\widehat X \overline \Clss) | 
&\leq C  \Big( \frac{|y|}{R_0} \Big)^{1-\eta} \big( |\nabla \widehat X  \overline \Clss| + |\widehat X  \nabla \overline \Clss |\big) \\
&\leq C\frac{|y|}{R_0}\frac{1}{|y|^\Lambda} \leq \frac{C}{R_0^{\Lambda}} \leq C\smallc_0.
\end{aligned}
\end{equation}
Combining \eqref{ES:weighted 1derPerS}--\eqref{ES:weighted 1derS} gives  
\begin{equation}\label{phiClss1}
\| \phi^{\frac{1}{2}} \nabla \Clss \|_{L^\infty (B^c(0, R_0))}\leq \smallc_0^{\frac{9}{10}}.
\end{equation}
By the same argument, it also follows that
\begin{equation}\label{phiU1}
\| \phi^{\frac{1}{2}} \nabla U \|_{L^\infty (B^c(0, R_0))} \leq \smallc_0^{\frac{9}{10}}. 
\end{equation}
Thus, we obtain from \eqref{def:Lambda*}, \eqref{Para-ineq}, \eqref{Rochoice estimate}, \eqref{lowerestimate1}, and \eqref{phiClss1}--\eqref{phiU1} that
\begin{align*}
    &\big| e^{(\Lambda-1)(\tau-\bar{\tau})} U(\tau, y_0 e^{(\tau-\bar{\tau})}) \nabla \Clss (\tau , y_0 e^{(\tau-\bar{\tau})}) \big| \\
    &\leq C \big( \frac{ y_0 e^{(\tau-\bar{\tau})} }{R_0} \big)^{-1+\eta} e^{(\Lambda-1) (\tau-\bar{\tau})} \smallc_0^{\frac{19}{10}} \\
    &\leq  C \smallc_0^{\frac{19}{10}} \big\langle \frac{y_0}{R_0} \big\rangle^{-1+\eta} e^{(\tau-\bar{\tau})(\eta+\Lambda-2)}\leq \frac{1}{2}\smallc_0^{\frac{9}{5}} \big\langle \frac{y_0}{R_0} \big\rangle^{1-\Lambda}e^{-\frac{\tau-\bar{\tau}}{10}},
\end{align*}
and
\begin{equation*}
| \alpha e^{(\Lambda-1)(\tau-\bar{\tau})}\Clss(\tau , y_0 e^{(\tau-\bar{\tau})})\, \dive U(\tau , y_0 e^{(\tau-\bar{\tau})}) | \leq  \frac{1}{2}\smallc_0^{\frac 95} \big\langle \frac{y_0}{R_0} \big\rangle ^{1-\Lambda}e^{-\frac{\tau-\bar{\tau}}{10}}.
\end{equation*}
Collecting the above two inequalities yields
\begin{equation}\label{dtauwClss}
    |\partial_\tau \omega_{\Clss,y_0}(\tau) |\leq  \smallc_0^{\frac 95} \big\langle \frac{y_0}{R_0} \big\rangle^{1-\Lambda}e^{-\frac{\tau-\bar{\tau}}{10}}.
\end{equation}
Integrating \eqref{dtauwClss} over $[\bar\tau, \tau]$, we obtain
\begin{equation} \label{eq:cadiz}
  |  \omega_{\Clss,y_0}(\tau) - \omega_{\Clss,y_0}(\bar{\tau}) | \leq  \smallc_0^{\frac 95} \big\langle \frac{y_0}{R_0} \big\rangle^{1-\Lambda},
\end{equation}
which holds uniformly for all $\tau \geq \bar\tau$.

\smallskip
\textbf{3.} Since $\overline \Clss$ is smooth in $\mathbb{R}^3$, it follows from \eqref{Rochoice estimate} that
$$ 
2 \smallc_0 \leq \overline \Clss \leq C\smallc_0 \qquad \text{ for $|y| = R_0$}.
$$
Consequently,
$$
\smallc_0\leq \omega_{\Clss,y_0}(\bar{\tau}) \leq C \smallc_0 \qquad\text{ for $|y_0|=R_0$}.
$$ 
By \eqref{eq:cadiz}, we obtain that, for all $\tau \ge \bar\tau$, 
$$
\frac12\smallc_0 \leq\omega_{\Clss,y_0}(\tau) \leq C\smallc_0 \qquad\text{ for $|y_0|=R_0$}.
$$ 
Equivalently,
\begin{equation} \label{eq:almeria}
\frac1C\smallc_0e^{-(\Lambda-1)(\tau-\bar{\tau})} \leq \Clss ( \tau, y_0 e^{(\tau-\bar{\tau})}  ) \leq C\smallc_0 e^{-(\Lambda-1)(\tau-\bar{\tau})}.
\end{equation}
Noting that $\frac{y}{y_0}=e^{\tau -\bar \tau}$, we have
\begin{equation} \label{eq:Sbound1}
\frac{\smallc_0}{C} \big\langle
\frac{y}{R_0}\big\rangle^{-(\Lambda-1)}\leq \Clss (\tau , y  ) \leq C\smallc_0 \big\langle
\frac{y}{R_0}\big\rangle^{-(\Lambda-1)}.
\end{equation}
Next, taking $\bar \tau =\tau_0$ and considering $|y_0| \geq R_0$, it follows 
from the definition of the initial data \eqref{initialdatacondition1} that
$$
\frac{\smallc_1}{2}\big\langle \frac{y_0}{R_0}\big\rangle^{1-\Lambda} \leq \Clss(\tau_0, y_0) \leq C\smallc_0. 
$$
Therefore, we have
\begin{equation} \label{eq:Sbound2}
e^{-(\Lambda-1)(\tau-\tau_0)}\frac{\smallc_1}{C}\big\langle \frac{y_0}{R_0}\big\rangle^{1-\Lambda} \leq \Clss(\tau , y)=e^{-(\Lambda-1)(\tau-\tau_0)}w_{\Clss,y_0}(\tau) \leq C\smallc_0e^{-(\Lambda-1)(\tau-\tau_0)}.
\end{equation}
The upper bound follows directly. For the lower bound, we use the identity $e^{(1-\Lambda)(\tau-\tau_0)}=(\frac{y}{y_0})^{1-\Lambda}$. This completes the proof.
\end{proof}

We next derive pointwise decay estimates for the first-order spatial derivatives of solution $(\Clss, U)$.
\begin{Lemma} \label{lemma:Sprimebounds} For all $(\tau,y)\in[\tau_0, \tau_*]\times \mathbb{R}^3$, 
\begin{align}  
    &|\nabla \widetilde \Clss | + |\nabla \Clss |  \leq   C\big\langle \frac{y}{R_0} \big\rangle^{-\Lambda}, \label{eq:voltaire1} \\[1mm]
    &|\nabla \widetilde U |+ |\nabla U |  \leq  C\big\langle \frac{y}{R_0} \big\rangle^{-\Lambda}. \label{eq:voltaire2}
  \end{align}
\end{Lemma}

\begin{proof}
We divide the proof into three steps.

\smallskip
\textbf{1}. {\it Bounds of $\Clss$, $\widetilde{\Clss}$, $U$, and $\widetilde{U}$ when $|y| \le R_0$.} 
For $|y|\leq R_0$, it follows from \eqref{Para-ineq} and the Gagliardo-Nirenberg inequality \eqref{eq:GNresultinfty_simplified} between
\eqref{lowerestimate1} and
\eqref{perturbationenergy} with
$$
\varphi= \phi^{\frac{1}{2}},\,\,\, l = K, \,\,\,   i=1, \,\,\, \theta = \frac{1}{K-3/2},
$$
that
\begin{equation*}
     | \nabla \widetilde{\Clss} | + | \nabla  \widetilde{U} | \leq C(K)\big(\smallc_0^{1 - \frac{1}{K-3/2}} E^{\frac{1}{2K-3}} \phi^{- \frac{K}{2(K-3/2)} } + \smallc_0 \langle y \rangle^{-j}\big) \leq C(K)\smallc_0^{\frac{9}{10}} \leq C \leq C\big\langle \frac{y}{R_0} \big\rangle^{-\Lambda}.
\end{equation*}
This, together with the properties of
$(\widehat{X}\overline{\Clss},\, \widehat{X}\overline{U})$ stated in \eqref{profile decay}, yields \eqref{eq:voltaire2}.

\smallskip
\textbf{2}. {\it Bounds of $\Clss$ and $\widetilde{\Clss}$ when $|y| \geq R_0$.} 
In the exterior region $B^c(0, R_0)$, \eqref{Rochoice estimate}, and \eqref{lowerestimate1} yield 
 $\max\{|\Clss|, |U|\} \leq C\smallc_0$. We then interpolate these bounds with \eqref{higherestimate3}. More precisely, applying the Gagliardo-Nirenberg inequality \eqref{eq:GNresultinfty_simplified} with
$$
\varphi= \phi^{\frac{1}{2}}, \,\,\, l = K, \,\,\,   i=j, \,\,\, \theta = \frac{j}{K-3/2},
$$
yields that, for $j \in \{ 1, 2, 3\}$,
\begin{equation*}
     | \nabla^j \Clss | + | \nabla^j  U | \leq C(K)\big(\smallc_0^{1 - \frac{j}{K-3/2}} E^{\frac{j}{2K-3}} \phi^{- \frac{Kj}{2(K-3/2)} } + \smallc_0 \langle y \rangle^{-j}\big).
\end{equation*}
We obtain from \eqref{Para-ineq} that
\begin{equation} \label{eq:rough_decay2}
    | \nabla \Clss | + | \nabla U | \leq \smallc_0^{\frac{9}{10}} \phi^{-\frac{1}{2}}, \quad
    | \nabla^2 \Clss | + | \nabla^2 U | \leq \smallc_0^{\frac{9}{10}} \phi^{-1}, \quad
    | \nabla^3 \Clss | + | \nabla^3 U | \leq \smallc_0^{\frac{9}{10}} \phi^{-\frac{3}{2}} .
\end{equation}
Applying $\partial_{y_i}$ to $\eqref{selfsimilar eq}_1$ yields
\begin{equation*}
    \partial_\tau \partial_{y_i}  \Clss = -\Lambda \partial_{y_i} \Clss - y\cdot \nabla \partial_{y_i} \Clss - \partial_{y_i} (U \cdot \nabla \Clss ) - \alpha \partial_{y_i} ( \Clss \,\dive \,U).
\end{equation*}
    For any $|y| \geq R_0$, we consider $y = y_0 e^{(\tau-\bar{\tau})}$, 
    where either $\bar{\tau} = \tau_0$ or $|y_0| = R_0$. Fix $i \in \{1, 2, 3\}$, define   
    $$\omega_{\Clss,y_0,i}(\tau) := e^{\Lambda(\tau-\bar{\tau})} \partial_{y_i} \Clss( \tau,  y_0 e^{\tau-\bar\tau} ).$$
    Then $\omega_{\Clss,y_0,i}$ satisfies
    \begin{equation*}
       \partial_{\tau} \omega_{\Clss,y_0,i} = -e^{\Lambda(\tau-\bar{\tau})} \big( \partial_{y_i} U\cdot \nabla \Clss + U \cdot \nabla \partial_{y_i} \Clss + \alpha \partial_{y_i}  \Clss \,\dive \,U + \alpha \Clss  \,\dive (\partial_{y_i} U) \big),
    \end{equation*}
    where all the terms are evaluated at $( \tau, \,y_0 e^{\tau-\bar\tau})$. We obtain from $| \Clss |, | U | \leq C\smallc_0$ and \eqref{eq:rough_decay2} that 
    \begin{equation*}
        | \partial_\tau \omega_{\Clss,y_0,i} | \leq C\smallc_0^{\frac 95} \phi ( y_0 e^{\tau-\bar{\tau}} )^{-1} e^{\Lambda(\tau-\bar{\tau})} \leq C\smallc_0^{\frac 95}e^{(\Lambda-2+2\eta)(\tau-\bar{\tau})} \big\langle\frac{y_0}{R_0} \big\rangle^{-2+2\eta}.
    \end{equation*}
    Consequently,
    \begin{equation*} \begin{split}
  & \big|\partial_{y_i}\Clss (\tau, y_0e^{\tau- \bar\tau}) - e^{-\Lambda(\tau-\bar{\tau})}\partial_{y_i}\Clss (\bar{\tau}, y_0))\big|\\
  &= e^{-\Lambda(\tau-\bar{\tau})}|  \omega_{\Clss,y_0,i}(\tau) - \omega_{\Clss,y_0,i}(\bar{\tau}) | \leq C\smallc_0^{\frac 95}e^{-\Lambda(\tau-\bar{\tau})} \big\langle \frac{y_0}{R_0} \big\rangle^{-\Lambda}.
   \end{split}
\end{equation*}
It follows from \eqref{choiceR0} and \eqref{initialdatacondition2} that, for all $|y_0|\ge R_0$, 
    $$
    |\partial_{y_i}\Clss(\tau_0, y_0)|\leq C  \langle y_0 \rangle^{-\Lambda} \leq C\langle\frac{y_0}{R_0} \rangle^{-\Lambda}.
    $$ 
Moreover, by the result obtained in Step 1, we have 
$$
|\partial_{y_i}\Clss(\bar{\tau}, y_0)|\leq C \qquad\,  \text{for  $|y_0|=R_0$ and $\bar\tau \ge \tau_0$}.
$$ 
Combining these estimates, we derive that, for all $|y|\geq R_0$, 
\begin{equation}\label{Order1-S}
    |\nabla \Clss|\leq C\big\langle\frac{y}{R_0} \big\rangle^{-\Lambda}.
\end{equation}
     
For $\nabla \widetilde{\Clss}$, it suffices to prove $$|\nabla(\widehat{X}\overline{\Clss})|\leq C\langle\frac{y}{R_0} \rangle^{-\Lambda}.$$
Since $$\nabla(\widehat{X}\overline{\Clss})=\nabla\overline{\Clss}\widehat{X}+\nabla\widehat{X}\overline{\Clss},$$
the first term inherits the decay of $\overline{\Clss}$, while the second term is controlled using the definition of $\widehat{X}$, which yields the desired decay.

\smallskip
\textbf{3}. {\it Bounds of $U$ and $\widetilde{U}$ when $|y| \geq R_0$.}      Applying $\partial_{y_i}$ to $\eqref{selfsimilar eq}_2$ gives
\begin{equation*}
\begin{aligned}
     \partial_\tau \partial_{y_i} U =& -\Lambda \partial_{y_i} U - y\cdot \nabla \partial_{y_i} U - \partial_{y_i} U\cdot \nabla U -U\cdot \nabla (\partial_{y_i} U) - \alpha \partial_{y_i} \Clss \nabla \Clss -\alpha  \Clss \nabla (\partial_{y_i} \Clss)\\
     &+C_{\rm{dis}}e^{-\delta_{\rm{dis}}\tau}\Big(\Clss^{\frac{\delta-1}{\alpha}}L(\partial_{y_i} U)+ \frac{\delta-1}{\alpha} \Clss^{\frac{\delta-1}{\alpha}-1}\partial_{y_i} \Clss L(U)+ \frac{\delta}{\alpha}\Clss^{\frac{\delta-1}{\alpha}-1}\nabla \Clss \cdot \md(\partial_{y_i} U)\Big.\\
     & \qquad\qquad\qquad\,\,+ \Big.\frac{\delta}{\alpha} \Clss^{\frac{\delta-1}{\alpha}-1}\nabla (\partial_{y_i} \Clss)\cdot \md(U) +\frac{\delta}{\alpha}\big(\frac{\delta-1}{\alpha}-1\big) \Clss^{\frac{\delta-1}{\alpha}-2}\,\partial_{y_i} \Clss\,\nabla \Clss \cdot \md(U) \Big).
\end{aligned}
\end{equation*}
    For any $|y| \geq R_0$, we consider $y = y_0 e^{(\tau-\bar{\tau})}$ for either $\bar{\tau} = \tau_0$ or $|y_0| = R_0$. Fix $i \in \{1, 2, 3\}$, define 
    $$\omega_{U,y_0,i}(\tau) := e^{\Lambda(\tau-\bar{\tau})} \partial_{y_i} U(\tau,  y_0 e^{\tau-\bar\tau} ).$$ 
    Then $\omega_{U,y_0,i}$ satisfies
    \begin{equation*}
    \begin{aligned}
       \partial_{\tau} \omega_{U,y_0,i} =& -e^{\Lambda(\tau-\bar{\tau})} \big( (\partial_{y_i} U \cdot \nabla) U + (U\cdot \nabla )(\partial_{y_i} U) + \alpha \partial_{y_i}  \Clss \nabla \Clss + \alpha \Clss \nabla(\partial_{y_i} \Clss) \big)\\
       &+C_{\rm{dis}}e^{-\delta_{\rm{dis}}\tau}e^{\Lambda(\tau-\bar\tau)}\Big(\Clss^{\frac{\delta-1}{\alpha}}L(\partial_{y_i} U)+ \frac{\delta-1}{\alpha} \Clss^{\frac{\delta-1}{\alpha}-1}\partial_{y_i} \Clss L(U)\\
       &\qquad \qquad \qquad \qquad \quad\,\, + \frac{\delta}{\alpha}\Clss^{\frac{\delta-1}{\alpha}-1}\nabla \Clss \cdot \md(\partial_{y_i} U)\Big.+ \Big.\frac{\delta}{\alpha} \Clss^{\frac{\delta-1}{\alpha}-1}\nabla (\partial_{y_i} \Clss)\cdot \md(U) \\
       &\qquad \qquad \qquad \qquad \quad\,\, +\frac{\delta}{\alpha}\big(\frac{\delta-1}{\alpha}-1\big) \Clss^{\frac{\delta-1}{\alpha}-2}\,\partial_{y_i} \Clss\,\nabla\Clss \cdot \md(U) \Big).
     \end{aligned}
\end{equation*}
 It follows from \eqref{S-para}, \eqref{Rochoice estimate}, \eqref{lowerestimate1}, Lemma \ref{lemma:Sbounds}, and  \eqref{eq:rough_decay2}--\eqref{Order1-S} that
\begin{align}
&e^{\Lambda(\tau-\bar{\tau})}\big|(\partial_{y_i} U \cdot \nabla) U + (U\cdot \nabla )(\partial_{y_i} U) + \alpha \partial_{y_i}  \Clss \nabla \Clss + \alpha \Clss \nabla (\partial_{y_i} \Clss)\big|\nonumber\\
         & \,\,\,\leq  C\smallc_0^{\frac{9}{5}}e^{(\Lambda-2+2\eta)(\tau-\bar{\tau})} \big\langle\frac{y_0}{R_0} \big\rangle^{-2+2\eta},\label{Order1U_1}\\[1mm]
        &e^{-\delta_{\rm{dis}}\tau}e^{\Lambda(\tau-\bar{\tau})}\big|\Clss^{\frac{\delta-1}{\alpha}}L(\partial_{y_i} U)\big|\nonumber\\
        &\,\,\, \leq  Ce^{-\delta_{\rm{dis}}\bar\tau}\smallc_1^{\frac{\delta-1}{\alpha}}\smallc_0^{\frac{9}{10}}\big\langle\frac{y_0}{R_0} \big\rangle^{-1+3\eta+\delta_{\rm{dis}}-\Lambda}e^{(-1+3\eta)(\tau-\bar{\tau})}, \label{Order1U_2}\\[1mm] 
         &e^{-\delta_{\rm{dis}}\tau}e^{\Lambda(\tau-\bar{\tau})}\big| \Clss^{\frac{\delta-1}{\alpha}-1}\partial_{y_i} \Clss L(U)+ \Clss^{\frac{\delta-1}{\alpha}-1} \nabla \Clss \cdot \md(\partial_{y_i} U) \big|\nonumber\\
         &\,\,\,\leq   Ce^{-\delta_{\rm{dis}}\tau}e^{\Lambda(\tau-\bar{\tau})}\smallc_1^{\frac{\delta-1}{\alpha}-1}\smallc_0^{\frac{9}{10}+1}\big\langle\frac{y_0e^{\tau-\bar\tau}}{R_0} \big\rangle^{-3+2\eta+(\frac{1-\delta}{\alpha})(\Lambda-1)}\nonumber\\
         &\,\,\, \leq  Ce^{-\delta_{\rm{dis}}\bar\tau}\smallc_1^{\frac{\delta-1}{\alpha}-1}\smallc_0^{\frac{19}{10}}\big\langle\frac{y_0}{R_0} \big\rangle^{-1+2\eta+\delta_{\rm{dis}}-\Lambda}e^{(-1+2\eta)(\tau-\bar{\tau})},\label{Order1U_3}\\[1mm]
 &e^{-\delta_{\rm{dis}}\tau} e^{\Lambda(\tau-\bar{\tau})}\big| \Clss^{\frac{\delta-1}{\alpha}-2}\,\partial_{y_i} \Clss\,\nabla \Clss \cdot \md(U) \big|\nonumber\\
         & \,\,\,\leq  Ce^{-\delta_{\rm{dis}}\tau}e^{\Lambda(\tau-\bar{\tau})}\smallc_1^{\frac{\delta-1}{\alpha}-2}\smallc_0^{\frac{9}{10}+2}\big\langle\frac{y_0e^{\tau-\bar\tau}}{R_0} \big\rangle^{-3+\eta+(\frac{1-\delta}{\alpha})(\Lambda-1)}\nonumber\\
         &\,\, \,\leq  Ce^{-\delta_{\rm{dis}}\bar\tau}\smallc_1^{\frac{\delta-1}{\alpha}-2}\smallc_0^{\frac{29}{10}}\big\langle\frac{y_0}{R_0} \big\rangle^{-1+\eta+\delta_{\rm{dis}}-\Lambda}e^{(-1+\eta)(\tau-\bar\tau)}. \label{Order1U_4}
  \end{align}       
For the remaining term, we apply \eqref{Para-ineq}, \eqref{choiceR0}, \eqref{higherestimate3}, \eqref{Order1-S}, and the Gagliardo-Nirenberg inequality \eqref{eq:GNresultinfty} between $\|\langle y \rangle^{\Lambda}\nabla \Clss\|_{L^{\infty}}$ and $\|\phi^{\frac{K}{2}}\nabla^K \Clss\|_{L^2}$ with
$$
\varphi =\phi^{\frac{K}{2(K-1)}}, \  \ \psi =\langle y\rangle^{\frac{\Lambda}{K-1}},\   \  i=1,\  \  p=\infty, \   \ q=2, \   \  l=K-1, \   \  \theta =\frac{1}{K-5/2},
$$
to obtain
    \begin{align}\label{Order2-S}
        |\nabla^2 \Clss| &\leq   C(K)\|\langle y \rangle^{\Lambda}\nabla \Clss\|^{1-\frac{1}{K-5/2}}_{L^{\infty}} \|\phi^{\frac{K}{2}}\nabla^K \Clss\|^{\frac{1}{K-5/2}}_{L^2}\langle y \rangle^{-\Lambda(1-\frac{1}{K-5/2})}\phi^{-\frac{K}{2(K-5/2)}} \nonumber \\
        &\quad + C(K)\|\langle y \rangle^{\Lambda}\nabla \Clss\|_{L^{\infty}}\,\langle y \rangle^{-1-\Lambda}  \nonumber \\
        & \leq C(K)\Big(E^{\frac{1}{2K-5}}\big\langle \frac{y}{R_0} \big\rangle^{-\Lambda(1-\frac{1}{K-5/2})}\big\langle \frac{y}{R_0} \big\rangle^{-\frac{K(1-\eta)}{K-5/2}}+ \big\langle \frac{y}{R_0} \big\rangle^{-1-\Lambda}\Big)\\
        &= C(K) \Big(E^{\frac{1}{2K-5}}\langle \frac{y}{R_0} \rangle^{-1+\eta-\Lambda}\langle \frac{y}{R_0} \rangle^{\frac{2\Lambda-5(1-\eta)}{2K-5}}+ \big\langle \frac{y}{R_0} \big\rangle^{-1-\Lambda}\Big) \nonumber \\
        &\leq  C(K)\smallc_0^{-\frac{1}{20}}\big\langle \frac{y}{R_0} \big\rangle^{-1+\eta-\Lambda} \leq \smallc_0^{-\frac{1}{10}}\big\langle \frac{y}{R_0} \big\rangle^{-1+\eta-\Lambda}. \nonumber
    \end{align}    
Therefore, it follows from Lemma~\ref{lemma:Sbounds}, \eqref{eq:rough_decay2}, and \eqref{Order2-S} that
\begin{equation}
    \begin{aligned}
        &e^{-\delta_{\rm{dis}}\tau}e^{\Lambda(\tau-\bar{\tau})}\big|\frac{\delta}{\alpha}\Clss^{\frac{\delta-1}{\alpha}-1}\nabla(\partial_{y_i} \Clss)\cdot \md(U) \big|  \\
        & \leq  C e^{-\delta_{\rm{dis}}\bar\tau}\smallc_1^{\frac{\delta-1}{\alpha}-1}\smallc_0^{\frac{4}{5}}\big\langle\frac{y_0}{R_0} \big\rangle^{-1+2\eta+\delta_{\rm{dis}}-\Lambda}e^{(-1+2\eta)(\tau-\bar{\tau})}.\label{Order1U_5}
    \end{aligned}
\end{equation}    
Moreover, when $1< \gamma < \frac{5}{3}$, we obtain
\begin{equation}\label{eq:-1+delta_dis1}
    \begin{aligned}
        -1+\delta_{\rm{dis}} &= \big(\frac{1-\delta}{\alpha}+1\big)(\Lambda-1)-2
        \leq \frac{2}{(1+\sqrt{\frac{2}{\gamma-1}})^2}\,\frac{2(1-\delta)+\gamma-1}{\gamma-1}-2\\
        & =\frac{2\gamma+2-4\delta}{\gamma+1+2\sqrt{2(\gamma-1)}}-2 <0.
    \end{aligned}
\end{equation}
When $\frac{5}{3}< \gamma < 1+\frac{2}{\sqrt{3}}$, we similarly have
\begin{equation}\label{eq:-1+delta_dis2}
    \begin{aligned}
        -1+\delta_{\rm{dis}} &= \big(\frac{1-\delta}{\alpha}+1\big)(\Lambda-1)-2\leq \frac{(3-\sqrt{3})(\gamma-1)+(6-2\sqrt{3})(1-\delta)}{2+\sqrt{3}(\gamma-1)}-2\\
        & =\frac{(3-3\sqrt{3})(\gamma-1)+(6-2\sqrt{3})(1-\delta)-4}{2+\sqrt{3}(\gamma-1)} <0.
    \end{aligned}
\end{equation}
Consequently, combining \eqref{range of Lambda}--\eqref{def:Lambda*}, \eqref{Order1U_1}--\eqref{Order1U_4}, and \eqref{Order1U_5}--\eqref{eq:-1+delta_dis2}, we choose $\eta \leq \frac{1}{10}(1-\delta_{\rm{dis}})$ such that
\begin{equation}\label{Order1-U'}
    \begin{aligned}
        &\big|\partial_{y_i}U (y_0e^{\tau-\bar \tau},\tau) - e^{-\Lambda(\tau-\bar{\tau})}\partial_{y_i}U (y_0,\bar{\tau}) |=e^{-\Lambda(\tau-\bar{\tau})}\big|  \omega_{U,y_0,i}(\tau) - \omega_{U,y_0,i}(\bar{\tau}) |\\
        &\leq C e^{-\Lambda(\tau-\bar{\tau})}\smallc_0^{\frac 95}\big\langle \frac{y_0}{R_0} \big\rangle^{-2+2\eta}+ Ce^{-\Lambda(\tau-\bar{\tau})}e^{-\delta_{\rm{dis}}\bar\tau}\smallc_1^{\frac{\delta-1}{\alpha}}\big\langle\frac{y_0}{R_0} \big\rangle^{-\Lambda}\big(\smallc_0^{\frac{9}{10}}+\smallc_1^{-1}\smallc_0^{\frac{9}{5}}+\smallc_1^{-2}\smallc_0^{\frac{29}{10}}\big)\\
        &\quad +Ce^{-\Lambda(\tau-\bar{\tau})}e^{-\delta_{\rm{dis}}\bar\tau}\smallc_1^{\frac{\delta-1}{\alpha}-1}\smallc_0^{\frac{4}{5}}\big\langle \frac{y_0}{R_0}\big\rangle^{-\Lambda}\\
        &\leq  \big\langle\frac{y_0e^{\tau-\bar\tau}}{R_0} \big\rangle^{-\Lambda}.
    \end{aligned}
\end{equation}
Here we used the choice of $\eta$ together with the fact that $\delta_{\rm dis}>0$ and $\bar{\tau} \ge \tau_0$
 is sufficiently large to absorb all the large constants appearing in the above estimates.
For $U$, we also have 
$$
|\partial_{y_i}U(\bar{\tau}, y_0)|\leq C \qquad\, \text{for  $|y_0|=R_0$ and $\bar\tau \ge \tau_0$},
$$
and 
$$
|\partial_{y_i}U(\tau_0, y_0)|\leq C \langle\frac{y_0}{R_0} \rangle^{-\Lambda} \qquad 
\text{for  $|y_0|\ge R_0$},
$$
which, along with \eqref{Order1-U'} and $y=y_0e^{\tau-\bar\tau}$, yields that, for all $|y| \geq R_0$,
\begin{equation}\label{Order1-U}
    |\nabla U|\leq C\big\langle\frac{y}{R_0} \big\rangle^{-\Lambda}.
\end{equation}
Analogously to the proof of $\nabla \widetilde{\Clss}$, we obtain the decay of $\nabla \widetilde{U}$.

This completes the proof of Lemma \ref{lemma:Sprimebounds}.
\end{proof}

In addition, we derive estimates for the higher-order derivatives of  $(\Clss,U)$.

\begin{Lemma}\label{higher-order}
    For every $1 \leq j \leq K-2$ and for all $ (\tau, y)\in[\tau_0, \tau_*]\times \mathbb{R}^3$, 
    \begin{equation} \label{eq:lessK-2}
         | \nabla^j  \Clss |+| \nabla^j U | \leq C(\smallc_0)  \phi^{-\frac j2} \langle y \rangle^{-(\Lambda-1)}.
    \end{equation}
    For $ j = K-1$, $\bar{\varepsilon}>0$, and for all $ (\tau, y)\in[\tau_0, \tau_*]\times \mathbb{R}^3$, 
    \begin{equation}\label{eq:K-1}
\big\| \langle y \rangle^{K(1-\eta)\frac{K-2}{K-1}-\bar{\varepsilon}}|\nabla^{K-1}\Clss|
\big\|_{L^{2+\frac{2}{K-2}}}+ \big\|\langle y \rangle^{K(1-\eta)\frac{K-2}{K-1}
-\bar{\varepsilon}}|\nabla^{K-1}U| \big\|_{L^{2+\frac{2}{K-2}}}\leq C(\smallc_0,\bar{\varepsilon}).
\end{equation}
\end{Lemma}

\begin{proof}
    For \eqref{eq:lessK-2},  in the Gagliardo-Nirenberg inequality \eqref{eq:GNresultinfty}, choose 
    $$
    \varphi =\phi^{\frac{K}{2(K-1)}}, \  \  \psi= \langle y \rangle^{\frac{\Lambda}{K-1}},\  \   l=K-1, \ \  i=j-1,\ \ p=\infty, \ \ q=2, \  \ \theta = \frac{j-1}{K- 5/2}. 
    $$
Then it follows from   \eqref{choiceR0}, \eqref{higherestimate3}, and \eqref{Order1-U} that 
\begin{equation}\label{decayestimatederj}
\begin{aligned}
       | \nabla^j U | &\leq   C(K)\big(\| \langle y \rangle^{\Lambda}\nabla U\|^{1-\theta}_{L^{\infty}} \|\phi^{\frac{K}{2}}\nabla^K U\|^{\theta}_{L^2}\langle y \rangle^{-\Lambda(1-\theta)}\phi^{-\frac{K\theta}{2}}\\
       &\qquad\qquad\,\, + \|\langle y \rangle^{\Lambda}\nabla U\|_{L^{\infty}}\langle y \rangle^{-j+1-\Lambda}\big)\\
        &\leq  C(K)\big(R_0^{\Lambda} E^{\frac{\theta}{2}}\langle y \rangle^{-\Lambda(1-\theta)}\phi^{-\frac{K}{2}\theta} + R_0^{\Lambda}\langle y \rangle^{-j+1-\Lambda}\big)\\
        &\leq C(\smallc_0)  \phi^{-\frac{j}{2}} \langle y \rangle^{-(\Lambda-1)}\phi^{-\frac{K\theta}{2} + \frac{j}{2}}  \langle y \rangle^{-\Lambda(1-\theta)+\Lambda-1}\\
        &\leq C(\smallc_0)\phi^{-\frac{j}{2}} \langle y \rangle^{-(\Lambda-1)}\phi^{\frac{K - 5j/2 }{2(K-5/2)}} \langle y \rangle^{\frac{\Lambda j-\Lambda - K + 5/2}{K-5/2}} \\ 
        &\leq C(\smallc_0) \phi^{-\frac{j}{2}} \langle y \rangle^{-(\Lambda-1)} \langle y \rangle^{\frac{\Lambda j-\Lambda - K + 5/2 + (1-\eta) (K - 5j/2 )}{K-5/2}}.
 \end{aligned}
\end{equation}
For $1\leq j \leq K-2$, using $\frac{1}{K} \ll \eta \ll 1$ yields
\begin{equation*}
    \Lambda j-\Lambda - K + \frac 52 + (1-\eta) (K - \frac52 j)=(\Lambda-\frac{5}{2})(j-2)+\eta(\frac{5}{2}j-K) <0.
\end{equation*}
The estimate for $\Clss$ is analogous. The proof of  \eqref{eq:lessK-2} is complete.

For \eqref{eq:K-1}, we use the assumption that $1 \ll K\eta $, 
together with the Gagliardo-Nirenberg inequality \eqref{eq:GNresult} applied to $\|\nabla U \langle y \rangle^\Lambda\|_{L^{\infty}}$ and $\|\phi^{\frac{K}{2}}\nabla^K U \|_{L^{2}}$. Choosing  
\begin{equation*}
\begin{split}
&\varphi =\phi^{\frac{K}{2(K-1)}}, \  \ \psi= \langle y \rangle^{\frac{\Lambda}{K-1}},\  \ l=K-1,\  \  i=K-2, \\
&  \bar{r}=2+\frac{2}{K-2},   \  \ p=\infty,\  \   q=2, \  \  \theta=\frac{K-2}{K-1},
\end{split}
\end{equation*}
we obtain
\begin{equation}
    \| \langle y \rangle^{-\bar\varepsilon+\frac{\Lambda}{K-1}+K(1-\eta)\frac{K-2}{K-1}}\nabla^{K-1}U\|_{L^{2+\frac{2}{K-2}}} \leq C(\smallc_0,\bar{\varepsilon}). 
\end{equation}
The same bound holds for $\Clss$ by an identical argument. 
Therefore, \eqref{eq:K-1} follows from $\Lambda>1$.
\end{proof}

\section{Global-In-Time Energy Estimates }\label{Section6}
This section is devoted to establishing refined estimates that strengthen the \textit{a priori} assumptions in the bootstrap argument. These refinements allow us to close the bootstrap argument 
and thereby obtain the desired global-in-time energy estimates of the solution.

\subsection{Proof of Proposition~\ref{prop:bootstrap}}\label{pro41proof}
In \S \ref{pro41proof},     the initial data function 
$(\widetilde{\Clss}_0, \widetilde{U}_0)$  under consideration is constructed in  \eqref{eq:tilde_is_stable}--\eqref{initialdatacondition2}, $1< \gamma <1+\frac{2}{\sqrt{3}}$, $\Lambda \in (1, \Lambda^*(\gamma))$ is introduced in {\rm Lemma~\ref{thm:existence_profiles}}, $\delta\in (0,\frac{1}{2})$ satisfies $\delta_{\rm dis}>0$,
and assumptions \eqref{Para-ineq}--\eqref{choiceR0} and \eqref{unsestimate1}--\eqref{higherestimate3} always hold.

\subsubsection{Lower order temporal decay estimates}\label{subsection6.1}
With the help of Lemmas \ref{per-U,S}--\ref{higher-order}, we first obtain the 
estimates for $\mathcal{E}$, $ \mathcal{N}$, and $\mathcal{F}_{\rm dis}$.

\begin{Lemma}\label{Forcingterm1} For all  $\tau\in[\tau_0, \tau_*]$, 
\begin{equation}\label{eq:est-E}
    \begin{aligned}
        &\chi_2 \mathcal E_{\subclss}=0,\ \  \chi_2 \mathcal E_{u}=0,\\
&\max\big\{\|\mathcal{E}_\subclss\|_{L^{\infty}},\ \  \|\nabla^4\mathcal{E}_\subclss\|_{L^{2}(B^c(0,C_0))}\big\} \leq \smallc_1 e^{-(\tau-\tau_0)},\\   
        &\max\big\{\|\mathcal{E}_u\|_{L^{\infty}},\ \ \|\nabla^4\mathcal{E}_u\|_{L^{2}(B^c(0,C_0))}\big\} \leq \smallc_1 e^{-(\tau-\tau_0)}.
    \end{aligned}
\end{equation}   
\end{Lemma}
\begin{proof}
    The decay of profile \eqref{profile decay} gives
    \begin{equation}\label{OrderJ-profile}
        |\nabla^j \overline{\Clss}|+|\nabla^j \overline{U}| \leq C\langle y \rangle^{1-\Lambda-j}.
    \end{equation}
It follows from \eqref{estimatehatX}--\eqref{estimatehatX2} that
\begin{equation*}
    \|(\widehat{X}^2-\widehat{X})\overline{U}\cdot \nabla \overline{U}\|_{L^{\infty}} \leq C \|(\widehat{X}^2-\widehat{X})\langle y \rangle^{1-2\Lambda}\|_{L^{\infty}}  \leq C e^{(1-2\Lambda)\tau},
\end{equation*}
and
\begin{equation*}
    \|\widehat{X}\overline{U}\cdot \nabla\widehat{X}\,\overline{U}\|_{L^{\infty}} \leq  C\||\widehat{X}| |\nabla \widehat{X}| \langle y \rangle^{2-2\Lambda} \|_{L^{\infty}} \leq Ce^{(1-2\Lambda)\tau}.
\end{equation*}
All the remaining terms in $\mathcal E_{u}$ and $\mathcal E_{\subclss}$ can be estimated in a similar way. 
Since $\Lambda >1$, and $\tau_0 \gg \smallc_1^{-1}\gg 1$, we deduce 
\begin{equation*}
    \|\mathcal{E}_\subclss\|_{L^{\infty}}+ \|\mathcal{E}_u\|_{L^{\infty}} \leq Ce^{(1-2\Lambda)\tau} \leq Ce^{-\tau_0}e^{-(\tau-\tau_0)} \leq \frac{C}{\tau_0}e^{-(\tau-\tau_0)} \leq \smallc_1 e^{-(\tau-\tau_0)}.
\end{equation*}
We obtain from \eqref{estimatehatX}--\eqref{estimatehatX2} and \eqref{OrderJ-profile} that
\begin{equation*}
    \big|\nabla^4\big((\widehat{X}^2-\widehat{X})\overline{U}\cdot\nabla\overline{U}\big)\big|
    \leq C 
\sum_{0\leq j \leq 4} |\nabla^{4-j} (\widehat{X}^2-\widehat{X})||\nabla^{j}(\overline{U}\cdot\nabla\overline{U})|\leq C\sum_{0\leq j \leq 4}e^{-(4-j) \tau}\langle y \rangle^{1-2\Lambda-j}.
\end{equation*}
Consequently, 
\begin{equation*}
\begin{aligned}
\big\|\nabla^4 \big((\widehat{X}^2-\widehat{X})\overline{U}\cdot \nabla \overline{U}\big)\big\|_{L^{2}}^2
&\leq C \sum_{0\le j \le 4}e^{-2(4-j)\tau}\int_{\frac{1}{2}e^\tau}^{e^\tau}r^{2(1-2\Lambda-j)}r^2 \text{d}r\\
&\leq C\sum_{0\le j \le 4}e^{(-2(4-j)+5-4\Lambda-2j)\tau}\\
&\leq Ce^{(-4\Lambda-3)\tau}\leq Ce^{-2\tau}.
\end{aligned}
\end{equation*}
The rest of the terms in $\nabla^4\mathcal{E}_\subclss$ and $\nabla^4\mathcal{E}_u$ can be estimated 
in a similar way. We obtain from \eqref{Para-ineq} and $\tau_0 \gg 1$ that
\begin{equation*}
    \|\nabla^4\mathcal{E}_\subclss\|_{L^{2}(B^c(0,C_0))}+ \|\nabla^4\mathcal{E}_u\|_{L^{2}(B^c(0,C_0))} \leq Ce^{-\tau} \leq Ce^{-\tau_0}e^{-(\tau-\tau_0)}\leq \frac{C}{\tau_0}e^{-(\tau-\tau_0)}\leq \smallc_1 e^{-(\tau-\tau_0)}.
\end{equation*}
\end{proof}

\begin{Lemma}\label{eq:est-NF}
    For all  $\tau\in[\tau_0, \tau_*]$, 
    \begin{equation}
\max\big\{\|\mathcal{N}_\subclss\|_{L^{\infty}}, \   \|\chi_2\mathcal{N}_\subclss\|_{X},\  \|\mathcal{N}_u\|_{L^{\infty}},  \  \|\chi_2\mathcal{N}_u\|_{X}\big\}  \leq \smallc_1^{\frac{6}{5}} e^{-\frac{3}{2}\varepsilon(\tau-\tau_0)},
    \end{equation}
    and
    \begin{equation}
        \max\big\{\|\mathcal{F}_{\rm dis}\|_{L^{\infty}},\    \|\chi_2\mathcal{F}_{\rm dis}\|_{X}\big\} \leq \smallc_1^{\frac{6}{5}} e^{-\delta_{\rm{dis}}\frac{\tau}{2}}.
    \end{equation}
\end{Lemma}

\begin{proof}
    First, using Lemmas \ref{lemma:Sbounds}--\ref{higher-order} yields
  \begin{equation}\label{eq:Fdis1}
        \begin{aligned}
            e^{-\delta_{\rm{dis}}\tau}\bigl| \Clss^{\frac{\delta-1}{\alpha}}L(U)\bigr| &\leq  C(\smallc_0) e^{-\delta_{\rm{dis}}\tau}\big\langle \frac{y}{R_0}\big\rangle^{-2(1-\eta)-(\Lambda-1)+\frac{(1-\delta)(\Lambda-1)}{\alpha}}   \\
            & = C(\smallc_0)e^{-\delta_{\rm{dis}}\tau}\big\langle \frac{y}{R_0}\big\rangle^{\delta_{\rm{dis}}-2\Lambda+1+2\eta} \leq  C(\smallc_0)e^{-\delta_{\rm{dis}}\tau}, 
        \end{aligned}
   \end{equation}     
  and
  \begin{equation}\label{eq:Fdis2}
        \begin{aligned}
            e^{-\delta_{\rm{dis}}\tau}\big| \Clss^{\frac{\delta-1}{\alpha}}\frac{\nabla \Clss}{\Clss}\md(U)\big| &\leq C(\smallc_0) e^{-\delta_{\rm{dis}}\tau}\big\langle \frac{y}{R_0}\big\rangle^{-2+\eta-(\Lambda-1)+\frac{(1-\delta)(\Lambda-1)}{\alpha}}  \\
            & = C(\smallc_0) e^{-\delta_{\rm{dis}}\tau}\big\langle \frac{y}{R_0}\big\rangle^{\delta_{\rm{dis}}-2\Lambda+1+\eta} \leq  C(\smallc_0)e^{-\delta_{\rm{dis}}\tau},
        \end{aligned}
  \end{equation}
where we have used the facts that $\delta_{\rm{dis}}-2\Lambda+1+2\eta<0$ and $\eta>0$.

It follows from \eqref{Para-ineq}, $\tau\ge \tau_0 \gg 1$, and \eqref{eq:Fdis1}--\eqref{eq:Fdis2} that
 \begin{equation}\label{eq:Fdis-inf}
 \begin{aligned}
        \|\mathcal{F}_{\rm dis}\|_{L^{\infty}}  &\leq  Ce^{-\delta_{\rm{dis}}\tau}\big\|\Clss^{\frac{\delta-1}{\alpha}}L(U)+\Clss^{\frac{\delta-1}{\alpha}}\frac{\nabla \Clss}{\Clss}\md(U)\big\|_{L^{\infty}}\\ &\leq C(\smallc_0) e^{-\delta_{\rm{dis}}\tau} \leq C(\smallc_0)\smallc_0^{-\frac{9}{5}}e^{-\delta_{\rm dis}\frac{\tau}{2}}\smallc_0^{\frac{9}{5}}e^{-\delta_{\rm dis}\frac{\tau}{2}} \\
        &\leq C(\smallc_0)e^{-\delta_{\rm dis}\frac{\tau_0}{2}} \smallc_1^{\frac{6}{5}}e^{-\delta_{\rm{dis}}\frac{\tau}{2}} \leq \smallc_1^{\frac{6}{5}}e^{-\delta_{\rm{dis}}\frac{\tau}{2}}.
\end{aligned}
\end{equation}
By the choice of $C_0$ in Lemma \ref{C0choice}, the decay estimate of profile \eqref{profile decay}, \eqref{lowerestimate1}, and $\smallc_0\ll 1$, we have
 $$
 \Clss \geq \overline{\Clss}-|\widetilde{\Clss}| \ge C^{-1}-\smallc_0 \geq C^{-1} \qquad \text{for $|y|\leq 3C_0$}.
 $$ 
Thus, we obtain from \eqref{Para-ineq} and Lemmas \ref{lemma:Sbounds}--\ref{higher-order} that
\begin{align}\label{eq:Fdis-X}
        \|\chi_2\mathcal{F}_{\rm dis}\|_{H^{m}} &\leq C(m)\big(  \|\mathcal{F}_{\rm dis}\|_{L^{\infty}(B(0, 3C_0))} + \|\nabla^m\mathcal{F}_{\rm dis}\|_{L^{\infty}(B(0, 3 C_0))}\big) \nonumber \\ 
        &\leq C(m)e^{-\delta_{\rm{dis}}\tau}\big\|\Clss^{\frac{\delta-1}{\alpha}}L(U)+\Clss^{\frac{\delta-1}{\alpha}}\frac{\nabla \Clss}{\Clss}\md(U)\big\|_{L^{\infty}(B(0, 3C_0))} \nonumber \\
        & \quad + C(m) e^{-\delta_{\rm{dis}}\tau}\big\|\nabla^m(\Clss^{\frac{\delta-1}{\alpha}}L(U))+\frac{\delta}{\alpha}\nabla^m(\Clss^{\frac{\delta-1}{\alpha}}\frac{\nabla \Clss}{\Clss}\md(U))\big\|_{L^{\infty}(B(0, 3C_0))}  \\
        &\leq C(\smallc_0) e^{-\delta_{\rm{dis}}\tau} \leq C(\smallc_0)\smallc_0^{-\frac{9}{5}}e^{-\delta_{\rm dis}\frac{\tau}{2}}\smallc_0^{\frac{9}{5}}e^{-\delta_{\rm dis}\frac{\tau}{2}} \nonumber\\
        &\leq C(\smallc_0)e^{-\delta_{\rm dis}\frac{\tau_0}{2}} \smallc_1^{\frac{6}{5}}e^{-\delta_{\rm{dis}}\frac{\tau}{2}} \leq \smallc_1^{\frac{6}{5}}e^{-\delta_{\rm{dis}}\frac{\tau}{2}}.\nonumber
\end{align}
Next, we obtain from the Gagliardo-Nirenberg inequality \eqref{eq:GNresultinfty_simplified}, \eqref{lowerestimate1}, and Lemma \ref{per-U,S} that, for any integer $0\leq j\leq m$,
\begin{align*}
\big\|\nabla^{j}\big((\widetilde{U}\cdot \nabla) \widetilde{U}\big)\big\|_{L^{\infty}}
&\leq C(m) \sum_{0\leq \ell \leq j\leq m}\|\nabla^{\ell}\widetilde{U}\|_{L^{\infty}}\|\nabla^{j-\ell+1}\widetilde{U}\|_{L^{\infty}}\\
&\leq C(m) E^{\frac{1}{20}}\smallc_0
^{\frac{39}{20}}e^{-\frac{39}{20}\varepsilon(\tau-\tau_0)}\\
&\leq C(m)\smallc_0
^{\frac{19}{10}}e^{-\frac{39}{20}\varepsilon(\tau-\tau_0)}.
\end{align*}  
A similar estimate holds for $\widetilde{U}\cdot \nabla\widetilde{\Clss}$, $\widetilde{\Clss} \nabla\widetilde{\Clss}$,  and $\widetilde{\Clss}\,\dive \widetilde{U}${\rm:}
\begin{align}
&\|\nabla^{j}(\widetilde{U}\cdot\nabla \widetilde{\Clss})\|_{L^{\infty}}+\|\nabla^{j}(\widetilde{\Clss}\nabla \widetilde{\Clss})\|_{L^{\infty}}+\|\nabla^{j}(\widetilde{\Clss}\,\dive \widetilde{U})\|_{L^{\infty}}\nonumber \\
&\leq C(m) \sum_{0\leq \ell \leq j\leq m}
\big(\|\nabla^{\ell}\widetilde{U}\|_{L^{\infty}}\|\nabla^{j-\ell +1}\widetilde{\Clss}\|_{L^{\infty}}+\|\nabla^{\ell}\widetilde{\Clss}\|_{L^{\infty}}\|\nabla^{j-\ell +1}\widetilde{\Clss}\|_{L^{\infty}}\big)  \nonumber\\
&\quad + C(m) \sum_{0\le \ell \leq j\leq m}\|\nabla^{\ell}\widetilde{\Clss}\|_{L^{\infty}}\|\nabla^{j-\ell+1}\widetilde{U}\|_{L^{\infty}} \nonumber\\
&\leq C(m)E^{\frac{1}{20}} \smallc_0
^{\frac{39}{20}}e^{-\frac{39}{20}\varepsilon(\tau-\tau_0)}\leq C(m)\smallc_0
^{\frac{19}{10}}e^{-\frac{39}{20}\varepsilon(\tau-\tau_0)}\nonumber.
\end{align}
Consequently, it follows from \eqref{Para-ineq} that
\begin{equation*}
    \|\mathcal{N}_\subclss\|_{L^{\infty}}+ \|\mathcal{N}_u\|_{L^{\infty}}\leq C(m) \smallc_0
^{\frac{19}{10}}e^{-\frac{39}{20}\varepsilon(\tau-\tau_0)}\leq \smallc_1^{\frac{6}{5}} e^{-\frac{3}{2}\varepsilon(\tau-\tau_0)}.
\end{equation*}
Moreover, $\text{supp}\, \chi_2\subset B(0,3C_0)$ and \eqref{Para-ineq} imply 
\begin{equation*}
\|\chi_2\mathcal{N}_u\|_{H^{m}}\leq C(m)\sum_{j\leq m}\|\nabla^j(\mathcal{N}_u)\|_{L^{\infty}(B(0, 3C_o))}\leq C(m)\smallc_0
^{\frac{19}{10}}e^{-\frac{39}{20}\varepsilon(\tau-\tau_0)}\leq \smallc_1^{\frac{6}{5}} e^{-\frac{3}{2}\varepsilon(\tau-\tau_0)}.
\end{equation*}
$\|\chi_2\mathcal{N}_\subclss\|_{H^{m}}$ can be estimated analogously.

This completes the proof.
\end{proof}

We now perform the bootstrap argument. We first carry out the proof in region $B(0, C_0)$.

\begin{Lemma}\label{linfitnityinside}
For any integer $0 \leq j\leq 4$ and all $\tau \in [\tau_0, \tau_*]$,
\begin{equation*}
\|\nabla^{j}\widetilde{\Clss}\|_{L^{\infty}(B(0,C_0))}+\|\nabla^{j}\widetilde{U}\|_{L^{\infty}(B(0,C_0))}\leq C\frac{\smallc_1}{\smallc_g}e^{-\varepsilon(\tau-\tau_0)}.
\end{equation*}
\end{Lemma}

\begin{proof}
It follows from Lemma~\ref{truncatedestimate} and the definition of $X$ in \eqref{X space} that
$$
\|\nabla^{j}\widetilde{U}\|_{L^{\infty}(B(0,C_0))}=\|\nabla^{j}\overset{\approx}{U}\|_{L^{\infty}(B(0,C_0))}\leq C \|\overset{\approx}{U}\|_{X}.
$$
Since $(\overset{\approx}{\Clss},\overset{\approx}{U})$ satisfies
\begin{equation*}
\partial_{\tau}(\overset{\approx}{\Clss},\overset{\approx}{U})=(\overset{\approx}{\Clss},\overset{\approx}{U})+\chi_2\mathcal{F}(\widetilde{\Clss},\widetilde{U})
\end{equation*}
with $
\chi_2\mathcal{F}_{\subclss}=\chi_2\mathcal{N}_\subclss$ and $\chi_2\mathcal{F}_{u}=\chi_2\mathcal{F}_{\rm dis}+\chi_2\mathcal{N}_{u}$, we project the equation onto the stable subspace to obtain 
\begin{align*}
\partial_{\tau}P_{\rm{sta}}(\overset{\approx}{\Clss},\overset{\approx}{U})=\mathcal{L}P_{\rm{sta}}(\overset{\approx}{\Clss},\overset{\approx}{U})+P_{\rm{sta}}(\chi_2\mathcal{F}(\widetilde{\Clss},\widetilde U)).
\end{align*}
Applying the Duhamel formula yields 
\begin{align*}
P_{\rm{sta}}(\overset{\approx}{\Clss},\overset{\approx}{U})(\tau)=\aleph(\tau-\tau_0)P_{\rm{sta}}(\overset{\approx}{\Clss},\overset{\approx}{U})(\tau_0)+\int_{\tau_0}^{\tau}\aleph(\tau-\bar{\tau})(P_{\rm{sta}}(\chi_2\mathcal{F}(\widetilde{\Clss},\widetilde{U})))(\bar{\tau})\,{\rm d}\bar{\tau},
\end{align*}
where $\aleph$ is the semigroup generated by $\mathcal{L}$. Based on Lemmas~\ref{Forcingterm1}--\ref{eq:est-NF}, $\smallc_g = \frac{25}{12}\varepsilon$, and the fact that $\aleph$ is a contraction semigroup in $V_{\rm{sta}}$ determined in Lemma~\ref{prop:maxdissmooth}, we have
\begin{align*}
\|P_{\rm{sta}}(\overset{\approx}{\Clss},\overset{\approx}{U})\|_{X} &\leq \smallc_1 e^{-\frac{\smallc_g}{2}(\tau-\tau_0)}+\int_{\tau_0}^{\tau}2 \smallc_1 e^{-\frac{\smallc_g}{2}(\tau-\bar{\tau})}e^{-\frac{3}{2}\varepsilon(\bar{\tau}-\tau_0)}\,{\rm d}\bar{\tau}\\
&\leq C\frac{\smallc_1}{\smallc_g}e^{-\frac{\smallc_g}{2}(\tau-\tau_0)} \leq C\frac{\smallc_1}{\smallc_g}e^{-\varepsilon(\tau-\tau_0)} .
\end{align*}
The unstable part of the solution is controlled directly from the assumption in Proposition \ref{prop:bootstrap}.
\end{proof}

Next, we derive the estimates in  $B^c(0,C_0)$.

\begin{Lemma}\label{boosstrapestimate1}
 For $\tau \in [\tau_0, \tau_*]$ and $|y|\geq C_0$,
\begin{equation*}
\max\big\{|\widetilde{\Clss}|, \    |\widetilde{U}|\big\} \leq \frac{{\smallc_0}}{200}e^{-\varepsilon(\tau-\tau_0)}.
\end{equation*}
\end{Lemma}
\begin{proof}
The perturbation $(\widetilde{\Clss}, \widetilde{U})$ satisfies
\begin{equation*}
    \begin{aligned}
        \partial_{\tau} \widetilde{\Clss}&= -(\Lambda-1)\widetilde{\Clss} -y\cdot \nabla\widetilde{\Clss} -(\widehat{X}\overline{U})\cdot \nabla\widetilde{\Clss}-\alpha(\widehat{X}\overline{\Clss})\,\dive \,\widetilde{U}-\widetilde{U} \cdot \nabla(\widehat{X}\overline{\Clss})\\
        & \quad\,\, -\alpha \widetilde{\Clss}  \,\dive(\widehat{X}\overline{U})+\mathcal{N}_\subclss+\mathcal{E}_\subclss,\\
        \partial_{\tau} \widetilde{U}&= -(\Lambda-1)\widetilde{U} -y\cdot \nabla\widetilde{U}- (\widehat{X}\overline{U})\cdot \nabla\widetilde{U}-\alpha(\widehat{X}\overline{\Clss})\nabla\widetilde{\Clss}-\widetilde{U} \cdot \nabla(\widehat{X}\overline{U})\\
        &\quad\,\, -\alpha \widetilde{\Clss} \nabla(\widehat{X}\overline{\Clss})+\mathcal{N}_u+\mathcal{E}_u+ \mathcal{F}_{\rm dis}.    \end{aligned}
\end{equation*}
    Applying Lemma \ref{lemma:GN_generalnoweightwholespace} between $\|(\widetilde{\Clss}, \widetilde{U})\|_{L^{\infty}}$ and $\|(\nabla^4\widetilde{\Clss}, \nabla^4\widetilde{U})\|_{L^2(B^c(0,C_0))}$ with 
$$
p=\infty,\quad q=2,\quad \theta=\frac{2}{5},\quad \bar r=\infty,\quad i=1,\quad l=4,
$$
 we obtain
    \begin{equation*}
        \max\big\{\|\nabla\widetilde{\Clss}\|_{L^\infty(B^c(0,C_0))},\   
\|\nabla\widetilde{U}\|_{L^\infty(B^c(0,C_0))} \big\}       
        \leq \smallc_0 \big(\frac{1}{100} \big)^{\frac{1}{20}}e^{-\varepsilon(\tau- \tau_0)}.
    \end{equation*}
   Combining this bound with \eqref{profilecondition1} and \eqref{lowerestimate1} yields
   \begin{equation*}
   \begin{aligned}
& \big|\widehat{X}\overline{U} \cdot\nabla\widetilde{\Clss}\big|+\big|\alpha(\widehat{X}\overline{\Clss})\,\dive \widetilde{U}\big|+\big|\widetilde{U} \cdot \nabla(\widehat{X}\overline{\Clss})\big|+\big|\alpha \widetilde{\Clss}\, \dive(\widehat{X}\overline{U})\big| \leq \smallc_0\frac{1}{C_1} \big(\frac{1}{100} \big)^{\frac{1}{20}}e^{-\varepsilon(\tau- \tau_0)},\\
&\big|\widehat{X}\overline{U}\cdot\nabla\widetilde{U}\big|+\big|\alpha(\widehat{X}\overline{\Clss})\nabla\widetilde{\Clss}\big|+\big|\widetilde{U} \cdot \nabla(\widehat{X}\overline{U})\big|+\big|\alpha \widetilde{\Clss}\nabla(\widehat{X}\overline{\Clss})\big|\leq \smallc_0\frac{1}{C_1} \big(\frac{1}{100} \big)^{\frac{1}{20}}e^{-\varepsilon(\tau- \tau_0)}.
    \end{aligned}
\end{equation*}
Then, according to Lemmas \ref{Forcingterm1}--\ref{eq:est-NF}, for $1\leq i \leq 3$ and $\tau_0 \leq \bar{\tau} \leq \tau$, we have   
\begin{equation}\label{eq:Uitau}
\begin{split}
    &|\partial_\tau \widetilde{U}_i(e^{\tau-\bar\tau}y_0)+(\Lambda-1)\widetilde{U}_i(e^{\tau-\bar\tau}y_0)|\\
    &\leq  \smallc_0\frac{1}{C_1} \big(\frac{1}{100} \big)^{\frac{1}{20}}e^{-\varepsilon(\tau- \tau_0)}+\smallc_1e^{-(\tau- \tau_0)}+\smallc_1^{\frac{6}{5}} e^{-\frac{3\varepsilon}{2}(\tau-\tau_0)}+\smallc_1^{\frac{6}{5}} e^{-\delta_{\rm{dis}}\frac{\tau}{2}},
\end{split}
\end{equation}
and
\begin{equation}\label{eq:Clsstau}
\begin{split}
    &|\partial_\tau \widetilde{\Clss}(e^{\tau-\bar\tau}y_0)+(\Lambda-1)\widetilde{\Clss}(e^{\tau-\bar\tau}y_0)|\\
    &\leq \smallc_0\frac{1}{C_1} \big(\frac{1}{100} \big)^{\frac{1}{20}}e^{-\varepsilon(\tau- \tau_0)}+\smallc_1e^{-(\tau- \tau_0)}+\smallc_1^{\frac{6}{5}} e^{-\frac{3\varepsilon}{2}(\tau-\tau_0)}.
\end{split}
\end{equation}
It follows from \eqref{eq:Uitau}, $\smallc_g\ll \delta_{\rm dis}$,  and $\varepsilon =\frac{12 \smallc_g}{25}$ that
\begin{equation*}
\begin{aligned}
        &|\partial_\tau \big(\widetilde{U}_i(e^{\tau-\bar\tau}y_0)e^{\varepsilon(\tau-\tau_0)}\big)+(\Lambda-1-\varepsilon)\widetilde{U}_i(e^{\tau-\bar\tau}y_0)e^{\varepsilon(\tau-\tau_0)}|\\
        &=|e^{\varepsilon(\tau-\tau_0)}\big(\partial_\tau \widetilde{U}_i(e^{\tau-\bar\tau}y_0)+(\Lambda-1)\widetilde{U}_i(e^{\tau-\bar\tau}y_0)\big)|\\
        &\leq  \smallc_0 \frac{1}{C_1} \big(\frac{1}{100} \big)^{\frac{1}{20}}+ \smallc_1 e^{-(1-\varepsilon)(\tau- \tau_0)}+ \smallc_1^{\frac{6}{5}} e^{-\frac{\varepsilon}{2}(\tau-\tau_0)}+ \smallc_1^{\frac{6}{5}} e^{-(\delta_{\rm{dis}}\frac{\tau}{2}-\varepsilon(\tau-\tau_0))}\\
        &\leq  2\smallc_0\frac{1}{C_1} \big(\frac{1}{100} \big)^{\frac{1}{20}}.
\end{aligned}
\end{equation*}
Multiplying the above by $e^{(\Lambda-1-\varepsilon)(\tau-\tau_0)}$ leads to
\begin{equation}\label{eq:Uitau1}
    |\partial_\tau \big(\widetilde{U}_i(e^{\tau-\bar\tau}y_0)e^{(\Lambda-1)(\tau-\tau_0)}\big)| \leq 2\smallc_0\frac{1}{C_1} \big(\frac{1}{100} \big)^{\frac{1}{20}}e^{(\Lambda-1-\varepsilon)(\tau-\tau_0)}.
\end{equation}
Integrating \eqref{eq:Uitau1} over $[\bar\tau, \tau]$ and then multiplying by $e^{-(\Lambda-1-\varepsilon)(\tau-\tau_0)}$ yield
\begin{equation*}
    \begin{aligned}
        &|\widetilde{U}_i(e^{\tau-\bar\tau}y_0, \tau)e^{\varepsilon(\tau-\tau_0)}|\\
        &\leq |\widetilde{U}_i(y_0, \bar\tau)e^{(\Lambda-1)(\bar\tau-\tau_0)}e^{-(\Lambda-1-\varepsilon)(\tau-\tau_0)}| + \frac{2\smallc_0}{\Lambda-1-\varepsilon}\frac{1}{C_1} \big(\frac{1}{100} \big)^{\frac{1}{20}}(1-e^{-(\Lambda-1-\varepsilon)(\tau-\bar\tau)})\\
        &\leq |\widetilde{U}_i(y_0, \bar\tau)e^{-(\Lambda-1-\varepsilon)(\tau-\bar\tau)}e^{\varepsilon(\bar\tau-\tau_0)}|+\frac{8\smallc_0}{\Lambda-1}\frac{1}{C_1} \big(\frac{1}{100} \big)^{\frac{1}{20}}(1-e^{-(\Lambda-1-\varepsilon)(\tau-\bar\tau)})\\
        &=  e^{-(\Lambda-1-\varepsilon)(\tau-\bar\tau)}\Big(|\widetilde{U}_i(y_0, \bar\tau)e^{\varepsilon(\bar\tau-\tau_0)}|-\frac{8\smallc_0}{\Lambda-1}\frac{1}{C_1} \big(\frac{1}{100} \big)^{\frac{1}{20}}\Big)+ \frac{8\smallc_0}{\Lambda-1}\frac{1}{C_1} \big(\frac{1}{100} \big)^{\frac{1}{20}}.
    \end{aligned}
\end{equation*}
For $|y_0|=C_0$ and $\bar\tau\geq \tau_0$, using Lemma \ref{linfitnityinside} and $\smallc_1\ll \smallc_0 \ll \smallc_{g}$, we have
\begin{equation*}
    \max\big\{|\widetilde{\Clss}(y_0, \bar\tau)e^{\varepsilon(\bar\tau-\tau_0)}|, \   |\widetilde{U}_i(y_0, \bar\tau)e^{\varepsilon(\bar\tau-\tau_0)}|\big\} \leq C\frac{\smallc_1}{\smallc_g} \leq \frac{8 \smallc_0}{\Lambda-1}\frac{1}{C_1} \big(\frac{1}{100} \big)^{\frac{1}{20}}.
\end{equation*}
For $|y_0|> C_0$ and $\bar\tau=\tau_0$, according to the initial condition \eqref{initialdatacondition1} and $\smallc_1 \ll \smallc_0$, we have
\begin{equation*}
    \max\big\{|\widetilde{\Clss}(y_0, \tau_0)|,\  |\widetilde{U}_i(y_0, \tau_0)|\big\} \leq \smallc_1 \leq  \frac{8\smallc_0}{\Lambda-1}\frac{1}{C_1} \big(\frac{1}{100} \big)^{\frac{1}{20}}.
\end{equation*}
Then the upper bound of $\widetilde{U}$ follows from the chosen constant parameters \eqref{parameterchosen}{\rm:} 
$$
\frac{32}{\Lambda-1}\frac{1}{C_1} \big( \frac{1}{100} \big)^{\frac{1}{20}}= \frac{1}{100}.
$$
The upper bound of $\widetilde{\Clss}$ can be proved by a similar way. 

This completes the proof.
\end{proof}

Now we control $\|\nabla^4\widetilde{\Clss}\|_{L^2(B^c(0,C_0))} $ and $\|\nabla^4\widetilde{U}\|_{L^2(B^c(0,C_0))}$. 

\begin{Lemma}\label{boosstrapestimate2}
    For $\tau \in [\tau_0, \tau_*]$,
    \begin{equation*}
\max\big\{\|\nabla^4\widetilde{\Clss}\|_{L^2(B^c(0,C_0))}, \ \|\nabla^4\widetilde{U}\|_{L^2(B^c(0,C_0))} \big\}\leq  \frac{\smallc_0}{2}e^{-\varepsilon(\tau-\tau_0)}.
\end{equation*}
\end{Lemma}

\begin{proof}
We divide the proof into four steps.

\smallskip
\textbf{1.}   
First, applying the fourth-order derivative $\partial_{\beta}$ with multi-index $\beta=(\beta_1,\beta_2,\beta_3,\beta_4)$ satisfying $|\beta|=4$ to \eqref{eq: PerS}--\eqref{eq: PerU}, we obtain
\begin{equation*}
\begin{aligned}
\partial_{\tau}\partial_{\beta} \widetilde{\Clss}&=B_{\subclss,0}(\widetilde{U},\widetilde{\Clss})+B_{\subclss,1}(\widetilde{U},\widetilde{\Clss})+B_{\subclss,2}(\widetilde{U},\widetilde{\Clss})+B_{\subclss,3}(\widetilde{U},\widetilde{\Clss})+\partial_{\beta}\mathcal N_\subclss+\partial_{\beta}\mathcal E_\subclss,\\
\partial_{\tau}\partial_{\beta} \widetilde{U}&=B_{u,0}(\widetilde{U},\widetilde{\Clss})+B_{u,1}(\widetilde{U},\widetilde{\Clss})+B_{u,2}(\widetilde{U},\widetilde{\Clss})+B_{u,3}(\widetilde{U},\widetilde{\Clss})+\partial_{\beta}\mathcal N_u+\partial_{\beta}\mathcal E_u+\partial_{\beta}\mathcal F_{\rm dis}.
\end{aligned}
\end{equation*}
Here $B_{\subclss,0}$ and $B_{u,0}$ collect the linear terms with the fifth-order derivatives 
fall on $\widetilde{\Clss}$ and $\widetilde{U}$. 
The terms $B_{\subclss,1}$ and $B_{u,1}$ consist of the linear terms 
with the fourth-order derivatives acting on $\widetilde{\Clss}$ and $\widetilde{U}$. 
Similarly, $B_{\subclss,2}$ and $B_{u,2}$ contain the linear terms with the third-order derivatives, 
while $B_{\subclss,3}$ and $B_{u,3}$ comprise the linear terms with less than or equal 
to the second-order derivatives acting on $\widetilde{\Clss}$ and $\widetilde{U}$.

Multiplying the above equations by $\partial_{\beta} \widetilde{\Clss}$ and 
$\partial_{\beta} \widetilde{U}$ respectively, summing over all 
$|\beta|=4$, and then integrating over $B^c(0,C_0)$, we obtain
    \begin{align}
       I&=\frac{\mathrm d}{\mathrm d\tau} \big( \|\nabla^4\widetilde{\Clss}\|_{L^2(B^c(0,C_0))}^{2}+ \|\nabla^4\widetilde{U}\|_{L^2(B^c(0,C_0))}^{2} \big) \nonumber \\ 
       &= \frac{\mathrm d}{\mathrm d\tau}\sum_{|\beta|=4}\int_{B^c(0,C_0)}\big(|\partial_{\beta}\widetilde{\Clss}|^2+|\partial_{\beta}\widetilde{U}|^2\big)\,\text{d}y \nonumber \\
       &=\sum_{|\beta|=4}\int_{B^c(0,C_0)}\frac{\mathrm d}{\mathrm d\tau}\big(|\partial_{\beta}\widetilde{\Clss}|^2+|\partial_{\beta}\widetilde{U}|^2\big)\,\text{d}y \nonumber \\
       &=2\sum_{|\beta|=4}\sum_{j=0}^{3} \bigg( \int_{B^c(0,C_0)}\partial_{\beta}\widetilde{U}\cdot B_{u,j}\,\text{d}y
       +\int_{B^c(0,C_0)} \partial_{\beta}\widetilde{\Clss}\,B_{\subclss,j}\,\text{d}y \bigg) \label{hignenergeest01}\\
       &\quad+2 \sum_{|\beta|=4}\bigg( \int_{B^c(0,C_0)}\partial_{\beta}\widetilde{U}\cdot \partial_{\beta}\mathcal{N}_u\,\text{d}y+\int_{B^c(0,C_0)}\partial_{\beta}\widetilde{\Clss}\, \partial_{\beta}\mathcal{N}_\subclss\,\text{d}y \bigg) \nonumber  \\
       &\quad+2 \sum_{|\beta|=4}\bigg( \int_{B^c(0,C_0)}\partial_{\beta}\widetilde{U}\cdot \partial_{\beta}\mathcal{E}_u\,\text{d}y
       +\int_{B^c(0,C_0)}\partial_{\beta}\widetilde{\Clss} 
       \,\partial_{\beta}\mathcal{E}_\subclss\,\text{d}y \bigg) \nonumber \\
       &\quad+2  \sum_{|\beta|=4}\int_{B^c(0,C_0)}\partial_{\beta}\widetilde{U}\cdot \partial_{\beta}\mathcal{F}_{\rm dis}\,\text{d}y \nonumber \\
       &:= \sum_{|\beta|=4}\Big(\sum_{j=0}^3 I_{j, \beta}+I_{N,\beta}+I_{E,\beta}+I_{F, \beta}\Big). 
       \nonumber
    \end{align}
    
\textbf{2.}
For $I_{0,\beta}$, it follows from  \eqref{radialrep} and integration by parts that 
    \begin{align*}
        I_{0,\beta}&=-2\int_{B^c(0,C_0)} \Big(\partial_{\beta}\widetilde{U}\cdot\big((y+\widehat{X}\overline{U})\cdot \nabla\big)\partial_{\beta}\widetilde{U} +(\partial_{\beta}\widetilde{U}\cdot  \nabla) \partial_{\beta}\widetilde{\Clss}(\alpha \widehat{X}\overline{\Clss})\Big)\,\text{d}y\\
        &\quad-2\int_{B^c(0,C_0)}\Big(\partial_{\beta}\widetilde{\Clss}(y+\widehat{X}\overline{U})\cdot \nabla \partial_{\beta}\widetilde{\Clss}+\partial_{\beta}\widetilde{\Clss}\,\dive (\partial_{\beta}\widetilde{U})\alpha \widehat{X}\overline{\Clss}\Big)\,\text{d}y\\
        &\leq \int_{B^c(0,C_0)}  \Big( 3+|\dive (\widehat{X}\overline{U})| \Big) \big(|\partial_{\beta}\widetilde{\Clss}|^2+|\partial_{\beta}\widetilde{U}|^2 \big)\,\text{d}y\\
        &\quad+2\int_{B^c(0,C_0)}| \nabla(\alpha\widehat{X}\overline{\Clss})||\partial_{\beta}\widetilde{\Clss}||\partial_{\beta}\widetilde{U}|\,\text{d}y\\
        &\quad+\int_{\partial B(0,C_0)} \big(C_0+|\widehat{X}\overline{\mathcal U}|\big) \big(|\partial_{\beta}\widetilde{\Clss}|^2+|\partial_{\beta}\widetilde{U}|^2 \big) \,\text{d}S
        +2\int_{\partial B(0,C_0)}|\alpha\widehat{X}\overline{\Clss}||\partial_{\beta}\widetilde{\Clss}||\partial_{\beta}\widetilde{U}|\,\text{d}S.
    \end{align*}
By \eqref{Para-ineq}, \eqref{profile decay}, Lemma \ref{linfitnityinside}, and the smoothness of $(\widetilde{\Clss}, \widetilde{U})$, we have
    \begin{align*}
        &\int_{\partial B(0,C_0)} \big(C_0+|\widehat{X}\overline{\mathcal U}|\big) \big(|\partial_{\beta}\widetilde{\Clss}|^2+|\partial_{\beta}\widetilde{U}|^2\big)\,\text{d}S
        +2\int_{\partial B(0,C_0)}|\alpha\widehat{X}\overline{\Clss}||\partial_{\beta}\widetilde{\Clss}||\partial_{\beta}\widetilde{U}|\,\text{d}S \\
        &\leq  C(C_0+|\widehat{X}\overline{\mathcal U}|+\alpha |\widehat{X}\overline{\Clss}|) \big( \frac{\smallc_1}{\smallc_g} \big)^{2}e^{-\smallc_g(\tau-\tau_0)}\\
        &\leq  \big( \frac{\smallc_1}{\smallc_g} \big)^{\frac{17}{10}}e^{-2\varepsilon(\tau-\tau_0)}.
    \end{align*}
Consequently,
\begin{equation}\label{termB0}
    \begin{aligned}
        I_{0,\beta}&\leq \int_{B^c(0,C_0)}  \big( 3+3|\nabla(\widehat{X}\overline{U})|+ |\nabla(\alpha\widehat{X}\overline{\Clss})| \big) \big(|\partial_{\beta}\widetilde{\Clss}|^2+|\partial_{\beta}\widetilde{U}|^2\big)\,\text{d}y 
        + \big( \frac{\smallc_1}{\smallc_g} \big)^{\frac{17}{10}}e^{-2\varepsilon(\tau-\tau_0)}. 
    \end{aligned}
\end{equation}

Next, $B_{\subclss,1}$ and $B_{u,1}$ can be rewritten as
\begin{align*}
B_{\subclss,1}&=-(\Lambda-1)\partial_{\beta}\widetilde{\Clss}-(\partial_{\beta}\widetilde{U}\cdot \nabla )(\widehat{X}\overline{\Clss})-\alpha\partial_{\beta}\widetilde{\Clss}\,\dive (\widehat{X}\overline{U})\\
&\quad-\sum_{j=1}^{4}\sum_{l=1}^{3}\partial_{y_{\beta_j}}(\widehat{X}\overline{U}_l+y_l)\partial_{y_l}\partial_{\beta^{(j)}} \widetilde{\Clss}-\alpha\sum_{j=1}^{4}\partial_{y_{\beta_j}}(\widehat{X}\overline{\Clss})\partial_{\beta^{(j)}} \dive\,\widetilde{U}\\
&=-(\Lambda+3)\partial_{\beta}\widetilde{\Clss}-(\partial_{\beta}\widetilde{U}\cdot \nabla )(\widehat{X}\overline{\Clss})-\alpha\partial_{\beta}\widetilde{\Clss}\,\dive (\widehat{X}\overline{U})\\
&\quad-\sum_{j=1}^{4}\sum_{l=1}^{3}\partial_{y_{\beta_j}}(\widehat{X}\overline{U}_l)\partial_{y_l}\partial_{\beta^{(j)}} \widetilde{\Clss}-\alpha\sum_{j=1}^{4}\partial_{y_{\beta_j}}(\widehat{X}\overline{\Clss})\partial_{\beta^{(j)}} \dive\,\widetilde{U},
\end{align*}
and
\begin{align*}
B_{u,1}&=-(\Lambda-1)\partial_{\beta}\widetilde{U}-(\partial_{\beta}\widetilde{U}\cdot\nabla)(\widehat{X}\overline{U})-\alpha\partial_{\beta}\widetilde{\Clss}\nabla(\widehat{X}\overline{\Clss})\\
&
\quad-\sum_{j=1}^{4}\sum_{l=1}^{3}\partial_{y_{\beta_j}}(\widehat{X}\overline{U}_l+y_l)\partial_{y_l}\partial_{\beta^{(j)}} \widetilde{U}-\alpha \sum_{j=1}^{4}\partial_{y_{\beta_j}}(\widehat{X}\overline{\Clss})\nabla(\partial_{\beta^{(j)}} \widetilde{\Clss})\\
&=-(\Lambda+3)\partial_{\beta}\widetilde{U}-(\partial_{\beta}\widetilde{U}\cdot\nabla)(\widehat{X}\overline{U})-\alpha\partial_{\beta}\widetilde{\Clss}\nabla(\widehat{X}\overline{\Clss})\\
&\quad-\sum_{j=1}^{4}\sum_{l=1}^{3}\partial_{y_{\beta_j}}(\widehat{X}\overline{U}_l)\partial_{y_l}\partial_{\beta^{(j)}} \widetilde{U}-\alpha \sum_{j=1}^{4}\partial_{y_{\beta_j}}(\widehat{X}\overline{\Clss})\nabla(\partial_{\beta^{(j)}} \widetilde{\Clss}),
\end{align*}
where the differential operator $\partial_{\beta^{(j)}}$ is defined in \S \ref{operator}. Since 
$$\alpha =\frac{\gamma-1}{2},\quad 1 <\gamma <1+\frac{2}{\sqrt{3}},
$$
we see that $0<\alpha <\frac{1}{\sqrt{3}}$.
Then it follows that
\begin{align}
        I_{1,\beta}&=2\int_{B^c(0,C_0)}\partial_{\beta}\widetilde{\Clss} B_{\subclss,1}\,\text{d}y+ 2\int_{B^c(0,C_0)}\partial_{\beta}\widetilde{U}\cdot B_{u,1}\,\text{d}y\nonumber \\
        &\leq -2(\Lambda+3)\int_{B^c(0,C_0)}\big(|\partial_{\beta}\widetilde{\Clss}|^2+|\partial_{\beta}\widetilde{U}|^2\big)\,\text{d}y \label{termB1}\\
&\quad+20\big(\|\nabla(\widehat{X}\overline{\Clss})\|_{L^{\infty}(B^c(0,C_0))}+ \|\nabla(\widehat{X}\overline{U})\|_{L^{\infty}(B^c(0,C_0))} \big)\nonumber \\
        &\qquad\,\,\, \times \big(\|\nabla^4\widetilde{\Clss}\|_{L^2(B^c(0,C_0))}^2+\|\nabla^4\widetilde{U}\|_{L^2(B^c(0,C_0))}^2 \big).\nonumber
\end{align}
For $B_{u,2}$ and $B_{\subclss,2}$, according to \eqref{profilecondition3}, the  H\"older inequality. 
and Lemma \ref{lemma:GN_generalnoweightwholespace} with 
$$
p=\infty,\quad q=2,\quad \theta=\frac{3}{4},\quad \bar r=\frac{8}{3},\quad i=3,\quad l=4,
$$
we derive
    \begin{align*}
        &\|B_{\subclss,2}\|_{L^2(B^c(0,C_0))}+\|B_{u,2}\|_{L^2(B^c(0,C_0))}\\
        & \leq 80 \big( \|\nabla^{2}(\widehat{X}\overline{U})\|_{L^{8}({B^c(0,C_0)})}+\|\nabla^2(\widehat{X}\overline{\Clss})\|_{L^{8}({B^c(0,C_0)})} \big)\\
        &\quad\,\,\, \times\big( \|\nabla^{3}\widetilde{U}\|_{L^{\frac{8}{3}}(B^c(0,C_0))}+\|\nabla^3{\widetilde{\Clss}}\|_{L^{\frac{8}{3}}(B^c(0,C_0))} \big) \\
        & \leq  80 \big( \|\nabla^{2}(\widehat{X}\overline{U})\|_{L^{8}({B^c(0,C_0)})}+\|\nabla^2(\widehat{X}\overline{\Clss})\|_{L^{8}({B^c(0,C_0)})}\big)   \smallc_0 100 ^{-\frac{1}{20}} e^{-\varepsilon(\tau-\tau_0)}.
    \end{align*}

  Thus, we have
    \begin{align}
I_{2,\beta}&=2\int_{B^c(0,C_0)}\partial_{\beta}\widetilde{U}\cdot B_{u,2}\,\text{d}y+2\int_{B^c(0,C_0)}\partial_{\beta}\widetilde{\Clss}\, B_{\subclss,2}\,\text{d}y \nonumber \\
        &\leq 160 \big( \|\nabla^2(\widehat{X}\overline{\Clss})\|_{L^{8}({B^c(0,C_0)})}+ \|\nabla^{2}(\widehat{X}\overline{U})\|_{L^{8}({B^c(0,C_0)})}\big) \\
        &\quad \, \times \Big(\int_{B^c(0,C_0)}\big(|\partial_{\beta}\widetilde{\Clss}|^2+|\partial_{\beta}\widetilde{U}|^2\big) \text{d}y\Big)^{\frac{1}{2}}  \smallc_0 100 ^{-\frac{1}{20}} e^{-\varepsilon(\tau-\tau_0)}  . \nonumber
    \end{align}
For $B_{u,3}$ and $B_{\subclss,3}$, we employ  \eqref{profilecondition3}, \eqref{lowerestimate1}--\eqref{lowerestimate2}, and Lemma \ref{lemma:GN_generalnoweightwholespace} with 
$$p=\infty,\quad q=2,\quad \theta=\frac{4}{5},\quad \bar{r}=\infty,\quad i=2, \quad l=4,$$
and 
$$p=\infty,\quad q=2,\quad \theta=\frac{2}{5},\quad \bar{r}=\infty,\quad i=1, \quad l=4,$$
to obtain
    \begin{align*}
        &\|B_{\subclss,3}\|_{L^2(B^c(0,C_0))}+\|B_{u,3}\|_{L^2(B^c(0,C_0))} \\
        & \leq   80 \big( \|\widetilde{U}\|_{W^{2,\infty}(B^c(0,C_0))}+\|\widetilde{\Clss}\|_{W^{2,\infty}(B^c(0,C_0))} \big)        \\
        &\quad\,\, \times \sum_{j=3}^{5} \big( \|\nabla^j(\widehat{X}\overline{\Clss})\|_{L^2(B^c(0,C_0))}+\|\nabla^j(\widehat{X}\overline{U})\|_{L^2(B^c(0,C_0))} \big) \\
        & \leq  80\sum_{j=3}^{5} \big( \|\nabla^j(\widehat{X}\overline{\Clss})\|_{L^2(B^c(0,C_0))}+\|\nabla^j(\widehat{X}\overline{U})\|_{L^2(B^c(0,C_0))} \big) \smallc_0 \, 100 ^{-\frac{1}{20}}e^{-\varepsilon(\tau-\tau_0)},
    \end{align*}
which implies
\begin{align}
I_{3,\beta}&=2\int_{B^c(0,C_0)}\partial_{\beta}\widetilde{U}\cdot B_{u,3}\,\text{d}y+2\int_{B^c(0,C_0)}\partial_{\beta}\widetilde{\Clss}\, B_{\subclss,3}\,\text{d}y  \nonumber\\
&\leq 160\sum_{j=3}^{5} \big( \|\nabla^j\widehat{X}\overline{\Clss}\|_{L^2(B^c(0,C_0))}+\|\nabla^j(\widehat{X}\overline{U})\|_{L^2(B^c(0,C_0))} \big) \label{termB3}\\
& \quad\,\, \times \Big(\int_{B^c(0,C_0)}\big(|\partial_{\beta}\widetilde{\Clss}|^2+|\partial_{\beta}\widetilde{U}|^2\big)\text{d}y\Big)^{\frac{1}{2}}\smallc_0 100 ^{-\frac{1}{20}} e^{-\varepsilon(\tau-\tau_0)} .\nonumber
\end{align}

\textbf{3. } For the nonlinear terms $(\mathcal{N}_\subclss, \mathcal{N}_u)$,  it follows from \eqref{Para-ineq}, \eqref{lowerestimate1}--\eqref{lowerestimate2}, and Lemmas \ref{per-U,S} and  \ref{lemma:GN_generalnoweightwholespace} that 
    \begin{align*}
        \|\nabla^2 \widetilde{U} \|_{L^8(B^c(0,C_0))} &\leq C\big( \|\widetilde{U} \|_{L^\infty(B^c(0,C_0))}^{\frac{7}{20}}\|\nabla^4 \widetilde{U} \|_{L^2(B^c(0,C_0))}^{\frac{13}{20}} + \|\widetilde{U} \|_{L^\infty(B^c(0,C_0))}\big) \\
        &\leq C\smallc_0 \Big(\frac{1}{100}+\big(\frac{1}{100}
        \big)^{\frac{7}{20}}\Big)e^{\varepsilon(\tau-\tau_0)},\\
        \|\nabla^3 \widetilde{U} \|_{L^{\frac{8}{3}}(B^c(0,C_0))} &\leq C\big(\|\widetilde{U} \|_{L^\infty(B^c(0,C_0))}^{\frac{1}{4}}\|\nabla^4 \widetilde{U} \|_{L^2(B^c(0,C_0))}^{\frac{3}{4}} + \|\widetilde{U} \|_{L^\infty(B^c(0,C_0))}\big) \\
        &\leq C\smallc_0 \Big(\frac{1}{100}+\big(\frac{1}{100} \big)^{\frac{1}{4}}\Big)e^{\varepsilon(\tau-\tau_0)}, \\
        \|\nabla^5 \widetilde{U} \|_{L^2(B^c(0,C_0))} &\leq C(K)\big( \|\nabla^4 \widetilde{U} \|_{L^2(B^c(0,C_0))}^{\frac{K-5}{K-4}} \|\nabla^K \widetilde{U} \|_{L^2(B^c(0,C_0))}^{\frac{1}{K-4}}+ \|\nabla^4\widetilde{U} \|_{L^2(B^c(0,C_0))}\big) \\
        &\leq C(K)(E^{\frac{1}{10}}\smallc_0^{\frac{9}{10}}e^{-\frac{9}{10}\varepsilon(\tau-\tau_0)}+\smallc_0e^{-\varepsilon(\tau-\tau_0)}) \leq C(K)\smallc_0^{\frac{4}{5}}e^{-\frac{9}{10}\varepsilon(\tau-\tau_0)}.
    \end{align*}
The corresponding estimates for $\widetilde{\Clss}$ follow analogously. Therefore, we obtain from $\smallc_0 \ll \frac{1}{K}$ that
\begin{align*}
&\|\partial_{\beta}\mathcal{N}_\subclss\|_{L^{2}(B^c(0,C_0))}+\|\partial_{\beta}\mathcal{N}_u\|_{L^{2}(B^c(0,C_0))}\\
&\leq C \big( \|\nabla ^{5}\widetilde{U}\|_{{L^{2}(B^c(0,C_0))}}+\|\nabla ^{5}\widetilde{\Clss}\|_{{L^{2}(B^c(0,C_0))}} \big) \big( \|\widetilde{U}\|_{{L^{\infty}(B^c(0,C_0))}}+\|\widetilde{\Clss}\|_{{L^{\infty}(B^c(0,C_0))}} \big) \\
 &\quad + C \big( \|\nabla ^{4}\widetilde{U}\|_{{L^{2}(B^c(0,C_0))}}+\|\nabla ^{4}\widetilde{\Clss}\|_{{L^{2}(B^c(0,C_0))}} \big) \big( \|\nabla\widetilde{U}\|_{{L^{\infty}(B^c(0,C_0))}}+\|\nabla \widetilde{\Clss}\|_{{L^{\infty}(B^c(0,C_0))}} \big)\\
& \quad+ C \big( \|\nabla ^{3}\widetilde{U}\|_{{L^{\frac{8}{3}}(B^c(0,C_0))}}+\|\nabla ^{3}\widetilde{\Clss}\|_{{L^{\frac{8}{3}}(B^c(0,C_0))}} \big) \big( \|\nabla ^{2}\widetilde{U}\|_{{L^{8}(B^c(0,C_0))}}+\|\nabla ^{2}\widetilde{\Clss}\|_{{L^{8}(B^c(0,C_0))}} \big) \\
& \leq C(K) \smallc_0^{\frac{9}{5}}e^{-\frac{19}{10}\varepsilon(\tau-\tau_0)} \leq \smallc_0^{\frac{3}{2}}e^{-\frac{3}{2}\varepsilon(\tau-\tau_0)}.
\end{align*}
Consequently,
\begin{equation}
\begin{aligned}\label{termN} 
I_{N,\beta}&=2\int_{B^c(0,C_0)}\partial_{\beta}\widetilde{U}\cdot \partial_{\beta}\mathcal N_u \,\text{d}y+2\int_{B^c(0,C_0)}\partial_{\beta}\widetilde{\Clss}\, \partial_{\beta}\mathcal N_\subclss \, \text{d}y  \\
&\leq 2{\smallc_0}^{\frac{3}{2}}e^{-\frac{3}{2}\varepsilon(\tau-\tau_0)}\Big(\int_{B^c(0,C_0)}\big(|\partial_{\beta}\widetilde{U}|^2+|\partial_{\beta}\widetilde{\Clss}|^2\big)\,
\text{d}y \Big)^{\frac{1}{2}}.
\end{aligned}
\end{equation}
For terms $\mathcal{E}_u$ and $\mathcal{E}_\subclss$, Lemma
\ref{Forcingterm1} gives
\begin{equation}
    \begin{aligned}\label{termE}
I_{E,\beta}&=2\int_{B^c(0,C_0)}\partial_{\beta}\widetilde{U}\cdot \partial_{\beta}\mathcal{E}_u \,\text{d}y+2\int_{B^c(0,C_0)}\partial_{\beta}\widetilde{\Clss}\, \partial_{\beta}\mathcal{E}_\subclss 
\,\text{d}y  \\
        & \leq 4\smallc_1e^{-(\tau-\tau_0)}\Big(\int_{B^c(0,C_0)}\big(|\partial_{\beta}\widetilde{U}|^2+|\partial_{\beta}\widetilde{\Clss}|^2\big)\,\text{d}y\Big)^{\frac{1}{2}}.
    \end{aligned}
\end{equation}
Before estimating the diffusion term $\mathcal{F}_{\rm dis}$, 
we first apply \eqref{decayestimatederj} to derive the following estimate{\rm:}
\begin{equation*}
    \begin{aligned}
        \left|\nabla^j(L(U))\right| &\leq C(\smallc_0, j) \langle y \rangle^{-\frac{K-(j+2)-3/2+K(1-\eta)(j+2-1)}{K- 5/2}}=C(\smallc_0, j)\langle y \rangle^{-\frac{K-j-7/2+K(1-\eta)(j+1)}{K- 5/2}}.
    \end{aligned}
\end{equation*}
Moreover, for any integer $\ell \in [1,j]$ and integers $k_i\ge 1$ ($1\leq i \leq \ell$) satisfying $k_1+\cdots +k_\ell = j$, it follows from Lemmas \ref{lemma:Sbounds}--\ref{higher-order} and $\smallc_0^{\frac{3}{2}}\ll \smallc_1 \ll \smallc_0 $ that
\begin{equation*}
    \begin{aligned}
        \big|\nabla^j(\Clss^{\frac{\delta-1}{\alpha}})\big| &\leq C(j)\sum_{k_1+\cdots +k_\ell = j}\frac{|\nabla^{k_1}\Clss||\nabla^{k_2}\Clss|\cdots |\nabla^{k_\ell}\Clss|}{\Clss^\ell}\Clss^{\frac{\delta-1}{\alpha}}\\
        &\leq C(\smallc_0, j) \sum_{ k_1+\cdots +k_\ell = j}\smallc_1^{-\ell+\frac{\delta-1}{\alpha}}\phi^{-\frac{k_1+ k_2+\cdots +k_\ell}{2}}\big\langle \frac{y}{R_0} \big\rangle^{-\frac{\delta-1}{\alpha}(\Lambda-1)} \\
        &\leq C(\smallc_0, j) \phi^{-\frac{j}{2}}\langle y \rangle^{\frac{1-\delta}{\alpha}(\Lambda-1)}.
    \end{aligned}
\end{equation*}
Combining the above two estimates with \eqref{range of Lambda}--\eqref{def:Lambda*} and \eqref{S-para}  yields
\begin{equation}\label{est:Fdis1}
    \begin{aligned}
        \big|\partial^{\beta}(\Clss^{\frac{\delta-1}{\alpha}}L(U))\big| &\leq C\sum_{0\leq j\leq 4} \big|\nabla^{4-j}(L(U))\big| \big|\nabla^{j}(\Clss^{\frac{\delta-1}{\alpha}})\big| \\
        &\leq C(\smallc_0) \sum_{0\leq j\leq 4}\langle y \rangle^{\frac{1-\delta}{\alpha}(\Lambda-1)-j(1-\eta)-\frac{K+j-15/2+K(1-\eta)(5-j)}{K- 5/2}}\\
        &\leq C(\smallc_0) \langle y \rangle^{-5(1-\eta)+2-\Lambda+\delta_{\rm{dis}}} \leq C(\smallc_0)\langle y \rangle^{-3+5\eta},
    \end{aligned}
\end{equation}
and
\begin{equation}\label{est:Fdis2}
    \begin{aligned}
        \big|\partial^{\beta}(\Clss^{\frac{\delta-1}{\alpha}-1}\nabla \Clss \md(U))\big| &\leq C \sum_{0\leq j\leq 4} \big|\nabla^{4-j}(\md(U))\big| \big|\nabla^{j}(\Clss^{\frac{\delta-1}{\alpha}-1}\nabla \Clss)\big| \\
        &\leq C(\smallc_0) \sum_{0\leq j\leq 4}\big|\nabla^{4-j}(\md(U))\big| \phi^{-\frac{j+1}{2}}\langle y \rangle^{\frac{1-\delta}{\alpha}(\Lambda-1)}\\
        &\leq C(\smallc_0) \sum_{0\leq j\leq 4}\langle y \rangle^{\frac{1-\delta}{\alpha}(\Lambda-1)-(j+1)(1-\eta)-\frac{K+j-13/2+K(1-\eta)(4-j)}{K- 5/2}}\\
        &\leq C(\smallc_0) \langle y \rangle^{-5(1-\eta)+2-\Lambda+\delta_{\rm{dis}}} \leq C(\smallc_0)\langle y \rangle^{-3+5\eta}.
    \end{aligned}
\end{equation}
Using \eqref{Para-ineq} and \eqref{est:Fdis1}--\eqref{est:Fdis2} and denoting $r=|y|$, we obtain
\begin{equation*}
    \begin{aligned}
    \|\partial^{\beta}\mathcal{F}_{\rm dis}\|_{L^{2}(B^c(0,C_0))} &\leq C(\smallc_0) e^{-\delta_{\rm{dis}}\tau}\Big(\int_{ C_0}^\infty \langle r\rangle^{-6+10\eta} r^2 \text{d}r\Big)^{\frac{1}{2}} \leq C(\smallc_0) e^{-\delta_{\rm{dis}}\tau}\\
    &\leq C(\smallc_0)\smallc_0^{-\frac{3}{2}}e^{-\delta_{\rm dis}\tau_0}\smallc_0^{\frac{3}{2}}e^{-\delta_{\rm dis}(\tau-\tau_0)}  \leq C(\smallc_0)e^{-\delta_{\rm dis}\tau_0}\smallc_1e^{-\smallc_g(\tau-\tau_0)} \\
    &\leq \smallc_1e^{-\smallc_g(\tau-\tau_0)}.
    \end{aligned}
\end{equation*}
Therefore, we have
\begin{equation}
    \begin{aligned}\label{termF}
        I_{F, \beta} &= 2\int_{B^c(0,C_0)}\partial_{\beta}\widetilde{U}\cdot \partial_{\beta}\mathcal{F}_{\rm dis}\, \text{d}y  \\
        &\leq 2\smallc_1e^{-\smallc_g(\tau-\tau_0)} \Big(\int_{B^c(0,C_0)}\big(|\partial_{\beta}\widetilde{\Clss}|^2+|\partial_{\beta}\widetilde{U}|^2\big)
        \,\text{d}y\Big)^{\frac{1}{2}}. 
    \end{aligned}
\end{equation}    
\textbf{4.} 
Gathering all the above bounds together and using $0<\alpha <\frac{1}{\sqrt{3}}$, we obtain
\begin{align*}
        I \leq &\, \Big( -2\Lambda-3+ 3|\nabla(\widehat{X}\overline{U})|+ |\nabla(\alpha\widehat{X}\overline{\Clss})| \Big)\Big(\sum_{|\beta|=4}\int_{B^c(0,C_0)}\big(|\partial_{\beta}\widetilde{U}|^2+|\partial_{\beta}\widetilde{\Clss}|^2\big)\,\text{d}y\Big) \\
        & + 81 \big( \frac{\smallc_1}{\smallc_g} \big)^{\frac{17}{10}}e^{-2\varepsilon(\tau-\tau_0)}\\
        &+ 1620\big(\|\nabla(\widehat{X}\overline{U})\|_{L^{\infty}(B^c(0,C_0))}+ \|\nabla(\widehat{X}\overline{\Clss})\|_{L^{\infty}(B^c(0,C_0))}\big)\\
        &\quad\,\, \times\big(\|\nabla^4\widetilde{U}\|_{L^2(B^c(0,C_0))}^2+ \|\nabla^4\widetilde{\Clss}\|_{L^2(B^c(0,C_0))}^2 \big)\\
        &+ \sum_{|\beta|=4} \Big(\int_{B^c(0,C_0)}\big(|\partial_{\beta}\widetilde{U}|^2+|\partial_{\beta}\widetilde{\Clss}|^2\big)\,
        \text{d}y\Big)^{\frac{1}{2}}\\
        &\quad\,\,\times\Big( 160 \big( \|\nabla^{2}(\widehat{X}\overline{U})\|_{L^{8}({B^c(0,C_0)})}+\|\nabla^2(\widehat{X}\overline{\Clss})\|_{L^{8}({B^c(0,C_0)})}\big)\Big. \smallc_0 e^{-\varepsilon(\tau-\tau_0)} 100 ^{-\frac{1}{20}}\\
        &\qquad\quad\, + 160\sum_{j=3}^{5} \big( \|\nabla^j(\widehat{X}\overline{U})\|_{L^2(B^c(0,C_0))}+\|\nabla^j(\widehat{X}\overline{\Clss})\|_{L^2(B^c(0,C_0))} \big) \smallc_0 e^{-\varepsilon(\tau-\tau_0)}100 ^{-\frac{1}{20}}\\
        &\qquad\quad\, + \Big. 2{\smallc_0}^{\frac{3}{2}}e^{-\frac{3}{2}\varepsilon(\tau-\tau_0)} + 4\smallc_1e^{-(\tau-\tau_0)} + 2\smallc_1e^{-\smallc_g(\tau-\tau_0)}\Big)\\
        \leq &\, \Big( -3-2\Lambda+ 2000 \|(\nabla(\widehat{X}\overline{U}),\nabla(\widehat{X}\overline{\Clss}))\|_{L^{\infty}(B^c(0,C_0))} \Big)\\
         &\,\,\times \big(\|\nabla^4\widetilde{U}\|_{L^2(B^c(0,C_0))}^2+ \|\nabla^4\widetilde{\Clss}\|_{L^2(B^c(0,C_0))}^2 \big)+ 81 \big( \frac{\smallc_1}{\smallc_g} \big)^{\frac{17}{10}}e^{-2\varepsilon(\tau-\tau_0)} \\
         &+ 160\big(\|\nabla^4\widetilde{U}\|_{L^2(B^c(0,C_0))}^2+ \|\nabla^4\widetilde{\Clss}\|_{L^2(B^c(0,C_0))}^2 \big)^{\frac{1}{2}}\\
         &\quad\,\,\times\Big(  \big( \|\nabla^{2}(\widehat{X}\overline{U})\|_{L^{8}({B^c(0,C_0)})}+\|\nabla^2(\widehat{X}\overline{\Clss})\|_{L^{8}({B^c(0,C_0)})}\big)\Big. \smallc_0 e^{-\varepsilon(\tau-\tau_0)} 100 ^{-\frac{1}{20}}\\
         &\qquad\,\,\,\,\, + \sum_{j=3}^{5} \big( \|\nabla^j(\widehat{X}\overline{U})\|_{L^{2}(B^c(0,C_0))}+\|\nabla^j(\widehat{X}\overline{\Clss})\|_{L^{2}(B^c(0,C_0))} \big) \smallc_0 e^{-\varepsilon(\tau-\tau_0)}100 ^{-\frac{1}{20}}\\
        &\qquad\,\,\,\,\, + \Big. 2{\smallc_0}^{\frac{3}{2}}e^{-\frac{3}{2}\varepsilon(\tau-\tau_0)} + 4\smallc_1e^{-(\tau-\tau_0)} + 2\smallc_1e^{-\smallc_g(\tau-\tau_0)}\Big).
\end{align*}
It follows from \eqref{profilecondition4}--\eqref{profilecondition3} that
\begin{equation*}
    \begin{aligned}
        &\frac{\mathrm d}{\mathrm d\tau}\big(\|\nabla^4\widetilde{U}\|_{L^2(B^c(0,C_0))}^2+ \|\nabla^4\widetilde{\Clss}\|_{L^2(B^c(0,C_0))}^2 \big)\\
        & \leq  -\big(\|\nabla^4\widetilde{U}\|_{L^2(B^c(0,C_0))}^2+ \|\nabla^4\widetilde{\Clss}\|_{L^2(B^c(0,C_0))}^2  \big)+  81 \big( \frac{\smallc_1}{\smallc_g} \big)^{\frac{17}{10}}e^{-2\varepsilon(\tau-\tau_0)}.\\
        & \quad +\frac{1}{20}100^{-\frac{1}{20}}\smallc_0 e^{-\varepsilon(\tau-\tau_0)}\big(\|\nabla^4\widetilde{U}\|_{L^2(B^c(0,C_0))}^2+ \|\nabla^4\widetilde{\Clss}\|_{L^2(B^c(0,C_0))}^2 \big)^{\frac{1}{2}}.
    \end{aligned}
\end{equation*}
For convenience, denote $A(\tau) =\big(\|\nabla^4\widetilde{U}\|_{L^2(B^c(0,C_0))}^2+ \|\nabla^4\widetilde{\Clss}\|_{L^2(B^c(0,C_0))}^2 \big)e^{2\varepsilon(\tau-\tau_0)}$. Then 
\begin{align*}
    \frac{\mathrm d}{\mathrm d\tau}A(\tau) \leq& -\frac{1}{2}A(\tau)+ \frac{\smallc_0}{20} 100^{-\frac{1}{20}} \sqrt{A(\tau)} +81 \big( \frac{\smallc_1}{\smallc_g} \big)^{\frac{17}{10}} \\
    \leq & -\frac{3}{8}A(\tau)+ \Big(\frac{\smallc_0}{10}100^{-\frac{1}{20}} \Big)^{2} + 81\big( \frac{\smallc_1}{\smallc_g} \big)^{\frac{17}{10}}.
\end{align*}
Then applying the Gr\"onwall inequality yields
\begin{align*}
    A(\tau) &\leq A(\tau_0)e^{-\frac{3}{8}(\tau-\tau_0)} + \frac{8}{3}\bigg( \Big(\frac{\smallc_0}{10}100^{-\frac{1}{20}} \Big)^{2} + 81\big( \frac{\smallc_1}{\smallc_g} \big)^{\frac{17}{10}} \bigg)
    \big(1-e^{-\frac{3}{8}(\tau-\tau_0)}\big)\\
    &\leq  A(\tau_0)e^{-\frac{3}{8}(\tau-\tau_0)} + 3\bigg( \Big(\frac{\smallc_0}{10}100^{-\frac{1}{20}} \Big)^{2} + 81\big( \frac{\smallc_1}{\smallc_g} \big)^{\frac{17}{10}} \bigg)\big(1-e^{-\frac{3}{8}(\tau-\tau_0)}\big)\\
    &\leq e^{-\frac{3}{8}(\tau-\tau_0)}\bigg(A(\tau_0)-\Big( \frac{3\smallc_0^2}{100}100^{-\frac{1}{10}}  + 243\big( \frac{\smallc_1}{\smallc_g} \big)^{\frac{17}{10}} \Big)\bigg)+ \Big( \frac{3\smallc_0^2}{100} 100^{-\frac{1}{10}}  + 243\big( \frac{\smallc_1}{\smallc_g} \big)^{\frac{17}{10}} \Big).
\end{align*}
Using \eqref{Para-ineq} and the Gagliardo-Nirenberg inequality \eqref{eq:GNresultnoweightwholespace} between $\|(\widetilde{U}_0,\widetilde{\Clss}_0)\|_{L^{\infty}}$ and $\widetilde{E}_{K}$, we have
\begin{equation*}
\begin{split}
A(\tau_0)&=\|\nabla^4\widetilde{U}_0\|_{L^2(B^c(0,C_0))}^2+ \|\nabla^4\widetilde{\Clss}_0\|_{L^2(B^c(0,C_0))}^2 \leq C(K)(\smallc_1^{\frac{2K-8}{K-3/2}}E^{\frac{5}{K-3/2}}+\smallc_1^2) \\
&\leq \smallc_1^{\frac{9}{5}}\leq  \frac{3\smallc_0^2}{100}100^{-\frac{1}{10}}  + 243\big( \frac{\smallc_1}{\smallc_g} \big)^{\frac{17}{10}}  \leq \frac{1}{4}\smallc_0^2.
\end{split}
\end{equation*}
Therefore, we conclude
\begin{equation}
    \Big(\|\nabla^4\widetilde{U}\|_{L^2(B^c(0,C_0))}^2+ \|\nabla^4\widetilde{\Clss}\|_{L^2(B^c(0,C_0))}^2  \Big)e^{2\varepsilon(\tau-\tau_0)}\leq \frac{1}{4}\smallc_0^2.
\end{equation}
This completes the proof of Lemma \ref{boosstrapestimate2}.
\end{proof}

\subsubsection{Higher-order weighted energy estimates}\label{subsection6.2}
Next, we establish the bootstrap estimate for the higher-order weighted energy $E_K$. 
For the higher--order weighted energy estimates, we first derive the following commutator estimates for the $K$-th order spatial derivatives
of the nonlinear convection term.

\begin{Lemma}\label{lemma:6.6}
Let $\beta = (\beta_1, \beta_2, \cdots, \beta_K)$ be a multi-index with 
$\beta_i\in\{1,2,3\}$ for $i=1,\cdots,K$, 
and  let $\beta^{(j)}$, $1\le j \le K$, 
be the $\beta$--related multi-index as defined in \eqref{indexbetai}. 
Then, for $\tau \in [\tau_0, \tau_*]$, 
       \begin{align}
    &\Big\|\sum_{i=1}^3\Big(\partial_\beta (U \cdot \nabla U_i) - U\cdot \nabla \partial_\beta U_i -\sum_{j=1}^K \sum_{k=1}^3 L_{k\beta_j}\partial_{y_k}\partial_{\beta^{(j)}}U_i\Big)\phi^{\frac{K}{2}}\Big\|_{L^2}
    \nonumber\\    &\,\,
    \leq E^{\frac{1}{2}-\frac{1}{4K}} + C\| \partial_\beta U \phi^{\frac{K}{2}}\|_{L^2},\label{Est:I21'}\\[1mm]
     &\Big\|\Big(\partial_\beta (U \cdot \nabla \Clss) - U\cdot \nabla \partial_\beta \Clss -\sum_{j=1}^K \sum_{k=1}^3 L_{k\beta_j}\partial_{y_k}\partial_{\beta^{(j)}}\Clss\Big)\phi^{\frac{K}{2}}\Big\|_{L^2}\nonumber\\ 
    & \,\,\leq  E^{\frac{1}{2}-\frac{1}{4K}} + C\| \partial_\beta \Clss \phi^{\frac{K}{2}}\|_{L^2},\label{Est:I22'}\\[1mm]
    &\Big\|\sum_{i=1}^3\Big(\partial_\beta (\Clss \partial_{y_i} \Clss ) - \Clss \partial_\beta \partial_{y_i} \Clss -\sum_{j=1}^K \big(\partial_r(\widehat{
    X}\overline{\Clss})\frac{y_{\beta_j}}{|y|}\big)\partial_{y_i}\partial_{\beta^{(j)}}\Clss\Big)\phi^{\frac{K}{2}}\Big\|_{L^2}\nonumber\\ 
    &\,\,\leq  E^{\frac{1}{2}-\frac{1}{4K}} + C\| \partial_\beta \Clss \phi^{\frac{K}{2}}\|_{L^2},\label{Est:I32'}\\[1mm]
    &\Big\|\Big(\partial_\beta (\Clss \,\dive \,U) - \Clss \partial_\beta \,\dive\, U -\sum_{j=1}^K \big(\partial_r(\widehat{
    X}\overline{\Clss})\frac{y_{\beta_j}}{|y|}\big)\partial_{\beta^{(j)}}(\dive U)\Big)\phi^{\frac{K}{2}}\Big\|_{L^2}\nonumber  \\
    &\,\,\leq E^{\frac{1}{2}-\frac{1}{4K}} + C\| \partial_\beta U \phi^{\frac{K}{2}}\|_{L^2},\label{Est:I31'}
     \end{align}
where $L_{k\beta_j}$ is given by
$$
L_{k\beta_j}:=
\big(\partial_r(\widehat{
    X}\overline{\mathcal U})- \frac{\widehat{
    X}\overline{\mathcal U}}{|y|}\big)\frac{y_k y_{\beta_j}}{|y|^2}+\delta_{\beta_j, k}\frac{\widehat{
    X}\overline{\mathcal U}}{|y|}.
$$
\end{Lemma}

\begin{proof}
Since all those estimates are analogous, we just show \eqref{Est:I21'}.
   Using the Leibniz rule to $\partial_\beta (U \cdot \nabla U_i)$ yields that, for any integers $2 \le \ell, \ell^* \le K-1$ satisfying $\ell +\ell^* =K+1$, 
\begin{align*}
    &\big|\partial_\beta (U \cdot \nabla U_i) - U\cdot \nabla \partial_\beta U_i -\partial_\beta U\cdot \nabla U_i -\sum_{j=1}^K \partial_{y_{\beta_j}}U \cdot \nabla \partial_{\beta^{(j)}}U_i\big| \\
    &\leq C(K) \sum_{\substack{2 \le \ell,\ \ell^* \le K-1,\\ \ \ell + \ell^* = K+1}} | \nabla^\ell U |  | \nabla^{\ell^*} U |.
\end{align*}
To control  the terms on the right-hand side of the above, we apply  Lemma \ref{lemma:GN_general} with
$$
\varphi=\phi^{\frac{1}{2}},\quad\psi = 1, \quad p=\infty,\quad q=2,\quad l=K,\quad \bar{r} = \frac{2(K+1)}{\ell}, \quad\theta = \frac{\ell - 3/\bar{r}}{K - 3/2},
$$
together with \eqref{range of Lambda}--\eqref{def:Lambda*}, \eqref{lowerestimate1}, and Lemma~\ref{perturbationenergy}, to obtain
\begin{equation}\label{preestI21}
    \|\langle y \rangle^{\frac{1-\Lambda}{4(K+1)}}\phi^{\frac{K\theta}{2}}\nabla^{\ell} \widetilde{U}\|_{L^{\bar{r}}}
    \leq C(K)\big( \|\widetilde{U}\|_{L^{\infty}}^{1-\theta} \|\phi^{\frac{K}{2}}\nabla^K \widetilde{U}\|_{L^{2}}^{\theta} 
    + \|\widetilde{U}\|_{L^{\infty}}\big) 
    \leq C(K)\smallc_0^{1-\theta}E^{\frac{\theta}{2}}.
\end{equation}
It follows from \eqref{profile decay} and \eqref{estimatehatX}--\eqref{estimatehatX2} that
\begin{equation*}
\begin{aligned}
    |\nabla^{\ell} (\widehat{X}\overline{U})|
    &\leq |\widehat{X}\nabla^\ell (\overline{U})| + \sum_{1\le j \le \ell}|\nabla^j(\widehat{X})\nabla^{\ell -j}(\overline{U}) | \\
    &\leq  C\langle y\rangle^{-(\Lambda-1)-\ell}\mathbf{1}_{\{ |y| \leq e^{\tau}\}} + C\sum_{1\le j \le \ell} e^{-j\tau}|y|^{-(\Lambda-1)-(\ell -j)}\mathbf{1}_{\{\frac{e^\tau}{2} \leq |y| \leq e^{\tau}\}},
\end{aligned}
\end{equation*}
where
\begin{equation*}
\mathbf{1}_{\{ |y| \le e^\tau \}}
=
\begin{cases}
1 & \mbox{ for $|y| \le e^\tau$},\\[6pt]
0 & \mbox{ for $|y| > e^\tau$},
\end{cases}\qquad
\mathbf{1}_{\{\frac{e^\tau}{2} \le |y| \le e^\tau \}}
=
\begin{cases}
1 & \mbox{ for $\frac{e^\tau}{2} \le |y| \le e^\tau$},\\[6pt]
0 & \mbox{ for $|y| < \frac{e^\tau}{2}$ or $|y| > e^\tau$}.
\end{cases}
\end{equation*}
Consequently, it follows from $K\eta \gg 1$ that
\begin{equation}
\begin{aligned}\label{preestI22}
    &\|\phi^{\frac{K\theta}{2}}\nabla^{\ell} (\widehat{X}\overline{U})\|^{\bar{r}}_{L^{\bar{r}}} \\
    & \leq C\int_0^{e^\tau} \phi^{\frac{K\theta}{2}\bar{r}}\langle r \rangle^{(-(\Lambda-1)-\ell)\bar r}r^2\,\text{d}r \\
    &\quad + C\sum_{1\le j \le \ell} e^{-j\tau \bar{r}}\int_{e^\tau/2}^{e^\tau} \phi^{\frac{K\theta }{2}\bar{r}} r^{(-(\Lambda-1)-(\ell-j))\bar{r}}r^2\,\text{d}r  \\
    &\leq C\int_{0}^{e^\tau} \langle r \rangle^{-(\Lambda-1)\bar{r}-2K\eta+\frac{2K-2K\eta}{K-3/2}}\,\text{d}r 
    +C(K)e^{(1-(\Lambda-1)\bar{r}-2K\eta+\frac{2K-2K\eta}{K-3/2})\tau} \\
    &\leq C\int_{0}^{e^\tau}\langle r \rangle^{-(\Lambda-1)\bar{r}-1}\text{d}r + C(K)e^{-(\Lambda-1)\bar{r}\tau}\leq C(K). 
\end{aligned}
\end{equation}
We obtain from \eqref{Para-ineq}, \eqref{choiceR0},  \eqref{preestI21}, and the H\"older inequality that, for any integers $2 \le \ell, \ell^* \le K-1$ satisfying $\ell +\ell^* =K+1$ 
and for $r^*=\frac{2(K+1)}{\ell^*}$,
    \begin{align}  \label{EstI21} 
    &\big\||\nabla^{\ell}\widetilde{U}||\nabla^{\ell^*}\widetilde{U}|\phi^{\frac{K}{2}}\big\|_{L^2} \nonumber \\
    &= \big\||\nabla^{\ell}\widetilde{U}||\nabla^{\ell^*}\widetilde{U}|
    \phi^{\frac{K}{2}\frac{\ell-3/{\bar{r}}}{K-3/2}}\phi^{\frac{K}{2}\frac{\ell^*-3/{r^*}}{K-3/2}}\phi^{-\frac{1}{2}\frac{K}{K-3/2}}\big\|_{L^2}  \\
    &\leq  C(K)
    R_0^\frac{\Lambda-1}{2(K+1)}\big\|\langle y \rangle^\frac{1-\Lambda}{4(K+1)}\phi^{\frac{K}{2}\frac{\ell-3/{\bar{r}}}{K-3/2}}\nabla^{\ell}\widetilde{U}\big\|_{L^{\frac{2(K+1)}{\ell}}} \big\|\langle y \rangle^\frac{1-\Lambda}{4(K+1)}\phi^{\frac{K}{2}\frac{\ell^*-3/{r^*}}{K-3/2}}\nabla^{\ell^*}\widetilde{U}\big\|_{L^{\frac{2(K+1)}{\ell^*}}} \nonumber \\
    &\leq C(K)\smallc_0^{-\frac{1}{2(K+1)}}\smallc_0^{\frac{K-5/2}{K-3/2}}E^{\frac{K-1/2}{2K-3}} \leq C(K)\smallc_0^{\frac{K-9/2}{2K-3}-\frac{1}{2(K+1)}} \leq C(K). \nonumber
    \end{align}
Similarly, combining \eqref{Para-ineq}, \eqref{choiceR0}, and \eqref{preestI21}--\eqref{preestI22} leads to 
    \begin{align}\label{EstI22}   
    &\big\||\nabla^{\ell}\widetilde{U}||\nabla^{\ell^*}(\widehat{X}\overline{U})|\phi^{\frac{K}{2}}\big\|_{L^2} \nonumber \\
    &\leq C(K)R_0^{\frac{\Lambda-1}{4(K+1)}}\big\|\langle y \rangle^\frac{1-\Lambda}{4(K+1)}
    \phi^{\frac{K}{2}\frac{\ell-3/{\bar{r}}}{K-3/2}}\nabla^{\ell}\widetilde{U}\big\|_{L^{\frac{2(K+1)}{\ell}}} \big\|\phi^{\frac{K}{2}\frac{\ell^*-3/{r^*}}{K-3/2}}\nabla^{\ell^*}(\widehat{X}\overline{U})\big\|_{L^{\frac{2(K+1)}{\ell^*}}} \\
    & \leq C(K) \smallc_0^{-\frac{1}{4(K+1)}}\smallc_0^{\frac{K-2}{(K-3/2)(K+1)}}E^{\frac{1}{2}\big(1-\frac{K-2}{(K-3/2)(K+1)}\big)}\nonumber\\
    & \leq C(K)\smallc_0^{\frac{K-3}{(2K-3)(K+1)}-\frac{1}{4(K+1)}}E^{\frac{1}{2}-\frac{1}{K+1}} \leq C(K)E^{\frac{1}{2}-\frac{1}{K+1}}, \nonumber
    \end{align}
and
\begin{equation}
    \begin{aligned}\label{EstI23}
       &\big\||\nabla^{\ell}(\widehat{X}\overline{U}))||\nabla^{\ell^*}(\widehat{X}\overline{U})|\phi^{\frac{K}{2}}\big\|_{L^2}  \\
       &\leq C(K)\big\|\phi^{\frac{K}{2}\frac{\ell-3/{\bar{r}}}{K-3/2}}\nabla^{\ell}(\widehat{X}\overline{U})\big\|_{L^{\frac{2(K+1)}{\ell}}} \big\|\phi^{\frac{K}{2}\frac{\ell^*-3/{r^*}}{K-3/2}}\nabla^{\ell^*}(\widehat{X}\overline{U})\big\|_{L^{\frac{2(K+1)}{\ell^*}}}\leq  C(K). 
    \end{aligned}
\end{equation}    
We obtain from \eqref{EstI21}--\eqref{EstI23} and $E \gg 1 $ that
\begin{equation}
    \big\| | \nabla^\ell U |  | \nabla^{\ell^*} U |\phi^{\frac{K}{2}}\big\|_{L^2} \leq C(K) E^{\frac{1}{2}-\frac{1}{K+1}}.
\end{equation}
Since $E$ is sufficiently large and depends on $K$, we have
\begin{equation}\label{eq:4.76}
    \Big\|\sum_{i=1}^3\big(\partial_\beta (U \cdot \nabla U_i) - U\cdot \nabla \partial_\beta U_i -\partial_\beta U\cdot \nabla U_i -\sum_{j=1}^K  \partial_{y_{\beta_j}}U \cdot \nabla \partial_{\beta^{(j)}}U_i\big)\phi^{\frac{K}{2}}\Big\|_{L^2} \leq E^{\frac{1}{2}-\frac{1}{2K}}.
\end{equation}
On the other hand, applying \eqref{Para-ineq} and the Gagliardo-Nirenberg inequality \eqref{eq:GNresult} between \eqref{lowerestimate1} and \eqref{perturbationenergy} gives 
\begin{equation}
    \|\nabla \widetilde{U}\|_{L^{\infty}}
    \leq C(K)\big( \smallc_0^{1 - \frac{1}{K-3/2}} E^{\frac{1}{2K-3}} + \smallc_0 \big)
    \leq \smallc_0^{\frac{9}{10}},
\end{equation}
hence
\begin{equation}
    \|\nabla U\|_{L^{\infty}}\leq \|\nabla \widetilde{U}\|_{L^{\infty}}+ \|\nabla (\widehat{X}\overline{U})\|_{L^{\infty}} \leq C .
\end{equation}
This implies
\begin{equation}\label{eq:4.79}
    \|\partial_\beta U\cdot \nabla U_i\, \phi^{\frac{K}{2}}\|_{L^{2}} \leq \| \nabla U \|_{L^\infty} \| \partial_\beta U \phi^{\frac{K}{2}}\|_{L^2} \leq C\| \partial_\beta U \phi^{\frac{K}{2}}\|_{L^2},
\end{equation}
and
\begin{equation}\label{eq:4.80}
    \Big\|\sum_{j=1}^{K}\partial_{y_{\beta_j}} \widetilde{U}\cdot \nabla \partial_{\beta^{(j)}} U_i\, 
    \phi^{\frac{K}{2}}\Big\|_{L^{2}} \leq C(K)\smallc_0^{\frac{9}{10}}E^{\frac{1}{2}} \leq \smallc_0^{\frac{2}{5}}.
\end{equation}
Finally, using that $\overline{U}$ is spherically symmetric (\textit{i.e.}, $\overline{U}=\overline{\mathcal{U}}\frac{y}{|y|}$), 
for each $k \in \{1, 2, 3\}$, we decompose
\begin{equation}\label{eq:4.81}
\begin{aligned}
    \partial_{y_{\beta_j}}U_k &= \partial_{y_{\beta_j}} \widetilde{U}_k + \partial_{y_{\beta_j}} (\widehat{
    X}\overline{U}_k)= \partial_{y_{\beta_j}} \widetilde{U}_k + \partial_{y_{\beta_j}} \big(\widehat{
    X}\overline{\mathcal U}\, \frac{y_k}{|y|}\big)\\
    &= \partial_{y_{\beta_j}} \widetilde{U}_k +\underbrace{ \big(\partial_r(\widehat{
    X}\overline{\mathcal U})- \frac{\widehat{
    X}\overline{\mathcal U}}{|y|}\big)\frac{y_k y_{\beta_j}}{|y|^2}+\delta_{\beta_j, k}\frac{\widehat{
    X}\overline{\mathcal U}}{|y|}}_{ L_{k\beta_j}}.
\end{aligned}
\end{equation}
Combining \eqref{eq:4.76} and \eqref{eq:4.79}--\eqref{eq:4.81} and using $\smallc_0 \ll \frac{1}{E}$, we obtain 
\begin{equation}
    \begin{aligned}
    &\Big\|\sum_{i=1}^3\Big(\partial_\beta (U \cdot \nabla U_i) - U\cdot \nabla \partial_\beta U_i -\sum_{j=1}^K \sum_{k=1}^3  L_{k\beta_j}\partial_{y_k}\partial_{\beta^{(j)}}U_i\Big)\phi^{\frac{K}{2}}\Big\|_{L^2}\\
    &\leq E^{\frac{1}{2}-\frac{1}{4K}} + C\| \partial_\beta U \phi^{\frac{K}{2}}\|_{L^2}.
    \end{aligned}
\end{equation}  
This completes the proof of Lemma \ref{lemma:6.6}.
\end{proof}

With the above commutator estimates, we can close the $K$-th order weighted energy estimates under the \textit{a priori} bootstrap assumptions.

\begin{Lemma}\label{boosstrapestimate3}
    For $\tau \in [\tau_0, \tau_*]$, \begin{equation*}
    E_{K} = \int
    \left( | \nabla^K \Clss |^2 + | \nabla^K U |^2 \right) \phi^K {\rm d}y \leq \frac{E}{2}.
\end{equation*}
\end{Lemma}

\begin{proof}
For the sake of simplicity in the proof,  we introduce a family of multi-indices $\boldsymbol{\bar\beta}^i$  $(i=0,1,\cdots\!,\ell)$ for some   integer  $\ell\in [1,K]$ and  $\beta = (\beta_1, \beta_2, \cdots, \beta_K)$ with $|\beta|=K$, which  satisfy
\begin{equation*}
\begin{split}
&|\boldsymbol{\bar\beta}^0| + | \boldsymbol{\bar\beta}^1 |  + \cdots + | \boldsymbol{\bar\beta}^\ell |= |\beta|=K.
\end{split}
\end{equation*}
Here $|\boldsymbol{\bar\beta}^i|$, for $i=0,1,\cdots\!,\ell$, and $|\beta|$ denote the lengths of the multi-indices $\boldsymbol{\bar\beta}^i$ and $\beta$, respectively,  as defined in \S \ref{operator}. 

$E_K$ can be rewritten as follows
\begin{equation}
    E_{K} = \sum_{|\beta|=K}  \int
    \left( | \partial_\beta \Clss |^2 + | \partial_\beta U |^2 \right) \phi^K \text{d}y,
\end{equation}
where $\partial_\beta U$ is the vector $(\partial_\beta U_1, \partial_\beta U_2, \partial_\beta U_3)^\top$.

Applying $\partial_\beta$ to \eqref{selfsimilar eq}  and using the fact that 
$$\partial_\beta (y\cdot \nabla f) = K \partial_\beta f + y \cdot \nabla \partial_\beta f,$$
we have 
\begin{equation}\label{eq:K-thderivative}
\begin{aligned}
     &(\partial_\tau +\Lambda-1+K)\partial_{\beta}\Clss +y\cdot \nabla \partial_{\beta}\Clss + \partial_{\beta}(U\cdot \nabla \Clss)+\alpha\partial_{\beta}(\Clss \,\dive U)=0,\\[1mm]
     &(\partial_\tau +\Lambda-1+K)\partial_{\beta}U +y\cdot \nabla \partial_{\beta}U + \partial_{\beta}(U\cdot \nabla U)+\alpha\partial_{\beta}(\Clss\nabla \Clss)= \partial_{\beta}\mathcal{F}_{\rm dis}.
     \end{aligned}
\end{equation}
Multiplying each equation of \eqref{eq:K-thderivative} by $\phi^K \partial_\beta \Clss$ and $\phi^K \partial_\beta U $ respectively, summing over all $|\beta|=K$, and integrating over $\mathbb{R}^3$, we obtain
    \begin{align}\label{eq:EK}
        &\big(\frac{1}{2}\frac{\mathrm d}{\mathrm d\tau} + \Lambda-1+K\big)E_K  \nonumber\\
       &= -\sum_{|\beta|=K}\int \phi^K y \cdot \big(\partial_{\beta} \Clss\nabla \partial_{\beta} \Clss + \sum_{i=1}^3 \partial_{\beta} U_i \nabla \partial_{\beta}U_i \big) \text{d}y \nonumber\\
        &\quad -\sum_{|\beta|=K} \int \phi^K \big(\partial_{\beta} \Clss \, \partial_{\beta}(U\cdot \nabla \Clss)+ \sum_{i=1}^3 \partial_{\beta} U_i\, \partial_{\beta}(U \cdot \nabla U_i) \big) \text{d}y\\
        &\quad - \sum_{|\beta|=K} \int \alpha\phi^K \big(\partial_{\beta} \Clss \,\partial_{\beta}(\Clss \,\dive\, U) + \sum_{i=1}^3 \partial_{\beta} U_i\, \partial_{\beta}(\Clss \partial_{y_i} \Clss) \big) \text{d}y \nonumber\\
       &\quad +\sum_{|\beta|=K}C_{\rm{dis}}e^{-\delta_{\rm{dis}}\tau}\int \phi^K\partial_{\beta}U\, \partial_{\beta}\big( \Clss^{\frac{\delta-1}{\alpha}}L(U) +\frac{\delta}{\alpha}\Clss^{\frac{\delta-1}{\alpha}-1}\nabla \Clss \cdot \md(U) \big)\text{d}y \nonumber\\
       &=-I_1-I_2-I_3+C_{\rm{dis}}e^{-\delta_{\rm{dis}}\tau}I_4. \nonumber
    \end{align}

We divide the rest of the proof into four steps.

\smallskip
\textbf{1}. \emph{Estimate of $I_1$}. 
By integrating by parts, it follows from \eqref{higherestimate3} that
\begin{equation}\label{eq:EstI1}
    \begin{aligned}
        -I_1  &=-\frac{1}{2}\sum_{|\beta|=K}\int \phi^K y \cdot \big( \nabla |\partial_\beta \Clss|^2 + \sum_{i=1}^{3} \nabla|\partial_\beta U_i|^2\big)\,\text{d}y\\
        & =\frac{1}{2}\sum_{|\beta|=K}\int \frac{\dive(\phi^K y )}{\phi^K}\big( |\partial_\beta \Clss|^2 + |\partial_\beta U|^2\big)\phi^K\,\text{d}y\\
        & \leq \frac{3}{2}E + \frac{K}{2} \sum_{|\beta|=K} \int \frac{y \cdot \nabla\phi }{\phi}\big(  |\partial_\beta \Clss|^2 +|\partial_\beta U|^2\big)\phi^K\,\text{d}y .
    \end{aligned}
\end{equation}

\textbf{2}. \emph{Estimate of $I_2$ and $I_3$}.
To reindex the summation, we denote by $\tilde \beta $ a multi-index with $|\tilde\beta|=K-1$. Moreover, in preparation for applying the repulsivity properties  \eqref{radial repulsivity}--\eqref{angular repulsivity}
of profile $(\overline{\Clss}, \overline{U})$, we define $\nabla_{\theta}$ as the angular gradient in spherical coordinates by
\begin{equation}\label{dri:angular}
    \frac{1}{r}\nabla_{\theta}
:= \nabla - \frac{y}{r}\,\partial_r,
\end{equation}
where $r=|y|$ and $\partial_r := \frac{y}{r}\cdot\nabla$ denotes 
the radial derivative.
Using \eqref{higherestimate3} and \eqref{Est:I21'} yields
\begin{align}
       &-\sum_{|\beta|=K} \int \phi^K \sum_{i=1}^3 \partial_\beta 
       (U \cdot \nabla U_i) \partial_\beta U_i\,\text{d}y \nonumber \\
       &=-\sum_{|\beta|=K} \int \phi^K \sum_{i=1}^3 \Big(\partial_\beta (U \cdot \nabla U_i) - U\cdot \nabla \partial_\beta U_i -\sum_{j=1}^K \sum_{k=1}^3 L_{k\beta_j}\partial_{y_k}\partial_{\beta^{(j)}}U_i\Big) \partial_\beta U_i\,\text{d}y \nonumber\\
       &\quad -\sum_{|\beta|=K} \int \phi^K \sum_{i=1}^3 \Big( U\cdot \nabla \partial_\beta U_i +\sum_{j=1}^K \sum_{k=1}^3 L_{k\beta_j}\partial_{y_k}\partial_{\beta^{(j)}}U_i\Big) \partial_\beta U_i\, \text{d}y \nonumber  \\
       & \leq  \sum_{|\beta|=K } \big( E^{\frac{1}{2}-\frac{1}{4K}} + C\| \partial_\beta U \phi^{\frac{K}{2}}\|_{L^2} \big) \| \partial_\beta U \phi^{\frac{K}{2}} \|_{L^2} -\sum_{|\beta|=K} \int \phi^K \sum_{i=1}^3  U\cdot \nabla \partial_\beta U_i\, \partial_\beta U_i \,\text{d}y \nonumber \\
       &\quad -  \sum_{j=1}^K \sum_{|\beta^{(j)}|=K-1}\sum_{\beta_j=1}^3\sum_{k=1}^3 \int \phi^K \sum_{i=1}^3 \bigg( L_{k\beta_j} \partial_{y_k} \partial_{\beta^{(j)}} U_i \,\partial_{y_{\beta_j}} \partial_{\beta^{(j)}} U_i \bigg) \,\text{d}y  \nonumber \\
       & \leq  \frac{3^K}{2}E^{1-\frac{1}{2K}}+C\sum_{|\beta|=K}\| \partial_\beta U \phi^{\frac{K}{2}}\|_{L^2}^2 +\sum_{|\beta|=K} \int \frac{\dive(\phi^KU) }{2\phi^K } |\partial_\beta U|^2 \phi^K\, \text{d}y  \nonumber \\
       &\quad - \sum_{j=1}^K \sum_{|\beta^{(j)}|=K-1}\sum_{i=1}^3\int \phi^K \big(\partial_r(\widehat{X}\overline{\mathcal U})-\frac{\widehat{X}\overline{\mathcal U}}{|y|}\big) \sum_{\beta_j, k=1}^3 \frac{y_k y_{\beta_j}}{|y|^2} \partial_{y_k} \partial_{\beta^{(j)}} U_i\, \partial_{y_{\beta_j}} \partial_{\beta^{(j)}} U_i \,\text{d}y  \nonumber \\
       &\quad - \sum_{j=1}^K \sum_{|\beta^{(j)}|=K-1}\sum_{i=1}^3\int \phi^K  \frac{\widehat{X}\overline{\mathcal U}}{|y|} \sum_{ k=1}^3  |\partial_{y_k} \partial_{\beta^{(j)}} U_i
       |^2\,\text{d}y  \label{Est:I21''} \\
       & \leq   CE+ \sum_{|\beta|=K} \int \frac{K\nabla \phi \cdot U+ \phi\,\dive\,U }{2\phi } |\partial_\beta U|^2 \phi^K\, \text{d}y \nonumber \\
       &\quad - K \sum_{|\tilde{\beta}|=K-1} \int \phi^K \big(\partial_r (\widehat{X}\overline{\mathcal U}) - \frac{\widehat{X}\overline{\mathcal U}}{|y|} \big)\sum_{i=1}^3|\partial_r \partial_{\tilde{\beta}}U_i|^2 \,\text{d}y  \nonumber \\
       &\quad -K \sum_{|\tilde{\beta}|=K-1} \sum_{i,k=1}^3 \int \phi^K \frac{\widehat{X}\overline{\mathcal U}}{|y|} |\partial_{y_k} \partial_{\tilde{\beta}}U_i|^2\, \text{d}y \nonumber \\
       & \leq  CE + \sum_{|\beta|=K} \int \big(\frac{K\nabla \phi \cdot \widetilde{U}}{2\phi }+\frac{1}{2}\dive \,U \big) |\partial_\beta U|^2 \phi^K\, \text{d}y \nonumber \\
       &\quad + K \sum_{|\tilde{\beta}|=K-1}\int \Big( \frac{\nabla \phi \cdot \widehat{X} \overline{U}}{2\phi }|\nabla \partial_{\tilde{\beta}}U|^2 -\partial_r (\widehat{X}\overline{\mathcal U})|\partial_r \partial_{\tilde{\beta}}U|^2 - \frac{\widehat{X}\overline{\mathcal U}}{|y|^3} |\nabla_{\theta} \partial_{\tilde{\beta}}U|^2 \Big)\,\phi^K\,\text{d}y \nonumber\\
       &\leq  CE + K \sum_{|\tilde{\beta}|=K-1}\int \Big( \frac{\nabla \phi \cdot \widehat{X} \overline{U}}{2\phi }|\nabla \partial_{\tilde{\beta}}U|^2 -\partial_r (\widehat{X}\overline{\mathcal U})|\partial_r \partial_{\tilde{\beta}}U|^2 - \frac{\widehat{X}\overline{\mathcal U}}{|y|^3} |\nabla_{\theta} \partial_{\tilde{\beta}}U|^2 \Big)\phi^K\, \text{d}y, \nonumber
\end{align}
where the second inequality (\textit{i.e.}, the 6th line of \eqref{Est:I21''}) follows from the Cauchy-Schwarz inequality, integration by parts, and rearranging the indices; in the third inequality (\textit{i.e.}, the 9th line of \eqref{Est:I21''}), we have used the identity: 
$$
\sum_{k,\beta_j=1}^3\frac{y_k y_{\beta_j}}{|y|^2}\partial_{y_k} f \partial_{y_{\beta_j}} f = |\partial_r f|^2;
$$
and, in  the last inequality (\textit{i.e.}, the  line 14  of \eqref{Est:I21''}), we have employed the bounds: 
$$
\|\dive\, U\|_{L^{\infty}} \leq C,\qquad  \big\|\frac{\nabla \phi}{2\phi}\widetilde{U}\big\|_{L^{\infty}}\leq C\smallc_0 
$$ 
with $\smallc_0 \ll K^{-1}$.

By an analogous argument, we obtain from \eqref{higherestimate3} and \eqref{Est:I22'} that   
    \begin{align}
    &-\sum_{|\beta|=K} \int \phi^K \partial_\beta \Clss \partial_\beta (U \cdot \nabla \Clss)\,\text{d}y\nonumber\\
    &=-\sum_{|\beta|=K} \int \phi^K \partial_\beta \Clss \Big(\partial_\beta (U \cdot \nabla \Clss) - U\cdot \nabla \partial_\beta \Clss -\sum_{j=1}^K \sum_{k=1}^3 L_{k\beta_j}\partial_{y_k}\partial_{\beta^{(j)}}\Clss\Big)\,\text{d}y\label{Est:I22''}\\
    &\ \quad-\sum_{|\beta|=K} \int \phi^K \partial_\beta \Clss \Big( U\cdot \nabla \partial_\beta \Clss +\sum_{j=1}^K \sum_{k=1}^3 L_{k\beta_j}\partial_{y_k}\partial_{\beta^{(j)}}\Clss\Big)\,\text{d}y\nonumber\\
    & \leq  CE + K \sum_{|\tilde{\beta}|=K-1}\int \Big( \frac{\nabla \phi \cdot \widehat{X} \overline{U}}{2\phi }|\nabla \partial_{\tilde{\beta}}\Clss|^2 -\partial_r (\widehat{X}\overline{\mathcal U})|\partial_r \partial_{\tilde{\beta}}\Clss|^2 - \frac{\widehat{X}\overline{\mathcal U}}{|y|^3} |\nabla_{\theta} \partial_{\tilde{\beta}}\Clss|^2 \Big)\phi^K \,\text{d}y.\nonumber
    \end{align}
Adding \eqref{Est:I21''}--\eqref{Est:I22''} yields 
\begin{equation}    
    \begin{aligned}
        -I_2 &\leq CE + K \int \Big( \frac{\nabla \phi \cdot \widehat{X} \overline{U}}{2\phi } -\partial_r (\widehat{X}\overline{\mathcal U})  \Big)\big(|\partial_r \nabla^{K-1} U|^2+|\partial_r \nabla^{K-1} \Clss|^2\big)\phi^K\,\text{d}y  \\
        &\quad + K  \int \Big( \frac{\nabla \phi \cdot \widehat{X} \overline{U}}{2\phi |y|^2}- \frac{\widehat{X}\overline{\mathcal U}}{|y|^3}   \Big)\big(|\nabla_{\theta} \nabla^{K-1} U|^2+|\nabla_{\theta} \nabla^{K-1} \Clss|^2\big)\phi^K\,\text{d}y. \label{Est:I2}
\end{aligned}
\end{equation}
Then we start estimating $I_3$. It follows from \eqref{higherestimate3} and \eqref{Est:I32'} that
\begin{align}
       &-\sum_{|\beta|=K} \int \phi^K \sum_{i=1}^3 \partial_\beta U_i \partial_\beta (\Clss \partial_{y_i} \Clss) \,\text{d}y\nonumber \\
       &=-\sum_{|\beta|=K} \int \phi^K \sum_{i=1}^3 \partial_\beta U_i \Big(\partial_\beta (\Clss \partial_{y_i} \Clss ) - \Clss \partial_\beta \partial_{y_i} \Clss -\sum_{j=1}^K  \big(\partial_r(\widehat{
    X}\overline{\Clss})\frac{y_{\beta_j}}{|y|}\big)\partial_{y_i}\partial_{\beta^{(j)}}\Clss\Big)\,\text{d}y\nonumber\\
       &\quad -\sum_{|\beta|=K} \int \phi^K \sum_{i=1}^3 \partial_\beta U_i \Big( \Clss \partial_\beta \partial_{y_i} \Clss +\sum_{j=1}^K  \big(\partial_r(\widehat{
    X}\overline{\Clss})\frac{y_{\beta_j}}{|y|}\big)\partial_{y_i}\partial_{\beta^{(j)}}\Clss\Big)\,\text{d}y\nonumber\\
       & \leq  \sum_{|\beta|=K } \big( E^{\frac{1}{2}-\frac{1}{4K}} + C\| \partial_\beta \Clss \phi^{\frac{K}{2}}\|_{L^2} \big) \| \partial_\beta U \phi^{\frac{K}{2}} \|_{L^2} -\sum_{|\beta|= K} \int \phi^K  \Clss \nabla \partial_\beta \Clss \cdot \partial_\beta U\, \text{d}y \nonumber \\
       &\quad - \sum_{j=1}^K \sum_{|\beta^{(j)}|=K-1}\sum_{\beta_j=1}^3\sum_{i=1}^3  \int \phi^K   \partial_r(\widehat{X}\overline{\Clss})\frac{y_{\beta_j}}{|y|}\partial_{y_i} \partial_{\beta^{(j)}} \Clss \,\partial_{y_{\beta_j}}\partial_{\beta^{(j)}}U_i\,\text{d}y  \label{Est:I31''} \\
       & \leq  \frac{3^K}{2}E^{1-\frac{1}{2K}}+C\sum_{|\beta|=K}\big(\| \partial_\beta \Clss \phi^{\frac{K}{2}}\|_{L^2}^2+ \| \partial_\beta U \phi^{\frac{K}{2}}\|_{L^2}^2 \big)-\sum_{|\beta|=K} \int \phi^K  \Clss \nabla \partial_\beta \Clss \cdot \partial_\beta U\, \text{d}y \nonumber \\
       &\quad - \sum_{j=1}^K \sum_{|\beta^{(j)}|=K-1}\sum_{i=1}^3 \int \phi^K  \partial_r(\widehat{X}\overline{\Clss})\sum_{\beta_j=1}^3\frac{y_{\beta_j}}{|y|}\partial_{y_{\beta_j}}\partial_{\beta^{(j)}}U_i \, \partial_{y_i} \partial_{\beta^{(j)}} \Clss
       \, \text{d}y  \nonumber \\
       &\leq  CE - \sum_{|\beta|=K} \int \phi^K  \Clss \nabla \partial_\beta \Clss \cdot \partial_\beta U \,\text{d}y \nonumber \\
       & \quad+ \frac{K}{2} \sum_{|\tilde{\beta}|=K-1}\sum_{i=1}^3 \int \phi^K |\partial_r(\widehat{X}\overline{\Clss})|\big(|\nabla \partial_{\tilde \beta} U_i |^2+|\partial_{y_i} \partial_{\tilde \beta} \Clss |^2\big) \,\text{d}y \nonumber \\
       &=  CE - \sum_{|\beta|=K} \int \phi^K  \Clss \nabla \partial_\beta \Clss \cdot \partial_\beta U \text{d}y+\frac{K}{2} \sum_{|\beta|=K}\int \phi^K |\partial_r(\widehat{X}\overline{\Clss})|\big(|\partial_{\beta} \Clss |^2+|\partial_{ \beta} U |^2\big) \,\text{d}y. \nonumber 
\end{align}
By an analogous argument, it follows from \eqref{higherestimate3} and \eqref{Est:I31'} that 
    \begin{align}
        &-\sum_{|\beta|=K} \int \phi^K \sum_{i=1}^3 \partial_\beta (\Clss\partial_{y_i}U_i) \partial_\beta \Clss \,\text{d}y\nonumber\\
        &=-\sum_{|\beta|=K} \int \phi^K \Big(\partial_\beta (\Clss \,\dive \,U) - \Clss \partial_\beta \,\dive\, U -\sum_{j=1}^K \big(\partial_r(\widehat{
    X}\overline{\Clss})\frac{y_{\beta_j}}{|y|}\big)\partial_{\beta^{(j)}}(\dive U)\Big)\partial_\beta \Clss \,\text{d}y\label{Est:I32''}\\
    &\quad -\sum_{|\beta|=K} \int \phi^K \Big(\Clss \partial_\beta \,\dive\, U +\sum_{j=1}^K \big(\partial_r(\widehat{
    X}\overline{\Clss})\frac{y_{\beta_j}}{|y|}\big)\partial_{\beta^{(j)}}(\dive U)\Big)\partial_\beta \Clss \,\text{d}y  \nonumber \\
    & \leq  CE - \sum_{|\beta|=K}\sum_{i=1}^3 \int \phi^K  \Clss  \partial_\beta \Clss \cdot \partial_{y_i} \partial_\beta U_i \,\text{d}y
    +\frac{K}{2} \sum_{|\beta|=K}\int \phi^K |\partial_r(\widehat{X}\overline{\Clss})|\big(|\partial_{\beta} \Clss |^2+ |\partial_{\beta} U |^2\big) \,\text{d}y. \nonumber
    \end{align}
Combining \eqref{Est:I31''}--\eqref{Est:I32''} and integrating by parts, along with the Cauchy-Schwarz inequality, give
    \begin{align}\label{Est:I3}
        -I_3 &\leq  CE +\alpha \sum_{|\beta|=K} \int \frac{\nabla (\phi^K \Clss)}{\phi^K}  \cdot \big(\phi^K\partial_\beta \Clss \partial_\beta U\big)\, \text{d}y \nonumber \\
        &\quad + \alpha K \sum_{|\beta|=K}\int \phi^K |\partial_r(\widehat{X}\overline{\Clss})|\big(|\partial_{\beta} \Clss |^2+|\partial_{\beta} U |^2\big) \,\text{d}y \nonumber \\
        &\leq   CE + \alpha \sum_{|\beta|=K} \int \Big(\frac{K\nabla \phi \widetilde{\Clss}}{\phi}+\frac{K\nabla \phi \widehat{X}\overline{\Clss}}{\phi}+\nabla \Clss \Big) \cdot \big(\phi^K\partial_\beta \Clss \partial_\beta U\big) \,\text{d}y  \\
        &\quad + \alpha K \sum_{|\beta|=K}\int \phi^K |\partial_r(\widehat{X}\overline{\Clss})|\big(|\partial_{\beta} \Clss |^2+|\partial_{ \beta} U |^2\big) \,\text{d}y  \nonumber \\
        &\leq  CE+ K\alpha \sum_{|\beta|=K}\int \phi^K \Big(\frac{|\nabla \phi| \widehat{X}\overline{\Clss}}{2\phi}+|\partial_r(\widehat{X}\overline{\Clss})|\Big)\big(|\partial_{\beta} \Clss |^2+|\partial_{ \beta} U |^2\big)\, \text{d}y, \nonumber
    \end{align}
where we have used the following facts in the last inequality (\textit{i.e.}, the line 5th of \eqref{Est:I3}):
$$
\|\nabla \Clss\|_{L^{\infty}} \leq C, \qquad 
\|K\frac{\nabla \phi}{2\phi}\widetilde{\Clss}\|_{L^{\infty}}\leq CK\smallc_0\leq 1.$$ 

\textbf{3}. \emph{Estimate of $I_4$}.
By the expression of $I_4$, after commuting $\partial_\beta$ with $\Clss^{\frac{\delta-1}{\alpha}}$and $\Clss^{\frac{\delta-1}{\alpha}-1}$ and then  integrating by parts, we obtain
    \begin{align}\label{Exp:diss}
        I_{4, \beta} = & \underbrace{ \int \phi^K  \big[ \partial_\beta, \Clss^{\frac{\delta-1}{\alpha}} \big] L(U)\cdot \partial_\beta U \,\text{d}y }_{\mathcal  D_{1, \beta}}-\underbrace{ \int \big(\nabla ( \phi^K \Clss^{\frac{\delta-1}{\alpha}} ) \cdot \partial_\beta \md(U)\big) \cdot \partial_\beta U\,\text{d}y }_{\mathcal D_{2, \beta}} \nonumber\\
     &- \underbrace{ \int \phi^K \Clss^{\frac{\delta-1}{\alpha}}\big(a_1|\nabla \partial_\beta U|^2 + (a_1+a_2)|\dive \partial_{\beta} U|^2\big)\,\text{d}y}_{\mathcal D_{3,\beta}}\nonumber\\
      & + \underbrace{ \frac{\delta}{\alpha}\int \phi^K  \big[ \partial_\beta, \Clss^{\frac{\delta-1}{\alpha}-1} \big] \big(\nabla \Clss \cdot \md(U)\big)\cdot \partial_\beta U\, \text{d}y }_{\mathcal  D_{4, \beta}}\\
      &+\underbrace{\frac{\delta}{\alpha} \int \phi^K \Clss^{\frac{\delta-1}{\alpha}-1} \big(\partial_\beta(\nabla \Clss)\cdot  \md(U)\big) \cdot \partial_\beta U\, \dy }_{\mathcal D_{5,\beta}}\nonumber \\
     &+ \underbrace{\frac{\delta}{\alpha}\int \phi^K \Clss^{\frac{\delta-1}{\alpha}-1} \big(\partial_\beta(\nabla \Clss \cdot \md(U))-\partial_\beta(\nabla \Clss)\cdot  \md(U)\big) \cdot \partial_{\beta} U\, \dy}_{\mathcal D_{6,\beta}}. \nonumber
    \end{align}
Since $\mathcal D_{3, \beta } > 0$, this is a dissipative term and already carries the favorable sign. We now estimate the remaining terms, starting with $\mathcal D_{2, \beta }$. Applying \eqref{higherestimate3} and the H\"older inequality yields
\begin{equation}
\begin{aligned}
    | \mathcal D_{2, \beta} | &=\bigg| \int \Big(\big( K \phi^{K-1} \nabla \phi \,\Clss^{\frac{\delta-1}{\alpha}} + \frac{\delta-1}{\alpha} \Clss^{\frac{\delta-1}{\alpha}-1} \phi^K \nabla \Clss \big) \cdot  \partial_\beta \md(U)\Big) \cdot \partial_\beta U \,\text{d}y\bigg|  \\
    &\leq C \,\mathcal D_{3, \beta}^{\frac{1}{2}}  K \Big( \int \phi^{K-2} |\nabla \phi |^2 \Clss^{\frac{\delta-1}{\alpha}}| \partial_\beta U |^2 \,\text{d}y\Big)^{\frac{1}{2}} \\
    &\quad + C\, \mathcal D_{3, \beta}^{\frac{1}{2}}\frac{1-\delta}{\alpha} \Big( \int \phi^K \Clss^{\frac{\delta-1}{\alpha}}\frac{|\nabla \Clss|^2}{\Clss^{2}} | \partial_\beta U |^2\,\text{d}y \Big)^{\frac{1}{2}}  \label{eq:D2beta} \\
    &\leq C \,\mathcal D_{3, \beta}^{\frac{1}{2}} E^{\frac{1}{2}} \bigg( K \Big\| \frac{| \nabla \phi |^2}{\phi^2  }\Clss^{\frac{\delta-1}{\alpha}}\Big\|^{\frac 12}_{L^\infty} + \frac{1-\delta}{\alpha} \Big\|\Clss^{\frac{\delta-1}{\alpha}} \frac{| \nabla \Clss |^2}{\Clss^{2}} \Big\|^{\frac 12}_{L^\infty} \bigg). 
\end{aligned}
\end{equation}
It follows from Lemmas \ref{lemma:Sbounds}--\ref{lemma:Sprimebounds} that
\begin{align*} \label{Est:D2beta'}
    \frac{| \nabla \phi |^2}{\phi^2 }\Clss^{\frac{\delta-1}{\alpha}} &\leq C(\smallc_0) \big\langle \frac{y}{R_0} \big\rangle^{-2+\frac{(1-\delta)}{\alpha}(\Lambda-1)} = C(\smallc_0) \big\langle \frac{y}{R_0} \big\rangle^{-\Lambda +\delta_{\rm{dis}}}  \leq C(\smallc_0) , \\
    \frac{| \nabla \Clss |^2}{\Clss^{2}}\Clss^{\frac{\delta-1}{\alpha}} &\leq  C(\smallc_0)  \big\langle \frac{y}{R_0} \big\rangle^{-2+\frac{(1-\delta)}{\alpha}(\Lambda-1)}\leq C(\smallc_0) ,
\end{align*}
where  we have used the fact that 
$$
-\Lambda+\delta_{\rm{dis}}<0.
$$
Since $\smallc_0$ is sufficiently small and depends only on $E$ and $K$, we obtain
\begin{equation} \label{Est:D2beta}
    |\mathcal D_{2, \beta}| \leq C(\smallc_0)  \mathcal D_{3, \beta}^{\frac{1}{2}}. 
\end{equation}
For $\mathcal D_{5, \beta}$, it follows from integration  by parts that 
\begin{equation}    
    \begin{aligned}\label{eq:D5beta}
       \mathcal D_{5, \beta}&= -\frac{\delta}{\alpha}\int \partial_{\beta}\Clss \big(\nabla(\phi^K \Clss^{\frac{\delta-1}{\alpha}-1})\cdot  \md(U)\big)\cdot \partial_{\beta}U\,\text{d}y \\
       &\quad -\frac{\delta}{\alpha}\int \phi^K \Clss^{\frac{\delta-1}{\alpha}-1} \partial_{\beta}\Clss\, \dive(  \partial_{\beta}U\cdot\md(U))\,\text{d}y   \\
       &=  \mathcal D_{5, \beta, 1}+ \mathcal D_{5, \beta, 2}.
    \end{aligned}
\end{equation}    
We first estimate $\mathcal D_{5, \beta, 1}$. It follows from \eqref{higherestimate3} and Lemmas~\ref{lemma:Sbounds}--\ref{lemma:Sprimebounds} that
    \begin{align}
       |\mathcal D_{5, \beta, 1}|&=\frac{\delta}{\alpha}\bigg| \int \partial_{\beta}\Clss\Big(\big( K \phi^{K-1} \nabla \phi\, \Clss^{\frac{\delta-1}{\alpha}-1} + (\frac{\delta-1}{\alpha}-1) \phi^K \Clss^{\frac{\delta-1}{\alpha}-2}\nabla \Clss \big)\cdot  \md(U)\Big)\cdot \partial_{\beta}U \,\text{d}y \bigg| \nonumber \\
       &\leq  E^{\frac{1}{2}}\bigg( K \Big( \int \phi^{K-2} |\nabla \phi |^2 \Clss^{\frac{2(\delta-1)}{\alpha}}\frac{|\md(U)|^2}{\Clss^2}| \partial_\beta \Clss |^2 \,\text{d}y\Big)^{\frac{1}{2}} \bigg. \nonumber \\
       &\quad \bigg.+ \big(\frac{1-\delta}{\alpha}+1\big) \Big( \int \phi^K \Clss^{\frac{2(\delta-1)}{\alpha}}\frac{|\nabla \Clss|^2}{\Clss^{2}}\frac{|\md(U)|^2}{\Clss^2} | \partial_\beta \Clss |^2\,\text{d}y \Big)^{\frac{1}{2}} \bigg) \label{Est:D5beta1}  \\
       &\leq E\bigg(K \Big\| \frac{| \nabla \phi |^2}{\phi^2  }\frac{|\md(U)|^2}{\Clss^2}\Clss^{\frac{2(\delta-1)}{\alpha}}\Big\|^{\frac{1}{2}}_{L^\infty} + \big(\frac{1-\delta}{\alpha}+1 \big)\Big\| \frac{| \nabla \Clss |^2}{\Clss^{2}}\frac{|\md(U)|^2}{\Clss^2} \Clss^{\frac{2(\delta-1)}{\alpha}}\Big\|^{\frac{1}{2}}_{L^\infty}\bigg)  \nonumber \\
       &\leq C(\smallc_0) ,\nonumber
    \end{align}
where we used the facts that 
$$
\max\big\{|\frac{\nabla \Clss}{\Clss}|,\  |\frac{\nabla U}{U}|,\ |\frac{\nabla \phi}{\phi}|\big\}
\leq C(\smallc_0)  \langle \frac{y}{R_0} \rangle^{-1},\qquad -\Lambda+\delta_{\rm{dis}}<0,
$$
and  $\smallc_0$ is sufficiently small and depends on $E$ and $K$. 

Next, for $\mathcal D_{5, \beta, 2}$, we  obtain from \eqref{higherestimate3}, Lemmas~\ref{lemma:Sbounds}--\ref{higher-order}, and  $\smallc_0 \ll \frac{1}{E} \ll \frac{1}{K}$ that
    \begin{align}
       |\mathcal D_{5, \beta, 2}|&=\frac{\delta}{\alpha}\Big|\int \phi^K \Clss^{\frac{\delta-1}{\alpha}-1} \partial_{\beta}\Clss\, \dive\,\big(  \partial_{\beta}U \cdot \md(U)\big)\,\text{d}y \Big| \nonumber \\
       &\leq C  E^{\frac{1}{2}} \bigg( \int \phi^{K}  \Clss^{\frac{2(\delta-1)}{\alpha}}\frac{|\nabla^2U|^2}{\Clss^2}| \partial_\beta \Clss |^2\, \dy \bigg)^{\frac{1}{2}}\nonumber\\
       &\quad \bigg.+ C\,\mathcal{D}_{3, \beta}^{\frac{1}{2}}  \bigg( \int \phi^K \Clss^{\frac{\delta-1}{\alpha}}\frac{|\nabla U|^2}{\Clss^{2}}| \partial_\beta \Clss |^2\,\dy \bigg)^{\frac{1}{2}}  \label{Est:D5beta2} \\
       &\leq CE \Big\| \frac{| \nabla^2 U |^2}{\Clss^2  }\Clss^{\frac{2(\delta-1)}{\alpha}}\Big\|^{\frac{1}{2}}_{L^\infty} + C\,\mathcal{D}_{3, \beta}^{\frac{1}{2}} E^{\frac{1}{2}}\Big\| \frac{|\nabla U|^2}{\Clss^2} \Clss^{\frac{\delta-1}{\alpha}}\Big\|^{\frac{1}{2}}_{L^\infty} \nonumber \\
       &\leq C(\smallc_0) \big( 1 + \mathcal{D}_{3, \beta}^{\frac{1}{2}}\big). \nonumber
    \end{align}
Substituting \eqref{Est:D5beta1}--\eqref{Est:D5beta2} into \eqref{eq:D5beta} leads to
\begin{equation}\label{Est:D5beta}
    \mathcal D_{5, \beta} \leq C(\smallc_0)\big( 1 + \mathcal{D}_{3, \beta}^{\frac{1}{2}}\big).
\end{equation}
We now turn to the estimate of $\mathcal D_{1, \beta}$ in \eqref{Exp:diss}. Note that
\begin{equation*}
\begin{split}
   & \Big| \big[ \partial_\beta, \Clss^{\frac{\delta-1}{\alpha}} \big] f - \sum_{j=1}^K\partial_{y_{\beta_j}}\Clss^{\frac{\delta-1}{\alpha}} \partial_{\beta^{(j)}}f \Big| \\
   &\leq  C(K)3^K\max_{\substack{ \sum_{j=0}^\ell|\boldsymbol{\bar\beta}^j|
    = K \\ 0\le|  \boldsymbol{\bar\beta}^0 | \leq K-2}} \Big|\Clss^{\frac{\delta-1}{\alpha}} \partial_{\boldsymbol{\bar\beta}^0}f \frac{\partial_{\boldsymbol{\bar\beta}^1} \Clss}{\Clss} \frac{\partial_{\boldsymbol{\bar\beta}^2} \Clss}{\Clss} \cdots \frac{\partial_{\boldsymbol{\bar\beta}^\ell} \Clss}{\Clss}\Big|,
    \end{split}
\end{equation*}
where $|\boldsymbol{\bar{\beta}}^j|\ge 1$ for all $ 1\le j \le \ell$.

Based on the above inequality and \eqref{higherestimate3}, we can decompose $\mathcal D_{1, \beta}$ as
\begin{equation}   
    \begin{aligned}\label{eq:D1beta}
    \mathcal D_{1, \beta} &\leq C\sum_{j=1}^K \int |\phi^K \partial_\beta U\cdot  \partial_{\beta^{(j)}}L(U) \Clss^{\frac{\delta-1}{\alpha}-1} \partial_{y_{\beta_j}} \Clss|\,\text{d}y  \\
    &\quad \quad + C(K)3^K E^{\frac{1}{2}}\max_{\substack{\sum_{j=0}^\ell |\boldsymbol{\bar\beta}^j|
     = K \\ 0 \le| \boldsymbol{\bar\beta}^0 | \leq K-2}} \Big\|\Clss^{\frac{\delta-1}{\alpha}} \partial_{\boldsymbol{\bar\beta}^0} L(U) \frac{\partial_{\boldsymbol{\bar\beta}^1} \Clss}{\Clss} \frac{\partial_{\boldsymbol{\bar\beta}^2} \Clss}{\Clss} \cdots \frac{\partial_{\boldsymbol{\bar\beta}^\ell} \Clss}{\Clss}  \phi^{\frac{K}{2}} \Big\|_{L^2}  \\
    & = \mathcal D_{1, \beta, 1}+\mathcal D_{1, \beta, 2}.
    \end{aligned}
\end{equation}
For later use, we define
\begin{equation*}
\mathcal D_3 =   \sum_{|\beta|=K} \mathcal D_{3, \beta} = \int  \Clss^{\frac{\delta-1}{\alpha}}\phi^K \big(a_1|\nabla^{K+1} U|^2 + (a_1+a_2)|\nabla^{K}\dive  U|^2\big)\,\text{d}y . 
\end{equation*}
By \eqref{higherestimate3}, Lemmas~\ref{lemma:Sbounds}--\ref{lemma:Sprimebounds}, and the Cauchy-Schwarz inequality, we first estimate $\mathcal D_{1, \beta, 1}$:
\begin{equation}
\begin{aligned}
\mathcal D_{1, \beta, 1} &\leq C(K) \int \phi^K | \partial_\beta U | \, |\nabla^{K+1} U | \frac{| \nabla \Clss |}{\Clss} \Clss^{\frac{\delta-1}{\alpha}} \,\text{d}y \\
&\leq C(K) \mathcal D_3^{\frac{1}{2}} \bigg( \int \phi^K |\partial_\beta U|^2 \frac{| \nabla \Clss |^2}{\Clss^{2}}\Clss^{\frac{\delta-1}{\alpha}}\, \text{d}y \bigg)^{\frac{1}{2}}  \\
&\leq C(K)
\mathcal D_3^{\frac{1}{2}} E^{\frac{1}{2}} \Big\| \frac{| \nabla \Clss |^2}{\Clss^{2}} \Clss^{\frac{\delta-1}{\alpha}}\Big\|^{\frac{1}{2}}_{L^\infty} \leq C(\smallc_0) \mathcal D_3^{\frac{1}{2}}. \label{Est:D1beta1}
\end{aligned}
\end{equation}
We now proceed to estimate $\mathcal D_{1,\beta,2}$ by splitting the analysis into four cases.

\smallskip
\textit{Generic case. $0\le | \boldsymbol{\bar\beta}^0 | \leq K-4$ and $1\le | \boldsymbol{\bar\beta}^i | \leq K-2$ for all $i \geq 1$.} In this case, each term involves less than $(K-2)$-th order derivatives. By Lemmas \ref{lemma:Sbounds}--\ref{higher-order}, we derive 
\begin{align*}
    &\Big\| \Clss^{\frac{\delta-1}{\alpha}}\partial_{\boldsymbol{\bar\beta}^0} L(U) \frac{\partial_{\boldsymbol{\bar\beta}^1} \Clss}{\Clss} \frac{\partial_{\boldsymbol{\bar\beta}^2} \Clss}{\Clss} \cdots \frac{\partial_{\boldsymbol{\bar\beta}^\ell} \Clss}{\Clss}  \phi^{\frac{K}{2}} \Big\|_{L^2} \\
    &\leq C(\smallc_0) \big\| \phi^{-1}\langle y\rangle^{(\frac{1-\delta}{\alpha}-1)(\Lambda-1)}\langle y\rangle^{\frac{(K+2)\Lambda-(\ell+1)\Lambda-(\ell+1)K+\frac{5}{2}(\ell+1)+(1-\eta)(K(\ell+1)-\frac{5}{2}K-5)}{K-5/2}} \big\|_{L^2}\\
    &= C(\smallc_0) \big\| \phi^{-1}\langle y\rangle^{(\frac{1-\delta}{\alpha}-1)(\Lambda-1)}\langle y\rangle^{\frac{\Lambda(K+1-\ell)-\frac{5}{2}(1-\eta)(K+2)-K\eta(\ell+1)+\frac{5}{2}(\ell+1)}{K-5/2}} \big\|_{L^2}.
\end{align*}
If $\ell\geq 2$, since $K$ is sufficiently large with respect to $\eta$, we have 
$$
(\frac{3}{2} - \ell)K\eta < -100 \ell,
$$
which implies
\begin{equation}
    \Lambda(K+1-\ell)-\frac{5}{2}(1-\eta)(K+2)-K\eta(\ell+1)+\frac{5}{2}(\ell+1) \leq (\Lambda-\frac{5}{2})K.
\end{equation}
If $\ell = 1$, a similar computation yields
\begin{equation}
\Lambda(K+1-l)-\frac{5}{2}(1-\eta)(K+2)-K\eta(l+1)+\frac{5}{2}(l+1) \leq (\Lambda-\frac{5}{2}+\eta)K.
\end{equation}
Using $\eta$ is sufficiently small, $-1+\delta_{\rm{dis}}<0$, and $\Lambda>1$, we deduce
\begin{align*}
    &\big\| \Clss^{\frac{\delta-1}{\alpha}} \partial_{\boldsymbol{\bar\beta}^0} L(U) \frac{\partial_{\boldsymbol{\bar\beta}^1} \Clss}{\Clss} \frac{\partial_{\boldsymbol{\bar\beta}^2} \Clss}{\Clss} \cdots \frac{\partial_{\boldsymbol{\bar\beta}^\ell} \Clss}{\Clss}  \phi^{\frac{K}{2}} \big\|_{L^2} \\
    &\leq C(\smallc_0)   \big\| \langle y\rangle^{-2(1-\eta)+(\frac{1-\delta}{\alpha}-1)(\Lambda-1)+\Lambda-\frac{5}{2}+\eta}\big\|_{L^2}\\
    &= C(\smallc_0)  \big\| \langle y\rangle^{-\Lambda-\frac{3}{2}+3\eta+\delta_{\rm{dis}}}\big\|_{L^2} \\
    &= C(\smallc_0) \Big(\int_0^\infty \langle r\rangle^{-2\Lambda-3+6\eta+2\delta_{\rm{dis}}}r^2 \text{d}r\Big)^{\frac 12}\leq C(\smallc_0) .
\end{align*}

\textit{Case where some $| \boldsymbol{\bar\beta}^i | \in \{ K-1, K \}$ for $i \geq 1$.} Without loss of generality, we assume $i=1$. If $| \boldsymbol{\bar\beta}^1 | = K$, then, by \eqref{higherestimate3} and Lemmas \ref{lemma:Sbounds}--\ref{higher-order}, we have
\begin{equation*}
    \Big\| \phi^{\frac{K}{2}}\Clss^{\frac{\delta-1}{\alpha}} \partial_{\boldsymbol{\bar\beta}^1} \Clss \frac{L(U)}{\Clss}  \Big\|_{L^2} \leq C E^{\frac 12} \Big\|\Clss^{\frac{\delta-1}{\alpha}} \frac{| \nabla^2 U |}{\Clss} \Big\|_{L^\infty} \leq C(\smallc_0).
\end{equation*}
If $| \boldsymbol{\bar\beta}^1 | = K-1$, we divide the corresponding estimates into two subcases. According to  Lemmas \ref{lemma:Sbounds}--\ref{higher-order}, we see that,  for fixed $j\in \{1,2,3\}$,
\begin{align*}
    \Big\| \phi^{\frac{K}{2}} \partial_{\boldsymbol{\bar\beta}^1} \Clss \frac{\partial_{y_j} \Clss\, L(U)}{\Clss^{2}} \Clss^{\frac{\delta-1}{\alpha}} \Big\|_{L^2} 
    &\leq C(\smallc_0,\bar{\varepsilon}) \|\langle y\rangle^{K(1-\eta)\frac{K-2}{K-1}-\bar{\varepsilon}}\partial_{\boldsymbol{\bar\beta}^{1}}\Clss \|_{L^{2+\frac{2}{K-2}}} \\
    &\quad \times  \Big\|  \frac{ | \nabla \Clss |\langle y\rangle}{\Clss} \Big\|_{L^\infty} \Big\| \frac{\langle y\rangle^{K(1-\eta)\frac{1}{K-1}+\bar{\varepsilon}-1} |\nabla^2 U|}{\Clss} \Clss^{\frac{\delta-1}{\alpha}}\Big\|_{L^{2(K-1)}}  \\
    & \leq C(\smallc_0,\bar{\varepsilon}) \big\| \langle y\rangle^{-2(1-\eta)+\frac{1-\delta}{\alpha}(\Lambda-1)-1+\frac{K}{K-1}(1-\eta)+\bar{\varepsilon}}\big\|_{L^{2(K-1)}} \\
     & \leq C(\smallc_0,\bar{\varepsilon}) \big\| \langle y\rangle^{-\Lambda+2\eta+\bar{\varepsilon}+\frac{K}{K-1}(1-\eta)}\big\|_{L^{2(K-1)}} \leq C(\smallc_0,\bar{\varepsilon}),\\
    \Big\| \phi^{\frac{K}{2}} \partial_{\boldsymbol{\bar\beta}^1} \Clss \frac{\partial_{y_j} L(U)}{\Clss}\Clss^{\frac{\delta-1}{\alpha}}  \Big\|_{L^2} &\leq C(\smallc_0,\bar{\varepsilon}) \|\langle y\rangle^{K(1-\eta)\frac{K-2}{K-1}-\bar{\varepsilon}}\partial_{\boldsymbol{\bar\beta}^{1}}\Clss \|_{L^{2+\frac{2}{K-2}}} \\
    & \quad \times \Big\| \frac{ \langle y\rangle^{K(1-\eta)\frac{1}{K-1}+\bar{\varepsilon}} | \nabla^3 U | }{\Clss}\Clss^{\frac{\delta-1}{\alpha}} \Big\|_{L^{2(K-1)}}\\
    & \leq C(\smallc_0,\bar{\varepsilon})  \big\| \langle y\rangle^{-3(1-\eta)+\frac{1-\delta}{\alpha}(\Lambda-1)+\frac{K}{K-1}(1-\eta)+\bar{\varepsilon}}\big\|_{L^{2(K-1)}} \\
     & \leq C(\smallc_0,\bar{\varepsilon}) \big\| \langle y\rangle^{-\Lambda+3\eta+\bar{\varepsilon}+\frac{K}{K-1}(1-\eta)}\big\|_{L^{2(K-1)}} \leq C(\smallc_0,\bar{\varepsilon}).
\end{align*}

\textit{Case where $| \boldsymbol{\bar\beta}^0  | = K-2$.}
We obtain from \eqref{higherestimate3} and Lemmas \ref{lemma:Sbounds}--\ref{higher-order} that
 \begin{align*}
&\Big\|\phi^{\frac{K}{2}} \partial_{\boldsymbol{\bar\beta}^0} L(U)   \nabla^2 ( \Clss^{\frac{\delta-1}{\alpha}} ) \Big\|_{L^2}\\ 
&\leq \| \phi^{\frac{K}{2}} \nabla^K U \|_{L^2} \Big( \Big\| \frac{\nabla^2 \Clss}{\Clss}\Clss^{\frac{\delta-1}{\alpha}}\Big\|_{L^\infty} + \Big\| \frac{|\nabla \Clss|^2}{\Clss^{2}}\Clss^{\frac{\delta-1}{\alpha}}\Big\|_{L^\infty} \Big) \\
&\leq E^{\frac{1}{2}} \big( \big\| \langle y\rangle^{-2(1-\eta)+\frac{\delta-1}{\alpha}(\Lambda-1)}\big\|_{L^\infty} + \big\| \langle y\rangle^{-2+\frac{\delta-1}{\alpha}(\Lambda-1)}\big\|_{L^\infty} \big) \leq C(\smallc_0).
\end{align*}

\textit{Case where $| \boldsymbol{\bar\beta}^0  | = K-3$.}
It follows from Lemmas \ref{lemma:Sbounds}--\ref{higher-order} that
\begin{align*}
&\big\|\phi^{\frac{K}{2}} \partial_{\boldsymbol{\bar\beta}^0} L(U)  \nabla^3 ( \Clss^{\frac{\delta-1}{\alpha}} )  \big\|_{L^2}  \\
&\leq C(\smallc_0,\bar{\varepsilon})\big\| \langle y\rangle^{K(1-\eta)\frac{K-2}{K-1}-\bar\varepsilon}\nabla^{K-1}U\big\|_{L^{1+\frac{2}{K-2}}}\, \big\| \langle y\rangle^{(1-\eta)\frac{K}{K-1}+\bar\varepsilon+\frac{1-\delta}{\alpha}(\Lambda-1)-3(1-\eta)}\big\|_{L^{2(K-1)}}\\
&\leq C(\smallc_0,\bar{\varepsilon}) \big\| \langle y\rangle^{-\Lambda+3\eta+\bar\varepsilon+(1-\eta)\frac{K}{K-1}}\big\|_{L^{2(K-1)}}
\leq C(\smallc_0,\bar{\varepsilon}).
\end{align*}
Collecting the above estimates and choosing $\bar{\varepsilon}$ sufficiently small (depending on $\eta$, $\Lambda$, and $K$), we obtain
\begin{align}\label{Est:D1beta2}
\mathcal D_{1, \beta, 2} &\leq C(\smallc_0, \bar{\varepsilon})\leq C(\smallc_0). 
\end{align}
Repeating the procedure used in the estimate of $\mathcal D_{1, \beta, 2}$, we deduce
\begin{align}\label{Est:D4beta}
\mathcal D_{4, \beta}\leq C(\smallc_0). 
\end{align}
For $D_{6, \beta}$, it follows from \eqref{higherestimate3} and Lemmas~\ref{lemma:Sbounds}--\ref{lemma:Sprimebounds} that
\begin{equation}
\begin{aligned}\label{Eq:D6beta}
\mathcal D_{6, \beta} &= \frac{\delta}{\alpha}\int \phi^K \Clss^{\frac{\delta-1}{\alpha}-1} \big(\partial_\beta(\nabla \Clss \cdot \md(U))-\partial_\beta(\nabla \Clss)\cdot  \md(U)\big) \cdot \partial_{\beta} U\,\dy \\
& \leq CE^{\frac{1}{2}}\mathcal{D}_3^{\frac{1}{2}}\Big\|\frac{|\nabla \Clss|^2}{\Clss^2}\Clss^{\frac{\delta-1}{\alpha}}\Big\|^{\frac{1}{2}}_{L^\infty} +C(K)E^{\frac{1}{2}}\max_{\substack{ |\boldsymbol{\bar\beta}^0| + | \boldsymbol{\bar\beta}^1 |= K \\ 1\le|  \boldsymbol{\bar\beta}^0 | \leq K-1}} \Big\|\phi^{\frac{K}{2}}\Clss^{\frac{\delta-1}{\alpha}} \partial_{\boldsymbol{\bar\beta}^0}\md(U)\frac{\partial_{\boldsymbol{\bar\beta}^1}\nabla\Clss}{\Clss}\Big\|_{L^2} \\
&\leq C(\smallc_0) \mathcal{D}_3^{\frac{1}{2}}+ \mathcal D_{6, \beta, 2} \leq C(\smallc_0) \big(1+\mathcal{D}_3^{\frac{1}{2}}\big).
\end{aligned}
\end{equation}
The estimate of $\mathcal D_{6, \beta, 2}$ follows from the same argument as $\mathcal D_{1, \beta, 2}$ and is therefore omitted.

Consequently, it follows from \eqref{Est:D2beta}--\eqref{Eq:D6beta} that
\begin{equation*}
\begin{aligned}
    C_{\rm{dis}}e^{-\delta_{\rm{dis}}\tau}I_4 + C_{\rm{dis}}e^{-\delta_{\rm{dis}}\tau}\mathcal D_{3}
    &= C_{\rm{dis}}e^{-\delta_{\rm{dis}}\tau}\sum_{|\beta|=K}
    \big(\mathcal D_{1, \beta}-\mathcal D_{2, \beta}+\mathcal D_{4, \beta}
    +\mathcal D_{5, \beta}+\mathcal D_{6, \beta}\big)\\
    & \leq C(\smallc_0)e^{-\delta_{\rm{dis}}\tau}\big(\mathcal D_{3}^{\frac 12}+1\big),
    \end{aligned}
\end{equation*}
where $C(\smallc_0)$ is sufficiently large and depends on $\smallc_0$. Applying the Cauchy-Schwarz inequality gives
\begin{equation*}
  C_{\rm{dis}}e^{-\delta_{\rm{dis}}\tau}I_4 + C_{\rm{dis}} e^{-\delta_{\rm{dis}}\tau} \mathcal D_3  \leq C(\smallc_0) e^{-\delta_{\rm{dis}}\tau} + e^{-\delta_{\rm{dis}}\tau} \Big( \mathcal D_3\frac{C_{\rm{dis}}}{2} + \frac{2 C(\smallc_0)^2}{C_{\rm{dis}}} \Big).
\end{equation*}
Consequently,
\begin{equation*}
 C_{\rm{dis}}e^{-\delta_{\rm{dis}}\tau}I_4 + \frac{1}{2}C_{\rm{dis}}e^{-\delta_{\rm{dis}}\tau} \mathcal D_3 \leq   e^{-\delta_{\rm{dis}}\tau} \Big( C(\smallc_0) +  \frac{2 C(\smallc_0)^2}{C_{\rm{dis}}} \Big) 
 \leq C(\smallc_0) e^{-\delta_{\rm{dis}}\tau}.
\end{equation*}
Since $\tau_0$ is sufficiently large, depending on $\smallc_0$, we conclude 
\begin{equation} \label{eq:I4}
  C_{\rm{dis}}e^{-\delta_{\rm{dis}}\tau}I_4 +\frac{1}{2}C_{\rm{dis}}e^{-\delta_{\rm{dis}}\tau}\mathcal D_3 \leq e^{-\delta_{\rm{dis}} (\tau-\tau_0)}.
\end{equation}

\textbf{4}. \emph{Energy estimate}.
It follows from  the estimates for  $I_1, I_2,I_3,  C_{\rm{dis}}e^{-\delta_{\rm{dis}}\tau}I_4$ obtained above,   \eqref{eq:EK}, and \eqref{dri:angular} that 
    \begin{align}\label{energyest1}
        &\big(\frac{1}{2}\frac{\mathrm d}{\mathrm d\tau} + \Lambda-1+K\big)E_K  + \frac{1}{2}C_{\rm{dis}}e^{-\delta_{\rm{dis}}\tau}\int  \Clss^{\frac{\delta-1}{\alpha}}\phi^K \big(a_1|\nabla^{K+1} U|^2 + (a_1+a_2)|\nabla^{K}\dive  U|^2\big)\,\text{d}y \nonumber \\
        &=-I_1-I_2-I_3+C_{\rm{dis}}e^{-\delta_{\rm{dis}}\tau}I_4 \nonumber\\
        &\quad + \frac{1}{2}C_{\rm{dis}}e^{-\delta_{\rm{dis}}\tau}\int  \Clss^{\frac{\delta-1}{\alpha}}\phi^K \big(a_1|\nabla^{K+1} U|^2 + (a_1+a_2)|\nabla^{K}\dive  U|^2\big)\,\text{d}y  \nonumber \\
        &\leq CE +\frac{K}{2} \int \frac{y \cdot \nabla\phi }{\phi}\big(  |\nabla^K \Clss|^2 + |\nabla^K U|^2\big)\phi^K\,\text{d}y + e^{-\delta_{\rm{dis}} (\tau-\tau_0)} \nonumber  \\
        &\quad + K \int \Big( \frac{\nabla \phi \cdot \widehat{X} \overline{U}}{2\phi } -\partial_r (\widehat{X}\overline{\mathcal U})  \Big)\big(|\partial_r \nabla^{K-1} \Clss|^2+|\partial_r \nabla^{K-1} U|^2\big)\phi^K\,\text{d}y \nonumber \\
        &\quad + K  \int \Big( \frac{\nabla \phi \cdot \widehat{X} \overline{U}}{2\phi |y|^2}- \frac{\widehat{X}\overline{\mathcal U}}{|y|^3}   \Big)\big(|\nabla_{\theta} \nabla^{K-1} U|^2+|\nabla_{\theta} \nabla^{K-1} \Clss|^2\big)\phi^K \,\text{d}y \nonumber \\    
        &\quad  +K\alpha \int \phi^K \Big(\frac{|\nabla \phi| \widehat{X}\overline{\Clss}}{\phi}+|\partial_r(\widehat{X}\overline{\Clss})|\Big)\big(|\nabla^K U |^2+|\nabla^K \Clss |^2\big)\, \text{d}y\\
        &\leq  C E + e^{-\delta_{\rm{dis}} (\tau-\tau_0)} \nonumber \\
        &\quad + K\int \big( -\partial_r (\widehat{X}\overline{\mathcal U})  + \alpha|\partial_r(\widehat{X}\overline{\Clss})| \big) 
        \big(|\partial_r \nabla^{K-1} U|^2+|\partial_r \nabla^{K-1} \Clss|^2\big)\phi^K \,\text{d}y \nonumber \\
        & \quad + K\int \Big( -\frac{\widehat{X}\overline{\mathcal U}}{|y|^3}  + \frac{\alpha|\partial_r(\widehat{X}\overline{\Clss})|}{|y|^2} \Big) \big(|\nabla_\theta \nabla^{K-1} U|^2+|\nabla_\theta \nabla^{K-1} \Clss|^2\big)\phi^K \,\text{d}y \nonumber \\
        & \quad + K\int \frac{ \nabla\phi\cdot (y + \widehat{X}\overline{U})+ |\nabla\phi|\widehat{X}\overline{\Clss}}{2\phi}\big(|\nabla^K U |^2+|\nabla^K \Clss |^2\big)\phi^K \,\text{d}y.\nonumber 
    \end{align}
    
In  $B(0, R_0)$, the weight function satisfies $\phi = 1$. Moreover, we know that 
$|\partial_r \widehat X| \leq C e^{-\tau_0}$ for $r=|y|$. Therefore, by \eqref{Para-ineq}--\eqref{Para-ineq1} and the standard repulsivity properties \eqref{radial repulsivity}--\eqref{angular repulsivity}, we obtain
\begin{align*} 
  -\partial_r (\widehat{X}\mathcal{\overline U}) + \alpha |\partial_r (\widehat{X}\overline \Clss) |+ \frac{\nabla \phi \cdot (y+\widehat{X}\overline U) + | \nabla \phi | \widehat{X}\overline \Clss}{2\phi} \leq 1 - \frac{ \tilde \eta}{2}, \\
    -\frac{\widehat{X}\mathcal{\overline U}}{r} + \alpha |\partial_r (\widehat{X}\overline \Clss) |+ \frac{\nabla \phi \cdot (y+\widehat{X}\overline U) + | \nabla \phi | \widehat{X}\overline \Clss}{2\phi} \leq 1 - \frac{ \tilde \eta  }{2},
\end{align*}
where $r=|y|$.
In the exterior region $B^c(0,R_0)$, it follows from \eqref{weightdefi}, \eqref{Para-ineq}, \eqref{Rochoice estimate}, and \eqref{estimatehatX2} that
\begin{equation*}
-\partial_r (\widehat{X}\mathcal{\overline U}) + \alpha |\partial_r (\widehat{X}\overline \Clss )|+ \frac{\nabla \phi \cdot (y+\overline U) + | \nabla \phi | \overline \Clss}{2\phi}  \leq \frac{\eta}{4} + \frac{r\partial_r \phi}{2\phi}  \leq \frac{\eta}{4} + 1-\eta < 1-\frac{\eta}{2},
\end{equation*}
and
\begin{equation*}
-\frac{\widehat{X}\mathcal{\overline U}}{r} + \alpha |\partial_r (\widehat{X}\overline \Clss )|+ \frac{\nabla \phi \cdot (y+\overline U) + | \nabla \phi | \overline \Clss}{2\phi}  \leq \frac{\eta}{4} + \frac{r\partial_r \phi}{2\phi}  \leq \frac{\eta}{4} + 1-\eta < 1-\frac{\eta}{2},
\end{equation*}
where $r=|y|$.
Combining the above estimates, we derive 
\begin{equation*}
\begin{aligned}
       & \big(\frac{1}{2}\frac{\mathrm d}{\mathrm d\tau} + \Lambda-1+K\big)E_K +\frac{1}{2}C_{\rm{dis}}e^{-\delta_{\rm{dis}}\tau}\int  \Clss^{\frac{\delta-1}{\alpha}}\phi^K \big(a_1|\nabla^{K+1} U|^2 + (a_1+a_2)|\nabla^{K}\dive  U|^2\big)\,\text{d}y  \\
        &\leq  CE +K\Big(1- \frac{\min \{ \tilde \eta , \eta \} }{2}
        \Big)E_K+ e^{-\delta_{\rm{dis}} (\tau-\tau_0)}. 
\end{aligned}        
\end{equation*}
Equivalently,
\begin{equation}\label{energyest2}
\begin{aligned}
        &\frac{\mathrm d}{\mathrm d\tau} E_K + C_{\rm{dis}}e^{-\delta_{\rm{dis}}\tau}\int  \Clss^{\frac{\delta-1}{\alpha}}\phi^K \big(a_1|\nabla^{K+1} U|^2 + (a_1+a_2)|\nabla^{K}\dive  U|^2\big)\,\text{d}y \\
        &\leq - K\min \{ \tilde \eta , \eta \}E_K +CE.
\end{aligned}        
\end{equation}
Since $K$ is sufficiently large and depends on $(\eta, \tilde \eta)$, applying the Gr\"onwall inequality yields
\begin{equation}\label{energyest3}
\begin{aligned}
         &E_K(\tau)+ C\int_{\tau_0}^{\tau}e^{-\delta_{\rm{dis}}\tilde{\tau}-K\min \{ \tilde \eta , \eta \}(\tau-\tilde \tau)}\|\phi^{\frac{K}{2}}\Clss^{\frac{\delta-1}{2\alpha}}\nabla^{K+1}U(\tilde{\tau}, \cdot)\|_{L^2}^2\,\text{d}\tilde{\tau}\\
         &\leq  e^{-K\min \{ \tilde \eta , \eta \}(\tau-\tau_0)}E_K(\tau_0)+\frac{CE }{K\min \{ \tilde \eta , \eta \}}\big(1-e^{-K\min \{ \tilde \eta , \eta \}(\tau-\tau_0)}\big)\\
         & \leq e^{-K\min \{ \tilde \eta , \eta \}(\tau-\tau_0)}E_K(\tau_0)
         +\frac{E }{2}\big(1-e^{-K\min \{ \tilde \eta , \eta \}(\tau-\tau_0)}\big). 
\end{aligned} 
\end{equation}
If $E_K(\tau_0) \leq E/2$, then $E_K(\tau) \leq E/2$ for all $\tau \geq \tau_0$.

This completes the proof of Lemma \ref{boosstrapestimate3}.
\end{proof}

\subsubsection{Control of the unstable modes of $(\overset{\approx}{\Clss},\overset{\approx}{U})$ by the selection of  initial data}\label{subsection6.3}
Since the truncated linear operator $\mathcal L$ generates a semigroup on space $X$ of functions supported in $B(0, 3C_0)$, we consider the Cauchy problem \eqref{eq:truncated} with initial data
$( \overset{\approx}{\Clss}_{0},\overset{\approx}{U}_{0})$.
Recalling the invariant decomposition
$$X=V_{\mathrm{sta}}\oplus V_{\mathrm{uns}},$$
the initial data for \eqref{eq:truncated} admit a decomposition into stable and unstable components. The stable component has already been chosen in \S\ref{subsection4.2} so that its evolution is controlled by the estimates established above.

We now turn to the selection of appropriate unstable components in order to control the solution to \eqref{eq:truncated}.
Recall that $\{(\varphi_{i,\subclss}, \varphi_{i,u})\}_{i=1}^{N}$ forms a normalized basis of $V_{\rm{uns}}$. For any $(\overset{\approx}{\Clss}(\cdot, \tau),\overset{\approx}{U}(\cdot,\tau))$, we denote 
$$
k(\tau)=P_{\rm{uns}}(\overset{\approx}{\Clss}(\cdot,\tau),\overset{\approx}{U}(\cdot,\tau))=\sum_{i=1}^{N}k_i(\tau)(\varphi_{i,\subclss},\varphi_{i,u}).
$$
By the initial data condition \eqref{initialdatatruncated}, together with ($\chi_2 \widetilde U_0^*, \chi_2 \widetilde \Clss_0^* ) \in V_{\rm{sta}}$, we have 
\begin{equation*}
k_i(\tau_0)=\hat{k}_i.
\end{equation*} 
From Lemma \ref{prop:maxdissmooth}, there exists a metric $\Upsilon$ such that 
$$
\langle \mathcal{L}k(\tau),k(\tau) \rangle_{\Upsilon}\geq -\frac{3\smallc_g}{5}\langle k(\tau),k(\tau) \rangle_{\Upsilon}.
$$
Since the unstable space $V_{\rm{uns}}$ only has finite dimension and $X\subset H^m (B(0, 3C_0))$, 
all the norms on $V_{\rm{uns}}$ are equivalent. 
In particular, there exists a constant $C(m)> 0$ such that
\begin{equation}\label{normequivalence}
\|k(\tau)\|_{X}\leq C(m) \|k(\tau)\|_{\Upsilon}\leq C(m)\|k(\tau)\|_{X},
\end{equation}
where we have denoted the norm $\| k(\tau) \|_X:=\| \sum_{i=1}^N k_i(\tau) \varphi_i \|_X$, and 
the dependence on $m$ arises from the definition of space $X$.

We now restrict the initial coefficients $\hat{k}_i$ such that 
\begin{equation*}
\| \hat{k} \|_\Upsilon \leq \smallc_1^{\frac{11}{10}}.
\end{equation*}
By \eqref{normequivalence},  this yields 
$$\| \hat{k} \|_X\leq C(m)\smallc_1^{\frac{11}{10}}.$$ 
Since $\varphi_i$ are compactly supported, $\|\varphi_i\|_X=1$, and $\smallc_1$ is sufficiently small (depending on $m$ and $E$), the initial conditions \eqref{initialdatacondition1}--\eqref{initialdatacondition2} are automatically satisfied for any such choice of $\hat{k}_i$ as long as
\begin{align} \begin{split} \label{eq:conditions_for_tilde}
& \max\big\{\|\widetilde{\Clss}_0^* \|_{L^{\infty}},\ \|\widetilde{U}_0^* \|_{L^{\infty}}\big\}\leq  \smallc_1, \qquad \| \nabla^K\widetilde{\Clss}_0^*\phi^{K/2} \|_{L^2}+\| \nabla^K\widetilde{U}_0^* \phi^{K/2}\|_{L^2} \leq\frac{E}{4}, 
\\
&\widetilde{\Clss}_0^*+\widehat{X}\overline{\Clss}\geq \frac{\smallc_1}{2}, \qquad|\nabla(\widetilde{\Clss}_0^*+\widehat{X} \overline{\Clss})|+|\nabla(\widetilde{U}_0^*+\widehat{X} \overline{U})|\leq C\langle y \rangle^{-\Lambda}.
\end{split} \end{align}

Let
$$
\mathcal R(\tau)=\big\{k(\tau) :\, \|k(\tau)\|_{X}\leq \smallc_1^{\frac{21}{20}} e^{-\frac{4}{3}\varepsilon(\tau-\tau_0)}\big\},
$$
$$
\widetilde{\mathcal R}(\tau)=\big\{k(\tau) :\, \|k(\tau)\|_{\Upsilon}\leq \smallc_1^{\frac{11}{10}}e^{-\frac{4}{3}\varepsilon(\tau-\tau_0)}\big\}.
$$
By \eqref{normequivalence} and the choice of parameter \eqref{Para-ineq}, we obtain
$$
\widetilde{ R}(\tau)\Subset \mathcal R(\tau).
$$

Throughout the argument, we choose the initial coefficients $\{ \hat{k}_i \} = k(\tau_0)\in \widetilde{\mathcal R}(\tau_0)$. As long as $k(\tau)$ remains in $\mathcal R(\tau)$, the unstable modes are controlled. 
Then we have the estimate for the forcing term{\rm:} 
\begin{equation}\label{forcingest}
\|\chi_2 \mathcal F\|_{X}=\|(\chi_2\mathcal{F}_{\subclss}, \chi_2\mathcal{F}_{u})\|_{X}\leq C \smallc_1^{\frac{6}{5}}e^{-\frac{3}{2}\varepsilon(\tau-\tau_0)}.
\end{equation}

Next, we show the following outgoing property{\rm:}
\begin{Lemma}\label{outgoingpro}
Suppose that $k(\tau)\in \widetilde{\mathcal R}(\tau)$ at $\tau \in[\tau_0,\tau_*]$, and 
 $k(\tau_*)\in \partial{\widetilde{\mathcal R}(\tau_*)}$ at time $\tau_*$, that is, 
$$
\|k(\tau_*)\|_{\Upsilon}=\smallc_1^{\frac{11}{10}}e^{-\frac{4}{3}\varepsilon(\tau_*-\tau_0)}.
$$
Then $k(\tau)$ does not belong to $\widetilde{\mathcal R}(\tau)$ 
for $\tau$ close enough to $\tau_*$ from  above. 
\end{Lemma}

\begin{proof}
Since $(\overset{\approx}{\Clss},\overset{\approx}{U})$ satisfies the Cauchy problem \eqref{eq:truncated}
and $V_{\rm{uns}}$ is invariant with respect to the linear operator $\mathcal L$, the unstable component $k(\tau)$ satisfies
\begin{align*}
\begin{cases}
\displaystyle
\partial_\tau k(\tau) = \mathcal L k(\tau) + P_{\rm{uns}}(\chi_2 \mathcal F), \\
\displaystyle
k(\tau_0) = (\,\sum_{i=1}^{N}\hat{k}_i
\varphi_{i,\subclss},\,\sum_{i=1}^{N}\hat{k}_i
\varphi_{i,u}).
\end{cases}
\end{align*}
The forcing estimate \eqref{forcingest} leads to
\begin{align*}
&\|P_{\rm{uns}}(\chi_2 \mathcal F)\|_{\Upsilon} \leq C(m)\|P_{\rm{uns}}(\chi_2 \mathcal F)\|_{X} \leq C(m) \|\chi_2 \mathcal F\|_{X}\leq C(m)\smallc_{1}^{\frac{6}{5}}e^{-\frac{3}{2}\varepsilon(\tau-\tau_0)}.
\end{align*}
Consequently, we obtain from Lemma \ref{prop:maxdissmooth} that
\begin{align*}
\big\langle \frac{\mathrm d}{\mathrm d\tau}k(\tau), \,k(\tau) \big\rangle_{\Upsilon}\bigg|_{\tau=\tau_*}\geq -\frac{3\smallc_g}{5}\langle k(\tau_*),k(\tau_*) \rangle_{\Upsilon}-C(m)  \smallc_{1}^{\frac{6}{5}}e^{-\frac{3}{2}\varepsilon(\tau_*-\tau_0)}\|k(\tau_*)\|_{\Upsilon}.
\end{align*}
Since $\|k(\tau_*)\|_{\Upsilon}=\smallc_1^{\frac{11}{10}}e^{-\frac{4}{3}\varepsilon(\tau_*-\tau_0)}$,
it follows that
\begin{align*}
\big\langle \frac{\mathrm d}{\mathrm d\tau}k(\tau), \,k(\tau) \big\rangle_{\Upsilon}\bigg|_{\tau=\tau_*}&\geq -\frac{3\smallc_g}{5}\langle k(\tau_*),k(\tau_*) \rangle_{\Upsilon}-C(m)\smallc_{1}^{\frac{6}{5}-\frac{11}{10}}e^{-\varepsilon(\frac{3}{2}-\frac{4}{3})(\tau_*-\tau_0)} \langle k(\tau_*),k(\tau_*) \rangle_{\Upsilon}.\\
&\geq -\frac{31}{50}\smallc_g \langle k(\tau_*),k(\tau_*) \rangle_{\Upsilon}.
\end{align*}
 When $\tau=\tau_*$,
$$\frac{31}{50}\smallc_g< \frac{4}{3}\varepsilon=\frac{16}{25}\smallc_g$$ 
yields
$$
\frac{\mathrm d}{\mathrm d\tau}\big(\langle k(\tau), \,k(\tau)\rangle_{\Upsilon}e^{\frac{8}{3}\varepsilon(\tau-\tau_0)}\big)\geq -\frac{31}{25}\smallc_g\langle k(\tau),k(\tau)\rangle_{\Upsilon} e^{\frac{8}{3}\varepsilon(\tau-\tau_0)}+\frac{8}{3} \varepsilon\langle k(\tau),k(\tau) \rangle_{\Upsilon}e^{\frac{8}{3}\varepsilon(\tau-\tau_0)}>0.
$$
Therefore, $k(\tau)$ exits $\widetilde{\mathcal R}(\tau)$ at $\tau=\tau_*$.

This completes the proof.

\end{proof}

We claim that coefficients $\{\hat{k}_i\}$ can be selected so that $k(\tau) \in \mathcal{R}(\tau)$ for all $\tau \in [\tau_0, \infty)$. The argument is identical to that of \cite[Proposition 8.15]{buckmaster}, and thus we use that result directly.
\begin{Proposition} \label{prop:ci}
There exist specific initial conditions $\{\hat{k}_i\} \in \widetilde{\mathcal R} (\tau_0)$ such that $ k(\tau_0)=\{\hat{k}_i\}$ and $k(\tau)\in \mathcal R(\tau)$ for all $\tau>\tau_0$.
\end{Proposition}

Moreover, restriction $(\chi_2 \widetilde{\Clss}_0^*, \chi_2 \widetilde{U}_0^*) \in V_{\mathrm{sta}}$  is equivalent to 
$$
P_{\mathrm{uns}} (\widetilde{\Clss}_0^*, \widetilde{U}_0^*) = 0,
$$
which is a finite-dimensional closed restriction. Since projection
$P_{\mathrm{uns}}$ maps onto the finite-dimensional subspace 
$V_{\mathrm{uns}}$,  the manifold of initial data that give rise to finite-time implosion has finite codimension.

The proof of Proposition~\ref{prop:bootstrap} is now complete. By a standard continuity argument, the bootstrap assumptions are therefore preserved on the whole time interval of the existence. This enables us to establish the global-in-time existence of solution $(\Clss, U)$ of the Cauchy problem \eqref{selfsimilar eq}--\eqref{SSfar}.

\subsection{Global-In-Time Regularity}\label{global62}
In this subsection, we establish the global-in-time regularity of the solution. 
In \S \ref{global62}, the initial data function $(\widetilde{\Clss}_0, \widetilde{U}_0)$ under consideration
is constructed in  \eqref{eq:tilde_is_stable}--\eqref{initialdatacondition2},
$1< \gamma <1+\frac{2}{\sqrt{3}}$, $\Lambda \in (1, \Lambda^*(\gamma))$ is introduced in {\rm Lemma~\ref{thm:existence_profiles}}, $\delta\in (0,\frac{1}{2})$ satisfies $\delta_{\rm dis}>0$,
and assumptions \eqref{Para-ineq}--\eqref{choiceR0} always hold.
Based on the conclusions obtained in Proposition \ref{prop:bootstrap},
we are now ready to establish the desired global-in-time regularities.

\begin{Lemma}\label{lem:globalregularities}
     The solution of the Cauchy problem \eqref{selfsimilar eq}{\rm--}\eqref{SSfar} satisfies \eqref{unsestimate1}{\rm--}\eqref{higherestimate3} and, for $\tau \in [ \tau_0, \tau_*]$,
    \begin{equation}\label{eq:HKestimate1}
        \| \Clss - \Clss^*\|^2_{H^K} + \|  U \|^2_{H^K}+\int_{\tau_0}^{\tau}\|  U(\tilde{\tau}, \cdot) \|^2_{H^{K+1}}\,{\rm d}\tilde{\tau}\leq C(\tau),
    \end{equation}
    \begin{equation}\label{eq:HKestimate2}
        \|\Clss_{\tau} \|_{H^{K-1}}+ \|  U_\tau \|_{H^{K-2}}+\int_{\tau_0}^{\tau}\|  U_\tau(\tilde{\tau}, \cdot) \|_{H^{K-1}}\,{\rm d}\tilde{\tau}\leq C(\tau).
    \end{equation}
\end{Lemma}
\begin{proof}
    
 Combining \eqref{selfsimilar eq} with $Q^* = \Lambda e^{-(\Lambda-1)\tau}\bar{\lss}$, we derive
\begin{equation}\label{eq:L2}
\begin{aligned}
    (\partial_\tau + \Lambda-1)(\Clss-\Clss^*)=&-(y+U)\cdot \nabla (\Clss-\Clss^*) -\alpha\Clss \, \dive\, U,\\
    (\partial_\tau + \Lambda-1)U=&-(y+U)\cdot\nabla U -\alpha\Clss \nabla \Clss+\mathcal{F}_{\rm{dis}}.
    \end{aligned}
\end{equation}
Multiplying each equation of \eqref{eq:L2} by  $\Clss-\Clss^*$  and $U$ respectively, summing over all, and integrating over $\mathbb{R}^3$, we obtain
\begin{align}\label{eq:L2energy}
   &\frac{1}{2}\frac{\rm d}{\rm d\tau} \int \big(|\Clss-\Clss^*|^2 +|U|^2\big)\,\dy \nonumber \\
   & = -(\Lambda-1)\int \big(|\Clss-\Clss^*|^2 +|U|^2\big)\,\dy \,- \frac{1}{2}\int (y+U)\cdot \nabla\big( |\Clss-\Clss^*|^2+|U|^2\big)\, \dy \\
   &\quad- \alpha\int  \big(\Clss(\Clss-\Clss^*)\,\dive \, U+\Clss \nabla \Clss \cdot U \big)\,\dy +C_{\rm{dis}}e^{-\delta_{\rm{dis}}}\int \Clss^{-\frac{1}{\delta}}\dive \,\big(\Clss^{\frac{\delta}{\alpha}}\md(U)\big)\cdot U \,\dy \nonumber \\
   &=J_1+J_2+J_3+J_4. \nonumber
\end{align}
For simplicity, we denote
\begin{equation*}
    \bar{E}(\tau) = \frac{1}{2}\int \big(|\Clss-\Clss^*|^2 +|U|^2\big)(\tau, y)\,\dy,
\end{equation*}
where $\tau \ge \tau_0$.
Since $\Lambda>1$, $J_1$ carries the negative sign. We now estimate the remaining terms. 
Integrating by parts and using Lemma~\ref{lemma:Sprimebounds}, we have
\begin{equation}
\begin{aligned}\label{est:J2}
   J_2 & =  -\frac{1}{2}\int (y+U)\cdot \nabla\big( |\Clss-\Clss^*|^2+|U|^2\big)\, \dy \\
   & = \frac{1}{2}\int \dive(y+U)\big( |\Clss-\Clss^*|^2+|U|^2\big)\, \dy \leq  C \bar{E}(\tau).
\end{aligned}
\end{equation}
For $J_3$, it follows from the Young inequality, Lemmas \ref{lemma:Sbounds} and  \ref{linfitnityinside},
the definition of $\widehat{X}$, and the properties of profile $(\overline{\Clss}, \overline{U})$ that
\begin{align}
    J_3 & = - \alpha\int  \big(\Clss(\Clss-\Clss^*)\,\dive \, U+\Clss \nabla \Clss \cdot U \big)\,\dy \nonumber\\
    & \leq (\Lambda-1)\bar E(\tau) + C\int \big(|\nabla \Clss|^2 + |\nabla U|^2\big)\, \dy \nonumber \\
    & \leq (\Lambda-1)\bar E(\tau) + C\int \big(|\nabla (\widehat{X}\overline{\Clss})|^2 + |\nabla (\widehat{X}\overline{U})|^2\big)\, \dy + C\|(\nabla \widetilde{\Clss}, \nabla \widetilde{U}) \|_{L^2}^2\nonumber \\
    & \leq (\Lambda-1)\bar E(\tau) + C(\tau) + C(C_0)\|(\nabla \widetilde{\Clss}, \nabla \widetilde{U}) \|_{L^\infty(B(0,C_0))}^2 + C\|(\nabla \widetilde{\Clss}, \nabla \widetilde{U}) \|_{L^2(B^c(0,C_0))}^2\nonumber \\
    & \leq (\Lambda-1)\bar E(\tau) + C(\tau)+C\|(\nabla \widetilde{\Clss}, \nabla \widetilde{U}) \|_{L^2(B^c(0,C_0))}^2.\nonumber 
\end{align}
In the Gagliardo-Nirenberg inequality \eqref{eq:GNresultnoweightwholespace}, choosing
$$
\bar r=2, \, \,  l=4, \,  \, p=2, \,  \, q=2, \,  \,  \theta = \frac{1}{4},
$$
we have
\begin{align*}
    \|\nabla \widetilde{\Clss}\|_{L^2(B^c(0,C_0))} 
    & \leq C\big(\| \nabla^4 \widetilde{\Clss} \|^{\frac{1}{4}}_{L^2(B^c(0,C_0))}\|  \widetilde{\Clss}-\Clss^* \|^{\frac{3}{4}}_{L^2(B^c(0,C_0))}
    +\|\widetilde{\Clss}-\Clss^* \|_{L^2(B^c(0,C_0))}\big) \\
    & \leq C\big(\|\widetilde{\Clss}-\Clss^* \|_{L^2(B^c(0,C_0))}+ \| \nabla^4 \widetilde{\Clss} \|_{L^2(B^c(0,C_0))}\big)\\
    & \leq C\big(\|\Clss-\Clss^* \|_{L^2(B^c(0,C_0))} +\|\widehat{X}\overline{ \Clss} \|_{L^2(B^c(0,C_0))} 
    + \| \nabla^4 \widetilde{\Clss} \|_{L^2(B^c(0,C_0))}\big)\\
    &\leq C\|\Clss-\Clss^* \|_{L^2}+ C(\tau).
\end{align*}
Similarly, we can also  obtain
\begin{align*}
    \|\nabla \widetilde{U}\|_{L^2(B^c(0,C_0))}  \leq C\|U \|_{L^2}+ C(\tau).
\end{align*}
Thus, we conclude
\begin{equation}\label{est:J3}
    J_3 \leq C\bar E(\tau)+C(\tau).
\end{equation}
For $J_4$, we have
\begin{align}\label{est:J4}
    J_4& = C_{\rm{dis}}e^{-\delta_{\rm{dis}}\tau}\int \Clss^{-\frac{1}{\delta}}\dive \,\big(\Clss^{\frac{\delta}{\alpha}}\md(U)\big)\cdot U \,\dy \nonumber \\
    & = -C_{\rm{dis}}e^{-\delta_{\rm{dis}}\tau}\Big(\int \Clss^{\frac{\delta-1}{\alpha}}\big(a_1|\nabla U|^2+ (a_1+a_2)|\dive U|^2\big)\,\dy \,+\int \nabla(\Clss^{-\frac{1}{\alpha}})\big(\Clss^{\frac{\delta}{\alpha}}\md (U)\big)\cdot U \,\dy\Big) \nonumber\\
    & = -C_{\rm{dis}}e^{-\delta_{\rm{dis}}\tau}\Big(\int \Clss^{\frac{\delta-1}{\alpha}}\big(a_1|\nabla U|^2+ (a_1+a_2)|\dive U|^2\big)\,\dy \,-\int \frac{1}{\alpha}\big(\frac{\nabla \Clss}{\Clss}\cdot \Clss^{\frac{\delta-1}{\alpha}}\md (U)\big)\cdot U \,\dy\Big) \nonumber\\
    & \leq -\frac{1}{2}C_{\rm{dis}}e^{-\delta_{\rm{dis}}\tau}\int \Clss^{\frac{\delta-1}{\alpha}}\big(a_1|\nabla U|^2+ (a_1+a_2)|\dive U|^2\big)\,\dy  \\
    & \quad +Ce^{-\delta_{\rm{dis}}\tau}\int \frac{|\nabla \Clss|^2}{\Clss^2}\Clss^{\frac{\delta-1}{\alpha}}|U|^2 \,\dy \nonumber \\
    &\leq -\frac{1}{2}C_{\rm{dis}}e^{-\delta_{\rm{dis}}\tau}\int \Clss^{\frac{\delta-1}{\alpha}}\big(a_1|\nabla U|^2+ (a_1+a_2)|\dive U|^2\big)\,\dy \nonumber\\
    &\quad + Ce^{-\delta_{\rm{dis}}\tau}\int \langle y \rangle^{-2+\delta_{\rm{dis}}-(\Lambda-2)} |U|^2 \,\dy \nonumber \\
    &\leq -\frac{1}{2}C_{\rm{dis}}e^{-\delta_{\rm{dis}}\tau}\int \Clss^{\frac{\delta-1}{\alpha}}\big(a_1|\nabla U|^2+ (a_1+a_2)|\dive U|^2\big)\,\dy+ C\bar E(\tau),\nonumber
\end{align}
where the last inequality follows from Lemmas~\ref{lemma:Sbounds}--\ref{lemma:Sprimebounds}, the definition of $\delta_{\rm dis}$, and $\Lambda>1$.

Combining \eqref{eq:L2energy}--\eqref{est:J4} yields
\begin{equation}
    \frac{\rm d}{\rm d\tau}\bar E(\tau)+ \frac{1}{2}C_{\rm{dis}}e^{-\delta_{\rm{dis}}\tau}\int \Clss^{\frac{\delta-1}{\alpha}}\big(a_1|\nabla U|^2+ (a_1+a_2)|\dive U|^2\big)\,\dy \leq C\bar E(\tau)+C(\tau),
\end{equation}
which, along with the Gr\"onwall inequality, implies
\begin{equation}
    \bar E(\tau)=  \frac{1}{2}\int \big(|\Clss-\Clss^*|^2 +|U|^2\big)\,\dy \leq C(\tau).
\end{equation}

Moreover, using \eqref{weightdefi}, \eqref{eq:Sbound1}, and \eqref{energyest3}, we obtain
\begin{equation}
    \int \big(|\nabla^K \Clss|^2+ |\nabla^K U|^2\big)\,\dy \leq E_K(\tau)\leq E,
\end{equation}
and 
\begin{equation}
\begin{aligned}
   & \int_{\tau_0}^{\tau}\|\nabla^{K+1} U(\tilde{\tau}, \cdot)\|^2_{L^2} \,\text{d}\tilde{\tau} \\
   &\leq C(\tau)\int_{\tau_0}^{\tau}e^{-\delta_{\rm{dis}}\tilde{\tau}-K\min \{ \tilde \eta , \eta \}(\tau-\tilde \tau)}\|\phi^{\frac{K}{2}}\Clss^{\frac{\delta-1}{2\alpha}}\nabla^{K+1}U(\tilde{\tau}, \cdot)\|_{L^2}^2\,\text{d}\tilde{\tau} \\
    &\leq C(\tau, E).
\end{aligned}
\end{equation}
Together with the $L^2$--control of $(\Clss-\Clss^*, U)$ and $\tau\ge \tau_0 \gg E$, the uniform bound on
$E_K(\tau)$ allows us to apply the standard interpolation inequalities to obtain
\begin{equation}
\| \Clss - \Clss^*\|^2_{H^K} + \|  U \|^2_{H^K}+\int_{\tau_0}^{\tau}\|  U(\tilde{\tau}, \cdot) \|^2_{H^{K+1}}\,\text{d}\tilde{\tau}\leq C(\tau).
\end{equation}
By the equation \eqref{selfsimilar eq}, we also derive 
\eqref{eq:HKestimate2}.

This completes the proof of Lemma~\ref{lem:globalregularities}.
\end{proof}

\section{Global Existence of Smooth  Solutions of  the Reformulated Problem}\label{Section7}

In this section, we give the proof of  Theorem~\ref{globalexist} based on  the local well-posedness obtained in Theorem~\ref{Local(U,S)} and  the \textit{a priori} estimates obtained in Proposition~\ref{prop:bootstrap}.

According to Theorem~\ref{Local(U,S)}, there exists a unique smooth   solution $(\Clss, U)$ of  the  Cauchy problem \eqref{selfsimilar eq}--\eqref{SSfar} in  $[\tau_0, \tau_*]\times \mathbb{R}^3$. To extend this solution to be a global-in-time one, we define the set
\begin{equation}
\mathscr{S} = \{ \tau > \tau_0 : \text{the solution exists on } [\tau_0, \tau] \text{ and satisfies }  \eqref{lowerestimate1} \text{ and } \eqref{higherestimate3}  \}.
\end{equation}
Let $\bar{\tau}_{*} = \sup \mathscr{S}$. We aim to show that $\bar{\tau}_{*} = \infty$ by contradiction.

Let $\hat\tau$ be any fixed time satisfying $\hat\tau \in (\tau_0, \bar{\tau}_*)$. Collecting the uniform \textit{a priori} bounds obtained in Proposition~\ref{prop:bootstrap} and Lemma~\ref{lem:globalregularities} yields that, for any $\tau \in [\tau_0, \hat\tau]$,
\begin{equation}\label{uniformbounds}
    \begin{aligned}
         \| (\Clss-\Clss^*)(\tau, \cdot) \|^2_{H^K}+\|U (\tau, \cdot)\|^2_{H^K} + \int_{\tau_0}^{\tau}\|  U(\upsilon, \cdot)\|^2_{H^{K+1}} \text{d}\upsilon &\leq C(\hat\tau),\\
        \| \partial_\tau \Clss (\tau, \cdot) \|^2_{H^{K-1}} + \|\partial_{\tau} U(\tau, \cdot)\|^2_{H^{K-2}} +\int_{\tau_0}^{\tau}\| \partial_{\tau} U(\upsilon, \cdot)\|^2_{H^{K-1}} \text{d}\upsilon &\leq C(\hat\tau),\\
        \|P_{\rm{uns}}( \overset{\approx}{\Clss}, \overset{\approx}{U})(\tau, \cdot)\|_{X} &\leq \frac{\smallc_1}{2},\\
        \max\big\{\|\widetilde{\Clss}(\tau, \cdot)\|_{L^\infty}, \  \|\widetilde{U}(\tau, \cdot)\|_{L^\infty}\big\}
        &\leq \frac{\smallc_0 }{200}, \\ 
        \max\big\{\|\nabla^4\widetilde{\Clss}(\tau, \cdot)\|_{L^2(B^c(0, C_0))}, \ \|\nabla^4\widetilde{U}(\tau, \cdot)\|_{L^2(B^c(0, C_0))}\big\} & \leq \frac{\smallc_0}{2} , \\
       E_{K}(\tau)=\int\big(|\nabla^{K}\Clss(\tau, y)|^2+|\nabla^{K}U(\tau,y)|^2\big)\phi^{K}(y)\dy &\leq \frac{E}{2}.
    \end{aligned}
\end{equation}

Clearly, $\bar \tau_{*} \ge \tau_*$. If $\bar \tau_{*}<\infty$, according to the uniform \textit{a priori} estimates \eqref{uniformbounds} and the standard weak convergence arguments, for any sequence $\{\tau_k\}_{k=1}^{\infty}$ satisfying 
$$
0<\tau_k <\bar \tau_*,\qquad\,\, \tau_k \to \bar\tau_*\,\,\,\,\, \text{as $k\to \infty$},
$$
there exists a subsequence (still denoted by) $\{\tau_k\}_{k=1}^{\infty}$ and limits $(\Clss, U)(\bar\tau_*, y)$ satisfying
\begin{equation*}
    \Clss(\bar\tau_*, y)-\Clss^* \in H^K(\mathbb{R}^3), \quad U(\bar\tau_*, y)\in H^K(\mathbb{R}^3),
\end{equation*}
and, by Lemma \ref{lem:weighted-weak-limit}, as $k\to \infty$,
\begin{equation}
    \begin{aligned}
        \Clss(\tau_k, y)-\Clss^* \to \Clss(\bar\tau_*, y)-\Clss^* &\quad\text{weakly   \  in}\  H^K(\mathbb{R}^3),\\
        U(\tau_k, y) \to U(\bar\tau_*, y) &\quad\text{weakly \    in}\   H^K(\mathbb{R}^3),\\
        P_{\mathrm{uns}}(\overset{\approx}{\Clss}, \overset{\approx}{U})(\tau_k, y)\to P_{\mathrm{uns}}(\overset{\approx}{\Clss}, \overset{\approx}{U})(\bar\tau_*, y)&\quad\text{weakly \    in}\   H_0^m(B(0, 3C_0)), \\
        (\widetilde{\Clss}, \widetilde{U})(\tau_k, y) \to  (\widetilde{\Clss}, \widetilde{U})(\bar\tau_*, y) &\quad\text{weakly$^*$  \   in}    \  L^\infty(\mathbb{R}^3), \\
        (\nabla^4\widetilde{\Clss}, \nabla^4\widetilde{U})(\tau_k, y) \to  (\nabla^4\widetilde{\Clss}, \nabla^4\widetilde{U})(\bar\tau_*, y)&\quad\text{weakly \    in}\   L^2(B^c(0, C_0)),\\
        (\nabla^K \Clss\phi^{\frac{K}{2}}, \nabla^K U \phi^{\frac{K}{2}})(\tau_k, y) \to  (\nabla^K \Clss\phi^{\frac{K}{2}}, \nabla^K U\phi^{\frac{K}{2}})(\bar\tau_*, y)&\quad\text{weakly   \  in}\   L^2(\mathbb{R}^3).
    \end{aligned}
\end{equation}
It follows from \eqref{uniformbounds} and the lower semi-continuity of norms that
\begin{equation}\label{uniformbounds2}
\begin{aligned}
     \|(\Clss-\Clss^*)(\bar\tau_*, \cdot)\|_{H^K}^2 +\|U(\bar\tau_*, \cdot)\|_{H^K}^2 &\leq C(\bar\tau_*),\\
     \|P_{\rm{uns}}( \overset{\approx}{\Clss}, \overset{\approx}{U})(\bar\tau_*, \cdot)\|_{X} &\leq \frac{\smallc_1}{2},\\
        \max\big\{\|\widetilde{\Clss}(\bar\tau_*, \cdot)\|_{L^\infty}, \  \|\widetilde{U}(\bar\tau_*, \cdot)\|_{L^\infty}\big\} &\leq \frac{\smallc_0 }{200}, \\ 
        \max\big\{\|\nabla^4\widetilde{\Clss}(\bar\tau_*, \cdot)\|_{L^2(B^c(0, C_0))}, \  \|\nabla^4\widetilde{U}(\bar\tau_*, \cdot)\|_{L^2(B^c(0, C_0))}\big\}&\leq \frac{\smallc_0}{2} , \\
       E_{K}(\bar\tau_*)=\int\big(|\nabla^{K}\Clss(  \bar\tau_*,y)|^2+|\nabla^{K}U(  \bar\tau_*,y)|^2\big)\phi^{K}(y)\,\dy &\leq \frac{E}{2}.
\end{aligned}
\end{equation}

We take the limit $(\Clss, U)(\bar\tau_*, y)$  as new initial data at $\tau=\bar\tau_*$. 
Theorem~\ref{Local(U,S)} ensures the existence of a positive time $\iota_*>0$ 
such that the smooth solution exists on $[\bar\tau_*, \bar\tau_*+\iota_*]$.
$(\Clss, U)(\bar\tau_*, y)$ satisfies \eqref{initialdatacondition1}--\eqref{initialdatacondition2}, and the bounds in \eqref{uniformbounds} yield  
$$(\Clss-\Clss^*, U) \in C([\bar\tau_*, \bar\tau_*+\iota_*]; H^K(\mathbb{R}^3)).$$ 
It follows from \eqref{EKRbound}--\eqref{EKbound} and \eqref{uniformbounds2}$_5$ that, for all $\tau \in [\bar\tau_*, \bar\tau_*+\iota_*]$,
\begin{equation}\label{EKbounds}
\begin{aligned}
E_K(\tau) &\leq e^{C(K, \smallc_1, \iota_*)(\tau-\bar\tau_*)}E_K(\bar\tau_*)+e^{C(K, \smallc_1, \iota_*)(\tau-\bar\tau_*)}-1\\
& \leq e^{C(K, \smallc_1, \iota_*)(\tau-\bar\tau_*)}\frac{E}{2}
+e^{C(K, \smallc_1, \iota_*)(\tau-\bar\tau_*)}-1.
\end{aligned}
\end{equation}
By the continuity of factor $e^{C(K, \smallc_1, \iota_*)(\tau-\bar\tau_*)}$, there exists $\iota_1 \in (0, \iota_*]$ 
such that estimate  \eqref{higherestimate3} remains valid on the extended interval $ [\bar\tau_*, \bar\tau_*+\iota_1]$.

On the other hand, since 
$$(\Clss-\Clss^*,\,U) \in C([\bar\tau_*, \bar\tau_*+\iota_*]; H^K(\mathbb{R}^3))$$
and $K \ge 4$, the Sobolev embedding and the bounds of $(\widehat{X}\overline{\Clss},\,\widehat{X}\overline{U})$ 
imply 
$$
(\widetilde{\Clss},\,\widetilde{U})\in C([\bar\tau_*, \bar\tau_*+\iota_*]; L^\infty(\mathbb{R}^3)).
$$
Together with \eqref{uniformbounds2}$_3$, this yields that there exists $\iota_2 \in (0, \iota_*]$ such that 
estimate  \eqref{lowerestimate1} remains valid on the extended interval $ [\bar\tau_*, \bar\tau_*+\iota_2]$. 
Consequently, there exists $\iota = \min\{\iota_1, \iota_2\}$ such that estimates $\eqref{lowerestimate1}$ and $\eqref{higherestimate3}$ 
are both valid on the extended interval $[\bar\tau_*, \bar\tau_*+\iota]$.

This contradicts the definition of $\bar\tau_*$ as the supremum of $\mathscr{S}$. Therefore, we conclude that $\bar\tau_*=\infty$, yielding the global-in-time existence of the solution. This completes the proof of Theorem~\ref{globalexist}.

\section{Development   of Implosions of Solutions}\label{Section8}

Based on the global-in-time stability result obtained in Theorem~\ref{globalexist} for the reformulated system \eqref{selfsimilar eq} in self-similar coordinates,
we are now ready to consider the corresponding development of implosions of solutions to  the degenerate \textbf{CNS} \eqref{eq:1.1} in Eulerian coordinates.

\smallskip
We divide the proof into two steps.

\smallskip
\textbf{1. Condition $(P_1)$}. Under condition $(P_1)$ in {\rm Theorem~\ref{Thm1.1}}, for any fixed  $\gamma\in (1, 1+\frac{2}{\sqrt 3})$, there exists $\Lambda\in (1, \Lambda^*(\gamma))$ introduced in {\rm Lemma~\ref{thm:existence_profiles}} such that the steady Euler system \eqref{Profile} admits a smooth, spherically symmetric solution $(\overline{\Clss},\,\overline{U})$. Furthermore, $\delta\in (0, \frac{1}{2})$ satisfies $\delta_{\rm dis}>0$. 
Then all the assumptions for parameters $(\gamma, \Lambda, \delta)$ required in  Theorem~\ref{globalexist} are satisfied.

For any vector function  $(\widetilde{\Clss}^*_0, \widetilde{U}^*_0)$ satisfying conditions \eqref{eq:tilde_is_stable}--\eqref{initialdatacondition2}, we can take $\{\hat{k}_i\}$ as in Proposition \ref{prop:ci}. Then the initial data function for the Cauchy problem \eqref{selfsimilar eq}--\eqref{SSfar} is given by
$$
\Clss_0 = \widehat{X}\overline{\Clss}+ \widetilde{\Clss}_0^* +\sum_{i=1}^N \hat{k}_i\varphi_{i, \subclss},\qquad U_0 = \widehat{X}\overline{U}+ \widetilde{U}_0^* +\sum_{i=1}^N \hat{k}_i\varphi_{i, u}.
$$
According to the self-similar scaling \eqref{scaling-coordinat} and Theorem \ref{globalexist}, 
we see that, for all $\tau \ge \tau_0$,
\begin{equation*}
    \max\big\{ \|\widetilde{\Clss}(\tau, \cdot)\|_{L^\infty},\   \|\widetilde{U}(\tau, \cdot)\|_{L^\infty}\big\} \leq \frac{\smallc_0}{100}e^{-\varepsilon(\tau-\tau_0)}  \leq \frac{\smallc_0}{100}\left(\frac{T-t}{T}\right)^{\frac{\varepsilon}{\Lambda}}.
\end{equation*}
It follows from \eqref{scaling-coordinat}--\eqref{scaling} and $\widehat{X}(y, \tau)=\mathscr{X}(e^{-\tau}y)$ that
\begin{equation*}
    \begin{aligned}
        \lss(t, x)&= \frac{1}{\Lambda(T-t)^{1-\frac{1}{\Lambda}}}\Big(\mathscr{X}(x)\overline{\Clss}(\frac{x}{(T-t)^{\frac{1}{\Lambda}}})+ \widetilde{\Clss}(-\frac{\log (T-t)}{\Lambda}, \frac{x}{(T-t)^{\frac{1}{\Lambda}}})\Big),\\
         u(t, x)&= \frac{1}{\Lambda(T-t)^{1-\frac{1}{\Lambda}}}\Big(\mathscr{X}(x)\overline{U}(\frac{x}{(T-t)^{\frac{1}{\Lambda}}})+ \widetilde{U}(-\frac{\log (T-t)}{\Lambda}, \frac{x}{(T-t)^{\frac{1}{\Lambda}}})\Big).
         \end{aligned}
\end{equation*}
Moreover, recalling that $\lss = \frac{1}{\alpha}\rho^\alpha$, we have
\begin{align*}
    \lim_{t\to T^-}(\alpha^{-1}\Lambda(T-t)^{1-\frac{1}{\Lambda}})^{\frac{1}{\alpha}}\rho(t, x)&= \mathscr{X}^{\frac{1}{\alpha}}(x)\lim_{t\to T^-}\overline{\Clss}^{\frac{1}{\alpha}}(\frac{x}{(T-t)^{\frac{1}{\Lambda}}}),\\
     \lim_{t\to T^-}\Lambda(T-t)^{1-\frac{1}{\Lambda}}u(t, x)&= \mathscr{X}(x)\lim_{t\to T^-}\overline{U}(\frac{x}{(T-t)^{\frac{1}{\Lambda}}}).
     \end{align*}

In view of the properties of the self-similar profile $(\overline{\Clss}, \overline{U})$, 
these asymptotics show that, as $t \to T^{-}$,  density $\rho(t, x)$ blows up at the origin, 
while velocity $u(t, x)$ becomes unbounded in any arbitrarily small neighborhood of the origin. 
Consequently, the solution to \textbf{CNS} develops an implosion-type singularity at time $T$.

\smallskip
\smallskip
\textbf{2. Condition $(P_2)$}. Now we consider the analysis under condition $(P_2)$. 
For any fixed  $\gamma \in (1+\frac{2}{\sqrt{3}})\setminus \mathcal{J}$, 
where the set $\mathcal J$ (possibly empty) introduced in {\rm Lemma~\ref{thm:existence_profiles}},  
there exists a discrete sequence of scaling parameters $\{\Lambda_n\}_{n=1}^{\infty}$ satisfying 
\begin{equation*}
    1<\Lambda_n <\Lambda^*(\gamma), \qquad \lim_{n\to \infty}\Lambda_n=\Lambda^*(\gamma),
\end{equation*}
such that the steady Euler system \eqref{Profile} with $\Lambda_n$ ($n=1,2, \cdots$) admits a smooth, spherically symmetric solution $(\overline{\Clss},\,\overline{U})$. For the range of $\delta>0$, in order to make sure that $\delta_{\rm dis}>0$, it is required that 
\begin{equation*}
    0<\delta<
\frac{\Lambda_n(\gamma+1)-2\gamma}{2(\Lambda_n-1)} <\frac{\Lambda^*(\gamma)(\gamma+1)-2\gamma}{2(\Lambda^*(\gamma)-1)}:=\delta^*(\gamma)<\frac{1}{2}.
\end{equation*}

Since $\lim_{n\to \infty}\Lambda_n=\Lambda^*(\gamma)$, when $\gamma \in (1+\frac{2}{\sqrt{3}})\setminus \mathcal{J}$ is fixed,  for any  $\delta\in (0, \delta^*(\gamma))$, there exists a suitable $\Lambda_n \in \{\Lambda_n\}_{n=1}^{\infty}$ such that, taking $\Lambda=\Lambda_n$, we have
$$
\delta_{\rm dis} (\gamma, \delta, \Lambda)=\frac{2(1-\delta)(\Lambda-1)}{\gamma-1}+\Lambda-2>0.$$
Consequently, the assumptions for parameters $(\gamma, \Lambda, \delta)$ required in  Theorem~\ref{globalexist} are satisfied. By an analogous argument in Step~1, we can establish the singularity formation of solutions to \textbf{CNS}.

\section{Remarks on the Periodic Problem}\label{remarkperiodic}
This section is devoted to the proof of the development of implosions for solutions 
of the periodic problem stated in Theorem \ref{Thm1.2peordic}. 
The argument is largely analogous to that of Theorem \ref{Thm1.1}. 
Accordingly, we provide only a brief outline of the proof and highlight 
the minor differences between the Cauchy and the periodic settings.

\subsection{Reformulation of the Periodic Problem}

First,  problem \eqref{periodicproblem} can be rewritten as 
\begin{equation}\label{eq:1.2periodic-cc}
\begin{cases}
\partial_t\lss=-u\cdot \nabla_x \lss -\alpha \lss \, \dive_x u  \ \ \ \ \  \  \  \  \  \  \  \ \ \qquad \qquad  \  \   \qquad \qquad  \ \qquad \qquad  \ \  \  \quad \ \ \ \text{in $\mathbb{R}^+\times \mathbb{T}^3_{10}$},\\[6pt]
 \partial_tu = -u\cdot \nabla_x u -\alpha \lss \nabla_x \lss+ \alpha^{\frac{\delta-1}{\alpha}}\lss^{\frac{\delta-1}{\alpha}}L_xu+\frac{\delta}{\alpha}\alpha^{\frac{\delta-1}{\alpha}}\lss^{\frac{\delta-1-\alpha}{\alpha}}\nabla_x \lss \cdot \md_x(u)  \quad \ \ \ \text{in $\mathbb{R}^+\times \mathbb{T}^3_{10}$},\\[6pt]
(c,u)(0,x)=(c_0,u_0)(x):=(\alpha^{-1}\rho^\alpha_0,u_0)(x)
   \ \ \    \   \qquad \qquad  \  \   \qquad \qquad  \ \qquad \quad  \text{in $\mathbb T^3_{10}$},
\end{cases}
\end{equation}
where $(L_x u, \md_x(u))$ are given in \eqref{Lame operator}, $\alpha =\frac{\gamma-1}{2}$, 
and  $\lss = \frac{1}{\alpha}\rho^{\alpha}$.

Second, based on the self-similar scaling \eqref{scaling-coordinat}--\eqref{scaling},
in the self-similar variables 
$$
(\tau, y)\in \Omega(\tau_0):=\{(\tau, y):\ \tau\in  [\tau_0, \infty), \ y\in  e^\tau\mathbb T^3_{10}\},
$$  \eqref{eq:1.2periodic-cc} can be rewritten as 
\begin{equation}\label{selfsimilar eq-periodic}
    \begin{cases}
    \displaystyle
        \partial_\tau \Clss =-(\Lambda-1)\Clss-(y+U)\cdot\nabla \Clss- \alpha \Clss \,\dive U \qquad  \ \ \ \ \ \  \quad \ \ \ \text{in $\Omega(\tau_0)$},
        \\[6pt]
       \displaystyle
        \partial_\tau U =-(\Lambda-1)U-(y+U)\cdot\nabla U- \alpha \Clss \nabla \Clss\\[6pt]
        \qquad\quad\,+ C_{\rm{dis}}e^{-\delta_{\rm{dis}}\tau}\big(\Clss^{\frac{\delta-1}{\alpha}}L(U)+\frac{\delta}{\alpha}\Clss^{\frac{\delta-1-\alpha}{\alpha}}\nabla \Clss \cdot \md(U)\big)
        \qquad \text{in $\Omega(\tau_0)$},       
        \\[6pt]
        \displaystyle
         (\Clss, U)(\tau_0, y)= (\Clss_0, U_0)(y)\\[6pt]
         \ \ \ \ \ \  :=(\Lambda\alpha^{-1} e^{-(\Lambda-1)\tau_0}\rho_0^{\alpha}, \ \Lambda e^{-(\Lambda-1)\tau_0}u_0)(e^{-\tau_0}y)  \quad \ \  \ \  \ \ \ \ \ \ \ \text{in} \  e^{\tau_0}\mathbb{T}_{10}^3,          \end{cases}
\end{equation}
where $(L(U),\md(U))$ and $(C_{\rm{dis}},\delta_{\rm{dis}})$ are given in \eqref{Lame operator-y}--\eqref{S-para}.

The original problem \eqref{eq:1.2periodic-cc} is defined in the fixed periodic domain
$$\mathbb T^3_{10}=\{(x_1, x_2, x_3)^\top: -5 \leq x_i <5, \ i=1,2,3\}.$$
Under the self-similar scaling of variables $y=e^\tau x$, the spatial domain corresponding to the rescaled variable $y$ becomes time-dependent, \textit{i.e.},
$$
e^\tau\mathbb T^3_{10}=\big\{(y_1, y_2, y_3)^\top:\, -5e^\tau \leq y_i < 5e^\tau, \ i=1,2,3\big\}. 
$$

Let $(\overline{\Clss}, \overline{U})$ be the $C^\infty$  self-similar profile  solving \eqref{Profile} obtained in {\rm Lemma \ref{lem:existofselfsimilar}} (see  {\rm Appendix \ref{appendix B}}) when  the  scaling parameter $\Lambda$ introduced in \eqref{scaling1}{\rm--}\eqref{scaling-coordinat1} satisfies  \eqref{range of Lambda}{\rm--}\eqref{def:Lambda*}. Since $(\overline{\Clss}, \overline{U})$ 
is spherically symmetric,  then there exists a scalar function $\overline{\mathcal{U}}$ such that 
\[
\overline{\Clss}(y)=\overline\Clss(r),\quad 
\overline{U}(y)=\overline{\mathcal U}(r)\frac{y}{r} \qquad\,\,\mbox{ for $r=|y|$}.
\]
The cutoff self-similar profile $(\widehat{X}\overline{\Clss}, \widehat{X}\overline{U})$ with $\widehat{X}$ introduced in \eqref{cutoff function} is compactly supported in $B(0, e^\tau)\subset e^\tau\mathbb T^3_{10}$. Since its support is strictly contained in the interior of the periodic domain, we may regard it as a periodic function on $e^\tau\mathbb T^3_{10}$ via periodic extension. Consequently, it satisfies the periodic boundary condition automatically.

Motivated by the proof  for Theorem \ref{Thm1.1},  
\begin{itemize}
\item
the study on  the development of implosions for the original periodic problem of \textbf{CNS} can also  be reduced to a stability analysis around the cutoff self-similar profile $(\widehat{X}\overline \Clss, \widehat{X}\overline U)$.
\end{itemize}
Based on this viewpoint, the key issue is to analyze the mathematical structure satisfied by the corresponding perturbation
\begin{equation}\label{truncated profile-periodic}
     \widetilde{\Clss}= \Clss- \widehat{X}\overline \Clss, \quad \quad  \widetilde{U}= U- \widehat{X}\overline U.
\end{equation}
According to  \eqref{Profile} and the equations in \eqref{selfsimilar eq-periodic}, 
 $(\widetilde{\Clss}, \widetilde{U})$ satisfies the following
equations{\rm:}
\begin{align}
    \partial_\tau \widetilde{\Clss} 
    =&\underbrace{-(\Lambda-1)\widetilde{\Clss}- (y + \widehat{X}\overline{U})\cdot \nabla \widetilde{\Clss}- \alpha(\widehat{X}\overline{\Clss})\,\dive \,\widetilde{U}
     - \widetilde{U}\cdot \nabla (\widehat{X}\overline{\Clss}) - \alpha\widetilde{\Clss}\,\dive (\widehat{X}\overline{U})}_{\mathcal{L}_\subclss^e(\widetilde{\Clss}, \widetilde{U})}\nonumber\\
    & \underbrace{- \widetilde{U}\cdot \nabla \widetilde{\Clss} - \alpha \widetilde{\Clss}\,\dive \,\widetilde{U}}_{\mathcal{N}_\subclss(\widetilde{\Clss}, \widetilde{U})}\label{eq: PerS-periodic}\\
    &\underbrace{-(\widehat{X}^2- \widehat{X})\overline{U}\cdot \nabla \overline{\Clss}- (1+\alpha)\widehat{X}\overline{U}\cdot \nabla\widehat{X}\overline{\Clss}-(\widehat{X}^2- \widehat{X})\overline{\Clss}\,\dive \,\overline{U}}_{\mathcal{E}_\subclss},  \nonumber\\
    \partial_\tau \widetilde{U} =&\underbrace{-(\Lambda-1)\widetilde{U}- (y + \widehat{X}\overline{U})\cdot \nabla \widetilde{U}- \alpha(\widehat{X}\overline{\Clss})\nabla \widetilde{\Clss}- \widetilde{U}\cdot \nabla (\widehat{X}\overline{U}) - \alpha\widetilde{\Clss}\nabla (\widehat{X}\overline{\Clss})}_{\mathcal{L}_u^e(\widetilde{\Clss}, \widetilde{U})}\nonumber\\
    &\underbrace{-\widetilde{U}\cdot \nabla \widetilde{U} - \alpha \widetilde{\Clss}\nabla \widetilde{\Clss}}_{\mathcal{N}_u(\widetilde{\Clss}, \widetilde{U})}\label{eq: PerU-periodic}\\
    &\underbrace{-(\widehat{X}^2- \widehat{X})\overline{U}\cdot \nabla \overline{U}-\widehat{X}\overline{U}\cdot \nabla\widehat{X}\overline{U}-\alpha(\widehat{X}^2- \widehat{X})\overline{\Clss} \nabla \overline{\Clss} -\alpha \widehat{X}\overline{\Clss} \nabla\widehat{X}\overline{\Clss} }_{\mathcal{E}_u} \nonumber \\
    &\underbrace{+C_{\rm{dis}}e^{-\delta_{\rm{dis}}\tau}\big(\Clss^{\frac{\delta-1}{\alpha}}L(U)+\frac{\delta}{\alpha}\Clss^{\frac{\delta-1-\alpha}{\alpha}}\nabla \Clss \cdot \md(U)
        \big)}_{\mathcal{F}_{\rm{dis}}}. \nonumber
\end{align}

We are now ready to state the desired global stability for problem \eqref{selfsimilar eq-periodic}.

\begin{Theorem}\label{globalexist-periodic}
We assume that 
$1< \gamma <1+\frac{2}{\sqrt{3}}$, $\Lambda \in (1, \Lambda^*(\gamma))$ is the scaling parameter introduced in {\rm Lemma~\ref{thm:existence_profiles}} such that there exists a smooth, spherically symmetric self-similar profile $(\overline{\Clss}, \overline{U})$ solving \eqref{Profile}, $\delta\in (0,\frac{1}{2})$ satisfies $\delta_{\rm dis}>0$,
$(m, J)$ are  two sufficiently large constants determined in {\rm Lemma~\ref{prop:maxdissmooth}}, and  $C_0>1$ is  a sufficiently large constant depending on profile $(\overline{\Clss}, \overline{U})$ and parameter $\Lambda$. 
Then there exist positive constants $\smallc_1, E$, and $K$, which depend only on $(m, J)$, and satisfy 
  $$
  K\geq 4, \qquad \smallc_1 \ll \frac{1}{E} \ll \frac{1}{K}\ll 1,
  $$ such that, 
  if the initial data function $(\Clss_0,U_0)$ satisfies
  $$
 \inf_{y\in e^{\tau_0}\mathbb{T}_{10}^3}\Clss_0>0,\quad   \Clss_0-\Clss^* \in   H^K(e^{\tau_0}\mathbb{T}_{10}^3), \quad U_0\in  H^K(e^{\tau_0}\mathbb{T}_{10}^3),     
 $$
and 
\begin{align}    
&\max\big\{ \|\Clss_0-\widehat{X}\overline{\Clss}\|_{L^{\infty}(e^{\tau_0}\mathbb{T}_{10}^3)}, 
\ \|U_0-\widehat{X}\overline{U}\|_{L^{\infty}(e^{\tau_0}\mathbb{T}_{10}^3)}\big\} 
\leq \smallc_1, \label{initialdata1-periodic}\\
&\Clss_0 \geq \frac{\smallc_1}{2}, \quad    
\langle y \rangle^{\Lambda}\big(|\nabla \Clss_0|+|\nabla U_0|\big)< \infty,\label{initialdata1*-periodic}\\
&\max\big\{\|P_{\rm sta}(\chi_2(\Clss_0-\widehat{X}\overline{\Clss}))\|_{X},\  \|P_{\rm sta}(\chi_2(U_0-\widehat{X}\overline{U}))\|_{X}\big\} \leq \smallc_1,\label{initialdata2-periodic}\\
&E_K(\tau_0)=\int_{e^{\tau_0}\mathbb{T}_{10}^3}\big(|\nabla^K \Clss(\tau_0,y)|^2+|\nabla^K U(\tau_0,y)|^2\big)\phi^K(y) \,\dy \leq \frac{E}{2},\label{initialdata3-periodic}
\end{align}  
where the function space $X$ is defined in \eqref{X space} based on the choice of $(m,C_0)$, 
and the  weight function $\phi \in C^1(\mathbb{R}^3)$ is defined in \eqref{weightdefi},
    then  the periodic problem \eqref{selfsimilar eq-periodic} admits a unique global-in-time smooth solution $(\Clss, U)(\tau,y)$ in $\Omega(\tau_0)$,   
which satisfies
    \begin{equation}\label{peri0dic-A}
        \begin{split}
            &\inf_{y\in e^{\tau}\mathbb{T}_{10}^3} \Clss(\tau,y)>0, \ \ 
\Clss-\Clss^*\in C([\tau_0,\tau];H^K(e^{\tau}\mathbb{T}_{10}^3)),\\
&\partial_\tau \Clss\in C([\tau_0,\tau];H^{K-1}(e^{\tau}\mathbb{T}_{10}^3)),\ \ U\in C([\tau_0,\tau];H^K(e^{\tau}\mathbb{T}_{10}^3))\cap L^2([\tau_0,\tau];H^{K+1}(e^{\tau}\mathbb{T}_{10}^3)),\\[2pt]
& \partial_\tau U\in C([\tau_0,\tau];H^{K-2}(e^{\tau}\mathbb{T}_{10}^3))\cap L^2([\tau_0,\tau];H^{K-1}(e^{\tau}\mathbb{T}_{10}^3)),
\end{split}
    \end{equation}
and
    \begin{equation}\label{periodic-B}
        \max\big\{\|\Clss-\widehat{X}\overline{\Clss}\|_{L^\infty(e^{\tau}\mathbb{T}_{10}^3)},\   \|U-\widehat{X}\overline{U}\|_{L^\infty(e^{\tau}\mathbb{T}_{10}^3)}\big\}\leq \frac{\smallc_0}{100}e^{-\varepsilon(\tau-\tau_0)}, \qquad   E_K(\tau) \leq E,
    \end{equation}
    for all $\tau>\tau_0$, where $\smallc_0>0$ is a sufficiently small positive constant depending only on $E$, and $\varepsilon>0$ is a small positive constant depending only on $\delta_{\rm dis}$, so that they satisfy $$\smallc_0^{3/2}\ll \smallc_1 \ll \smallc_0 \ll \frac{1}{E}\ll \frac{1}{K}\ll \frac{1}{m}\ll \varepsilon \ll 1.$$
    Moreover, there exists a finite-codimension set of initial data satisfying the above conditions {\rm(}see {\rm{Appendix \ref{appendix D} }}for more details{\rm)}.
\end{Theorem}

\subsection{Local-In-Time Well-Posedness of the Reformulated Problem}

The desired local-in-time well-posedness  in this section can be stated as follows:
\begin{Theorem}\label{Local(U,S)-periodic}
Assume that  parameters $(\Lambda,\alpha,\delta, a_1,a_2)$ satisfy  \eqref{physical}
and  $K\geq 4$ is an integer.   
Let the initial data $(\Clss_0,U_0)$ satisfy 
\begin{equation}\label{initial'-periodic}
\inf_{y\in e^{\tau_0}\mathbb{T}_{10}^3} \Clss_0(y)>0, \quad  \Clss_0-\Clss^*\in H^s(e^{\tau_0}\mathbb{T}_{10}^3),\quad  U_0\in H^s(e^{\tau_0}\mathbb{T}_{10}^3)
\end{equation}
for some integer $s\geq K$. Then there exists  $\tau_*>\tau_0$ such that  problem \eqref{selfsimilar eq-periodic} admits a unique smooth  solution $(\Clss, U)(\tau, y)$ in 
$$
\Omega(\tau_0,\tau_*)=\big\{(\tau, y):\ \tau\in  [\tau_0, \tau_*], \ y\in  e^\tau\mathbb T^3_{10}\big\},
$$
which satisfies  
\begin{equation}\label{regularity'-periodic}
\begin{split}
&\inf_{(\tau, y)\in \Omega(\tau_0,\tau_*)} \Clss(\tau,y)>0, \ \ 
\Clss-\Clss^*\in C([\tau_0,\tau_*];H^s(e^{\tau}\mathbb{T}_{10}^3)),\\
&\partial_\tau \Clss \in C([\tau_0,\tau_*];H^{s-1}(e^{\tau}\mathbb{T}_{10}^3)),\ \ U\in C([\tau_0,\tau_*];H^s(e^{\tau}\mathbb{T}_{10}^3))\cap L^2([\tau_0,\tau_*];H^{s+1}(e^{\tau}\mathbb{T}_{10}^3)),\\[2pt]
& \partial_\tau U\in C([\tau_0,\tau_*];H^{s-2}(e^{\tau}\mathbb{T}_{10}^3))\cap L^2([\tau_0,\tau_*];H^{s-1}(e^{\tau}\mathbb{T}_{10}^3)). \end{split}
\end{equation}
Moreover, if $E_{K}(\tau_0)< \infty$, then
\begin{equation}\label{weightenergy-periodic}
   \text{\rm ess}\!\!\!\sup_{\tau_0\leq \tau \leq \tau_*} E_K(\tau) < \infty.
\end{equation}
\end{Theorem}

The corresponding local-in-time well-posedness of smooth solutions can be obtained by a similar argument as in \S \ref{Section3}. Moreover, since  the weight function $\phi(y)$ is bounded in $ e^\tau \mathbb{T}^3_{10}$, then the boundedness of the weighted energy $E_K(\tau)$  follows from the $H^s$--regularity of solutions.

\subsection{Global-In-Time Stability of the Reformulated Problem}
We assume that all the related  parameters still satisfy \eqref{Para-ineq}--\eqref{Para-ineq1}. Since $\tau_0$ is sufficiently large such that
$$C(\smallc_0) \le e^{\delta_{\rm dis}\tau_0},$$
it follows that, for all $\tau \ge \tau_0$,
\begin{equation}\label{Ballsinperiodic}
B(0, C_0)\Subset e^\tau\mathbb T^3_{10},\qquad B(0, R_0) \Subset e^\tau\mathbb T^3_{10}.
\end{equation}

Consequently, we may apply the same region segmentation method as in the Cauchy problem to establish the global-in-time stability for the periodic problem. The proof is analogous except the following two points: 
\begin{itemize}
    \item  When the interpolation estimates are needed, we employ the Gagliardo-Nirenberg interpolation 
    inequalities on the periodic domain $e^\tau \mathbb T^3_{10}$, as stated in 
    Lemmas~\ref{lemma:GN_general_period}--\ref{lemma:GN_generalnoweightwholespace_period}. 
    
   \item Since the periodic domain $e^\tau\mathbb{T}^3_{10}$ depends on time, 
    differentiating the energy functional with respect to $\tau$ generates an additional boundary 
    term associated with the moving boundary. 
    Moreover, to control the higher-order derivatives in certain integrals, 
    one must perform integration by parts. When the integrand contains non-periodic components, 
    such as the spatial variable $y$,  additional new boundary terms may arise 
    on $\partial(e^\tau\mathbb T^3_{10})$, which do not appear in the Cauchy problem.
    A direct computation shows that the boundary terms generated by integration by parts 
    exactly cancel those arising from the time differentiation of the expanding domain. 
    Consequently, the same energy estimates as in the Cauchy problem remain valid in the periodic setting.
\end{itemize}

\subsubsection{$4$-th order derivative estimates in $B^c(0, C_0)\cap e^\tau\mathbb T^3_{10}$}
To control 
$$
\|\nabla^4 \widetilde{\Clss}\|_{L^2(B^c(0, C_0)\cap e^\tau\mathbb T^3_{10})}+ \|\nabla^4 \widetilde{U}\|_{L^2(B^c(0, C_0)\cap e^\tau\mathbb T^3_{10})},
$$
applying  $\partial_{\beta}$ with multi-index $\beta=(\beta_1,\beta_2,\beta_3,\beta_4)$ satisfying $|\beta|=4$ to \eqref{eq: PerS-periodic}--\eqref{eq: PerU-periodic}, we obtain
\begin{equation*}
\begin{aligned}
\partial_{\tau}\partial_{\beta} \widetilde{\Clss}&=B_{\subclss,0}(\widetilde{U},\widetilde{\Clss})+B_{\subclss,1}(\widetilde{U},\widetilde{\Clss})+B_{\subclss,2}(\widetilde{U},\widetilde{\Clss})+B_{\subclss,3}(\widetilde{U},\widetilde{\Clss})+\partial_{\beta}\mathcal N_\subclss+\partial_{\beta}\mathcal E_\subclss,\\
\partial_{\tau}\partial_{\beta} \widetilde{U}&=B_{u,0}(\widetilde{U},\widetilde{\Clss})+B_{u,1}(\widetilde{U},\widetilde{\Clss})+B_{u,2}(\widetilde{U},\widetilde{\Clss})+B_{u,3}(\widetilde{U},\widetilde{\Clss})+\partial_{\beta}\mathcal N_u+\partial_{\beta}\mathcal E_u+\partial_{\beta}\mathcal F_{\rm dis},
\end{aligned}
\end{equation*}
where, for $0\leq i \le 3$, $B_{\subclss,i}$ and $B_{u,i}$  denote the collections of the terms as 
defined in Step 1 of the proof of Lemma~\ref{boosstrapestimate2}.

Multiplying the above equations by $\partial_{\beta} \widetilde{\Clss}$ and 
$\partial_{\beta} \widetilde{U}$ respectively, summing over all 
$|\beta|=4$, and then integrating over $B^c(0,C_0)\cap e^\tau\mathbb T^3_{10}$, we obtain from \eqref{Ballsinperiodic} that
    \begin{align}
       I&=\frac{\mathrm d}{\mathrm d\tau} \Big( \|\nabla^4\widetilde{\Clss}\|_{L^2(B^c(0, C_0)\cap e^\tau\mathbb T^3_{10})}^{2}+ \|\nabla^4\widetilde{U}\|_{L^2(B^c(0, C_0)\cap e^\tau\mathbb T^3_{10})}^{2} \Big) \nonumber \\ 
       &= \frac{\mathrm d}{\mathrm d\tau}\sum_{|\beta|=4}\int_{B^c(0, C_0)\cap e^\tau\mathbb T^3_{10}}\big(|\partial_{\beta}\widetilde{\Clss}|^2+|\partial_{\beta}\widetilde{U}|^2\big)
       \,\text{d}y \nonumber \\
       &=\sum_{|\beta|=4}\int_{B^c(0, C_0)\cap e^\tau\mathbb T^3_{10}}\frac{\mathrm d}{\mathrm d\tau}\big(|\partial_{\beta}\widetilde{\Clss}|^2+|\partial_{\beta}\widetilde{U}|^2\big)\,
       \text{d}y \nonumber\\
       &\quad +\sum_{|\beta|=4}5e^\tau \int_{\partial(e^\tau \mathbb{T}^3_{10})}\big(|\partial_{\beta}\widetilde{\Clss}|^2+|\partial_\beta \widetilde{U}|^2\big) \,\mathrm{d}S\nonumber \\
       &=2\sum_{|\beta|=4}\sum_{j=0}^{3} \bigg( \int_{B^c(0, C_0)\cap e^\tau\mathbb T^3_{10}}\partial_{\beta}\widetilde{U}\cdot B_{u,j}\text{d}y+\int_{B^c(0, C_0)\cap e^\tau\mathbb T^3_{10}} \partial_{\beta}\widetilde{\Clss}\,B_{\subclss,j}\,\text{d}y \bigg) \label{hignenergeest01-periodic}\\
       &\quad+2 \sum_{|\beta|=4}\bigg( \int_{B^c(0, C_0)\cap e^\tau\mathbb T^3_{10}}\partial_{\beta}\widetilde{U}\cdot \partial_{\beta}\mathcal{N}_u\text{d}y+\int_{B^c(0, C_0)\cap e^\tau\mathbb T^3_{10}}\partial_{\beta}\widetilde{\Clss}\, \partial_{\beta}\mathcal{N}_\subclss\,\text{d}y \bigg) \nonumber  \\
       &\quad+2 \sum_{|\beta|=4}\bigg( \int_{B^c(0, C_0)\cap e^\tau\mathbb T^3_{10}}\partial_{\beta}\widetilde{U}\cdot \partial_{\beta}\mathcal{E}_u \,\text{d}y
       +\int_{B^c(0, C_0)\cap e^\tau\mathbb T^3_{10}}\partial_{\beta}\widetilde{\Clss} \,\partial_{\beta}\mathcal{E}_\subclss \,\text{d}y \bigg) \nonumber \\
       &\quad+2  \sum_{|\beta|=4}\int_{B^c(0, C_0)\cap e^\tau\mathbb T^3_{10}}\partial_{\beta}\widetilde{U}\cdot \partial_{\beta}\mathcal{F}_{\rm dis} \,\text{d}y \nonumber \\
       &\quad +\underbrace{\sum_{|\beta|=4}5e^\tau \int_{\partial(e^\tau \mathbb{T}^3_{10})
       }\big(|\partial_{\beta}\widetilde{\Clss}|^2+|\partial_\beta \widetilde{U}|^2\big)\,\mathrm{d}S}_{I^c_b: \ \text{boundary term of the periodic domain}}\nonumber\\
       &:= \sum_{|\beta|=4}\Big(\sum_{j=0}^3 I_{j, \beta}+I_{N,\beta}+I_{E,\beta}+I_{F, \beta}\Big)+I^c_b. \nonumber
    \end{align}
    
For $I_{0, \beta}$, it follows from   integration by parts, the definition of $\widehat{X}$, \eqref{Para-ineq}, \eqref{radialrep}, \eqref{profile decay}, and Lemma~\ref{linfitnityinside} that
\begin{align*}
    I_{0, \beta}&=-2\int_{B^c(0,C_0)\cap e^\tau \mathbb{T}^3_{10}} \Big(\partial_{\beta}\widetilde{U}\cdot\big((y+\widehat{X}\overline{U})\cdot \nabla\big)\partial_{\beta}\widetilde{U} +(\partial_{\beta}\widetilde{U}\cdot  \nabla) \partial_{\beta}\widetilde{\Clss}(\alpha \widehat{X}\overline{\Clss})\Big)\,\text{d}y\\
        &\quad-2\int_{B^c(0,C_0)\cap e^\tau \mathbb{T}^3_{10}}\Big(\partial_{\beta}\widetilde{\Clss}(y+\widehat{X}\overline{U})\cdot \nabla \partial_{\beta}\widetilde{\Clss}+\partial_{\beta}\widetilde{\Clss}\,\dive (\partial_{\beta}\widetilde{U})\,\alpha \widehat{X}\overline{\Clss}\Big)\,\text{d}y\\
    & \leq \int_{B^c(0,C_0)\cap e^\tau \mathbb{T}^3_{10}}  \big( 3+|\dive (\widehat{X}\overline{U})| \big) \big(|\partial_{\beta}\widetilde{\Clss}|^2+|\partial_{\beta}\widetilde{U}|^2 \big)\, \text{d}y\\
        &\quad+2\int_{B^c(0,C_0)\cap e^\tau \mathbb{T}^3_{10}}| \nabla(\alpha\widehat{X}\overline{\Clss})||\partial_{\beta}\widetilde{\Clss}||\partial_{\beta}\widetilde{U}|
        \,\text{d}y\\
        &\quad+\int_{\partial B(0,C_0)} \big(C_0+|\widehat{X}\overline{\mathcal U}|\big) \big(|\partial_{\beta}\widetilde{\Clss}|^2+|\partial_{\beta}\widetilde{U}|^2 \big) \,\text{d}S\\
        &\quad +2\int_{\partial B(0,C_0)}|\alpha\widehat{X}\overline{\Clss}||\partial_{\beta}\widetilde{\Clss}||\partial_{\beta}\widetilde{U}|\text{d}S -5\int_{\partial(e^\tau \mathbb{T}^3_{10})}\big(|\partial_{\beta}\widetilde \Clss|^2+|\partial_{\beta}\widetilde U|^2\big)e^\tau\, \text{d}S\\
        & \leq \int_{\partial B(0,C_0)}\big(3+3|\nabla(\widehat{X}\overline{U})|+|\nabla(\alpha\widehat{X}\overline{\Clss})|\big)\big(|\partial_\beta \widetilde{\Clss}|^2+|\partial_\beta \widetilde{U}|^2\big)\, \dy + (\frac{\smallc_1}{\smallc_g})^{\frac{17}{10}}e^{-2\varepsilon(\tau-\tau_0)}\\
        &\quad -\underbrace{5e^\tau \int_{\partial(e^\tau \mathbb{T}^3_{10})}\big(|\partial_{\beta}\widetilde \Clss|^2+|\partial_{\beta}\widetilde U|^2\big)\,\text{d}S.}_{I^\beta_{bc}: \ \text{boundary term of the periodic domain}}
\end{align*}
It is observed that, after summing over all $|\beta|=4$, 
$$
I^c_b-\sum_{|\beta|=4}I^\beta_{bc}=0,
$$
which yields that 
the corresponding boundary terms arising from integration by parts cancel out.
Moreover, the estimates for all remaining terms are identical to those in the Cauchy problem.

\subsubsection{$K$-th order weighted energy estimate} 
Let $\beta = (\beta_1, \beta_2, \cdots, \beta_K)$ with $|\beta|=K$. 
For the periodic problem \eqref{selfsimilar eq-periodic}, we denote the $K$-th order weighted energy by
\begin{equation}\label{eq:weightedenergy-periodic}
   E_{K}(\tau)=\sum_{|\beta|=K}\int_{e^\tau\mathbb{T}^3_{10}}\big(|\partial_\beta \Clss|^2+ |\partial_\beta U|^2\big)\phi^K \,\dy.
\end{equation}
Applying $\partial_\beta$ to \eqref{selfsimilar eq-periodic}  and using the fact that 
$$\partial_\beta (y\cdot \nabla f) = K \partial_\beta f + y \cdot \nabla \partial_\beta f,$$
we have
\begin{equation}\label{eq:K-thderivative-periodic}
\begin{aligned}
     &(\partial_\tau +\Lambda-1+K)\partial_{\beta}\Clss +y\cdot \nabla \partial_{\beta}\Clss + \partial_{\beta}(U\cdot \nabla \Clss)+\alpha\partial_{\beta}(\Clss \,\dive U)=0,\\[1mm]
     &(\partial_\tau +\Lambda-1+K)\partial_{\beta}U +y\cdot \nabla \partial_{\beta}U + \partial_{\beta}(U\cdot \nabla U)+\alpha\partial_{\beta}(\Clss\nabla \Clss)= \partial_{\beta}\mathcal{F}_{\rm dis}.
     \end{aligned}
\end{equation}
Multiplying each equation of \eqref{eq:K-thderivative-periodic} by $\phi^K \partial_\beta \Clss$ and $\phi^K \partial_\beta U $ respectively, summing over all $|\beta|=K$, and integrating over $e^\tau\mathbb{T}^3_{10}$, we obtain
    \begin{align}\label{eq:EK-periodic}
        &\big(\frac{1}{2}\frac{\mathrm d}{\mathrm d\tau} + \Lambda-1+K\big)E_K  \nonumber\\
       &= -\sum_{|\beta|=K}\int_{e^\tau\mathbb{T}^3_{10}} \phi^K y\big(\partial_{\beta} \Clss\nabla \partial_{\beta} \Clss + \sum_{i=1}^3 \partial_{\beta} U_i \nabla \partial_{\beta}U_i \big) \,\text{d}y \nonumber\\
        &\quad -\sum_{|\beta|=K} \int_{e^\tau\mathbb{T}^3_{10}} \phi^K \big(\partial_{\beta} \Clss \, \partial_{\beta}(U\cdot \nabla \Clss)+ \sum_{i=1}^3 \partial_{\beta} U_i\, \partial_{\beta}(U \cdot \nabla U_i) \big) \,\text{d}y \nonumber\\
        &\quad - \sum_{|\beta|=K} \int_{e^\tau\mathbb{T}^3_{10}} \alpha\phi^K \big(\partial_{\beta} \Clss \,\partial_{\beta}(\Clss \,\dive\, U) + \sum_{i=1}^3 \partial_{\beta} U_i\, \partial_{\beta}(\Clss \partial_{y_i} \Clss) \big)\, \text{d}y \\
       &\quad +\sum_{|\beta|=K}C_{\rm{dis}}e^{-\delta_{\rm{dis}}\tau}\int_{e^\tau\mathbb{T}^3_{10}} \phi^K\partial_{\beta}U\, \partial_{\beta}\big( \Clss^{\frac{\delta-1}{\alpha}}L(U) +\frac{\delta}{\alpha}\Clss^{\frac{\delta-1}{\alpha}-1}\nabla \Clss \cdot \md(U) \big)\,\text{d}y \nonumber\\
       &\quad + \underbrace{\frac{5}{2}e^\tau \sum_{|\beta|=K}\int_{\partial(e^\tau\mathbb{T}^3_{10})}\big(|\partial_\beta \Clss|^2+|\partial_\beta U|^2\big)\phi^K \,\text{d}S}_{I^\phi_b:\ \text{boundary term  of the periodic domain}} \nonumber \\
       &=-I_1-I_2-I_3+C_{\rm{dis}}e^{-\delta_{\rm{dis}}\tau}I_4 +I^\phi_b. \nonumber
    \end{align}

For $I_1$, via  integration by parts, it follows that
\begin{equation}\label{eq:EstI1-periodic}
    \begin{aligned}
        -I_1  &=-\frac{1}{2}\sum_{|\beta|=K}\int_{e^\tau\mathbb{T}^3_{10}} \phi^K y \big( \nabla |\partial_\beta \Clss|^2 + \sum_{i=1}^{3} \nabla|\partial_\beta U_i|^2\big)\,\text{d}y\\
        & =\frac{1}{2}\sum_{|\beta|=K}\int_{e^\tau\mathbb{T}^3_{10}} \frac{\dive(\phi^K y )}{\phi^K}\big( |\partial_\beta \Clss|^2 + |\partial_\beta U|^2\big)\phi^K\,\text{d}y\\
        & \quad \underbrace{- \frac{5}{2}e^\tau\sum_{|\beta|=K} \int_{\partial(e^\tau\mathbb{T}^3_{10})} \big( |\partial_\beta \Clss|^2 + |\partial_\beta U|^2\big)\phi^K\, \text{d}S}_{I^1_{b\phi}:\ \text{boundary term of the periodic domain}}.
    \end{aligned}
\end{equation}
It is observed that
$$
I^\phi_b-I^1_{b\phi}=0,
$$
which yields that 
the corresponding  boundary terms  arising from  integration by parts cancel out.
In all remaining integrations by parts, no boundary terms appear due to the periodic boundary conditions.

\subsection{Development of Implosions of Solutions of the Original Periodic Problem}
Based on the global-in-time stability for problem \eqref{selfsimilar eq-periodic} obtained in Theorem \ref{globalexist-periodic}, the corresponding development of implosions of solutions of the periodic problem \eqref{periodicproblem}
 follows exactly the same argument used  in the Cauchy problem, which can be found in \S \ref{Section8}.

\bigskip
\appendix
\section{Basic Lemmas}\label{appendix A}
This appendix lists some useful lemmas that have been used frequently in the previous sections. 
The first  lemma gives  the well-known Fatou's lemma ({\it cf}. \cite{realrudin})   
\begin{Lemma}\label{Fatou}
Given a measure space $(V,\mathcal{H},\varkappa)$ and a set $\beth \in \mathcal{H}$,
let  $\{f_n\}$ be a sequence of $(\mathcal{H} , \mathcal{B}_{\mathbb{R}_{\geq 0}})$--measurable 
non-negative functions $f_n(\upsilon): \beth\rightarrow [0,\infty]$. Define a function $f(\upsilon): \beth\rightarrow [0,\infty]$ by setting
$$
f(\upsilon)= \liminf_{n\rightarrow \infty} f_n(\upsilon)\qquad \mbox{for every $\upsilon\in \beth$}.
$$
Then $f$ is $(\mathcal{H},  \mathcal{B}_{\mathbb{R}_{\geq 0}})$--measurable and   
$$
\int_\beth f(\upsilon)\,\text{\rm d}\varkappa 
\leq \liminf_{n\rightarrow \infty} \int_\beth f_n(\upsilon)\,\text{\rm d}\varkappa.
$$
\end{Lemma}

The second lemma shows the weak compactness in weighted Sobolev spaces for a suitably chosen weight function.
\begin{Lemma}\label{lem:weighted-weak-limit}
    Let $k\in \mathbb{N}$, let $\phi$ be a weighted function defined in \eqref{weightdefi}, 
    and let $\{f_n(y)\}_{n=1}^\infty$ be a sequence with $f_n(y):\mathbb{R}^3\to \mathbb{R}$. 
    Suppose that, for some constant $\bar C_1,\bar C_2>1$ and for all $n\ge 1$,
    \begin{equation}\label{HK-bound}
        \| f_n \|^2_{H^K} \leq \bar C_1,
    \end{equation}
    and for the multi-index $\beta=(\beta_1,\beta_2,\cdots,\beta_K)$ with 
$\beta_i\in\{1,2,3\}$ for $i=1,\cdots,K$,    \begin{equation}\label{weighted-energy-bound}
        \sum_{|\beta|=K}\int |\partial_\beta f_n|^2\phi^K \,\dy \le \bar C_2,
    \end{equation}
    where $\bar C_1, \bar C_2$, and $K$ are independent of $n$ and $\phi$.
   Then there exist a subsequence {\rm (}still denoted by $n${\rm )} 
   and functions $f \in H^K(\mathbb{R}^3)$ and $F_\beta \in L^2(\mathbb{R}^3)$ such that
   \begin{equation}\label{weak-conv}
   \begin{aligned}
       f_n \to f \quad &\textit{weakly in $H^K(\mathbb{R}^3)$},\\
       \partial_\beta f_n\phi^{\frac{K}{2}} \to F_{\beta} \quad &\textit{weakly in $L^2(\mathbb{R}^3)$}.
    \end{aligned}
   \end{equation}
Moreover,  for every multi-index $\beta=(\beta_1,\beta_2,\cdots,\beta_K)$ with 
$\beta_i\in\{1,2,3\}$ for $i=1,\cdots,K$, the weak limits in \eqref{weak-conv} are uniquely identified by
\begin{equation}
    F_\beta = \partial_{\beta}f \phi^{\frac{K}{2}}\quad \text{ in $L^2(\mathbb{R}^3)$}.
\end{equation}
\end{Lemma}
\begin{proof}
We divide the proof into two steps.

\smallskip
1. {\it Local weak limits and identification.}
Fix $R>0$ and a multi-index $\beta$ with $|\beta|=K$.
It follows from the uniform bounds in \eqref{HK-bound} that
$$
\|f_n\|_{H^K(B(0,R))}\le \|f_n\|_{H^K(\mathbb{R}^3)}\le \bar C ^{\frac{1}{2}}_1.
$$
Then there exist both a subsequence (still denoted by $n$) and a function $f^{(R)}\in H^K(B(0, R))$ such
that
$$
f_n \to f^{(R)} \qquad \text{weakly in $H^K(B(0, R))$}.
$$
In particular, for every multi-index $\beta$ with $|\beta|=K$, we also have
\begin{equation}\label{cvginBR1}
\partial_\beta f_n \to \partial_\beta f^{(R)} \qquad \text{weakly in $L^2(B(0, R))$}.
\end{equation}
Next, the weighted energy bound \eqref{weighted-energy-bound} implies 
$$
\|\partial_\beta f_n \phi^{\frac{K}{2}}\|_{L^2(B(0, R))}\leq \|\partial_\beta f_n \phi^{\frac{K}{2}}\|_{L^2(\mathbb R^3)} \leq \bar C_2^{\frac{1}{2}}.
$$
After extracting a further subsequence (still denoted by $n$), there exists $F_\beta^{(R)} \in L^2(B(0, R))$ such that
\begin{equation}\label{cvginBR2}
  \partial_\beta f_n\phi^{\frac{K}{2}} \to F_{\beta}^{(R)} \qquad \text{weakly in $L^2(B(0,R))$}.  
\end{equation}
By the definition of the weight function $\phi$ in \eqref{weightdefi}, we see that
$\phi^{\frac K2}\in L^\infty(B(0, R))$. 
Since the bounded linear operators preserve weak convergence, 
the weak convergence \eqref{cvginBR1} implies
$$
\partial_\beta f_n \phi^{\frac{K}{2}} \to \partial_\beta f^{(R)} \phi^{\frac{K}{2}}  
\qquad \text{weakly in $L^2(B(0, R))$}.
$$
By the uniqueness of weak limits in $L^2(B(0, R))$, we conclude  
\begin{equation}\label{local-identification}
F_\beta^{(R)}=\partial_\beta f^{(R)}\phi^{\frac K2} \qquad \text{ in $L^2(B(0, R))$}.
\end{equation}

\smallskip
2. {\it Diagonal subsequence and consistency of local limits.}
Perform the above extraction for $R=1,2,3,\cdots$ and use a diagonal argument to obtain a subsequence $\{f_n\}$ such that, for every $R\in \mathbb{N}$,
\begin{align*}
     f_n \to  f^{(R)} \qquad   &\text{weakly in $H^K(B(0, R))$},\\
    \partial_\beta f_n \phi^{\frac{K}{2}} \to F^{(R)}_\beta  \qquad &\text{weakly in $L^2(B(0, R))$}.
\end{align*}
For $1 \le R < R^*$, the restriction map from $H^K(B(0, R^*))$ to $H^K(B(0, R))$ is bounded and linear, hence preserves weak convergence. Therefore, it follows from the uniqueness of weak limits that
$$ 
f^{(R^*)}|_{B(0, R)}= f^{(R)} \qquad \text{{\it a.e.} in $B(0, R)$}.
$$
This compatibility allows us to define a global function $f\in H^K(\mathbb{R}^3)$ by setting $f = f^{(R)}$ in $B(0, R)$. Combining this with the local identification obtained in Step~1, we conclude that, for every $R$,
$$
F_\beta^{(R)}=\partial_\beta f\,\phi^{\frac K2}
\qquad \text{in $L^2(B(0,R))$}.
$$
Similarly, we also define a global function $F_\beta\in L^2(\mathbb{R}^3)$ by setting $F_\beta = F^{(R)}_\beta$ in $B(0, R)$. Since $R>0$ is arbitrary, we derive 
$$F_\beta = \partial_\beta f \phi^{\frac{K}{2}}\qquad \text{ in $L^2(\mathbb{R}^3)$}.
$$
\end{proof}

The next lemma provides a weighted Gagliardo-Nirenberg interpolation inequality, 
which plays an important role in our analysis.

\begin{Lemma}[\cite{shijia}] \label{lemma:GN_general} 
Let $f(y):\mathbb R^3 \to \mathbb R^3$, $I_{-1} = [0, 1]$,  and $I_j = [2^{j}, 2^{j+1}]$, 
and let $\varphi (y), \psi(y) > 0$ be two weight functions such that there exist values $\varphi_j, \psi_j$, 
and a constant $C^*$ satisfying $\varphi (y) \in [\varphi_j/C^*, C^* \varphi_j]$ 
and $\psi(y) \in [\psi_j/C^*, C^*\psi_j]$ for all $y$ with $|y| \in I_j$. 
Moreover, assume that $\varphi (y)$ and $\psi(y)$ are approximately $1$ at the origin, that is, 
$$\varphi_{-1} = \psi_{-1} = \varphi_0 = \psi_0 = 1.$$
For an integer $1 \leq i \leq l$, assume that parameters $(p, q, \bar{r}, \theta)$ satisfy
\begin{equation} \label{eq:GN_conditions}
    \frac{1}{\bar{r}} = \frac{i}{3} + \theta \Big( \frac{1}{q} - \frac{l}{3} \Big) + \frac{1-\theta}{p}, \qquad  \frac{i}{l} \leq \theta < 1.
\end{equation}

When $\bar{r} = \infty$, then
\begin{equation} \label{eq:GNresultinfty}
| \nabla^i f | \leq C_3\big( \| \psi^l f \|_{L^p}^{1-\theta} \| \varphi^l \nabla^l f \|_{L^q}^{\theta} \psi^{-l(1-\theta)} \varphi^{-l\theta}  + \| \psi^l f \|_{L^p} \, \langle y \rangle^{3\theta (\frac 1q-\frac 1p)-l\theta} \psi^{-l}\big).
\end{equation}
Moreover, if $p = \infty$, $q = 2$, and $\psi = 1$, then
\begin{equation} \label{eq:GNresultinfty_simplified}
    \left| \nabla^i f \right| \leq C_4\big( \| f \|_{L^\infty}^{1-\frac{i}{l-3/2}} \| \varphi^l \nabla^l f \|_{L^2}^{\frac{i}{l-3/2}}  \varphi^{-l\frac{i}{l-3/2}}  + \left\| f \right\|_{L^\infty} \,
    \langle y \rangle^{-i}\big).
\end{equation}
Here  $C_3$ and $C_4$ are both positive constants that depend on  $(p, q, i, l, \theta, C^*)$, 
but are independent of $(f, \psi, \varphi)$. 

When $\bar{r} \in [1,\infty)$, given any $\bar{\varepsilon} > 0$, under the extra assumption{\rm:}
\begin{equation}\label{eq:GN_extracond}
    \left( \frac{\varphi(y)}{\langle y \rangle \psi(y) }\right)^{l\theta} \langle y \rangle^{3\theta (\frac1q-\frac1p)} \leq C_5,
\end{equation}
the following weighted Gagliardo-Nirenberg inequality holds{\,\rm:}
\begin{equation} \label{eq:GNresult}
    \| \langle y \rangle^{-\bar{\varepsilon}} \psi^{l(1-\theta)} \varphi^{l\theta} \nabla^i f \|_{L^{\bar{r}}} \leq C_6\big(\| \psi^l f \|_{L^p}^{1-\theta} \| \varphi^l \nabla^l f \|_{L^q}^{\theta} + \| \psi^l f \|_{L^p}\big),
\end{equation}
where  $C_5$ is a positive constant depending on $(p, q, l, \theta, C^*)$, while $C_6$ depends on parameters 
$(p, q, i, l, \theta, \bar{r}, \bar{\varepsilon}, C^*)$,  but is independent of $(f,  \psi, \varphi)$. 
\end{Lemma}

Moreover, in the exterior region $B^c(0,C_0)$ for some constant $C_0\ge 1$, the following interpolation inequality holds:

\begin{Lemma}[\cite{shijia}]  \label{lemma:GN_generalnoweightwholespace} 
Let $f(y):\mathbb R^3 \to \mathbb R^3$ be in $H^{l}(\mathbb R^3 )$, and 
let $C_0\geq 1$ be a constant independent of $f$. Assume that parameters $(p, q, \bar{r}, \theta, l)$ satisfy
\begin{equation} \label{eq:GN_conditions02torus}
    \frac{1}{\bar{r}} = \frac{i}{3} + \theta \Big( \frac{1}{q} - \frac{l}{3} \Big) + \frac{1-\theta}{p},  \qquad \frac{i}{l} \leq \theta < 1,  \qquad \frac{l}{3}\geq \frac{1}{q}-\frac{1}{p}.
\end{equation}
Then the Gagliardo-Nirenberg inequalities hold
\begin{equation} \label{eq:GNresultnoweightwholespace}
    \|   \nabla^i f \|_{L^{\bar{r}}(B^c(0,C_0))} \leq C_7\big( \|  f \|_{L^p(B^c(0,C_0))}^{1-\theta} \|  \nabla^l f \|_{L^q(B^c(0,C_0))}^{\theta} + \| f \|_{L^p(B^c(0,C_0))}\big),\end{equation}
where $C_7$ is a positive constant that depends on $(p, q, i, l, \theta, \bar{r})$, but is independent of $f$ and $C_0$.
\end{Lemma}

We next state the corresponding weighted Gagliardo-Nirenberg interpolation inequalities on the periodic domain $\mathbb T^3_{2\bar L}$ of period $2\bar L$.

\begin{Lemma}[\cite{shijia}] \label{lemma:GN_general_period} 
Let $f(y):\mathbb T^3_{2\bar L} \to \mathbb R^3$ with $\bar L\ge 1$, $I_{-1} = [0, 1]$,  $I_j = [2^{j}, 2^{j+1}]$, 
and let $\varphi (y), \psi(y) > 0$ be two weight functions such that there exist values $\varphi_j, \psi_j$, and a constant $C^*$ 
satisfying $\varphi (y) \in [\varphi_j/C^*, C^* \varphi_j]$ and $\psi(y) \in [\psi_j/C^*, C^*\psi_j]$ for all $y$ with $|y| \in I_j$. 
Moreover, assume that $\varphi (y)$ and $\psi(y)$ are approximately $1$ at the origin, that is, 
$$\varphi_{-1} = \psi_{-1} = \varphi_0 = \psi_0 = 1.$$
For an integer $1 \leq i \leq l$, assume that parameters $(p, q, \bar{r}, \theta)$ satisfy \eqref{eq:GN_conditions}.

When $\bar{r} = \infty$, then
\begin{equation} \label{eq:GNresultinfty-periodic}
    \left| \nabla^i f \right| 
    \leq C_3\Big( \| \psi^l f \|_{L^p(\mathbb T^3_{2\bar L})}^{1-\theta} \| \varphi^l \nabla^l f \|_{L^q(\mathbb T^3_{2\bar L})}^{\theta} 
    \psi^{-l(1-\theta)} \varphi^{-l\theta}  
    + \| \psi^l f\|_{L^p(\mathbb T^3_{2\bar L})} \,  \langle y \rangle^{3\theta (\frac 1q-\frac 1p)-l\theta} \psi^{-l}\Big).
\end{equation}
Moreover, if $p = \infty$, $q = 2$, and $\psi = 1$, then
\begin{equation} \label{eq:GNresultinfty_simplified-periodic}
    \left| \nabla^i f \right| \leq C_4\Big( \| f \|_{L^\infty(\mathbb T^3_{2\bar L})}^{1-\frac{i}{l-3/2}} \| \varphi^l \nabla^l f \|_{L^2(\mathbb T^3_{2\bar L})}^{\frac{i}{l-3/2}}  \varphi^{-l\frac{i}{l-3/2}}  + \left\| f \right\|_{L^\infty(\mathbb T^3_{2\bar L})} \cdot  \langle y \rangle^{-i}\Big).
\end{equation}
Here  $C_3$ and $C_4$ are both positive constants that may depend on  $(p, q, i, l, \theta, C^*)$, 
but are independent of $(f, \bar L, \psi, \varphi)$. 

When $\bar{r} \in [1,\infty)$, given any $\bar{\varepsilon} > 0$, and under the extra assumption \eqref{eq:GN_extracond},
the following weighted Gagliardo-Nirenberg inequality holds{\,\rm:}
\begin{equation} \label{eq:GNresult-periodic}
    \| \langle y \rangle^{-\bar{\varepsilon}} \psi^{l(1-\theta)} \varphi^{l\theta} \nabla^i f \|_{L^{\bar{r}}(\mathbb T^3_{2\bar L})} 
    \leq C_6\Big(\| \psi^l f \|_{L^p(\mathbb T^3_{2\bar L})}^{1-\theta} \| \varphi^l \nabla^l f \|_{L^q(\mathbb T^3_{2 \bar L})}^{\theta} + \| \psi^l f \|_{L^p(\mathbb T^3_{2\bar L})}\Big),
\end{equation}
where  $C_5$ is a positive constant depending on $(p, q, l, \theta, C^*)$, while $C_6$ depends on 
parameters $(p, q, i, l, \theta, C^*, \bar{r}, \bar{\varepsilon})$,  but is independent of  $(f, \bar L, \psi, \varphi)$. 
\end{Lemma}

Moreover, in the exterior region $B^c(0,C_0)\cap \mathbb T^3_{2\bar L}$ for some constant $C_0\ge 1$, we have the following Gagliardo-Nirenberg interpolation inequality without weights:

\begin{Lemma}[\cite{shijia}]  \label{lemma:GN_generalnoweightwholespace_period} 
Let $f(y): \mathbb T^3_{2\bar L} \to \mathbb R^3$ be in $H^{l}(\mathbb T^3_{2\bar L})$ with $\bar{L}\ge 10C_0$, where $C_0\geq 1$ is 
a constant independent of $f$. 
Then, if parameters $(p, q, \bar{r}, \theta, l)$ satisfy \eqref{eq:GN_conditions02torus},
the Gagliardo-Nirenberg inequalities hold{\rm :}
\begin{equation} \label{eq:GNresultnoweight-periodic}
\begin{split}
    \|   \nabla^i f \|_{L^{\bar{r}}(B^c(0,C_0)\cap \mathbb T^3_{2\bar L})} & \leq C_7 \|  f \|_{L^p(B^c(0,C_0)\cap \mathbb T^3_{2\bar L})}^{1-\theta} \|  \nabla^l f \|_{L^q(B^c(0,C_0)\cap \mathbb T^3_{2\bar L})}^{\theta}\\
    &\quad+ C_7\| f \|_{L^p(B^c(0,C_0)\cap \mathbb T^3_{2\bar L})},
\end{split}
\end{equation}
where  $C_7$ is a positive constant that depends only on $(p, q, i, l, \theta, \bar{r})$, 
but is  independent of $(f, \bar L,  C_0)$.
\end{Lemma}

\section{Self-Similar Solutions to the Steady Compressible Euler Equations}\label{appendix B}
This appendix is devoted to showing some properties of the smooth and spherically symmetric self-similar solution to the steady Euler equations, which have been verified in \cites{buckmaster, merle1, shijia, shao}. The spherically symmetric self-similar solution 
$(\overline{\Clss}, \overline{U})$ takes the form:
\[
\overline{\Clss}(y)=\overline\Clss(r),\quad 
\overline{U}(y)=\overline{\mathcal U}(r)\frac{y}{r} \qquad\,\,\mbox{for $r=|y|$}.
\]
Thus, the steady Euler equations \eqref{Profile} reduce to the following
system of ordinary differential equations
$(\overline \Clss,\overline{\mathcal U})$:
\begin{equation}\label{eq:sym_profiles}
\begin{split}
&(\Lambda-1) \overline{\Clss} + (r + \overline{\mathcal U}) \partial_r\overline{\Clss}
+ \alpha \overline{\Clss} \big(\partial_r\overline{\mathcal U}+\frac{2}{r}\overline{\mathcal U}\big) = 0,\\
&(\Lambda-1) \overline{\mathcal U} + (r + \overline{\mathcal U})\partial_r\overline{\mathcal U} + \alpha \overline{\Clss} \,\partial_r\overline{\Clss} = 0.
\end{split}
\end{equation}

Following \cite{buckmaster} and \cite{merle1}, one introduces the Emden
transformation:
\begin{equation}
        \overline \Clss(r) = r\,\widehat{\Clss}(\xi),\quad
        \overline{\mathcal U}(r) =r\,\widehat{U}(\xi)\qquad\,\, 
        \mbox{ for $\xi = \log r$}, 
\end{equation}
which maps \eqref{eq:sym_profiles} onto the following \emph{autonomous} system:
\begin{equation}\label{atuosys}  
    \begin{cases}
    \displaystyle
\Lambda \widehat{\Clss}+(1+\widehat{U})\frac{{\rm d}\widehat{\Clss}}{{\rm d} \xi} +\widehat{\Clss}\widehat{U}+\alpha \widehat{\Clss}\big(\frac{d\widehat{U}}{d \xi} +3\widehat{U}\big)=0, \\[12pt]
\displaystyle
        \Lambda \widehat{U}+(1+\widehat{U})\frac{{\rm d}\widehat{U}}{{\rm d} \xi} +\widehat{U}^2+\alpha \widehat{\Clss}\big(\frac{{\rm d}\widehat{\Clss}}{{\rm d} \xi} +\widehat{\Clss}\big)=0,
    \end{cases} \Longleftrightarrow \quad \begin{cases}
        \displaystyle\frac{{\rm d}\widehat{\Clss}}{{\rm d} \xi} = \frac{N_\Clss(\widehat{\Clss}, \widehat{U})}{D(\widehat{\Clss}, \widehat{U})},\\[12pt]
        \displaystyle
        \frac{{\rm d}\widehat{U}}{{\rm d} \xi} = \frac{N_U(\widehat{\Clss}, \widehat{U})}{D(\widehat{\Clss}, \widehat{U})},    \end{cases}
    \end{equation}
where 
\begin{align*}
   N_{\Clss}(\widehat{\Clss}, \widehat{U}) &=-(1+\widehat{U})\widehat{\Clss}(\Lambda+(1+3\alpha)\widehat{U})+\alpha \widehat{\Clss}(\Lambda \widehat{U}+\widehat{U}^2+\alpha\widehat{\Clss}^2),\\
   N_{U}(\widehat{\Clss}, \widehat{U}) &=-(1+\widehat{U})(\Lambda \widehat{U}+\widehat{U}^2+\alpha\widehat{\Clss}^2)+\alpha \widehat{\Clss}^2(\Lambda+(1+3\alpha)\widehat{U}),\\
   D(\widehat{\Clss}, \widehat{U})&= (1+\widehat{U})^2-\alpha^2 \widehat{\Clss}^2.
\end{align*}

The compactified phase portrait in the $(\widehat{\Clss},\widehat U)$--plane provides a
geometric description of self-similar profiles through the Emden transformation. In particular, a smooth orbit of the autonomous system
\eqref{atuosys} connecting $P_0$ and $P_\infty$ corresponds to a smooth spherically symmetric self-similar profile that is regular at the origin and decays at infinity, which is stated in the following two lemmas.

  \begin{figure}
      \centering
      \begin{tikzpicture}
      \node[anchor=south west, inner sep=0] (image) 
      at (0,0)
     {\includegraphics[height=0.4\textheight]{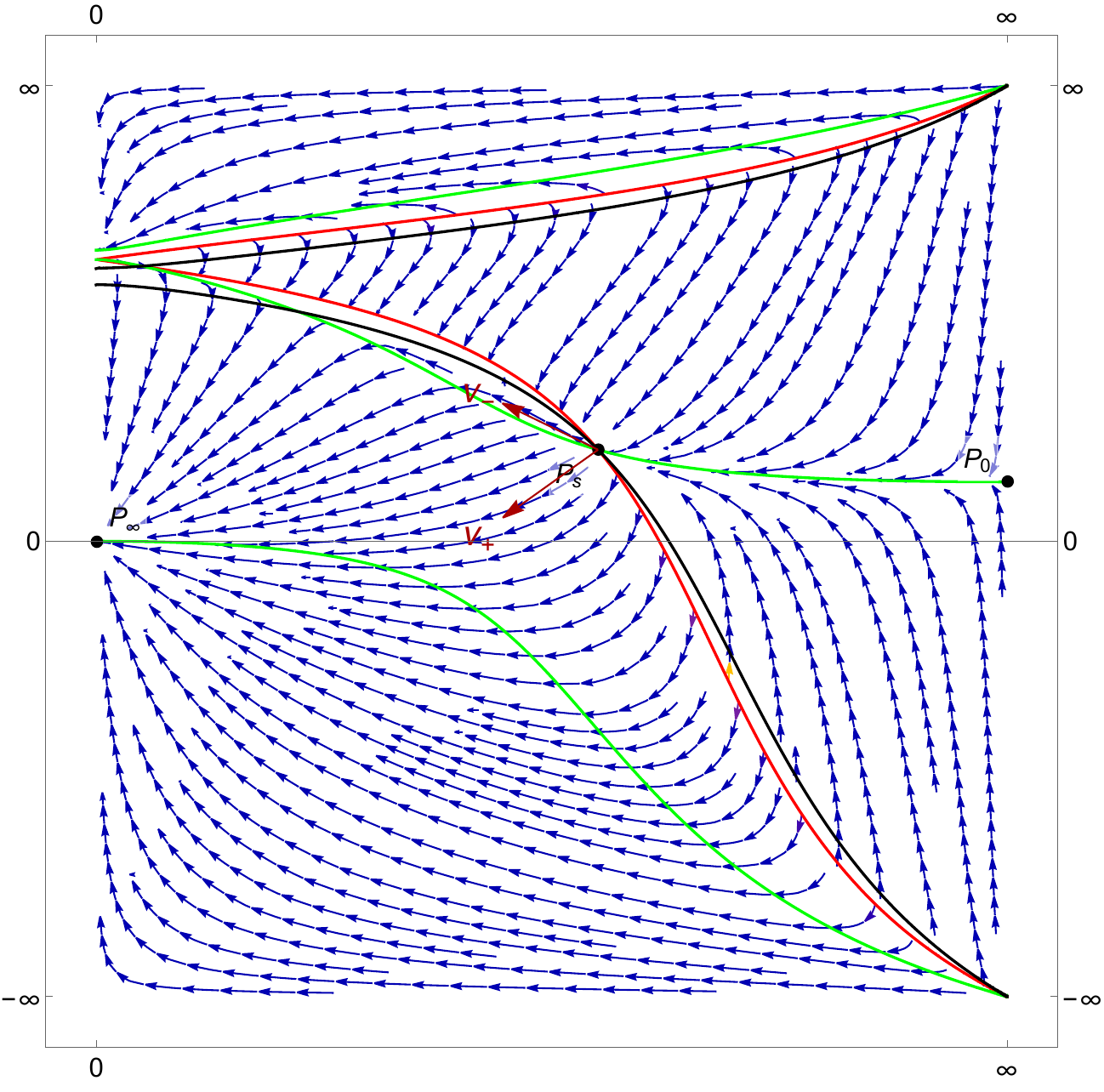}};
     \begin{scope}[x={(image.south east)},y={(image.north west)}]
      \node at (0.5,-0.03) {$\widehat{Q}$};
      \node at (-0.03,0.5) {$\widehat U$};
    \end{scope}
     \end{tikzpicture}
    \caption{Phase portrait trajectory from $P_\infty$ to $P_0$.}
    \label{fig2}
  \end{figure}

\begin{Lemma}[Existence of Smooth Solutions to \eqref{atuosys}]\label{thm:existence_profiles}
Let $\Lambda^*(\gamma)$ be defined as in \eqref{def:Lambda*}.
\begin{enumerate}
\item[\rm (i)]
There exists a {\rm(}possibly empty{\rm)} exceptional countable set $\mathcal{J}:=\{\gamma_k\}_{k\in \mathbb{N}}$ whose accumulation points can only be at $\{1, \frac{5}{3}, \infty\}$ and $\{\frac{7}{5}, \frac{5}{3}\}\cap \mathcal{J}=\varnothing$. Then, for each $\gamma \in (1, \infty)\setminus \mathcal{J}$, there exists a discrete sequence of scaling parameters 
$\{\Lambda_n\}_{n=1}^{\infty}$ satisfying
\begin{equation*}
    1 < \Lambda_n < \Lambda^*(\gamma), \qquad \lim_{n\to \infty}\Lambda_n = \Lambda^*(\gamma),
\end{equation*}
such that the autonomous system \eqref{atuosys} with $\Lambda=\Lambda_n$, $n=1,2, \cdots$, admits a smooth solution connecting $P_0$ to $P_\infty=(0,0)${\rm ;} 
see \cite[Theorem 1.3]{merle1}, \cite[Theorem 1.2]{buckmaster}, and \cite[Theorem 1.1]{shao}.
\item[\rm (ii)] Let $\gamma\in(1,\infty)$. 
There exists an interval $I_\gamma\subsetneq (1, \Lambda^*(\gamma))$ and a parameter $\Lambda \in I_\gamma$ such that the autonomous system \eqref{atuosys} admits a smooth solution connecting $P_0$ to $P_\infty=(0,0)${\rm ;} 
see \cite[Theorem 1.1]{buckmaster}.
\end{enumerate}
In {\rm Figure~\ref{fig2}}, $P_0$ is a point in the compactified phase portrait, with finite value of $\widehat U$ and $\widehat{\Clss}=\infty$, representing
the regular behavior of self-similar profiles at origin $r=0$,
whereas $P_\infty=(0,0)$ corresponds to the values of profiles at $r=\infty$. Point $P_s$ is a regular singular point, 
and there exist two smooth integral curves crossing $P_s${\rm :} one is tangent to the direction $\nu_-$ and the other one is tangent to $\nu_+$. The curve tangent to $\nu_+$ corresponds to a shock-type self-similar solution to the Euler equations, whereas the curve tangent to $\nu_-$ yields a smooth self-similar solution to the Euler equations.
\end{Lemma}

While Lemma~\ref{thm:existence_profiles} establishes the existence of a smooth orbit in the phase plane, it remains to verify that this orbit corresponds to a smooth profile in the spatial variables. By projecting this orbit back into the physical coordinates via the inverse Emden transformation, we obtain a globally defined, spherically symmetric smooth profile $(\overline{\Clss}, \overline{U})$.

\begin{Lemma}[Existence of Smooth Self-Similar Euler Profiles]
\label{lem:existofselfsimilar}
Suppose that parameters $(\gamma, \Lambda)$ in the autonomous system \eqref{atuosys} are 
chosen according to either case~{\rm (i)} or case~{\rm (ii)} in {\rm Lemma~\ref{thm:existence_profiles}.}
In either case, the system admits a smooth solution $(\widehat{\Clss},\widehat U)$ connecting $P_0$ to
$P_\infty=(0,0)$ in the $(\widehat{\Clss},\widehat U)$--phase plane.
Define
\[
\overline{\Clss}(y)=|y|\,\widehat{\Clss}(\xi),\quad
\overline U(y)=\widehat U(\xi)\,y \qquad\,\,\mbox{for $\xi=\log |y|$}.
\]
Then $(\overline{\Clss},\overline U)$ is a smooth, spherically symmetric solution
to the steady Euler system \eqref{Profile}. Moreover, the asymptotic behavior of
$(\widehat{\Clss},\widehat U)$ at $P_0$ implies 
\[
 \overline{\Clss}(0)=\Clss_*>0,\qquad \overline U(0)=0,
\]
and that profile $(\overline{\Clss},\overline U)$ is smooth in a
neighborhood of the origin. On the other hand, the convergence of the orbit to
$P_\infty$ yields
\[
(\overline{\Clss},\overline U)(y)\to(0,0)\qquad \text{as } |y|\to\infty,
\]
corresponding to the decay of the profile at spatial infinity.
\end{Lemma}

We now summarize the main
properties of the spherically symmetric self-similar solution $(\overline \Clss,\overline{U})$.

\begin{Lemma}[\cites{shijia, shao}]\label{prop:profiles-R}
Let $(\overline \Clss, \overline U)=(\overline{\Clss}(r), \overline{\mathcal{U}}(r)\frac{y}{r})$ be the smooth, spherically symmetric self-similar profile obtained in {\rm Lemma~\ref{lem:existofselfsimilar}}. Then the following properties hold{\,\rm :}
\begin{align}
&\overline \Clss >0, \label{barS lower bd}\\
& \overline \Clss \geq \widehat{C}^{-1}\langle r \rangle^{-\Lambda+1}, \qquad |\nabla^j  \overline \Clss|+|\nabla^j \overline U| \leq \widehat{C}(j) \langle r \rangle^{-(\Lambda-1)-j} \,\,\mbox{for any $j \geq 0$},  \label{profile decay}\\
& 1+ \partial_r\overline{\mathcal U}- \alpha | \partial_r\overline \Clss| > \tilde{\eta}, \label{radial repulsivity}\\
&1+\frac{\overline{\mathcal U}}{r}-\alpha |\partial_r\overline \Clss| > \tilde{\eta}, \label{angular repulsivity}
\end{align}
and, for $r>1$,
\begin{equation}
    r + \overline{\mathcal U}- \alpha\overline{\Clss
    }> (r-1) \tilde{\eta}.
\end{equation}
Here $r = |y|$ is the radial variable, $\alpha=\frac{\gamma-1}{2}$, $j$ is a natural number, and 
$\tilde{\eta}>0$ is a sufficiently small constant that depends on $(\Lambda, \gamma)$. 
Moreover, $\widehat{C}>1$ denotes a constant depending on $\Lambda$ and $(\overline{\Clss}, \overline{U})$, and $\widehat{C}(j)>1$ denotes a constant depending on $j$ and $\widehat{C}$.
\end{Lemma}

\section{Properties of the Truncated Linear Operator}
\label{appendix C}
In this appendix, we collect several properties of the truncated linear operator
$\mathcal L = (\mathcal L_\subclss, \mathcal L_{u})$ defined in \eqref{eq:cutoffL}.
In particular, we describe the decomposition of the phase space $X$ into the stable and unstable subspaces, 
and summarize the dissipative behavior and regularity features of $\mathcal L$ that have been used 
in our analysis. We first recall a decomposition result for $\mathcal L$.

\begin{Lemma}[\cite{shijia}, Lemma 2.9] \label{lemma:abstract_result} 
Let $\smallc_g \in (0,1)$, and let $(m,J)$ satisfy
\[
m \ge \widetilde{C}\tilde{\eta}^{-1},\qquad J \ge 2m,
\]
for some sufficiently large constant $\widetilde{C}>1$ independent of $(m, J, \tilde{\eta})$, 
where  $\tilde{\eta}$ arises from the properties of $(\overline \Clss, \overline U)$ in {\rm Lemma \ref{prop:profiles-R}}.
Then  
\begin{enumerate}
\item[\rm (i)]\label{item:spectrum} 
If $\sigma (\mathcal L)$ is the spectrum of $\mathcal L$, then the set 
$$ 
\Sigma  = \sigma (\mathcal L) \cap \{ \lambda \in \mathbb{C} : \, \Re (\lambda) > -\frac{\smallc_g}{ 2}  \}
$$
is finite and formed only by the eigenvalues of $\mathcal L$,
where $\Re(\lambda)$ denotes the real part of $\lambda$. 
Moreover, each $\lambda \in \Sigma$ has finite algebraic multiplicity. 
That is, if $\mu_{\lambda}$ is the smallest natural number such that 
$$
{\rm ker} (\mathcal L-\lambda  )^{\mu_{\lambda}} = {\rm ker} (\mathcal L- \lambda  )^{\mu_\lambda + 1},
$$
then the vector space
\begin{equation} \label{eq:spaceV}
V_{\rm{uns}} = \oplus_{\lambda \in \Sigma} {\rm ker} (\mathcal L - \lambda )^{\mu_{\lambda}}\,
\end{equation}
is finite-dimensional. 

\smallskip
\item[\rm (ii)] \label{item:invariance} 
Denote $\mathcal L^\ast$ as the adjoint of $\mathcal L$ and 
$$
\Sigma^\ast = \sigma (\mathcal L^\ast) \cap \big\{ \lambda \in \mathbb{C} :\, \Re (\lambda) >- \frac{\smallc_g}{ 2}\big\},
$$
and define
\begin{equation} \label{eq:spaceVast}
V_{\rm{sta}} = \big(\oplus_{\lambda \in \Sigma^\ast} {\rm ker} (\mathcal L^\ast - \lambda )^{\mu_{\lambda}^\ast}\big)^{\perp},
\end{equation}
where $\mu_\lambda^{\ast}$ is the smallest natural number such that
$$
{\rm ker} (\mathcal L^\ast-\lambda  )^{\mu_{\lambda}^\ast} = {\rm ker} (\mathcal L^\ast- \lambda  )^{\mu_\lambda^\ast + 1}.
$$
Then both $V_{\rm{uns}}$ and $V_{\rm{sta}}$ are invariant under $\mathcal L$, 
$\Sigma^\ast = \overline{\Sigma}$ {\rm{(}}$\lambda \in \Sigma$ if and only if its complex conjugate
$\overline{\lambda} \in \Sigma^\ast${\rm{)}},  $\mu_\lambda = \mu_{\overline{\lambda}}^\ast$, and  the decomposition holds{\rm :}
$$
X = V_{\rm{uns}} \oplus V_{\rm{sta}}.
$$

\item[\rm (iii)] \label{item:stability_outwards} 
The linear transformation $\mathcal L|_{V_{\rm{uns}}}: V_{\rm{uns}} \rightarrow V_{\rm{uns}}$ 
obtained by restricting $\mathcal L$ to the finite-dimensional space $V_{\rm{uns}}$ has all its eigenvalues 
with real part larger than $-\frac{\smallc_g}{2}$. 
Since $V_{\rm{uns}}$ is finite-dimensional, denote
$$N:=\dim V_{\rm{uns}}.$$
Then there exists a basis of $V_{\rm{uns}}$ such that $\mathcal L|_{V_{\rm{uns}}}$ is in the Jordan normal form 
with $\ell$ Jordan blocks, whose sizes satisfy $$\sum_{i=1}^{\ell} \dim J_i =N.$$
More precisely,
\begin{equation*}
\mathcal L|_{V_{\rm{uns}}}  = 
\begin{bmatrix}
J_1 &  & &  \\
 & J_2 & &  \\
 & & \ddots & \\
 &  & & J_\ell
 \end{bmatrix}  
 \qquad\,\, \mbox{ with} \quad J_i = 
 \begin{bmatrix}
\lambda_i & \frac{\smallc_g}{10} & & \\
 & \lambda_i & \ddots & \\
 & & \ddots & \frac{\smallc_g}{10} \\
 & & & \lambda_i
 \end{bmatrix},
\end{equation*}
where $\lambda_i$, $i=1,2,\cdots, \ell$, are the eigenvalues of $\mathcal L|_{V_{\rm{uns}}}$. In that basis, 
\begin{equation} \label{eq:chisinau}
\Re \big(\bar{w}^\top \cdot \mathcal L|_{V_{\rm{uns}}} \cdot w \big) 
\geq -\frac{3 \smallc_g}{5} \| w \|^2 \qquad \mbox{ for any $w \in \mathbb{C}^N$},
\end{equation}
where we have used $\Re$ to denote the real part and $\bar{w}$ is the complex conjugate of $w$.
Moreover, letting $\aleph(t)$ be the semigroup generated by $\mathcal L$, then 
\begin{equation} \label{eq:prim}
\| \aleph(t) v \|_X \leq e^{-t \frac{\smallc_g}{ 2}} \| v \|_X \qquad\,\mbox{ for any  $v \in V_{\rm{sta}}$}.
\end{equation}
\end{enumerate}
\end{Lemma}

The next lemma shows that, for sufficiently large parameters $m$ and $J$, the unstable subspace constructed above is in fact spanned by smooth functions.

\begin{Lemma}[\cite{shijia}]\label{prop:smooth} Let $\smallc_g \in (0, 1)$. 
Then there exist $\tilde{m}$ depending on $\tilde{\eta}$, and $\bar{J}$ depending on $(C_0,\smallc_g,\tilde{\eta})$ 
and profile $(\overline{\Clss},\overline{U})$ 
such that, for all $m \ge \tilde{m}$ and all $\bar{J}
 \ge J_0$,  the unstable space $V_{\rm{uns}} \subset X$ defined in {\rm Lemma \ref{lemma:abstract_result}} is formed by smooth functions.
 \end{Lemma}

As a direct consequence of Lemmas~\ref{lemma:abstract_result}--\ref{prop:smooth}, we obtain the following properties of $\mathcal L$.

\begin{Lemma}[\cite{shijia}, Proposition 1.8] \label{prop:maxdissmooth} 
Let $\smallc_g \in (0,1)$, and let $(m,J)$ be fixed as in  {\rm Lemmas~\ref{lemma:abstract_result}{\rm--}\ref{prop:smooth}}. 
Then the Hilbert space $X$ can be decomposed as $V_{\text{\rm sta}} \oplus V_{\text{\rm uns}}$, 
where both $V_{\rm{sta}}$ and $ V_{\rm{uns}}$ are invariant subspaces of $\mathcal{L} = (\mathcal{L}_u, \mathcal{L}_\subclss)$, 
and $V_{\rm{uns}}$ is finite-dimensional and spanned by smooth functions $\{(\varphi_{i,u}, \varphi_{i,\subclss})\}^N_{i=1}$. 
Moreover, there exists a metric  $\Upsilon$ of $V_{uns}$ such that the decomposition satisfies{\,\rm:}
\begin{align} \begin{split} \label{eq:decomposition_condition}
\Re\langle \mathcal{L} v,  v \rangle_\Upsilon &\geq -\frac{3}{5}\smallc_g \| v \|_{\Upsilon}^2  \qquad \,\,\, \text{for all}\,\, v \in V_{\rm{uns}}, \\
\left\| e^{t\mathcal{L}} v \right\|_X &\leq e^{-t\frac{\smallc_g}{ 2}} \| v \|_X \qquad\,\,\, \text{for all}\,\, v \in V_{\rm{sta}}.
\end{split} \end{align}
\end{Lemma}

\section{Remarks  on the Choice of the Initial Data}\label{appendix D}
In this section, we give some remarks on the initial data considered in Theorem \ref{Thm1.1}. 
First, for the reformulated system \eqref{selfsimilar eq}, 
we choose the initial data $(\Clss_0, U_0)$ in the self-similar coordinates $(\tau, y)$ with the following form{\rm:}
$$ \Clss_0 = \widehat{X}(\tau_0, y)\overline{\Clss}+ \widetilde{\Clss}_0^* +\sum_{i=1}^N \hat{k}_i\varphi_{i, \subclss},\quad U_0 = \widehat{X}(\tau_0, y)\overline{U}+ \widetilde{U}_0^* +\sum_{i=1}^N \hat{k}_i\varphi_{i, u},$$
where $\{\varphi_{i, \subclss}, \varphi_{i, u}\}_{i=1}^N$ denotes a normalized basis of $V_{\mathrm{uns}}$, consisting of smooth functions compactly 
supported in ball $B(0, 3C_0)$, and the corresponding coefficients ${\hat{k}_i}$ are  fixed in Proposition \ref{prop:ci}. 
Here,  $(\widetilde{\Clss}_0^*, \widetilde{U}_0^*)$ is required to satisfy the conditions: 
\begin{align} \label{smallness-D}
\begin{split} 
\max\big\{\|\widetilde{\Clss}_0^* \|_{L^{\infty}},\  \|\widetilde{U}_0^* \|_{L^{\infty}}\big\}
&\leq  \smallc_1, \\
\| \nabla^K\widetilde{\Clss}_0^*\phi^{\frac K2} \|_{L^2}+\| \nabla^K\widetilde{U}_0^* \phi^{\frac K2}\|_{L^2} &\leq\frac{E}{4}, 
 \\
\widetilde{\Clss}_0^*+\widehat{X}\overline{\Clss}&\geq \frac{\smallc_1}{2}, \\
|\nabla(\widetilde{\Clss}_0^*+\widehat{X} \overline{\Clss})|+|\nabla(\widetilde{U}_0^*+\widehat{X} \overline{U})|&\leq \frac{C}{\langle y \rangle^{\Lambda}},\\
(\chi_2 \widetilde{\Clss}_0^*, \chi_2 \widetilde{U}_0^*) &\in V_{\mathrm{sta}}.\end{split} \end{align}

 On one hand, we consider a Banach space $\mathcal{B}$ that incorporates all the norms appearing in 
\eqref{smallness-D}. 
There exists a sufficiently small ball $B_* \subset \mathcal{B}$ such that every element in this ball satisfies the required smallness assumptions shown in \eqref{smallness-D}. 
On the other hand,  the restriction 
$$(\chi_2 \widetilde{\Clss}_0^*, \,\chi_2 \widetilde{U}_0^*) \in V_{\mathrm{sta}}$$ shown in \eqref{smallness-D} is equivalent to 
\[
P_{\mathrm{uns}} (\widetilde{\Clss}_0^*, \widetilde{U}_0^*) = 0,
\]
which is a finite-dimensional closed restriction. Since projection $P_{\mathrm{uns}}$ maps onto the finite-dimensional subspace 
$V_{\mathrm{uns}}$, there always exists a representative of $(\widetilde{\Clss}_0^*, \widetilde{U}_0^*)$ in the quotient space 
$\mathcal{B}/V_{\mathrm{uns}}$ that satisfies both of the above conditions. 
By the form of initial data, $(\Clss_0, U_0)$ belongs to an equivalence class in the quotient space $\mathcal{B} / V_{\mathrm{uns}}$. Since $V_{\mathrm{uns}}$ is finite-dimensional, 
the manifold of initial data that give rise to finite-time implosion has finite codimension.

Second, it follows from  \eqref{scaling1},  $\tau_0=-\frac{\log T}{\Lambda}$, and $y= xT^{-\frac{1}{\Lambda}}=e^{\tau_0}x$ that 
\begin{equation*}
    \begin{aligned}
        \lss_0(x)&= \frac{1}{\Lambda T^{1-\frac{1}{\Lambda}}}\Big(\mathscr{X}(x)\overline{\Clss}(\frac{x}{T^{\frac{1}{\Lambda}}})+ \widetilde{\Clss}_0^*(\frac{x}{T^{\frac{1}{\Lambda}}})+\sum_{i=1}^N \hat{k}_i\varphi_{i,\subclss}(\frac{x}{T^{\frac{1}{\Lambda}}})\Big),\\
         u_0(x)&= \frac{1}{\Lambda T^{1-\frac{1}{\Lambda}}}\Big(\mathscr{X}(x)\overline{U}(\frac{x}{T^{\frac{1}{\Lambda}}})+ \widetilde{U}_0^*(\frac{x}{T^{\frac{1}{\Lambda}}})+ \sum_{i=1}^N \hat{k}_i\varphi_{i,u}(\frac{x}{T^{\frac{1}{\Lambda}}})\Big).
         \end{aligned}
\end{equation*}
By the construction of $(\Clss_0, U_0)$, we obtain  
$$
\lss_0 \geq \frac{\smallc_1}{2}\frac{1}{\Lambda T^{1-\frac{1}{\Lambda}}}
\,\, \Longrightarrow \,\, \rho_0 
\geq \alpha^{\frac{1}{\alpha}}\Big(\frac{\smallc_1}{2}\frac{1}{\Lambda T^{1-\frac{1}{\Lambda}}}\Big)^{\frac{1}{\alpha}},
$$
and
$$
\|(\rho_0, u_0)\|_{L^\infty}\leq C,\qquad 
\|(\nabla^K\rho_0\phi^{\frac K2}, \,\nabla^K u_0\phi^{\frac K2})\|_{L^2}\leq C(E).
$$
Moreover, $(\rho_0, u_0)$ belongs to an equivalence class in the quotient space $\mathcal{B} / V_{\mathrm{uns}}$.

\bigskip
\bigskip
\noindent{\bf Acknowledgements:}  
 The authors would like to thank Professor Zhifei Zhang for his truly helpful suggestions and remarks.
This research is partially supported by National Key R$\&$D Program of China (No. 2022YFA1007300). The research of Gui-Qiang G. Chen was also supported in part by the UK Engineering and Physical Sciences Research
Council Award EP/L015811/1, EP/V008854, and EP/V051121/1.
The research of Shengguo Zhu was also supported in part by 
the National Natural Science Foundation of China under the Grant  12471212, 
and the Royal Society (UK)-Newton International 
Fellowships NF170015.

\bigskip
\noindent{\bf Conflict of Interest:} The authors declare  that they have no conflict of
interest. 
The authors also  declare that this manuscript has not been previously published, 
and will not be submitted elsewhere before your decision.

\bigskip
\noindent{\bf Data availability:} Data sharing is not applicable to this article as no datasets were generated or analyzed during the current study.

\end{document}